\documentclass{article}%

\usepackage{tikz}
\usepackage{enumitem}
\usepackage{float}
\usepackage{subcaption}
\usetikzlibrary{calc}
\usetikzlibrary{intersections}

\usepackage{graphicx}
\usepackage{amsmath}
\usepackage{amsfonts}
\usepackage{amssymb}%
\providecommand{\U}[1]{\protect\rule{.1in}{.1in}}
\newtheorem{theorem}{Theorem}[section]

\newtheorem{case}{Case}
\newtheorem{claim}[theorem]{Claim}

\newtheorem{condition}[theorem]{Condition}

\newtheorem{corollary}[theorem]{Corollary}

\newtheorem{definition}[theorem]{Definition}
\newtheorem{example}[theorem]{Example}

\newtheorem{lemma}[theorem]{Lemma}
\newtheorem{notation}[theorem]{Notation}

\newtheorem{proposition}[theorem]{Proposition}
\newtheorem{remark}[theorem]{Remark}

\newenvironment{proof}[1][Proof]{\noindent\textbf{#1.} }{\ \rule{0.5em}{0.5em}}

 \newcommand{\vpcn}{\ensuremath{\text{VPC}(n)}}

  \newcommand{\C}{\ensuremath{\mathcal{C}}}
    \newcommand{\B}{\ensuremath{\mathcal{B}}}
    
    \DeclareMathOperator{\stab}{Stab}

    \newcommand{\cS}{\ensuremath{\mathcal{S}}}
    \newcommand{\cF}{\ensuremath{\mathcal{F}}}
    
 \newcommand{\df}[1]{\textbf{#1}}   
 \renewcommand{\|}{\mid}   
 \newcommand{\ccc}{cross connected component}   
 \newcommand{\comm}{Comm}

\begin{document}

\author{
Peter Scott\\Mathematics Department\\University of Michigan\\Ann Arbor, Michigan 48109, USA.\\email: pscott@umich.edu\\
\and Gadde Swarup
\\email: anandaswarupg@gmail.com}

\title{Adapted almost invariant sets and splittings of groups}
\maketitle

\begin{abstract}
We prove relative versions of many earlier results about almost invariant sets
and splittings of groups. In particular, we prove a relative version of the
algebraic torus theorem, and we prove the existence and uniqueness of relative
versions of algebraic regular neighbourhoods and of JSJ decompositions.

\end{abstract}
\date{}
\tableofcontents

\section*{Forward}
This paper was written mostly by Peter Scott before he sadly passed away in September 2023. Section 9 was the last section written by him. The initial drafts with Swarup date back to 2005--2007. After that, Peter Scott spent a considerable amount of time on this paper as he felt  that `adapted almost invariant sets' are a  useful idea. Section 16 and the `Further Remarks' in section 17 were written by Swarup after earlier discussions with Peter Scott. These connect the work here after section 15 with the work of Guiradel and Levitt \cite{Guirardel-Levitt2} and suggest
an alternate approach to later results in section 17. These are more suggestions rather than proofs and is a topic for further study. This paper greatly benefited from discussions with Vincent Guiradel who pointed out various errors in the earlier work of Scott and Swarup \cite{SS2} which helped shape the present paper. The technique  of  `putting things in a circle' of section 12.1 was suggested by him. The cone metric used in section 11 was also suggested by him.

\section{Introduction}

In \cite{Stallings1, Stallings2} Stallings proved that a finitely generated
group $G$ splits over a finite subgroup if and only if it has at least two
ends. In \cite{Swan} Swan generalised Stallings' result to all groups. One
ingredient of his arguments was a relative version of Stallings' result in
which one is given a subgroup $S$ of $G$ and looks for a splitting of $G$ over
a finite subgroup such that $S$ is conjugate into a vertex group of the
splitting. This is the first example of the type of result in which we will be
interested in this paper. In \cite{Swarup2}, Swarup generalised Swan's result
to the situation where $G$ has a finite family $\mathcal{S}=\{S_{1}%
,\ldots,S_{m}\}$ of subgroups and one looks for a splitting of $G$ over a
finite subgroup such that each $S_{i}$ is conjugate into a vertex group of the
splitting. In \cite{Muller}, Muller introduced the term "adapted to
$\mathcal{S}$" to describe such a splitting. In \cite{SSpdnold}, Scott and
Swarup introduced the idea of an almost invariant subset of a group being
adapted to a family of subgroups, and they gave a somewhat special relative
version of the algebraic torus theorem of Dunwoody and Swenson
\cite{D-Swenson}.

We recall that several authors have produced algebraic decompositions of
groups which are analogous to the JSJ decomposition of Haken $3$--manifolds
described in the work of Jaco and Shalen \cite{JS} and of Johannson \cite{JO}.
This field was initiated by Kropholler \cite{K} who studied analogous
decompositions for Poincar\'{e} duality groups of any dimension greater than
$2$. But the current interest in this kind of decomposition started with the
work of Sela \cite{S1} on one-ended torsion-free hyperbolic groups. His
results were generalised by Rips and Sela \cite{RS}, Bowditch \cite{B1, B3},
Dunwoody and Sageev \cite{D-Sageev}, Dunwoody and Swenson \cite{D-Swenson},
and Scott and Swarup \cite{SS2, SSpdnold, SSpdn}. In a series of papers
\cite{Guirardel-Levitt1, Guirardel-Levitt2, Guirardel-Levitt3,
Guirardel-Levitt4}, Guirardel and Levitt generalised such results even further.

However there are very few algebraic analogues of the relative versions of the
topological decomposition which were also described in \cite{JS} and
\cite{JO}. The main discussion of such analogues has been in
\cite{Guirardel-Levitt4}. To describe these relative JSJ decompositions, we
fix some notation. Let $M$ be a Haken $3$--manifold, and let $S$ be a compact
subsurface (not necessarily connected) of $\partial M$ such that both $S$ and
$\partial M-S$ are incompressible in $M$. Let $\Sigma$ denote the closure of
$\partial M-S$. One considers $\pi_{1}$--injective maps of tori and annuli
into $M$. For annuli, one assumes in addition that the boundary is mapped into
$S$. Such a map of a torus or annulus into $M$ is said to be \textit{essential
in }$(M,S)$ if it is not homotopic, as a map of pairs, into $S$ or into
$\Sigma$. Jaco and Shalen, and Johannson, defined what can reasonably be
called the JSJ decomposition of the pair $(M,S)$. One can think of it as a
kind of regular neighbourhood of all essential such maps. In most cases, an
embedded annulus or torus in $M$ which is essential in $(M,S)$ determines a
splitting of $\pi_{1}(M)$ over $\mathbb{Z}$ or $\mathbb{Z}\times\mathbb{Z}$,
as appropriate, which is clearly adapted to the family of subgroups of
$\pi_{1}(M)$ which consists of the fundamental groups of the components of
$\Sigma$. The standard JSJ decomposition is the special case where $S$ equals
$\partial M$ so that $\Sigma$ is empty.

The opposite extreme occurs when $S$ is empty. There are no annuli in
$(M,\varnothing)$, so that the JSJ decomposition of the pair $(M,\varnothing)$
is a kind of regular neighbourhood of all the essential tori in $M$. The
algebraic decomposition of $\pi_{1}(M)$ determined by the frontier of the JSJ
decomposition of $(M,\varnothing)$ has a direct generalisation to Poincar\'{e}
duality pairs, given by Kropholler \cite{K} and by Scott and Swarup
\cite{SSpdnold, SSpdn}. If $G$ is the fundamental group of a Haken
$3$--manifold $M$, then the decomposition of $G$ given by Fujiwara and
Papasoglu \cite{FujiwaraPapasoglu} is closely related to the JSJ decomposition
of $(M,\varnothing)$. These are the only algebraic generalisations of the JSJ
decomposition of $(M,\varnothing)$ which are in the literature, apart from the
discussion in \cite{Guirardel-Levitt4}.

If $S$ is non-empty and not equal to $\partial M$, the only algebraic
generalisations of the JSJ decomposition of $(M,S)$ have been by Guirardel and
Levitt in \cite{Guirardel-Levitt4}. Several different cases have been studied
topologically. The most important is the case where $\Sigma$ consists of a
finite collection of disjoint annuli, no two of which are parallel in
$\partial M$. This case has played an important role in the theory of
hyperbolic structures on $3$--manifolds with accidental parabolics, where the
components of $\Sigma$ correspond to the rank $1$ cusps associated to the
accidental parabolic elements of $\pi_{1}(M)$.

In this paper we develop the background theory needed to give full algebraic
analogues of the relative JSJ decompositions of a Haken $3$--manifold. We
obtain relative versions of the results of Scott and Swarup in \cite{SS2} on
algebraic regular neighbourhoods and JSJ decompositions of groups. In
addition, we are also able to generalise some of the results in \cite{SS2} in
the absolute case. An important part of this work is a full relative version
of the algebraic torus theorem of Dunwoody and Swenson \cite{D-Swenson}. A new
feature in this paper is that some of the groups we consider need not be
finitely generated, though they do need to be finitely generated relative to
some family of subgroups. As most of the published theory in this area treats
only the finitely generated case, we give a careful discussion of this. This
seems worth doing because a graph of groups decomposition of a finitely
generated group may have vertices whose associated groups are not finitely generated.

An important role in this paper will be played by the results of Scott and
Swarup in \cite{SS2}. Since that paper was published, the authors have
discovered some errors. They discuss and correct these in Appendix
\ref{correctionstomonster}. We should also note that in section
\ref{algregnbhds}, we give a slightly modified approach to developing the theory.

In section \ref{preliminaries} we discuss the basic theory of ends of groups
which need not be finitely generated. In section \ref{enclosing} we discuss
the definition of enclosing. This concept was defined in \cite{SS2}, but we
need to slightly reformulate the definition of enclosing in order to handle
the situation where groups need not be finitely generated. In section
\ref{cubings}, we recall Sageev's construction of a cubing in
\cite{Sageev-cubings}, and observe that his construction works even when the
groups involved are not finitely generated. In section
\ref{adaptedalmostinvariantsets}, we give the definitions and basic properties
of adapted almost invariant subsets of a group. We discuss how this theory
works in the case when the ambient group is not finitely generated. In section
\ref{groupsystems}, we introduce the idea of a group system of finite type and
of a relative form of the Cayley graph of a group. In section
\ref{extensionsandcrossing}, we discuss crossing of almost invariant sets and
show that the symmetry of crossing extends to some situations where the
ambient group is not finitely generated but satisfies some relative finite
generation condition. In section \ref{goodposition}, we discuss good position
and show that this theory too extends to situations where the ambient group
satisfies some relative finite generation condition. We do the same for very
good position.

In section \ref{algregnbhds}, we give a slightly modified approach to
developing the theory of algebraic regular neighbourhoods, which helps in the
extension of the ideas of \cite{SS2} to the setting of group systems. We also
show how the theory of very good position, which was introduced by Niblo,
Sageev, Scott and Swarup in \cite{NibloSageevScottSwarup}, can be used to
substantially simplify some of the work in \cite{SS2}. In section
\ref{coends}, we discuss the theory of coends and give the key definition
which will be used throughout the following sections. In section
\ref{crossingofa.i.setsoverVPCgroups}, we prove some important lemmas on the
symmetry of strong crossing in certain situations. These are the analogue for
adapted almost invariant sets of lemmas in chapters 7, 13 and 14 of
\cite{SS2}. In section \ref{CCC'swithstrongcrossing}, we consider a
cross-connected family of adapted almost invariant subsets in which all
crossing is strong, and show that this yields an action of a convergence group
on the circle. This was first done by Dunwoody and Swenson \cite{D-Swenson} in
the absolute case, but we need to discuss the entire argument in order to
clarify the situation in the adapted case. In section
\ref{accessibilityresults}, we discuss some accessibility results needed in
order to prove the existence of our JSJ decompositions. In section
\ref{CCC'swithweakcrossing}, we consider cross-connected families of adapted
almost invariant subsets in which all crossing is weak. In section
\ref{section:relativetorustheorem}, we give a proof of a relative version of
the algebraic torus theorem, which is far more general than the special case
proved in \cite{SSpdnold}.

In section \ref{multi-lengthdecompositions}, we discuss adapted versions of the
main decomposition results of Scott and Swarup in \cite{SS2} including a JSJ
decomposition result. As in \cite{SS2}, we state these results under the
assumption that the ambient group is finitely presented. The restriction to
finitely presented groups is natural because we use certain accessibility
results. However it is now standard that the accessibility results that we use
extend to almost finitely presented groups. Thus all the results of this paper
which are stated under the assumption that $G$ is finitely presented are valid
when $G$ is almost finitely presented. In particular, the results are valid
for Poincar\'{e} duality groups as these are known to be almost finitely
presented. These results for almost finitely presented groups will be used in
a subsequent paper \cite{SSpdn} on the analogues of JSJ-decompositions for
Poincar\'{e} duality pairs. In \cite{Guirardel-Levitt2}, Guirardel and Levitt
described a new approach to the one-stage decompositions of \cite{SS2}. In
\cite{Guirardel-Levitt4}, they showed how to extend their ideas to the adapted
case, and we use their approach. This gives more information on $V_{1}%
$--vertices, namely that they have finitely generated stabilizers.

In section \ref{multi-lengthdecompositions}, we discuss multi-length JSJ
decompositions analogous to those in chapters 13 and 14 of \cite{SS2}, and we
also discuss generalisations. It would be interesting to extend the approach
of Guirardel and Levitt in \cite{Guirardel-Levitt4} to these multi-length
decompositions. It is possible that their approach may suggest further
applications. So far, one interesting application we have of these
multi-length decompositions is to the canonical decomposition of Poincar\'{e}
Duality Groups first described in \cite{SSpdnold}, which was updated in
\cite{SSpdn}. In fact some of the ideas of this paper were first developed for
the applications in those papers.

Finally in section 17
we compare our new algebraic
decompositions with topological decompositions of $3$--manifolds. As in the
absolute case, there are very close similarities between the two theories, but
there are a few extra differences in the relative case.

\section{Preliminaries\label{preliminaries}}

As stated in the introduction, we will be considering groups which need not be
finitely generated. Thus we need to give a careful introduction to the theory
of ends of such groups.

Throughout this paper, groups are not assumed to be finitely generated unless
that is specifically stated.

The following small technical result will be used in this paper at several
points. It is due to Neumann \cite{Neumann}, but we reproduce his proof for completeness.

\begin{lemma}
\label{Neumannlemma}Let $G$ be a group with subgroups $S,H,H_{1},\ldots,H_{n}$.

\begin{enumerate}
\item If $S$ is contained in the union of finitely many cosets $Hg$ of $H$,
then $S\cap H$ has finite index in $S$.

\item If $S$ is contained in the union of finitely many cosets $H_{i}g_{i}$,
then there is $i$ such that $S\cap H_{i}$ has finite index in $S$.
\end{enumerate}
\end{lemma}

\begin{proof}
1) We have an inclusion $S\subset Hg_{1}\cup\ldots Hg_{n}$. Without loss of
generality, we can assume that $S\cap Hg_{j}$ is non-empty for each $j$. Now,
by replacing the $g_{j}$'s if necessary, we can arrange that each $g_{j}$ lies
in $S$. Thus we have $S\cap Hg_{j}=Sg_{j}\cap Hg_{j}=(S\cap H)g_{j}$, so that
the $g_{j}$'s are a collection of coset representatives for $S\cap H$ in $S$.
It follows that $S\cap H$ has finite index in $S$, as required.

2) This proof is by induction on the number $k$ of distinct subgroups of $G$
among the $H_{i}$'s. Part 1) of this lemma handles the case when $k=1$. As in
part 1) we can arrange that each $g_{i}$ lies in $S$, so that $S\cap
H_{i}g_{i}=(S\cap H_{i})g_{i}$.

Now suppose that $k\geq2$, that there are $k$ distinct subgroups of $G$ among
the $H_{i}$'s, and that the lemma holds whenever there are at most $k-1$
distinct subgroups. By renumbering the $H_{i}$'s we can suppose that there are
$k-1$ distinct subgroups of $G$ among $H_{1},\ldots,H_{m}$ and that
$H_{m+1},\ldots,H_{n}$ are all equal. We denote this common group by $H$. We
have an inclusion $S\subset H_{1}g_{1}\cup\ldots\cup H_{m}g_{m}\cup
Hg_{m+1}\cup\ldots\cup Hg_{n}$. If $S\subset Hg_{m+1}\cup\ldots\cup Hg_{n}$,
the result follows by part 1). Otherwise there is $g\in S-(Hg_{m+1}\cup
\ldots\cup Hg_{n})$. It follows that $(S\cap H)g$ is also disjoint from
$Hg_{m+1}\cup\ldots\cup Hg_{n}$, so that $(S\cap H)g$ must be contained in
$H_{1}g_{1}\cup\ldots\cup H_{m}g_{m}$. But then $(S\cap H)g_{i}=(S\cap
H)g(g^{-1}g_{i})$ must be contained in $H_{1}g_{1}(g^{-1}g_{i})\cup\ldots\cup
H_{m}g_{m}(g^{-1}g_{i})$, for $m+1\leq i\leq n$, so that $S$ is contained in
some finite union of cosets of $H_{1},\ldots,H_{m}$. As there are only $k-1$
distinct subgroups of $G$ among $H_{1},\ldots,H_{m}$, our induction hypothesis
implies the result.
\end{proof}

\subsection{Ends\label{ends}}

We start by recalling the definition of the number of ends of a group. Let $G$
be a group and let $E$ be a set on which $G$ acts on the right. Let $PE$
denote the power set of $E$, i.e. the set of all subsets of $E$. Under Boolean
addition (\textquotedblleft symmetric difference\textquotedblright) this is an
additive group of exponent $2$. Write $FE$ for the additive subgroup of finite
subsets. We refer to two sets $A$ and $B$ whose symmetric difference lies in
$FE$ as \textit{almost equal}, and write $A\overset{a}{=}B$. This amounts to
equality in the quotient group $PE/FE$. Now define
\[
QE=\{A\subset E:\forall g\in G,\;A\overset{a}{=}Ag\}.
\]
The action of $G$ on $PE$ by right translation preserves the subgroups
$QE\ $and $FE$, and $QE/FE$ is the subgroup of elements of $PE/FE$ fixed under
the induced action. Elements of $QE$ are said to be \textit{almost invariant}.
If we take $E$ to be $G$ with the action of $G$ being right multiplication,
then the number of ends of $G$ is
\[
e(G)=\dim_{\mathbb{Z}_{2}}\;(QG/FG).
\]

If $H$ is a subgroup of $G$, and we take $E$ to be the coset space
$H\backslash G$ of all cosets $Hg$, still with the action of $G$ being right
multiplication, then the number of ends of the pair $(G,H)$ is
\[
e(G,H)=\dim_{\mathbb{Z}_{2}}\;\left(  \frac{Q(H\backslash G)}{F(H\backslash
G)}\right)  .
\]

Next we recall the basic definitions for almost invariant sets.

\begin{definition}
Let $G$ be a group and $H$ a subgroup.

A subset of $G$ is $H$\textsl{--finite} if it is contained in a finite union
of cosets $Hg$.

Two subsets of $G$ are $H$\textsl{--almost equal} if their symmetric
difference is $H$--finite.

A subset $X$ of $G$ is $H$\textsl{--almost invariant} if $HX=X$ and, for all
$g\in G$, the sets $X$ and $Xg$ are $H$--almost equal.

For any subset $X$ of $G$, we denote the complement $G-X$ by $X^{\ast}$.

A $H$--almost invariant subset $X$ of $G$ is \textsl{nontrivial} if both $X$
and $X^{\ast}$ are $H$--infinite.
\end{definition}

\begin{remark}
If $H$ is trivial, then a $H$--almost invariant subset $X$ of $G$ is an
element of the set $QG$ defined above.

In general, if $X$ is a $H$--almost invariant subset of $G$, then the quotient
set $H\backslash X$ is an element of the set $Q(H\backslash G)$ defined above,
such that $H\backslash X\ $is an almost invariant subset of $H\backslash G$
under the action of $G$ by right translation.
\end{remark}

\subsection{Equivalence}

As we discuss later in this subsection, if a group $G$ splits over a subgroup
$H$, then there is a natural way to associate to this splitting a nontrivial
$H$--almost invariant subset of $G$ described by using the $G$--tree given by
the splitting. But this set is not uniquely determined. This leads naturally
to the idea of equivalence of almost invariant sets. Before defining
equivalence, we will need the following technical result.

\begin{lemma}
\label{equivalenta.i.setshavecommensurablestabiliser}Let $G$ be a group and
let\ $H$ and $K$ be subgroups. Let $X$ be a nontrivial $H$--almost invariant
subset of $G$, and let $Y$ be a $K$--almost invariant subset of $G$ such that
$X$ and $Y$ are $H$--almost equal. Then $H$ and $K$ are commensurable
subgroups of $G$.
\end{lemma}

\begin{proof}
As $X$ and $Y$ are $H$--almost equal, it follows that, for all $g\in G$, the
sets $Xg$ and $Yg$ are also $H$--almost equal. As $X$ is $H$--almost
invariant, we know that $X$ and $Xg$ are $H$--almost equal. It follows that,
for all $g\in G$, the sets $Y\ $and $Yg$ are $H$--almost equal.

As $Y$ is $K$--almost invariant, we know that $KY=Y$. Thus each of $Y\ $and
$Y^{\ast}$ is a union of cosets $Kg$ of $K$ in $G$. As $X$ is assumed to be
nontrivial, each of $Y\ $and $Y^{\ast}$ must be non-empty. Let $Ku$ be a coset
of $K$ contained in $Y$, and let $Kv$ be a coset of $K$ contained in $Y^{\ast
}$. Thus $Y(u^{-1}v)-Y$ contains $Kv$. Recall from the first paragraph of this
proof that, for all $g\in G$, the sets $Y\ $and $Yg$ are $H$--almost equal.
Taking $g=u^{-1}v$ shows that $Kv$ must be $H$--finite. Hence $K$ itself is
$H$--finite, so that $K\subset Hg_{1}\cup\ldots\cup Hg_{n}$, for some
$g_{i}\in G$. Now part 1) of Lemma \ref{Neumannlemma} shows that $K\cap H$ has
finite index in $K$, so that $K$ is commensurable with a subgroup of $H$.

Next we consider $X$. As $X\ $is $H$--almost invariant, we know that it is a
union of cosets $Hg$ of $H$ in $G$, and as $X$ is assumed to be nontrivial, we
know that $X$ and $X^{\ast}$ must each be the union of infinitely many such
cosets. As $X$ and $Y$ are $H$--almost equal, it follows that $Y$ contains
some coset $Hw$ of $H$ in $G$. Similarly $Y^{\ast}$ must contain some coset
$Hz$. Hence $Y(w^{-1}z)-Y$ contains $Hz$, and as above this implies that $H$
is commensurable with a subgroup of $K$.

Now it follows that $H$ and $K$ must be commensurable subgroups of $G$, as required.
\end{proof}

If two subgroups $H$ and $K$ of $G\ $are commensurable, it is trivial that a
subset of $G$ is $H$--finite if and only if it is $K$--finite. Thus the
following result is immediate.

\begin{corollary}
\label{HaiiffKai}Let $G$ be a group and let\ $H$ and $K$ be subgroups. Let $X$
be a nontrivial $H$--almost invariant subset of $G$, and let $Y$ be a
nontrivial $K$--almost invariant subset of $G$. Then $X$ and $Y$ are
$H$--almost equal if and only if they are $K$--almost equal.
\end{corollary}

Now we can define the idea of equivalence of almost invariant sets. Previous
discussions have always used the coboundary, but the coboundary is only a
useful idea in the setting of finitely generated groups.

\begin{definition}
Let $G$ be a group and let\ $H$ and $K$ be subgroups. Let $X$ be a nontrivial
$H$--almost invariant subset of $G$, and let $Y$ be a nontrivial $K$--almost
invariant subset of $G$. We will say that $X$ and $Y$ are \textsl{equivalent}
if they are $H$--almost equal, and will write $X\thicksim Y$.
\end{definition}

\begin{remark}
Corollary \ref{HaiiffKai} implies that this is an equivalence relation, which
justifies our terminology.
\end{remark}

We also note the following useful consequence of Lemma
\ref{equivalenta.i.setshavecommensurablestabiliser}.

\begin{corollary}
\label{H-aisetequalsKaisetimpliescommensurable}Let $G$ be a group and let\ $H$
and $K$ be subgroups. Let $X$ be a nontrivial $H$--almost invariant subset of
$G$, and suppose that $X$ is also $K$--almost invariant. Then $H$ and $K$ are
commensurable subgroups of $G$.
\end{corollary}

We next recall some ideas from chapter 2 of \cite{SS2}. Consider a group $G$
and a minimal $G$--tree $T$ on which $G$ acts without inversions. After fixing
a vertex $w$ of $T$, we define a $G$--equivariant map $\varphi:G\rightarrow
V(T)$ by the formula $\varphi(g)=gw$. An oriented edge $s$ of $T$ determines a
natural partition of $V(T)$ into two sets, namely the vertices of the two
subtrees obtained by removing the interior of $s$ from $T$. Let $Y_{s}$ denote
the collection of all the vertices of the subtree which contains the terminal
vertex of $s$, and let $Y_{s}^{\ast}$ denote the complementary collection of
vertices. We define the sets $Z_{s}=\varphi^{-1}(Y_{s})$ and $Z_{s}^{\ast
}=\varphi^{-1}(Y_{s}^{\ast})$. Lemma 2.10 of \cite{SS2} shows that, if $S$
denotes the stabiliser of $s$, then $Z_{s}$ is a $S$--almost invariant subset
of $G$, and its equivalence class is independent of the choice of basepoint
$w$. By definition, we know that $\varphi(Z_{s})\subset Y_{s}$. But if $X$ is
an almost invariant subset of $G$ which is equivalent to $Z_{s}$, then
$\varphi(X)$ need not be contained in $Y_{s}$. However we show below that
$\varphi(X)$ must be contained in a bounded neighbourhood of $Y_{s}$, where we
put the usual metric on $T$ in which each edge has length $1$. We will need
this fact later in this paper.

\begin{lemma}
\label{boundednbhdofY_s}Let $G$ be a group and let $T$ be a minimal $G$--tree
on which $G$ acts without inversions. Let $w$ be a vertex of $T$, and let $s$
be an oriented edge of $T$ with stabiliser $S$. If $X$ is a $S$-almost
invariant subset of $G$ which is equivalent to $Z_{s}$, then there is an
integer $M$ such that $\varphi(X)$ is contained in the $M$--neighbourhood of
$Y_{s}$.
\end{lemma}

\begin{proof}
Recall that we have the $G$--equivariant map $\varphi:G\rightarrow V(T)$ given
by the formula $\varphi(g)=gw$. As $X$ and $Z_{s}$ are equivalent, their
symmetric difference is $S$--finite. Thus $X$ is contained in the union of
$Z_{s}$ with a finite union of cosets $Sg$ of $S$. As $S$ fixes the edge $s$
of $T$, the image $\varphi(Sg)$ of the coset $Sg$ consists of vertices of $T$
all at the same distance from $s$ as the vertex $\varphi(g)$. As
$\varphi(Z_{s})$ is contained in $Y_{s}$, it follows that there is an integer
$M$ such that $\varphi(X)$ is contained in the $M$--neighbourhood of $Y_{s}$,
as required.
\end{proof}

\section{Enclosing\label{enclosing}}

Let $A$ be a nontrivial $H$--almost invariant subset of a group $G$ and let
$v$ be a vertex of a $G$--tree $T$. In \cite{SS2}, Scott and Swarup defined
what it means for $A$ to be enclosed by $v$. However we will need to slightly
modify this definition in order to allow us to handle the theory of adapted
almost invariant sets in section \ref{adaptedalmostinvariantsets}. Having
discussed this, we will then show that in the context of the work done in
\cite{SS2}, the new definition is equivalent to the old. More precisely, in
Lemma \ref{twodefnsofenclosing}, we show that the two definitions of enclosing
agree if $H$ is finitely generated, or if $A$ is associated to a splitting of
$G$ over $H$.

Before discussing the details, here is an explanation of the intuition behind
them for the special case of surface groups. We start with the basic
connection between topology and almost invariant sets. Consider a closed
orientable surface $M$ with base point $w$. Pick a finite generating set for
$G=\pi_{1}(M,w)$ and for each chosen generator pick a loop in $M\ $which
represents that generator. For simplicity we will assume that these loops are
simple and intersect each other only in $w$. Let $\widetilde{M}$ denote the
universal cover of $M$, and let $\widetilde{w}$ be a point of the pre-image of
$w$. We will take $\widetilde{w}$ to be the basepoint of $\widetilde{M}$. The
pre-image of all these loops in $M$ is a copy of the Cayley graph $\Gamma$ of
$G$ with respect to the chosen generating set and with basepoint at
$\widetilde{w}$. Now let $\lambda$ be an essential simple loop in $M$ which
does not meet $w$, and let $l$ be a component of the pre-image of $\lambda$ in
$\widetilde{M}$. Thus $l$ is an embedded line with infinite cyclic stabilizer
$H$. In particular $l$ cuts $\widetilde{M}$ into two pieces which we denote by
$Z\ $and $Z^{\ast}$. We associate to $Z$ the subset $X$ of $G$ given by
$X=\{g\in G:g\widetilde{w}\in Z\}$. As $HZ=Z$, we have $HX=X$. As the quotient
$H\backslash l$ is compact, it follows that the coboundary of $X$ in $\Gamma$
is $H$--finite. Thus $X$ is a $H$--almost invariant subset of $G$. As we
assumed the loop $\lambda$ does not meet $w$, the subset of $G$ associated to
$Z^{\ast}$ is equal to $X^{\ast}$, the complement of $X$ in $G$.

Now consider an incompressible subsurface $N$ of $M$. This means that the
boundary of $N$ is a finite family of disjoint $\pi_{1}$--injective circles
embedded in $M$. We will assume $\partial N$ does not meet $w$. Note that $N$
is not assumed to be connected. The algebraic idea of enclosing corresponds to
the topological idea of a connected subset of $M$ lying in the subsurface $N$.
More precisely, let $\alpha$ be a nontrivial element of $G=\pi_{1}(M)$ such
that there is a loop representing $\alpha$ which is contained in $N$. We will
show how to reformulate this idea in terms of almost invariant sets and
$G$--trees. Let $M_{\alpha}$ denote the based cover of $M$ with $\pi
_{1}(M_{\alpha})$ equal to the cyclic group $H$ generated by $\alpha$. As $M$
is closed and orientable, $M_{\alpha}$ is an annulus. In particular,
$e(G,H)=2$, so that, up to equivalence and complementation, there is a unique
nontrivial $H$--almost invariant subset $A$ of $G$. Pick a simple loop in
$M_{\alpha}$ which represents $\alpha$, and does not pass through any point
above $w$, and let $l$ denote the pre-image of this loop in $\widetilde{M}$.
Thus $l$ is a $H$--invariant line in $\widetilde{M}$, and the two sides of $l$
in $\widetilde{M}$ determine $H$--almost invariant subsets of $G$ equivalent
to $A$ and $A^{\ast}$.

The pre-image in $\widetilde{M}$ of $\partial N$, consists of disjoint lines
which divide $\widetilde{M}$ into copies of the universal cover of the
components of $M-\partial N$. The dual graph in $\widetilde{M}$ to this family
of lines is a $G$--tree $T$. As above each of these lines determines a
complementary pair of almost invariant subsets of $G$. Recall that we are
assuming that there is a loop representing $\alpha$ which is contained in $N$,
and so does not meet $\partial N$. If $l$ lies above such a loop, then $l$ is
disjoint from the pre-image of $\partial N$. In particular, if $m$ denotes a
component of this pre-image, and $Y\ $and $Y^{\ast}$ denote the associated
almost invariant subsets of $G$, then one of these sets must be contained in
$A$ or in $A^{\ast}$. In general, any loop representing $\alpha$ is homotopic
into $N$. If $l$ is the pre-image of such a loop, it may meet the pre-image of
$\partial N$, in which case one of $Y$ or $Y^{\ast}$ must be "almost"
contained in $A$ or in $A^{\ast}$. Note that $Y$ and $Y^{\ast}$ are associated
to the splitting of $G$ determined by the corresponding edge of $T$. This is
the intuition behind our definition of enclosing.

Choose a basepoint $w$ for $T$. For each edge $s$ of $T$, this determines the
subsets $Z_{s}$ and $Z_{s}^{\ast}$ of $G$ as described immediately before
Lemma \ref{boundednbhdofY_s}. In Definition 4.2 of \cite{SS2}, Scott and
Swarup said that the vertex $v$ encloses $A$, if for all edges $s$ of $T$
which are incident to $v$ and directed towards $v$, we have $A\cap Z_{s}%
^{\ast}$ or $A^{\ast}\cap Z_{s}^{\ast}$ is $H$--finite. They showed that this
definition did not depend on the choice of basepoint $w$. They also showed
that if $v$ encloses $A$, then it encloses any $H^{\prime}$--almost invariant
subset $A^{\prime}$ of $G$ which is equivalent to $A$.

As the definition of the sets $Z_{s}$ and $Z_{s}^{\ast}$ depends on the choice
of basepoint $w$, we will not use this notation in the rest of this section.
Instead we will refer to $\varphi^{-1}(Y_{s})$ and $\varphi^{-1}(Y_{s}^{\ast
})$ which makes the dependence on the base point somewhat clearer.

Now we begin discussing our new definition of enclosing.

\begin{definition}
\label{defnofstrictenclosing}Let $A$ be a nontrivial $H$--almost invariant
subset of a group $G$ and let $v$ and $w$ be vertices of a $G$--tree $T$. Let
$\varphi:G\rightarrow V(T)$ denote the function determined by choosing $w$ as
the basepoint of $T$, so that $\varphi(g)=gw$. We say that the vertex $v$
\textsl{strictly encloses} $A$ \textsl{with basepoint} $w$ if, for all edges
$s$ of $T$ which are incident to $v$ and directed towards $v$, we have
$\varphi^{-1}(Y_{s}^{\ast})\subset A$ or $\varphi^{-1}(Y_{s}^{\ast})\subset
A^{\ast}$.
\end{definition}

\begin{remark}
It follows immediately from this definition that if $t$ is any edge of $T$
which is directed towards $v$, and if $v$ strictly encloses $A$ with basepoint
$w$, then we must have $\varphi^{-1}(Y_{t}^{\ast})\subset A$ or $\varphi
^{-1}(Y_{t}^{\ast})\subset A^{\ast}$.
\end{remark}

This definition obviously depends crucially on the choice of basepoint $w$. We
will discuss what happens if the basepoint is changed. But first we need to
introduce some notation. If the vertex $v$ of $T$ strictly encloses $A$ with
basepoint $w$, we let $\Sigma_{v}(A)$ denote the collection of all vertices
$u$ of $T-\{v\}$ such that the path from $u$ to $v$ ends in an edge $s$ such
that when $s$ is directed towards $v$, we have $\varphi^{-1}(Y_{s}^{\ast
})\subset A$. Note that if $u$ lies in $\Sigma_{v}(A)$, and $t$ is any edge on
the path from $u$ to $v$ which is directed towards $v$, then we will also have
$\varphi^{-1}(Y_{t}^{\ast})\subset A$. We define $\Sigma_{v}(A^{\ast})$
similarly. Clearly $\Sigma_{v}(A)$ and $\Sigma_{v}(A^{\ast})$ partition the
vertices of $T-\{v\}$. As $H$ fixes $A$ and $v$, it also follows that
$\Sigma_{v}(A)$ and $\Sigma_{v}(A^{\ast})$ are each $H$--invariant, under the
left action of $H$.

\begin{lemma}
\label{strictenclosingdoesnotdependonbasepoint}Let $T$ be a minimal $G$--tree
and let $X$ be a nontrivial $H$--almost invariant subset of $G$. Let $v$, $x$
and $w$ be vertices of $T$. Suppose that the vertex $v$ strictly encloses $X$
with basepoint $x$. Then there is a $H$--almost invariant subset $W$ of $G$
such that $W$ is equivalent to $X$ and $v$ strictly encloses $W$ with
basepoint $w$.
\end{lemma}

\begin{proof}
Let $\varphi:G\rightarrow V(T)$ denote the map given by $\varphi(g)=gx$, and
define $\Sigma_{v}(X)$ as above using the map $\varphi$. As $v$ strictly
encloses $X$ with basepoint $x$, it follows that an element $g$ of $G$ such
that $gx\neq v$ lies in $X$ if and only if $gx$ lies in $\Sigma_{v}(X)$. Thus
$X=\varphi^{-1}(\Sigma_{v}(X))\cup(X\cap\varphi^{-1}(v))$.

Now let $\psi:G\rightarrow V(T)$ denote the map given by $\psi(g)=gw$. We
define a subset $W$ of $G$ by replacing $\varphi$ by $\psi$ in the above
formula for $X$. Thus $W=\psi^{-1}(\Sigma_{v}(X))\cup(X\cap\psi^{-1}(v))$. As
$H$ fixes $v$ and preserves $\Sigma_{v}(X)$, it follows that $HW=W$.

We claim that $W$ is $H$--almost equal to $X$, so that $W$ is $H$--almost
invariant and is equivalent to $X$. Now the definition of $W$ shows that for
all edges $s$ of $T$ which are incident to $v$ and directed towards $v$, we
have $\psi^{-1}(Y_{s}^{\ast})\subset W$ or $\psi^{-1}(Y_{s}^{\ast})\subset
W^{\ast}$. Thus $v$ strictly encloses $W$ with basepoint $w$, which will
complete the proof of the lemma.

Now we come to the proof of our claim. We need to consider various cases
depending on which, if any, of $x$, $w$ and $v$ lie in the same $G$--orbit as
each other.

Case 1: There is $k$ in $G$ such that $w=kx$, and $v$ is not in the $G$--orbit
of $x$ and $w$.

In this case $\varphi^{-1}(v)$ and $\psi^{-1}(v)$ are both empty. Thus we have
$X=\varphi^{-1}(\Sigma_{v}(X))$ and $W=\psi^{-1}(\Sigma_{v}(X))$. Let $L$
denote the stabiliser of $x$, so that $kLk^{-1}$ is the stabiliser of $w$.
Thus if $z=gw=gkx$, then $\varphi^{-1}(z)=gkL$ and $\psi^{-1}(z)=g(kLk^{-1})$.
Hence $\psi^{-1}(z)=\varphi^{-1}(z)k^{-1}$. As this holds for all $z$ in the
orbit of $x$ and $w$, it follows that $W=Xk^{-1}$, so that $W$ is $H$--almost
equal to $X$, as required.

Case 2: There are $a$ and $k$ in $G$ such that $w=kx$ and $v=aw=akx$.

Recall that $X=\varphi^{-1}(\Sigma_{v}(X))\cup(X\cap\varphi^{-1}(v))$ and
$W=\psi^{-1}(\Sigma_{v}(X))\cup(X\cap\psi^{-1}(v))$. For ease of reading we
let $P$, $Q$, $R$ and $S$ denote $\varphi^{-1}(\Sigma_{v}(X))$, $X\cap
\varphi^{-1}(v)$, $\psi^{-1}(\Sigma_{v}(X))$ and $X\cap\psi^{-1}(v)$
respectively, so that $X=P\cup Q$ and $W=R\cup S$. The argument in Case 1
applies to show that $R=Pk^{-1}$. Also $Q=X\cap akL$ and $S=X\cap a(kLk^{-1}%
)$. Now $X$ is $H$--almost equal to $Xk^{-1}$, so that $S$ is $H$--almost
equal to $Xk^{-1}\cap a(kLk^{-1})=(X\cap akL)k^{-1}=Qk^{-1}$. It follows that
$R\cup S$ is $H$--almost equal to $(P\cup Q)k^{-1}$, so that $W$ is
$H$--almost equal to $Xk^{-1}$. Thus $W$ is $H$--almost equal to $X$, as required.

Case 3: $x$ and $w$ lie in distinct $G$--orbits.

Let $z$ be a translate of $x$ such that $w$ lies between $x$ and $z$.
Considering $z$ as a new choice of basepoint, we define $\theta:G\rightarrow
V(T)$ by $\theta(g)=gz$, and define $Z=\theta^{-1}(\Sigma_{v}(X))\cup
(X\cap\theta^{-1}(v))$. As $z$ is a translate of $x$, we can apply Case 1 or
Case 2 to $X\ $and $Z$ to conclude that $Z$ is $H$--almost equal to $X$.

Next we will show that $X-W$ is $H$--finite by showing that $X-W\subset X-Z$.
To prove this, we note that if $g\in X-W$ then $gw$ cannot equal $v$, as then
$g$ would lie in $X\cap\psi^{-1}(v)$ which is contained in $W$. Thus if $g\in
X-W$, then $gx$ lies in $\Sigma_{v}(X)\cup\{v\}$ and $gw$ lies in $\Sigma
_{v}(X^{\ast})$. As $w$ lies between $x$ and $z$, it follows that $gz$ also
lies in $\Sigma_{v}(X^{\ast})$, so that $g\in X-Z$, proving that $X-W\subset
X-Z$, as required.

Similarly, by reversing the roles of $X\ $and $W$, we can show that $W-X$ is
$H$--finite by showing that $W-X\subset Z-X$. It follows that $X\ $and $W$ are
$H$--almost equal, which completes the proof of the lemma.
\end{proof}

Having proved this result we now make the following definition of enclosing.

\begin{definition}
\label{defnofenclosing}Let $A$ be a nontrivial $H$--almost invariant subset of
a group $G$ and let $v$ a vertex of a $G$--tree $T$. We say that the vertex
$v$ \textsl{encloses} $A$, if $A$ is equivalent to a $H^{\prime}$--almost
invariant subset $A^{\prime}$ of $G$ such that $A^{\prime}$ is strictly
enclosed by $v$ with some basepoint $w$.
\end{definition}

\begin{remark}
The result of Lemma \ref{strictenclosingdoesnotdependonbasepoint} shows that
if $v$ encloses $A$ according to this new definition, and we fix a basepoint
$u$ for $T$, then $A$ is equivalent to a $H^{\prime\prime}$--almost invariant
subset $A^{\prime\prime}$ of $G$ such that $A^{\prime\prime}$ is strictly
enclosed by $v$ with basepoint $u$.
\end{remark}

For future reference, we set out the connection between this definition and
the old one.

\begin{lemma}
\label{twodefnsofenclosing}Let $A$ be a nontrivial $H$--almost invariant
subset of a group $G$ and let $v$ a vertex of a $G$--tree $T$.

\begin{enumerate}
\item If $v$ encloses $A$ in the new sense (Definition \ref{defnofenclosing}),
then $v$ encloses $A$ in the old sense (Definition 4.2 of \cite{SS2}).

\item If $v$ encloses $A$ in the old sense, and $H$ is finitely generated,
then $v$ encloses $A$ in the new sense.

\item If $v$ encloses $A$ in the old sense, and $A$ is associated to a
splitting of $G$ over $H$, then $v$ encloses $A$ in the new sense.
\end{enumerate}
\end{lemma}

\begin{proof}
1) If $v$ encloses $A$ in the new sense (Definition \ref{defnofenclosing}),
then $A$ is equivalent to a $H^{\prime}$--almost invariant subset $A^{\prime}$
of $G$ such that $A^{\prime}$ is strictly enclosed by $v$ with some basepoint
$w$. It is trivial that $v$ encloses $A^{\prime}$ in the old sense (Definition
4.2 of \cite{SS2}), so it follows that $v$ encloses $A$ itself in the old sense.

2) Now suppose that $v$ encloses $A$ in the old sense, and that $H\ $is
finitely generated. We would like to apply Lemma 4.14 of \cite{SS2}.
Unfortunately, as pointed out by Scott and Swarup in the appendix, the proof
given in \cite{SS2} is incorrect. But that proof does show that if $H$ is
finitely generated and we fix the basepoint $w$ for $T$ to be equal to $v$,
then $A$ is equivalent to a $H^{\prime}$--almost invariant subset $A^{\prime}$
of $G$ such that, for each $s$, we have $Z_{s}^{\ast}$ contained in
$A^{\prime}$ or its complement. Thus $A^{\prime}$ is strictly enclosed by $v$
with base point $v$, so that $A$ is enclosed by $v$ in the new sense.

3) Finally suppose that $v$ encloses $A$ in the old sense, and that $A$ is
associated to a splitting of $G$ over $H$. In this case, the proof of
Proposition 5.7 of \cite{SS2} yields the same conclusion as in part 2). As
discussed in the appendix, Proposition 5.7 of \cite{SS2} is incorrect as
stated, and so is the first part of the proof. But here we are assuming that
$A$ is enclosed by $v$, so we only need to use the second, correct, part of
the proof. Thus again $A$ is enclosed by $v$ in the new sense.
\end{proof}

\section{Cubings\label{cubings}}

In this section, we recall some basic facts about the cubings constructed by
Sageev in \cite{Sageev-cubings}. We will need the following definition.

\begin{definition}
\label{defnofultrafilter} Let $E$ be a partially ordered set, equipped with an
involution $A\rightarrow A^{\ast}$ such that $A\neq A^{\ast}$, and if $A\leq
B$ then $B^{\ast}\leq A^{\ast}$. An \textsl{ultrafilter} $V$ on $E$ is a
subset of $E$ satisfying

\begin{enumerate}
\item For every $A\in E$, we have $A\in V$ or $A^{\ast}\in V$ but not both.

\item If $A\in V$ and $A\leq B$, then $B\in V$.
\end{enumerate}
\end{definition}

Consider a group $G$ with subgroups $H_{1},\ldots,H_{n}$. For $i=1,\ldots,n$,
let $X_{i}$ be a nontrivial $H_{i}$--almost invariant subset of $G$, and let
$E$ denote the set

$\{gX_{i},gX_{i}^{\ast}:g\in G,1\leq i\leq n\}$. In \cite{Sageev-cubings},
Sageev gave a construction of a cubing from the set $E$ equipped with the
partial order given by inclusion. The involution $A\rightarrow A^{\ast}$ on
$E$ is complementation in $G$. Note that Sageev did not assume that the
$H_{i}$'s were finitely generated, but he did assume that $G$ was finitely
generated. However this is not necessary. The finite generation of $G$ allowed
him to discuss the number of ends of $G$ in terms of the number of ends of its
Cayley graph, but the Cayley graph does not appear after the initial
discussion and is not needed for any of the results in \cite{Sageev-cubings}.
Thus his construction works with no restrictions on $G$ and the $H_{i}$'s.

Here is a summary of Sageev's construction. Let $\mathcal{K}^{(0)}$ denote the
collection of all ultrafilters on $E=\{gX_{i},gX_{i}^{\ast}:g\in G,1\leq i\leq
n\}$. Construct $\mathcal{K}^{(1)}$ by attaching an edge to two vertices
$V,V^{\prime}\in\mathcal{K}^{(0)}$ if and only if they differ by replacing a
single element by its complement, i.e. there exists $A\in V$ such that
$V^{\prime}=(V-\{A\})\cup\{A^{\ast}\}$. Note that the fact that $V$ and
$V^{\prime}$ are both ultrafilters implies that $A$ must be a minimal element
of $V$. Also if $A$ is a minimal element of $V$, then the set $V^{\prime
}=(V-\{A\})\cup\{A^{\ast}\}$ is an ultrafilter on $E$. Now attach
$2$--dimensional cubes to $\mathcal{K}^{(1)}$ to form $\mathcal{K}^{(2)}$, and
inductively attach $n$--cubes to $\mathcal{K}^{(n-1)}$ to form $\mathcal{K}%
^{(n)}$. All such cubes are attached by an isomorphism of their boundaries
and, for each $n\geq2$, one $n$--cube is attached to $\mathcal{K}^{(n-1)}$ for
each occurrence of the boundary of an $n$--cube appearing in $\mathcal{K}%
^{(n-1)}$. The complex $\mathcal{K}$ constructed in this way need not be
connected, but one special component can be picked out in the following way.

For each element $g$ of $G$, define the ultrafilter $V_{g}=\{A\in E:g\in A\}$.
These special vertices of $\mathcal{K}$ are called \textit{basic}. Sageev
shows that two basic vertices $V$ and $V^{\prime}$ of $\mathcal{K}$ differ on
only finitely many complementary pairs of elements of $E$, so that there exist
elements $A_{1},\ldots,A_{m}$ of $E$ which lie in $V$ such that $V^{\prime}$
can be obtained from $V$ by replacing each $A_{i}$ by $A_{i}^{\ast}$. By
re-ordering the $A_{i}$'s if needed, we can arrange that $A_{1}$ is a minimal
element of $V$. It follows that $V_{1}=(V-\{A_{1}\})\cup\{A_{1}^{\ast}\}$ is
also an ultrafilter on $E$, and so is joined to $V$ by an edge of
$\mathcal{K}$. By repeating this argument, we will find an edge path in
$\mathcal{K}$ of length $m$ which joins $V$ and $V^{\prime}$. It follows that
the basic vertices of $\mathcal{K}$ all lie in a single component $C$. As the
collection of all basic vertices is preserved by the action of $G$ on
$\mathcal{K}$, it follows that this action preserves $C$. Finally, Sageev
shows in \cite{Sageev-cubings} that $C$ is simply connected and $CAT(0)$ and
hence is a cubing.

If $\sigma$ is a $k$--dimensional cube in a cubing $C$, viewed as a standard
unit cube in $\mathbb{R}^{k}$ and $\hat{\sigma}$ denotes the barycentre of
$\sigma$, then a \textit{dual cube} in $\sigma$ is the intersection with
$\sigma$ of a $(k-1)$--dimensional plane running through $\hat{\sigma}$ and
parallel to one of the $(k-1)$--dimensional faces of $\sigma$. Given a cubing,
one may consider the equivalence relation on edges generated by the relation
which declares two edges to be equivalent if they are opposite sides of a
square in $C$. Now given an equivalence class of edges, the
\textit{hyperplane} associated to this equivalence class is the collection of
dual cubes whose vertices lie on edges in the equivalence class. In
\cite{Sageev-cubings}, Sageev shows that hyperplanes are totally geodesic
subspaces. Moreover, he shows that hyperplanes do not self-intersect (i.e. a
hyperplane meets a cube in a single dual cube) and that a hyperplane separates
a cubing into precisely two components, which we call the \textit{half-spaces}
associated to the hyperplane.

An important aspect of Sageev's construction is that one can recover the
elements of $E$ from the action of $G$ on the cubing $C$. Recall that an edge
$f$ of $C$ joins two vertices $V$ and $V^{\prime}$ if and only if there exists
$A\in V$ such that $V^{\prime}=(V-\{A\})\cup\{A^{\ast}\}$. If $f$ is oriented
towards $V^{\prime}$, we will say that $f$ \textit{exits} $A$. We let
$\mathcal{H}_{A}$ denote the hyperplane associated to the equivalence class of
$f$. This equivalence class consists of all those edges of $C$ which, when
suitably oriented, exit $A$. Now let $X$ denote an $H$--almost invariant
subset of $G$ which is an element of $E$. Let $k$ denote an element of $X$,
and let $l$ denote an element of $X^{\ast}$. Thus $X$ lies in the basic vertex
$V_{k}=\{A\in E:k\in A\}$, and $X^{\ast}$ lies in the basic vertex $V_{l}$.
Now any path joining $V_{k}$ to $V_{l}$ must contain an edge which exits $X$,
so we can define the hyperplane $\mathcal{H}_{X}$ as above. Let $\mathcal{H}%
_{X}^{+}$ denote the half-space determined by $\mathcal{H}_{X}$ and by the
condition that every vertex of $\mathcal{H}_{X}^{+}$ is an ultrafilter on $E$
which contains $X$. Clearly the basic vertex $V_{g}$ lies in $\mathcal{H}%
_{X}^{+}$ if and only if $g$ lies in $X$. As $gV_{e}=V_{g}$, it follows that
$X$ equals the subset $\{g\in G:gV_{e}\in\mathcal{H}_{X}^{+}\}$ of $G$. Thus
we have recovered $X$ from the action of $G$ on the cubing $C$.

For future reference we note the following points. By definition an edge of
$C$ exits a single element $A$ of $E$. Thus if $X$ and $Y$ are distinct (and
non-complementary) elements of $E$, the hyperplanes $\mathcal{H}_{X}$ and
$\mathcal{H}_{Y}$ have no dual cubes in common. Further if $X$ and $Y$ are
elements of $E$ such that $X$ is properly contained in $Y$, then the
hyperplanes $\mathcal{H}_{X}$ and $\mathcal{H}_{Y}$ must be disjoint. In
particular, the half-space $\mathcal{H}_{X}^{+}$ is properly contained in the
half-space $\mathcal{H}_{Y}^{+}$. To see that these hyperplanes must be
disjoint, note that otherwise there is a square in $C$ whose four vertices lie
in the four regions into which $\mathcal{H}_{X}$ and $\mathcal{H}_{Y}$
together separate $C$. In particular, there is a vertex $V$ in $\mathcal{H}%
_{X}^{+}\cap\mathcal{H}_{Y}^{+}$, and edges $f$ and $f^{\prime}$ with one end
at $V$ which exit $X$ and $Y$ respectively. But this implies that both $X$ and
$Y$ are minimal elements of the ultrafilter $V$, which contradicts the fact
that $X$ is properly contained in $Y$.

\section{Adapted almost invariant sets\label{adaptedalmostinvariantsets}}

In this section we discuss the idea of an almost invariant subset of a group
being adapted to a family of subgroups. This is the beginning of a relative
theory of almost invariant subsets of a group. The idea of a splitting being
adapted to some family of subgroups was introduced by M\"{u}ller in
\cite{Muller}, but for general almost invariant sets the idea of being adapted
seems to be new. Here is M\"{u}ller's definition.

\begin{definition}
\label{defnofadaptedsplitting}Let $K$ be a group with a splitting $\sigma$
over a subgroup $H$, and let $\mathcal{S}=\{S_{i}\}$ be a family of subgroups
of $K$. The splitting $\sigma$ of $K$ is \textsl{adapted to} $\mathcal{S}$, or
is $\mathcal{S}$--\textsl{adapted}, if each $S_{i}$ is conjugate into a vertex
group of $\sigma$.
\end{definition}

A natural example, which motivated this definition, occurs when $K$ is the
fundamental group of a compact orientable surface $M$ with boundary, and the
splitting $\sigma$ of $K$ is given by a simple closed curve in the interior of
$M$. Such a splitting is automatically adapted to the family $\mathcal{S}$ of
subgroups of $K$ which are the images of the fundamental groups of the
boundary components of $M$. Now consider an essential closed curve on $M$
which need not be simple. It still determines an almost invariant subset $X$
of $K$ in a natural way, and our definition of adaptedness for almost
invariant subsets of a group is chosen to ensure that $X$ is adapted to the
same family $\mathcal{S}$ of subgroups of $K$. Note that we first define the
term "strictly adapted" and then use that to define "adapted" very much as in
the definition of enclosing, Definition \ref{defnofenclosing}.

Here is our definition.

\begin{definition}
\label{defnofadapted} Let $G$ be a group and let\ $H$ and $S$ be subgroups.
Let $\mathcal{S}=\{S_{i}\}_{i\in I}$ be a family of subgroups of $G$, with
repetitions allowed.

A $H$--almost invariant subset $X$ of $G$ is \textsl{strictly adapted to the
subgroup} $S$ if, for all $g\in G$, the coset $gS$ is contained in $X$ or in
$X^{\ast}$.

A $H$--almost invariant subset $X$ of $G$ is \textsl{adapted to} $S$, or is
$S$--\textsl{adapted}, if it is equivalent to a $H^{\prime}$--almost invariant
subset $X^{\prime}$ of $G$ such that $X^{\prime}$ is strictly adapted to $S$.

A $H$--almost invariant subset $X$ of $G$ is \textsl{adapted to the family}
$\mathcal{S}$, or is $\mathcal{S}$--\textsl{adapted}, if it is adapted to each
$S_{i}$.
\end{definition}

\begin{remark}
If $X$ is $\mathcal{S}$--adapted, then so is any almost invariant subset of
$G$ which is equivalent to $X$. Clearly $G$ and the empty set are strictly
adapted to any subgroup of $G$. Thus trivial $H$--almost invariant subsets of
$G$ are adapted to any family of subgroups of $G$.

If $X$ is adapted to the family $\mathcal{S}$, and we replace each $S_{i}$ by
some conjugate, then $X$ is also adapted to the new family. For if $X$ is
strictly adapted to a subgroup $S$, and $k\in G$, then $Xk$ is $H$--almost
invariant and equivalent to $X$, and is strictly adapted to $k^{-1}Sk$.

Note that if $X$ is adapted to the family $\mathcal{S}$, then, for each $i$,
there is a $K_{i}$--almost invariant subset $X_{i}$ of $G$ which is equivalent
to $X$ and is strictly adapted to $S_{i}$, but the $X_{i}$'s may all be
different. In general, it is difficult for an almost invariant set to be
strictly adapted to more than one subgroup of $G$.

In most applications of these ideas, the family $\mathcal{S}$ will be finite.
\end{remark}

Finally we set out one more consequence of the above definition as a lemma for
future reference.

\begin{lemma}
\label{S-adaptedimpliesgSintersectXisH-finite}Let $G$ be a group and let\ $H$
and $S$ be subgroups. Let $X$ be a $H$--almost invariant subset of $G$ which
is adapted to $S$. Then, for each $g\in G$, one of $gS\cap X$ and $gS\cap
X^{\ast}$ is $H$--finite.
\end{lemma}

\begin{proof}
The above definition implies that $X$ is $H$--almost equal to a $H^{\prime}%
$--almost invariant subset $X^{\prime}$ of $G$ such that $X^{\prime}$ is
strictly adapted to $S$. Thus, for each $g\in G$, one of $gS\cap X^{\prime}$
and $gS\cap(X^{\prime})^{\ast}$ is empty. The result follows.
\end{proof}

Note that if $S$ is the trivial subgroup of $G$, then any almost invariant
subset of $G$ is automatically strictly adapted to $S$. It is not hard to
generalise this to finite subgroups $S$ of $G$. Our next result implies this
but is more general.

\begin{lemma}
\label{adaptedtoSimpliesadaptedtoS'commS}Let $G$ be a group with subgroups
$H$, $S$ and $S^{\prime}$, such that $S\ $and $S^{\prime}\ $are commensurable.
Let $X$ be a $H$--almost invariant subset of $G$. Then $X$ is $S$--adapted if
and only if it is $S^{\prime}$--adapted.
\end{lemma}

\begin{proof}
It suffices to consider the case where $S$ is a subgroup of $S^{\prime}$ of
finite index, and $X$ is strictly adapted to $S$. We will construct a
$H$--almost invariant subset $Y$ of $K$ which is equivalent to $X$ such that
$Y$ is strictly adapted to $S^{\prime}$.

As $X$ is strictly adapted to $S$, we have $kS\subset X$ or $kS\subset
X^{\ast}$, for all $k$ in $G$. We can write $S^{\prime}$ as the disjoint union
$S\cup Sg_{1}\cup Sg_{2}\cup\ldots\cup Sg_{n}$, for some $n$. If $kS\subset
X$, then $kSg_{i}\subset Xg_{i}$. As $Xg_{i}$ is $H$--almost equal to $X$, it
follows that $(\cup Xg_{i})-X$ is a finite union of cosets $Hk_{j}$. Thus if
$kS\subset X$, the intersection $kS^{\prime}\cap X^{\ast}$ is $H$--finite.
Further, if $kS^{\prime}\not \subset X$, there must be $g_{i}$ and $k_{j}$
such that $kg_{i}\in Hk_{j}$, so that $k$ lies in one of the finitely many
cosets $Hk_{j}g_{i}^{-1}$. Hence the set of all $k$ such that $kS\subset X$
and $kS^{\prime}\not \subset X$ is contained in a finite union of double
cosets $HgS$.

Now we will construct the required set $Y$ by modifying $X$. For each $k$ such
that $kS\subset X$ and $kS^{\prime}\not \subset X$, we will add the
$H$--finite set $kS^{\prime}\cap X^{\ast}$ to $X$. Clearly the intersection
$kS^{\prime}\cap X^{\ast}$ is unchanged if we replace $k$ by any element of
the double coset $HkS$. Thus we will only add finitely many $H$--finite sets
to $X$, so in total we will only add a $H$--finite set to $X$. Similarly if
$kS\subset X^{\ast}$, then $kS^{\prime}\cap X$ is $H$--finite. We will
subtract this $H$--finite set from $X$, for each $k$ such that $kS\subset
X^{\ast}$ and $kS^{\prime}\not \subset X^{\ast}$, and this only removes a
$H$--finite set from $X$. Let $Y$ denote the result of these changes. Thus $Y$
is $H$--almost equal to $X$ and $HY=Y$, so that $Y$ is $H$--almost invariant
and equivalent to $X$. We need to observe that as distinct double cosets $HgS$
are disjoint, the sets we remove from $X$ do not overlap those we add. Thus
$Y$ has the property that $kS^{\prime}\subset Y$ or $kS^{\prime}\subset
Y^{\ast}$, for all $k$ in $K$, so that $Y$ is strictly adapted to $S^{\prime}%
$, as required.
\end{proof}

Next recall our discussion at the start of this section. Let $G$ be the
fundamental group of a compact orientable surface $M$ with boundary, and
consider an essential closed curve on $M$ which need not be simple. It
determines an almost invariant subset $X$ of $G$ in a natural way, and our
definition of adaptedness ensures that $X$ is adapted to the family
$\mathcal{S}$ of subgroups of $G$ which are the images of the fundamental
groups of the boundary components of $M$. Suppose we enlarge the surface $M$
to $\overline{M}$ by attaching surfaces to the boundary components. Trivially
any essential closed curve on $M$ gives rise to one on $\overline{M}$. From
the algebraic point of view, this means that the almost invariant subset $X$
of $G$ gives rise to an almost invariant subset $\overline{X}$ of
$\overline{G}=\pi_{1}(\overline{M})$. Clearly $\overline{X}\cap G=X$. Let
$\Gamma$ denote the dual graph to $\partial M$ in $\overline{M}$, and consider
the associated graph of groups structure on $\overline{G}$ with underlying
graph $\Gamma$. Let $V$ denote the vertex of $\Gamma$ which corresponds to
$M$. Then $\overline{X}$ is enclosed by $V$ (see Definition
\ref{defnofenclosing}).

Now let $\overline{G}$ be any group with subgroups $H$ and $G$, with $H$
contained in $G$, and let $\overline{X}$ be a $H$--almost invariant subset of
$\overline{G}$. Then $\overline{X}\cap G$ is a $H$--almost invariant subset
$X$ of $G$. In the preceding paragraph we gave an example of a converse
construction in the special case when $G=\pi_{1}(M)$ and $\overline{G}=\pi
_{1}(\overline{M})$. From the algebraic point of view, this example has three
special features. These are that $G$ is the group associated to a vertex $V$
of a graph of groups decomposition $\Gamma$ of $\overline{G}$, that
$\overline{X}$ is enclosed by $V$, and $X$ is adapted to the family of
subgroups of $G$ associated to the edges of $\Gamma$ which are incident to
$V$. In Lemma \ref{extensionsexist} we will show that these conditions are the
key to being able to make a reverse construction.

The following lemma clarifies the connections between these conditions.

\begin{lemma}
\label{XbarenclosedbyvimpliesthatXisadapted}Let $\overline{G}$ be a group with
a minimal graph of groups decomposition $\Gamma$. Let $T$ denote the universal
covering $\overline{G}$--tree of $\Gamma$, so that $\overline{G}$ acts on $T$
without inversions, and $T$ is a minimal $\overline{G}$--tree. Let $v$ denote
a vertex of $T$, and let $\mathcal{S}$ denote the family of stabilizers of
edges of $T$ incident to $v$. Suppose there is a subgroup $H$ of $\overline
{G}$ and a $H$--almost invariant subset $\overline{X}$ of $\overline{G}$ which
is enclosed by $v$. Let $X$ denote the intersection $\overline{X}\cap G$. Then
$X$ is a $H$--almost invariant subset of $G$ which is adapted to the family
$\mathcal{S}$.
\end{lemma}

\begin{remark}
Note that it is possible that $\overline{X}$ is a nontrivial $H$--almost
invariant subset of $\overline{G}$, but that $X$ is a trivial $H$--almost
invariant subset of $G$. For example, this will occur if $\overline{X}$ is
associated to an edge of $T$ incident to $v$.
\end{remark}

\begin{proof}
As discussed at the end of section \ref{preliminaries}, after fixing a base
vertex $w$ of $T$, we can define a $\overline{G}$--equivariant map
$\varphi:\overline{G}\rightarrow V(T)$ by the formula $\varphi(g)=gw$. We can
now associate the almost invariant subset $Z_{s}$ of $\overline{G}$ to an
oriented edge $s$ of $T$, and if $S$ denotes the stabiliser of $s$, then
$Z_{s}$ is $S$--almost invariant.

As $\overline{X}$ is enclosed by $v$, it follows that $H$ is a subgroup of
$G$. Let $e$ be an edge of $T$ incident to $v$, and oriented towards $v$. Let
$S$ denote the stabiliser of $e$, let $u$ denote the other vertex of $e$, and
choose $u$ to be the basepoint of $T$. Recall from Definition
\ref{defnofenclosing} and Lemma \ref{strictenclosingdoesnotdependonbasepoint}
that as $\overline{X}$ is enclosed by $v$ then $\overline{X}$ is equivalent to
some $H^{\prime}$--almost invariant subset $P$ which is strictly enclosed by
$v$ with base point $u$. This means that for every edge $s$ of $T$ which is
incident to $v$ and oriented towards $v$, we have $Z_{s}^{\ast}\subset P$ or
$Z_{s}^{\ast}\subset P^{\ast}$. As the stabiliser $S$ of $e$ fixes $u$, it
follows that $S\subset\varphi^{-1}(u)\subset\varphi^{-1}(Y_{e}^{\ast}%
)=Z_{e}^{\ast}$, so that $S$ is contained in $P$ or in $P^{\ast}$. As $S$ also
fixes $v$, it is a subgroup of $G$, so that $S$ is contained in $P\cap G$ or
$P^{\ast}\cap G$. Let $Q$ and $Q^{\ast}$ denote $P\cap G$ and $P^{\ast}\cap G$
respectively. Now consider an element $g$ of $G$, so that $ge$ is also
incident to $v$. Then $gS\subset\varphi^{-1}(gu)\subset\varphi^{-1}%
(Y_{ge}^{\ast})=Z_{ge}^{\ast}$, so that $gS$ is also contained in $P$ or
$P^{\ast}$. As $gS$ is contained in $G$, it follows that $gS$ must be
contained in $Q$ or $Q^{\ast}$. Note that as $\overline{X}$ and $P$ are
equivalent almost invariant subsets of $\overline{G}$, it follows that $X$ and
$Q$ are equivalent almost invariant subsets of $G$. Thus we conclude that
$X=\overline{X}\cap G$ is equivalent to an almost invariant subset $Q$ of $G$
such that, for all $g\in G$, the coset $gS$ is contained in $Q$ or in
$Q^{\ast}$. It follows from Definition \ref{defnofadapted} that $X$ is adapted
to $S$. As the above discussion applies to the stabiliser of any edge of $T$
incident to $v$, we conclude that $X$ is adapted to the family of all such
stabilizers, as required.
\end{proof}

Now we formulate this in terms of the graph of groups $\Gamma$. Let $V$ denote
the image of $v$ in $\Gamma$, so that $V$ has associated group $G$. Then $X$
is adapted to the family $\mathcal{S}$ consisting of those subgroups of
$\overline{G}$ associated to the edges of $\Gamma$ which are incident to $V$.

In the next section, in Lemma \ref{extensionsexist}, we will describe a
natural converse construction, in which we start with a $H$--almost invariant
subset $X$ of $G$ which is adapted to $\mathcal{S}$, and produce a $H$--almost
invariant subset $\overline{X}$ of $\overline{G}$ which is enclosed by $V$,
such that $\overline{X}\cap G$ is equivalent to $X$. At that point we will
restrict the cardinality of $G$ in order to allow the use of graphical arguments.

For the rest of this section we develop the idea of the number of adapted ends
of a group or pair of groups, which will play a crucial role in later
sections. This number can be defined very much as in the case of non-adapted
ends (see section \ref{preliminaries}).

Let $G$ be a group acting on itself on the right. Let $PG$ denote the power
set of $G$. Under Boolean addition (\textquotedblleft symmetric
difference\textquotedblright) this is an additive group of exponent $2$. Write
$FG$ for the additive subgroup of finite subsets. Let $\mathcal{S}$ denote
some family of subgroups of $G$, and let $QG(\mathcal{S})$ denote the subset
of $PG$ consisting of all almost invariant subsets of $G$ which are adapted to
$\mathcal{S}$. It is easy to see that $QG(\mathcal{S})$ is closed under the
sum operation in $PG$, so that $QG(\mathcal{S})/FG$ is a subspace of $PG/FG$.
Thus we can define the number of $\mathcal{S}$--adapted ends of $G$ to be%

\[
e(G:\mathcal{S})=\dim_{\mathbb{Z}_{2}}\;(QG(\mathcal{S})/FG).
\]

Clearly $e(G:\mathcal{S})\leq e(G)$. We recall that $e(G)$ can only take the
values $0$, $1$, $2$ or $\infty$. Essentially the same proof as in
\cite{Scott-Wall:Topological} shows that the same holds for the invariant
$e(G:\mathcal{S})$, so long as $G$ is finitely generated. It is easy to give
examples where the inequality is strict. For example if $G$ is infinite and
one of the groups in the family $\mathcal{S}$ has finite index in $G$, then
$e(G:\mathcal{S})=1$, irrespective of the value of $e(G)$.

If $H$ is a subgroup of $G$, we consider the coset space $H\backslash G$ of
all cosets $Hg$, still with the action of $G$ being right multiplication. Let
$Q(H\backslash G)(\mathcal{S})$ denote the subset of $P(H\backslash G)$
consisting of all almost invariant subsets of $H\backslash G$ whose pre-image
in $G$ is $\mathcal{S}$--adapted. Again it is easy to see that $Q(H\backslash
G)(\mathcal{S})$ is a subspace of $P(H\backslash G)$. Thus we can define the
number of $\mathcal{S}$--adapted ends of the pair $(G,H)$ to be%

\[
e(G,H:\mathcal{S})=\dim_{\mathbb{Z}_{2}}\;\left(  \frac{Q(H\backslash
G)(\mathcal{S})}{F(H\backslash G)}\right)  .
\]

Clearly $e(G,H:\mathcal{S})\leq e(G,H)$. It is again easy to give examples
where this inequality is strict. For example if $H$ has infinite index in $G$
and one of the groups in the family $\mathcal{S}$ has finite index in $G$,
then $e(G,H:\mathcal{S})=1$, irrespective of the value of $e(G,H)$.

As a splitting of $K$ over $H$ has a naturally associated $H$--almost
invariant subset of $K$, we now have two different ideas of what it means for
a splitting of $K$ to be adapted to a family $\mathcal{S}$ of subgroups. The
following lemma connects these ideas.

\begin{lemma}
\label{splittingisadaptediffa.i.setis}Let $K$ be a group, let $\mathcal{S}%
=\{S_{i}\}$ be a family of subgroups of $K$, and let $\sigma$ be a splitting
of $K$ over a subgroup $H$. Then the splitting $\sigma$ is $\mathcal{S}%
$--adapted if and only if any $H$--almost invariant subset of $K$ associated
to $\sigma$ is $\mathcal{S}$--adapted.
\end{lemma}

\begin{proof}
Let $T$ be the $K$--tree corresponding to the splitting $\sigma$, and let $s$
be an edge of $T$ with stabiliser $H$.

Suppose that the splitting $\sigma$ is $\mathcal{S}$--adapted, so that each
$S_{i}$ has a conjugate $L_{i}$ which fixes a vertex of $s$. If $L_{i}$ fixes
the vertex $v$ of $s$, we choose $v$ as the basepoint of $T$, so that
$L_{i}\subset\varphi^{-1}(v)$. If we orient $s$ towards $v$, then the
$H$--almost invariant subset $Z_{s}$ of $K$ is associated to $\sigma$, and it
clearly contains $L_{i}$. Further for each $k\in K$, we have $kL_{i}%
\subset\varphi^{-1}(kv)$ which must be contained in $Z_{s}$ or $Z_{s}^{\ast}$.
Thus $Z_{s}$ is strictly adapted to $S_{i}$, and hence any $H$--almost
invariant subset of $K$ which is associated to $\sigma$ is adapted to $S_{i}$.
Similar arguments for each $i$ show that any $H$--almost invariant subset of
$K$ which is associated to $\sigma$ must be $\mathcal{S}$--adapted. This
completes the proof of one half of the lemma.

For the second half, we suppose that any $H$--almost invariant subset of $K$
which is associated to $\sigma$ is $\mathcal{S}$--adapted. We need to show
that each $S_{i}$ fixes some vertex of $T$.

Fix a basepoint $w$ of $T$, so that we have the $K$--equivariant map
$\varphi:K\rightarrow V(T)$ given by the formula $\varphi(g)=gw$. Next fix an
edge $s$ of $T$ with stabiliser $H$ such that the associated $H$--almost
invariant subset $Z_{s}$ of $G$ is associated to the splitting $\sigma$. Our
hypothesis implies that $Z_{s}$ is adapted to the family $\mathcal{S}$, so
that, for each $i$, the set $Z_{s}$ is equivalent to an almost invariant
subset $Z_{i}$ of $G$ which is strictly adapted to $S_{i}$. Now Lemma
\ref{boundednbhdofY_s} implies that there is an integer $M$ (depending on $i$)
such that $\varphi(Z_{i})$ is contained in the $M$--neighbourhood of $Y_{s}$,
and $\varphi(Z_{i}^{\ast})$ is contained in the $M$--neighbourhood of
$Y_{s}^{\ast}$. As $Z_{i}$ is strictly adapted to $S_{i}$, we know that, for
each $k$ in $K$, we have $kS_{i}$ contained in $Z_{i}$ or $Z_{i}^{\ast}$. It
follows that, for each $k$ in $K$, the image $\varphi(kS_{i})$ is contained in
the $M$--neighbourhood of $Y_{s}$ or of $Y_{s}^{\ast}$. After replacing $k$ by
$k^{-1}$, this is equivalent to the condition that, for each $k$ in $K$, the
image $\varphi(S_{i})$ is contained in the $M$--neighbourhood of $kY_{s}%
=Y_{ks}$ or of $kY_{s}^{\ast}=Y_{ks}^{\ast}$. As every edge of $T$ is a
translate of $s$, it follows that for every edge $e$ of $T$, the image
$\varphi(S_{i})$ is contained in the $M$--neighbourhood of $Y_{e}$ or of
$Y_{e}^{\ast}$.

First suppose that $S_{i}$ is finitely generated. If $S_{i}$ does not fix a
vertex of $T$, then some element $k$ of $S_{i}$ also does not fix a vertex and
so has an axis $l$ in $T$. Pick an edge $e$ of $l$, and consider the vertices
$\varphi(k^{n})=k^{n}w$, for all integers $n$. This collection of vertices
contains vertices of $T$ which are arbitrarily far from $e$ on each side of
$e$. This contradicts the fact that $\varphi(S_{i})$ is contained in the
$M$--neighbourhood of $Y_{e}$ or of $Y_{e}^{\ast}$, thus proving that $S_{i}$
must fix some vertex of $T$.

Now consider the general case when $S_{i}$ need not be finitely generated, and
suppose that $S_{i}$ does not fix a vertex of $T$. The preceding paragraph
shows that each finitely generated subgroup of $S_{i}$ fixes some vertex of
$T$. Express $S_{i}$ as an ascending union of finitely generated subgroups
$S_{i}^{n}$, $n\geq1$, and let $T_{n}$ denote the fixed subtree of $S_{i}^{n}%
$. We have $T_{n}\supset T_{n+1}$, for $n\geq1$. As $S_{i}$ itself does not
fix a vertex, the intersection of all the $T_{n}$'s must be empty. Let $d_{n}$
denote the distance in $T$ of $w$ from $T_{n}$. Then $d_{n}\rightarrow\infty$
as $n\rightarrow\infty$. By passing to a subsequence we can assume that
$d_{n}$ is never zero. Consider the path from $w$ to $T_{n}$ and let $e_{n}$
denote the edge of this path which meets $T_{n}$. As $S_{i}^{n}$ does not fix
$e_{n}$, there is an element $k_{n}$ of $S_{i}^{n}$ which does not fix $e_{n}%
$. It follows that $k_{n}w$ lies distance $d_{n}$ from $e_{n}$ on the $T_{n}%
$--side of $e_{n}$. Of course the vertex $w$ lies at distance $d_{n}-1$ from
$e_{n}$ on the other side of $e_{n}$. If we choose $n$ large enough to ensure
that $d_{n}>M+1$, this contradicts the fact that, for every edge $e$ of $T$,
the image $\varphi(S_{i})$ is contained in the $M$--neighbourhood of $Y_{e}$ or
of $Y_{e}^{\ast}$. This contradiction shows that $S_{i}$ must fix some vertex
of $T$. As this holds for each $i$, it follows that $\sigma$ is $\mathcal{S}%
$--adapted as required, which completes the proof of part 2).
\end{proof}

There is also a natural generalisation of Definition
\ref{defnofadaptedsplitting} to graphs of groups.

\begin{definition}
\label{defnofadaptedgraphofgroups}Let $K$ be a group with a graph of groups
structure $\Gamma$, and let $\mathcal{S}=\{S_{i}\}$ be a family of subgroups
of $K$. Then $\Gamma$ is \textsl{adapted to} $\mathcal{S}$, or is
$\mathcal{S}$--\textsl{adapted}, if each $S_{i}$ is conjugate into a vertex
group of $\Gamma$.
\end{definition}

The following result connects this new definition with Definition
\ref{defnofadaptedsplitting}.

\begin{lemma}
\label{adaptedgraphofgroups}Let $K$ be a group with a family $\mathcal{S}%
=\{S_{i}\}$ of subgroups. Let $\Gamma$ be a finite minimal graph of groups
structure for $K$. Then $\Gamma$ is $\mathcal{S}$--adapted if and only if each
edge splitting of $\Gamma$ is $\mathcal{S}$--adapted.
\end{lemma}

\begin{proof}
Recall that if $e$ is an edge of a graph of groups structure $\Gamma$ for $K$,
then the splitting $\sigma_{e}$ of $K$ associated to $e$ is obtained by
collapsing to a point each component of the complement of the interior of $e$
in $\Gamma$. If $e$ separates $\Gamma$, the result is an interval and
$\sigma_{e}$ is an amalgamated free product. If $e$ fails to separate $\Gamma
$, the result is a loop and $\sigma_{e}$ is a HNN extension.

Now suppose that $\Gamma$ is $\mathcal{S}$--adapted. Thus each $S_{i}$ is
conjugate into a vertex group of $\Gamma$. It follows immediately that, for
each edge $e$ of $\Gamma$, each $S_{i}$ is conjugate into a vertex group of
$\sigma_{e}$, so that each splitting $\sigma_{e}$ is $\mathcal{S}$--adapted.

Conversely suppose that each edge splitting of $\Gamma$ is $\mathcal{S}%
$--adapted. We proceed by induction on the number $m$ of edges of $\Gamma$. If
$m=1$, the result is trivial. Now let $e_{1},\ldots,e_{m}$ denote the edges of
$\Gamma$, and consider the universal covering $K$--tree $T$ of $\Gamma$. Pick
an edge $e$ of $\Gamma$ and let $f$ denote an edge of $T$ in the pre-image of
$e$. Collapsing the edge $e$ yields a graph of groups decomposition
$\Gamma^{\prime}$ of $K$ with $m-1$ edges. The universal covering $K$--tree
$T^{\prime}$ of $\Gamma^{\prime}$ is obtained from $T$ by collapsing to a
point the edge $f$ and each of its translates. As the edge splittings of
$\Gamma^{\prime}$ are $\mathcal{S}$--adapted, it follows by induction that
each $S_{i}$ fixes a vertex $v_{i}^{\prime}$ of $T^{\prime}$. The pre-image in
$T$ of $v_{i}^{\prime}$ is a subtree, which must be preserved by $S_{i}$. This
subtree is a component of the pre-image in $T$ of $e$. Now we will apply this
argument twice with $e$ chosen to be $e_{1}$ and then to be $e_{2}$. We see
that $S_{i}$ preserves a subtree $T_{1}$ of $T$ which is a component of the
pre-image in $T$ of $e_{1}$, and it also preserves a subtree $T_{2}$ of $T$
which is a component of the pre-image in $T$ of $e_{2}$. As the intersection
$T_{1}\cap T_{2}$ contains no edges, it must be empty or a single vertex. If
it is not empty, then $T_{1}\cap T_{2}$ is a vertex of $T$ fixed by $S_{i}$.
If $T_{1}\cap T_{2}$ is empty, then $S_{i}$ must fix every vertex on the
geodesic path joining $T_{1}$ and $T_{2}$, so again $S_{i}$ fixes a vertex of
$T$. This completes the proof of the lemma.
\end{proof}

Next we need the following technical result.

\begin{lemma}
\label{finitelymanydoublecosets} Let $K$ be a group with subgroups $H$ and
$S$. Let $Y$ be a nontrivial $H$--almost invariant subset of $K$ which is
adapted to $S$. Then the set $\{k\in K:kS\not \subset Y,kS\not \subset
Y^{\ast}\}$ consists of the union of finitely many double cosets $HgS$.
\end{lemma}

\begin{remark}
Lemma 2.7 of \cite{Scott symmint} is a similar sounding result about two
almost invariant subsets of a group, but, unlike the present lemma, that
result requires that all the groups involved be finitely generated, whereas in
the present lemma none of the groups need be finitely generated.
\end{remark}

\begin{proof}
As $Y$ is adapted to $S$, there is a subgroup $H^{\prime}$ of $K$ and a
$H^{\prime}$--almost invariant subset $Z$ of $K$ such that $Y$ and $Z\ $are
equivalent and $Z$ is strictly adapted to $S$. Thus for each $k$ in $K$, we
have $kS$ is contained in $Z$ or in $Z^{\ast}$. Lemma
\ref{equivalenta.i.setshavecommensurablestabiliser} shows that $H$ and
$H^{\prime}$ must be commensurable. We let $H^{\prime\prime}$ denote $H\cap
H^{\prime}$. Thus $H^{\prime\prime}$ is of finite index in $H$ and in
$H^{\prime}$, and so $Y\ $and $Z$ are both $H^{\prime\prime}$--almost
invariant. Now suppose that $kS\subset Z$. Then $kS$ will also be contained in
$Y$ unless $kS$ meets $Z-Y$. As $Y\ $and $Z$ are equivalent, $Z-Y$ is equal to
the union of finitely many cosets $H^{\prime\prime}g$. Now the intersection of
$kS$ with $H^{\prime\prime}g$ is non-empty if and only if $k\in H^{\prime
\prime}gS$. It follows that $\{k\in K:kS\subset Z,kS\not \subset Y\}$ consists
of the union of finitely many double cosets $H^{\prime\prime}gS$. Similarly
$\{k\in K:kS\subset Z^{\ast},kS\not \subset Y^{\ast}\}$ also consists of the
union of finitely many double cosets $H^{\prime\prime}gS$. Hence $\{k\in
K:kS\not \subset Y,kS\not \subset Y^{\ast}\}$ is contained in the union of
finitely many double cosets $H^{\prime\prime}gS$. As $HY=Y$, the set $\{k\in
K:kS\not \subset Y,kS\not \subset Y^{\ast}\}$ certainly consists of some union
of double cosets $HgS$. As $H^{\prime\prime}$ is contained in $H$, it follows
that $\{k\in K:kS\not \subset Y,kS\not \subset Y^{\ast}\}$ consists of the
union of finitely many double cosets $HgS$, as required.
\end{proof}

If we restrict attention to the special case where $G$ is finitely generated,
we can formulate the above ideas in terms of the geometry of a Cayley graph
$\Gamma$ of $G$. As usual, we identify the vertex set of $\Gamma$ with $G$,
and use the path metric on $\Gamma$ in which each edge has length $1$. For any
subgraph $X$ of $\Gamma$, we let $N_{R}X$ denote the $R$--neighbourhood of $X$.

\begin{lemma}
\label{adaptediffR-nbhdswork}Let $G$ be a finitely generated group, and let
$\Gamma$ be a Cayley graph for $G$ with respect to some finite generating set
of $G$. Let $H$ and $S$ be subgroups of $G$, and let $X$ be a nontrivial
$H$--almost invariant subset of $G$.

Then $X$ is $S$--adapted if and only if there is $R$ such that, for all $g\in
G$, we have $gS\subset N_{R}X$ or $gS\subset N_{R}X^{\ast}$.
\end{lemma}

\begin{remark}
This lemma does not assume that $H$ or $S$ is finitely generated.
\end{remark}

\begin{proof}
First suppose that $X$ is $S$--adapted. Definition \ref{defnofadapted} tells
us that $X$ is equivalent to a $H^{\prime}$--almost invariant subset
$X^{\prime}$ of $G$ such that $X^{\prime}$ is strictly adapted to $S$. This
means that, for all $g\in G$, we have $gS\subset X^{\prime}$ or $gS\subset
X^{\prime\ast}$. As $X$ and $X^{\prime}$ are equivalent, their symmetric
difference is $H$--finite. It follows that there is $R$ such that $X\subset
N_{R}X^{\prime}$ and $X^{\prime}\subset N_{R}X$. Hence if $gS\subset
X^{\prime}$, we have $gS\subset N_{R}X$, and if $gS\subset X^{\prime\ast}$, we
have $gS\subset N_{R}X^{\ast}$, which proves the lemma in one direction.

For the converse direction, we suppose that there is $R$ such that, for all
$g\in G$, we have $gS\subset N_{R}X$ or $gS\subset N_{R}X^{\ast}$. As $X$ is a
$H$--almost invariant subset of $G$, its quotient $H\backslash X$ is an almost
invariant subset of $H\backslash G$. Now we consider how the cosets $gS$ meet
$X$ and $X^{\ast}$. Note that distinct cosets $gS$ are disjoint, and are
permuted by the action of $G$ on itself by left multiplication. Thus the
images in $H\backslash G$ of any two of these cosets must coincide or be
disjoint. As the image in $H\backslash\Gamma$ of $\delta X$ is finite, it
follows that the image in $H\backslash\Gamma$ of $N_{R}(\delta X)$ is finite.
As we know that $gS$ is contained in $N_{R}X$ or $N_{R}X^{\ast}$, we know that
the image in $H\backslash G$ of either $gS\cap X$ or $gS\cap X^{\ast}$ is
finite, and there are only finitely many images in $H\backslash G$ of cosets
$gS$ which meet $H\backslash X$ and $H\backslash X^{\ast}$. Now we construct a
subset $W$ of $H\backslash G$ from $H\backslash X$ by altering $H\backslash X$
for each $g\in G$ such that $gS$ meets both $X$ and $X^{\ast}$. If $gS\cap X$
has finite image in $H\backslash G$, we remove that image from $H\backslash
X$. If $gS\cap X$ has infinite image in $H\backslash G$, then $gS\cap X^{\ast
}$ has finite image in $H\backslash G$, and we add that image to $H\backslash
X$. As only finitely many images in $H\backslash G$ of cosets $gS$ meet both
$H\backslash X$ and $H\backslash X^{\ast}$, we have altered $H\backslash X$
finitely many times by adding or subtracting a finite set. It follows that the
new set $W$ is almost equal to $H\backslash X$, and so is an almost invariant
subset of $H\backslash G$. Thus the pre-image in $G$ of $W$ is a $H$--almost
invariant subset $Z$ which is $H$--almost equal to $X$, so that $Z$ is
equivalent to $X$. Note that because the images in $H\backslash G$ of any two
of the cosets $gS$ must coincide or be disjoint, the finite sets which were
added to or subtracted from $H\backslash X$ are disjoint from each other. In
particular, it follows that for each $g\in G$, the image in $H\backslash G$ of
the coset $gS$ is entirely contained in $W$ or in $W^{\ast}$. It follows
immediately that for each $g\in G$, we have one of the inclusions $gS\subset
Z$ or $gS\subset Z^{\ast}$, so that $Z$ is strictly adapted to $S$. It follows
that $X$ is adapted to $S$, as required. This completes the proof of the lemma.
\end{proof}

\begin{remark}
\label{Zisnotunique}If the image in $H\backslash G$ of both $gS\cap X$ and
$gS\cap X^{\ast}$ is finite, then the above proof arranges that $gS\subset
Z^{\ast}$. It would be possible to change the construction of $Z$ to arrange
that instead $gS\subset Z$. Thus whenever this situation occurs there is a
certain amount of choice in our construction of $Z$.
\end{remark}

\begin{remark}
In section 2 of \cite{SSpdnold}, Scott and Swarup considered a finitely
generated group $G$ with Cayley graph $\Gamma$, with subgroups $H$ and $S$,
and a nontrivial $H$--almost invariant subset $X$ of $G$. They defined $X$ to
be adapted to $S$ if there is $R$ such that, for all $g\in G$, we have
$gS\subset N_{R}X$ or $gS\subset N_{R}X^{\ast}$. This definition cannot be
directly extended to the case when $G$ is not finitely generated, which is why
we chose to define adapted in this paper using Definition \ref{defnofadapted}.
But Lemma \ref{adaptediffR-nbhdswork} shows that the two definitions are
equivalent in the case when $G$ is finitely generated.
\end{remark}

If we restrict $S$ to be finitely generated as well as $G$, then we can obtain
a more useful criterion for an almost invariant set to be adapted to $S$.

We will need the following technical result.

\begin{lemma}
\label{Sf.g.impliesfinitenumberofdoublecosets}Let $G$ be a finitely generated
group with subgroups $S$ and $H$ such that $S$ is finitely generated. Let $X$
be a nontrivial $H$--almost invariant subset of $G$. Then the set $\{g\in
G:gS\not \subset X,gS\not \subset X^{\ast}\}$ consists of the union of
finitely many double cosets $HgS$.
\end{lemma}

\begin{proof}
This result is very closely related to Lemma 2.30 of \cite{SS2}, and has a
similar proof, but does not follow from it as that was a result about two
almost invariant sets. Also we are not assuming that $H$ is finitely generated.

Let $\Gamma$ denote the Cayley graph of $G$ with respect to some finite
generating set for $G$. Note that if $gS\not \subset X$ and $gS\not \subset
X^{\ast}$, then any connected subcomplex of $\Gamma$ which contains $gS$ must
meet the coboundary of $X$. As $S$ is finitely generated, there is a finite
connected subgraph $C$ of $S\backslash\Gamma$ such that $C$ contains the
single point which is the image of $S$ and the natural map $\pi_{1}%
(C)\rightarrow S$ is onto. Thus the pre-image $D$ of $C$ in $\Gamma$ is
connected and contains $S$. Let $\Delta$ denote a finite subgraph of $D$ which
projects onto $C$, and let $\Phi$ denote a finite subgraph of $\delta X$ which
projects onto the quotient $H\backslash\delta X$. If $gD$ meets $\delta X$,
there must be elements $s$ and $h$ in $S$ and $H$ respectively such that
$gs\Delta$ meets $h\Phi$. Now $\{\gamma\in G:\gamma\Delta$ meets $\Phi\}$ is
finite, as $G$ acts freely on $\Gamma$. It follows that $\{g\in G:gD$ meets
$\delta X\}$ consists of a finite number of double cosets $HgS$, which proves
the required result.
\end{proof}

The following corollary is an easy consequence.

\begin{corollary}
\label{adaptedtoSiffH-finiteintersections}Let $G$ be a finitely generated
group with a subgroup $H$, and a nontrivial $H$--almost invariant subset $X$.
Let $S$ be a finitely generated subgroup of $G$. Then $X$ is adapted to $S$ if
and only if, for each $g\in G$, we have $gS\cap X$ or $gS\cap X^{\ast}$ is $H$--finite.
\end{corollary}

\begin{proof}
If $X$ is adapted to $S$, then Lemma \ref{adaptediffR-nbhdswork} tells us that
there is a finite generating set of $G$ with associated Cayley graph $\Gamma$,
and an integer $D$ such that, for each $g\in G$, we have $gS\subset N_{D}X$ or
$gS\subset N_{D}X^{\ast}$. As $\delta X$ is $H$--finite, it follows that
$N_{D}X-X\subset N_{D}\delta X$ is also $H$--finite. Thus if $gS\subset
N_{D}X$, it follows that $gS\cap X^{\ast}$ is $H$--finite, and similarly if
$gS\subset N_{D}X^{\ast}$, then $gS\cap X$ is $H$--finite.

Now suppose that, for each $g\in G$, we have $gS\cap X$ or $gS\cap X^{\ast}$
is $H$--finite, and choose a finite generating set of $G$ with associated
Cayley graph $\Gamma$. If $gS\cap X$ is $H$--finite, it follows that
$gS\subset N_{D}X^{\ast}$, for some $D$, and if $gS\cap X^{\ast}$ is
$H$--finite, it follows that $gS\subset N_{D}X$, for some $D$. We need to show
that there is a uniform bound $D$ for the sizes of these bounded
neighbourhoods. Now Lemma \ref{Sf.g.impliesfinitenumberofdoublecosets} tells
us that the set $\{g\in G:gS\not \subset X,gS\not \subset X^{\ast}\}$ consists
of the union of finitely many double cosets $HgS$. If $gS$ is contained in the
$d$--neighbourhood of $X$ or of $X^{\ast}$, the same holds for $hgS$, for
every $h$ in $H$, so that $HgS$ is contained in the $d$--neighbourhood of $X$
or of $X^{\ast}$. As we have only finitely many such double cosets, it follows
that there is a uniform bound for the sizes of the neighbourhoods of $X$ and
$X^{\ast}$ which contain the translates $gS$. Thus there is an integer $D$
such that, for each $g\in G$, we have $gS\subset N_{D}X$ or $gS\subset
N_{D}X^{\ast}$. Now Lemma \ref{adaptediffR-nbhdswork} tells us that $X$ is
adapted to $S$.
\end{proof}

A slightly surprising application of the above lemma is that if $S$ is a
Tarski monster, then one expects that any almost invariant subset of $G$ is
automatically adapted to $S$. Here is a specific example of such a result.

\begin{lemma}
Let $G$ be a finitely generated group with a subgroup $H$, and a nontrivial
$H$--almost invariant subset $X$. Let $S$ be a finitely generated subgroup of
$G$ such that every proper subgroup of $S$ is isomorphic to $Z_{p}$. Then $X$
is $S$--adapted.
\end{lemma}

\begin{remark}
In \cite{Olshanski}, Olshanski showed that infinite such groups $S$ exist for
every prime $p>10^{75}$.
\end{remark}

\begin{proof}
The point here is that any almost invariant subset of the group $S$ is
trivial. For suppose that $Y$ is a $K$--almost invariant subset of $S$, for
some subgroup $K$ of $S$. If $K$ is equal to $S$, then $Y\ $is automatically
trivial. If $K$ is trivial or finite cyclic, and $Y$ is nontrivial, it follows
that $S$ has more than one end. Now Stallings' theorem \cite{Stallings2}
implies that $S$ splits over a finite subgroup, contradicting the assumption
that $S$ is a torsion group. Thus in all cases $Y$ must be trivial.

Now the intersection $X\cap S$ is a $(H\cap S)$--almost invariant subset of
$S$, so that $X\cap S$ or $X^{\ast}\cap S$ is $(H\cap S)$--finite, and hence
is certainly $H$--finite. The same argument applies to show that, for each $g$
in $G$, one of $gX\cap S$ or $gX^{\ast}\cap S$ is $(H^{g}\cap S)$--finite, and
hence is certainly $H^{g}$--finite. It follows that, for each $g$ in $G$, one
of $X\cap gS$ or $X^{\ast}\cap gS$ is $H$--finite. Now Corollary
\ref{adaptedtoSiffH-finiteintersections} implies that $X$ is $S$--adapted, as required.
\end{proof}

We close this discussion of basic properties of adapted almost invariant sets
with the following result which connects these ideas with Sageev's cubing,
discussed in section \ref{cubings}.

\begin{lemma}
\label{cubinghasafixedvertex}Let $K$ be a group with subgroups $H$ and $S$,
and let $Y$ be a nontrivial $H$--almost invariant subset of $K$ which is
adapted to $S$. Let $C(Y)$ denote the cubing constructed by Sageev from the
set $E=\{gY,gY^{\ast}:g\in K\}$, as discussed in section \ref{cubings}. Thus
$K$ acts on $C(Y)$.

Suppose that $S$ is not conjugate commensurable with any subgroup of $H$. Then
there is a vertex of $C(Y)$ fixed by $S$.
\end{lemma}

\begin{proof}
For the purposes of this proof only, the following notation will be
convenient. Let $\Sigma$ be a subset of $K$, and let $X$ be a $L$--almost
invariant subset of $K$. We will write $\Sigma<X$ to mean that $\Sigma\cap
X^{\ast}$ is $L$--finite. Note that this is well defined. For if $X$ is also
$L^{\prime}$--almost invariant, then $L\ $and $L^{\prime}$ are commensurable,
by Corollary \ref{H-aisetequalsKaisetimpliescommensurable}, so that
$\Sigma\cap X^{\ast}$ is also $L^{\prime}$--finite. Intuitively, the
inequality $\Sigma<X$ means that $\Sigma$ is almost contained in $X$, but the
reader is warned that in general, the relation $<$ does not have good
properties. However we claim that this notation is $K$--equivariant in the
sense that if $k$ is an element of $K$ and if $\Sigma<X$, then we do have
$k\Sigma<kX$. For if $\Sigma<X$, then $\Sigma\cap X^{\ast}$ is $L$--finite, so
that $k\Sigma\cap kX^{\ast}=k(\Sigma\cap X^{\ast})$ is $kLk^{-1}$--finite. As
$kX$ is $kLk^{-1}$--almost invariant, it follows that $k\Sigma<kX$, as claimed.

Recall from Lemma \ref{S-adaptedimpliesgSintersectXisH-finite} that as $Y$ is
adapted to $S$, we know that, for each $k$ in $K$, we have $kS<Y$ or
$kS<Y^{\ast}$. Suppose first that both these inequalities hold for some $k$ in
$K$. This implies that $kS$ must be $H$--finite, so that $kSk^{-1}$ is also
$H$--finite. Now part 1) of Lemma \ref{Neumannlemma} implies that $kSk^{-1}$
is commensurable with a subgroup of $H$, so that $S$ must be conjugate
commensurable with a subgroup of $H$. Thus our assumption that $S$ is not
conjugate commensurable with any subgroup of $H$ implies that there is no
$k\in K$ such that $kS<Y$ and $kS<Y^{\ast}$.

Now we claim that $V(S)=\{U\in E:S<U\}$ is an ultrafilter on the set $E$
partially ordered by inclusion. The fact that there is no $k\in K$ such that
$kS<Y$ and $kS<Y^{\ast}$ implies that there is no $k\in K$ such that $S<kY$
and $S<kY^{\ast}$. Thus for every $U\in E$, exactly one of $S<U$ and
$S<U^{\ast}$ holds, so that $V(S)$ satisfies condition 1) of Definition
\ref{defnofultrafilter}. Now suppose that $S<A$ and $A\subset B$, where $A$
and $B$ are elements of $E$, so that $A$ is almost invariant over a conjugate
$H_{A}$ of $H$, and $B$ is almost invariant over a conjugate $H_{B}$ of $H$.
We claim that $S<B$, so that $B\in V(S)$. For suppose this claim is false.
Then we must have $S<A$ and $S<B^{\ast}$, so that $S\cap A^{\ast}$ is $H_{A}%
$--finite and $S\cap B$ is $H_{B}$--finite. As $A\subset B$, it follows that
$K=A^{\ast}\cup B$, so that $S$ is the union of a $H_{A}$--finite set with a
$H_{B}$--finite set. Now part 2) of Lemma \ref{Neumannlemma} implies that $S$
is commensurable with a subgroup of $H_{A}$ or of $H_{B}$, so that $S$ must be
conjugate commensurable with a subgroup of $H$. Thus our assumption that $S$
is not conjugate commensurable with any subgroup of $H$ implies that we must
have $S<B$, as claimed. It follows that $V(S)$ satisfies condition 2) of
Definition \ref{defnofultrafilter}. Hence $V(S)$ is an ultrafilter on $E$, as
claimed, and so is a vertex of $\mathcal{K}(Y)$. As the action of $S$ on $G$
preserves $S$, it must also fix the vertex $V(S)$.

Next we show that $V(S)$ is a vertex of the component $C(Y)$ of $\mathcal{K}%
(Y)$. Let $e$ denote the identity element of $K$, and recall that
$V_{e}=\{U\in E:e\in U\}$ is an ultrafilter on the set $E$ partially ordered
by inclusion, and is a basic vertex of $\mathcal{K}(Y)$ which lies in $C(Y)$.
In order to show that $V(S)$ lies in $C(Y)$, it suffices to show that $V_{e}$
can be joined to $V(S)$ by a path in $\mathcal{K}(Y)$. Equivalently it
suffices to show that $V_{e}$ differs from $V(S)$ on only finitely many pairs
$\{U,U^{\ast}\}$ of elements of $E$. Clearly $V(S)$ and $V_{e}$ coincide for
all $U\in E$ such that $S\subset U$. It remains to consider those elements $U$
of $E$ such that $S<U$ but $S\not \subset U$. Note that if $S<U$, we cannot
have $S<U^{\ast}$, and so we must have $S\not \subset U^{\ast}$. Thus the set
$\{g\in K:S<gY,S\not \subset gY\}$ is contained in the set $\{g\in
K:S\not \subset gY,S\not \subset gY^{\ast}\}$. Now Lemma
\ref{finitelymanydoublecosets} tells us that the set $\{k\in K:kS\not \subset
Y,kS\not \subset Y^{\ast}\}$ consists of the union of finitely many double
cosets $HgS$. Thus the set $\{g\in K:S\not \subset gY,S\not \subset gY^{\ast
}\}$ consists of the union of finitely many double cosets $SgH$. It follows
that the set $\{g\in K:S<gY,S\not \subset gY\}$ is contained in the union of
finitely many double cosets $Sk_{i}H$, $1\leq i\leq m$. For each $i$, we claim
that $e$ lies in all but finitely many of the translates $sk_{i}Y$ of $Y$,
where $s\in S$. Assuming this claim, it now follows that $V(S)$\ and $V_{e}$
differ on only finitely many pairs. Thus $V(S)$ is a vertex of $C(Y)$ which is
fixed by $S$, as required.

Now we prove the above claim. For each $i$, let $Y_{i}$ denote $k_{i}Y$, and
let $H_{i}$ denote $k_{i}Hk_{i}^{-1}$, so that $Y_{i}$ is a $H_{i}$--almost
invariant subset of $K$. As $S<Y_{i}$, the intersection $S\cap Y_{i}^{\ast}$
is $H_{i}$--finite. Denote this intersection by $W_{i}$. As $W_{i}$ is $H_{i}%
$--finite, there are elements $a_{1},\ldots,a_{n}$ of $K$ such that $W_{i}$ is
contained in the union $H_{i}a_{1}\cup\ldots\cup H_{i}a_{n}$. Now let $s$ be
an element of $S$, and suppose that $e$ does not lie in $sk_{i}Y$, so that $e$
must lie in $sk_{i}Y^{\ast}=sY_{i}^{\ast}$. As $e\in S$, we have $e\in S\cap
sY_{i}^{\ast}=sW_{i}$. Thus $s^{-1}$ lies in $W_{i}$, so that $s$ lies in the
union $a_{1}^{-1}H_{i}\cup\ldots\cup a_{n}^{-1}H_{i}$. Hence there is $j$ such
that $s$ lies in $a_{j}^{-1}H_{i}=a_{j}^{-1}k_{i}Hk_{i}^{-1}$, so that
$sk_{i}$ lies in $a_{j}^{-1}k_{i}H$. Thus the translates $sk_{i}Y$ of $Y$ such
that $S<sk_{i}Y$ but $e\notin sk_{i}Y$ are among the finitely many translates
$a_{j}^{-1}k_{i}Y$, for $1\leq j\leq n$. This completes the proof of the
claim, and hence completes the proof of the existence of a vertex of $Y$ fixed
by $S$.
\end{proof}

In the above lemma, it is easy to show that the vertex of $C(Y)$ fixed by $S$
is unique. If $S$ is conjugate commensurable with a subgroup of $H$, it must
still fix some point of $C(Y)$, but $S$ need not fix a vertex of $C(Y)$. The
following is an example of this phenomenon.

\begin{example}
\label{exampleofnofixedvertexinC(Y)}Let $K$ be a group with subgroups $H$ and
$S$, and let $Y$ be a nontrivial $H$--almost invariant subset of $K$ which is
adapted to $S$, and is in very good position. Let $C(Y)$ denote the cubing
constructed by Sageev from the set $E=\{gY,gY^{\ast}:g\in K\}$, as discussed
in section \ref{cubings}. Thus $K$ acts on $C(Y)$.

Suppose that there is $s$ in $S$ such that $sY=Y^{\ast}$. As $Y\ $is in very
good position, this implies that $s$ preserves the hyperplane $\Pi$ which is
associated to $Y$, and that $s$ interchanges the sides of $\Pi$. In particular
$s$ cannot fix any vertex of $C(Y)$. The simplest example where this occurs is
as follows. Let $G=\mathbb{Z}_{2}\ast\mathbb{Z}_{2}$, let $H$ be the trivial
subgroup of $G$, let $S$ denote one of the $\mathbb{Z}_{2}$ factors of $G$,
and let $s$ denote the nontrivial element of $S$. Then there is an almost
invariant subset $Y$ of $G$ associated to the given splitting of $G$ over $H$
such that $sY=Y^{\ast}$ and $Y$ is in very good position. Further any such
$Y\ $is automatically adapted to $S$. Note that in this example $C(Y)$ is a
line, and $s$ acts on $C(Y)$ by reflection.
\end{example}

Finally we note that the proof of Lemma \ref{cubinghasafixedvertex} applies
just as well to the case of several almost invariant sets. For future
reference we state the result.

\begin{lemma}
\label{cubingformanysetshasafixedpoint}Let $K$ be a group with subgroups $S$
and $H_{1},\ldots,H_{n}$, and, for $1\leq i\leq n$, let $Y_{i}$ be a
nontrivial $H_{i}$--almost invariant subset of $K$ which is adapted to $S$.
Let $C(Y_{1},\ldots,Y_{n})$ denote the cubing constructed by Sageev from the
set $E=\{gY_{i},gY_{i}^{\ast}:1\leq i\leq n,g\in K\}$, as discussed in section
\ref{cubings}. Thus $K$ acts on $C(Y_{1},\ldots,Y_{n})$.

Suppose that $S$ is not conjugate commensurable with any subgroup of any
$H_{i}$. Then there is a vertex of $C(Y_{1},\ldots,Y_{n})$ fixed by $S$.
\end{lemma}

\section{Group systems and Cayley graphs\label{groupsystems}}

Up to this point of the paper, we have made no restrictions on the cardinality
of the groups we consider. It will now be convenient to introduce the
following terminology.

\begin{definition}
A \textsl{group system} $(G,\mathcal{S})$ consists of a group $G$ and a family
$\mathcal{S}$ of subgroups of $G$, with repetitions allowed.
\end{definition}

\begin{definition}
A group system is \textsl{of finite type} if $G$ is countable, $\mathcal{S}$
is finite, and $G$ is finitely generated relative to $\mathcal{S}$ (i.e. $G$
is generated by the union of all the subgroups in the family $\mathcal{S}$
together with some finite subset of $G$).
\end{definition}

For later use, we also define what it means for $G$ to be finitely presented
relative to $\mathcal{S}$.

\begin{definition}
\label{defnoff.p.relfamilyofsubgroups} Let $(G,\mathcal{S})$ be a group system
of finite type, and let $\mathcal{S}=\{S_{1},\ldots S_{n}\}$. Then $G$ is
\textsl{finitely presented relative to} $\mathcal{S}$, if there is a finite
relative generating set $\Omega$ such that the kernel of the natural
epimorphism $F(\Omega)\ast S_{1}\ast\ldots\ast S_{n}\rightarrow G$ is normally
generated by a finite set.
\end{definition}

For much of the rest of this paper, we will consider only group systems of
finite type. The reason for this is that although the group $G$ need not be
finitely generated, one can understand almost invariant subsets of $G$ which
are adapted to $\mathcal{S}$ in a kind of Cayley graph of the system. Note
that the countability condition on $G$ is not needed in this section, but will
frequently be used in later sections.

Recall that given a group $G$ and a subset $A$, a graph $\Gamma(G:A)$ can be
defined with vertex set $G$ and edges $(g,gg_{\alpha})$ which join $g$ to
$gg_{\alpha}$, for each $g_{\alpha}\in A$. Clearly $G$ acts freely on the left
on $\Gamma(G:A)$. Further this graph is connected if and only if $A$ generates
$G$. In this case $\Gamma(G:A)$ is called the Cayley graph of $G$ with respect
to $A$. Note that in this case, an alternative way of obtaining $\Gamma(G:A)$
is to start with a connected $1$--vertex graph $R$ whose edges are labelled by
the elements $g_{\alpha}$ of $A$. Thus $\pi_{1}(R)$ is the free group on $A$,
and there is a natural map $\varphi:\pi_{1}(R)\rightarrow G$ which sends the
generators of $\pi_{1}(R)$ to the $g_{\alpha}$'s. Now the Cayley graph
$\Gamma(G:A)$ is just the cover of $R$ corresponding to the kernel of
$\varphi$. If the generating set $A$ is finite, then $\Gamma(G:A)$ is locally
finite. In this case, a subset $X$ of $G$ such that $HX=X$ is $H$--almost
invariant if and only if the coboundary $\delta X$ of $X$ in $\Gamma(G:A)$ is
$H$--finite. (Here $\delta X$ denotes the set of edges of $\Gamma$ with
precisely one vertex in $X$.) This beautiful fact is due to Cohen \cite{Cohen}.

Now suppose that $(G,\mathcal{S})$ is a group system of finite type, and let
$A=\{g_{1},\ldots,g_{n}\}$ be a generating set for $G$ relative to
$\mathcal{S}$. Then we define the relative Cayley graph $\Gamma(G,\mathcal{S}%
)$ to be the union of $\Gamma(G:A)$, whose vertex set equals $G$, with a cone
on each coset $gS_{j}$ of each $S_{j}$ in $\mathcal{S}$. More precisely,
$\Gamma(G,\mathcal{S})$ has vertex set $G$ together with $\{gS_{j}:g\in
G,S_{j}\in\mathcal{S}\}$, and has edge set consisting of the edges of
$\Gamma(G:A)$ together with edges $(g,gS_{j})$ which join $g$ to $gS_{j}$, for
each $g\in G$ and each $S_{j}\in\mathcal{S}$. Vertices of the form $gS_{j}$
will be called cone points, and any edge which is incident to a cone point
will be called a cone edge. The union of all the cone edges incident to a
given cone point $gS_{j}$ is a cone on the subset $gS_{j}$ of $G$. The graph
$\Gamma(G,\mathcal{S})$ is connected because $A$ generates $G$ relative to
$\mathcal{S}$, and $G$ again acts on the left, but the action is no longer
free. Nor is $\Gamma(G,\mathcal{S})$ locally finite (unless each $S_{j}$ is
finite). However it is still possible to interpret almost invariant subsets of
$G$ in terms of a coboundary if they are adapted to $\mathcal{S}$. Let $X$
denote a subset of $G$ and, as usual, let $X^{\ast}$ denote its complement in
$G$. Thus $X$ and $X^{\ast}$ are sets of vertices in $\Gamma(G,\mathcal{S})$.
Note that if $X$ is a $H$--almost invariant subset of $G$, then its coboundary
$\delta X$ in $\Gamma(G,\mathcal{S})$ is most unlikely to be $H$--finite, so
we will enlarge $X$ to try to arrange this. A collection $\widehat{X}$ of
vertices of $\Gamma(G,\mathcal{S})$ is called an \textit{enlargement} of $X$
if $\widehat{X}\cap G=X$. Thus $\widehat{X}$ contains $X$, and $\widehat{X}-X$
consists only of cone points.

The following result is what we need.

\begin{lemma}
\label{characterizingdX}Let $(G,\mathcal{S})$ be a group system of finite
type, and let $\Gamma(G,\mathcal{S})$ be a relative Cayley graph. Let $H$ be a
subgroup of $G$, and let $X$ be a subset of $G$ such that $HX=X$. Then there
is an enlargement $\widehat{X}$ of $X$ such that $\delta\widehat{X}$ is
$H$--finite if and only if $X$ is $H$--almost invariant and $\mathcal{S}$--adapted.
\end{lemma}

\begin{proof}
First we assume that $X$ is $H$--almost invariant and $\mathcal{S}$--adapted,
and will construct an enlargement $\widehat{X}$ such that $\delta\widehat{X}$
is $H$--finite. We start by considering the set $E^{\prime}$ of all edges of
the form $(g,gg_{i})$ which join points of $X$ and $X^{\ast}$. As $X$ is
$H$--almost invariant, we know that the symmetric difference $X+Xg_{i}$ is
$H$--finite. It follows that the set of edges of the form $(g,gg_{i})$ which
join points of $X$ and $X^{\ast}$ is $H$--finite, and hence that $E^{\prime}$
is $H$--finite. Now we consider the cone points of $\Gamma(G,\mathcal{S})$,
and will decide whether or not to assign them to $\widehat{X}$. As $X$ is
$\mathcal{S}$--adapted, Lemma \ref{finitelymanydoublecosets} shows there are
only $H$--finitely many cosets $gS_{j}$ such that $gS_{j}\cap X$ and
$gS_{j}\cap X^{\ast}$ are both non-empty. If $gS_{j}\cap X$ is empty, we will
choose the cone point $gS_{j}$ not to lie in $\widehat{X}$. If $gS_{j}\cap
X^{\ast}$ is empty, we will choose the cone point $gS_{j}$ to lie in
$\widehat{X}$. In either case, this ensures that no cone edge incident to
$gS_{j}$ lies in $\delta\widehat{X}$. If $gS_{j}\cap X$ and $gS_{j}\cap
X^{\ast}$ are each non-empty, the fact that $X$ is $\mathcal{S}$--adapted
tells us that at least one of them is $H$--finite. If $gS_{j}\cap X$ is
$H$--finite, we will choose the cone point $gS_{j}$ not to lie in
$\widehat{X}$. If $gS_{j}\cap X$ is $H$--infinite, so that $gS_{j}\cap
X^{\ast}$ must be $H$--finite, we will choose the cone point $gS_{j}$ to lie
in $\widehat{X}$. In either case, this ensures only a $H$--finite number of
cone edges incident to $gS_{j}$ lie in $\delta\widehat{X}$. Thus
$\delta\widehat{X}$ consists of $E^{\prime}$ and a $H$--finite collection of
cone edges, so that $\delta\widehat{X}$ is $H$--finite, as required.

Now suppose that $\widehat{X}$ is an enlargement of $X$ such that
$\delta\widehat{X}$ is $H$--finite. We need to show that $X$ is $H$--almost
invariant and $\mathcal{S}$--adapted. Fix $g\in G$, and consider the symmetric
difference $X+Xg$. Observe that $x\in X+Xg$ if and only if exactly one of $x$
and $xg^{-1}$ lies in $X$. Now write $g^{-1}$ as a word $w$ in the $g_{i}$'s
and in elements of the $S_{j}$'s. Thus $w=a_{1}a_{2}\ldots a_{k}$, where each
$a_{k}$ equals some $g_{i}$ or its inverse, or some element of some $S_{j}$.
The word $w$ determines an edge path $\lambda=e_{1}\ldots e_{n}$ in
$\Gamma(G,\mathcal{S})$ from $e$ to $g^{-1}$ which contains one edge for each
$g_{i}$ in $w$ and contains two cone edges for each element of some $S_{j}$.
Thus the path $x\lambda$ joins $x$ to $xg^{-1}$. It follows that if $x\in
X+Xg$, then $x\lambda$ contains an edge of $\delta\widehat{X}$. Thus $xe_{i}$
lies in $\delta\widehat{X}$, for some $i$. As $G$ acts freely on the edges of
$\Gamma(G,\mathcal{S})$, and $\delta\widehat{X}$ is $H$--finite, the set of
all elements $x$ in $G$ such that $xe_{i}$ lies in $\delta\widehat{X}$ is also
$H$--finite, for each $i$. Hence $X+Xg$ is $H$--finite, for every $g$ in $G$,
so that $X$ is $H$--almost invariant.

Next fix a group $S$ in $\mathcal{S}$. We will show that $X$ is equivalent to
a $H$--almost invariant subset $Y$ of $G$ which is strictly adapted to $S$,
which will imply that $X$ itself is adapted to $S$. First, by replacing
$\widehat{X}$ by $H\widehat{X}$ if needed, we can arrange that $H\widehat{X}%
=\widehat{X}$, without losing the $H$--finiteness of $\delta\widehat{X}$. This
is because $H\widehat{X}$ is obtained from $\widehat{X}$ by adding cone
points, each of which contributes only a $H$--finite number of new edges to
the coboundary, and all but a $H$--finite number of these added cone points
contribute no new edges. Now consider all cone edges of $\Gamma(G,\mathcal{S}%
)$ which lie in $\delta\widehat{X}$. If a coset $gS$ is entirely contained in
$X$, and the cone point $gS$ does not lie in $\widehat{X}$, we add $gS$ to
$\widehat{X}$. This removes from $\delta\widehat{X}$ all cone edges which are
incident to the vertex $gS$, without changing $X$. If a coset $gS$ is entirely
contained in $X^{\ast}$, and the cone point $gS$ lies in $\widehat{X}$, we
remove $gS$ from $\widehat{X}$. This again removes from $\delta\widehat{X}$
all cone edges which are incident to the vertex $gS$, without changing $X$. If
a coset $gS$ meets both $X$ and $X^{\ast}$ and if $gS\cap X$ is $H$--finite,
we will remove $gS\cap X$ from $X$. In addition, if the cone point $gS$ lies
in $\widehat{X}$, we will remove it from $\widehat{X}$. This removes from
$\delta\widehat{X}$ all cone edges which are incident to the vertex $gS$. If a
coset $gS$ meets both $X$ and $X^{\ast}$ and if $gS\cap X$ is $H$--infinite,
then $gS\cap X^{\ast}$ must be $H$--finite. In this case, we add $gS\cap
X^{\ast}$ to $X$, and also add the cone point $gS$ to $\widehat{X}$ if needed.
This again removes from $\delta\widehat{X}$ all cone edges which are incident
to the vertex $gS$. Let $Y$ denote the subset of $G$ obtained from $X$ in this
way, and let $\widehat{Y}$ denote the vertex set obtained from $\widehat{X}$
in this way, so that $\widehat{Y}$ is an enlargement of $Y$. As $\delta
\widehat{X}$ is $H$--finite, $Y$ is $H$--almost equal to $X$. As
$H\widehat{X}=\widehat{X}$ and $HX=X$, we have $HY=Y$, so that $Y$ is a
$H$--almost invariant subset of $G$ which is equivalent to $X$. Finally, by
construction, $\delta\widehat{Y}$ contains no cone edges incident to any cone
point $gS$. It follows that for each coset $gS$ of $S$ in $G$, the
corresponding cone can meet only one of $Y$ and $Y^{\ast}$ in a non-empty set.
Thus $gS$ is contained in one of $Y$ and $Y^{\ast}$, so that $Y$ is strictly
adapted to $S$, as required.
\end{proof}

We summarize some basic facts about enlargements in the following remark.

\begin{remark}
\label{enlargements}Using the notation of Lemma \ref{characterizingdX},
suppose there is an enlargement $\widehat{X}$ of $X$ in $\Gamma(G,\mathcal{S}%
)$ such that $\delta\widehat{X}$ is $H$--finite.

\begin{enumerate}
\item By replacing $\widehat{X}$ by $H\widehat{X}$ if needed, we can arrange
that $H\widehat{X}=\widehat{X}$.

\item If $gS_{i}\cap X$ is $H$--infinite, then the cone point $gS_{i}$ must
lie in $\widehat{X}$.

\item If $gS_{i}\cap X^{\ast}$ is $H$--infinite, then the cone point $gS_{i}$
cannot lie in $\widehat{X}$.

\item If there is no $g$ in $G$ and $S_{i}$ in $\mathcal{S}$ such that
$gS_{i}$ is $H$--finite, then $X$ has only one enlargement $\widehat{X}$ such
that $\delta\widehat{X}$ is $H$--finite.
\end{enumerate}
\end{remark}

In case 4) of the above remark, we call this unique enlargement, the
\textit{canonical enlargement} of $X$ in $\Gamma(G,\mathcal{S})$.

Now we are ready to start the construction which we promised immediately
before Definition \ref{defnofadaptedsplitting}.

It will be convenient to introduce the following terminology.

\begin{definition}
Let $\overline{G}$ be a group with a minimal graph of groups decomposition
$\Gamma$, with a vertex $V$ whose associated group is $G$. Let $X$ be a
nontrivial $H$--almost invariant subset of $G$. An \textsl{extension} of $X$
to $\overline{G}$ is a $H$--almost invariant subset $\overline{X}$ of
$\overline{G}$ which is enclosed by $V$ such that $\overline{X}\cap G=X$.
\end{definition}

In the setting of this definition, let $\mathcal{S}$ denote the family of
subgroups of $G$ associated to the edges of $\Gamma$ which are incident to
$V$. Lemma \ref{XbarenclosedbyvimpliesthatXisadapted} implies that if $X$ has
an extension, then $X$ must be $\mathcal{S}$--adapted. We will use the
relative Cayley graph to prove a converse under the assumption that the group
system $(G,\mathcal{S})$ is of finite type.

\begin{lemma}
\label{extensionsexist}Let $\overline{G}$ be a group with a minimal graph of
groups decomposition $\Gamma$, with a vertex $V$ whose associated group is
$G$. Let $\mathcal{S}$ denote the family of subgroups of $G$ associated to the
edges of $\Gamma$ which are incident to $V$, and let $X$ be a nontrivial
$\mathcal{S}$--adapted $H$--almost invariant subset of $G$. If the group
system $(G,\mathcal{S})$ is of finite type, then $X$ has an extension
$\overline{X}$ to $\overline{G}$.
\end{lemma}

\begin{remark}
In this result, the countability assumption on $G$ is not needed, as noted
immediately after the definition of a group system of finite type.
\end{remark}

\begin{proof}
By collapsing all edges of $\Gamma$ which do not meet $V$, we can arrange that
each edge of $\Gamma$ meets $V$. For simplicity we will restrict to this case.
Now we will pick a generating set (not necessarily finite) for each vertex
group of $\Gamma$ other than $G$, and will pick a finite subset $A$ of $G$
which generates $G$ relative to $\mathcal{S}$. For each chosen generating set
including $A$, choose a rose, i.e. a $1$--vertex graph whose edges are
labelled by the generators. Now construct a connected graph $C$ from these
roses by joining the roses in pairs by a single edge corresponding to each
edge of $\Gamma$. This graph has a natural map to $\Gamma$ in which the
pre-image of each vertex is a rose, and the pre-image of each edge is a single
edge of $C$. Thus $\pi_{1}(C)$ is a free group and there is a natural
surjection $\varphi:\pi_{1}(C)\rightarrow\overline{G}$ which sends each
generator of $\pi_{1}(C)$ represented by a loop in a rose to the appropriate
element of the vertex group. The covering space of this graph corresponding to
the kernel of $\varphi$ is a connected graph $\overline{\Gamma}$ on which
$\overline{G}$ acts freely with quotient $C$. Let $T$ denote the $\overline
{G}$--tree associated with the graph of groups $\Gamma$. Then there is a
natural $\overline{G}$--equivariant map $p:\overline{\Gamma}\rightarrow T$.
The pre-image of each vertex of $T$ not above $V$ is a copy of the Cayley
graph of the vertex group. The pre-image of each vertex of $T$ above $V$ is a
copy of the graph $\Gamma(G:A)$. Recall that $\Gamma(G:A)$ will be
disconnected unless $A$ generates $G$. Finally the pre-image of each edge of
$T$ consists of disjoint edges corresponding to the elements of the edge
group. Now we consider the vertex $V_{0}$ of $T$ with stabilizer $G$, let
$\Gamma(G:A)$ denote its pre-image in $\overline{\Gamma}$, and let
$\overline{\Gamma(G:A)}$ denote the closure in $\overline{\Gamma}$ of the
pre-image of the open star of $V_{0}$ in $T$. This consists of $\Gamma(G:A)$
together with edges attached, such that each vertex $g$ of $\Gamma(G:A)$ has
one edge $(g,S)$ attached for each group $S$ in $\mathcal{S}$. Thus the
relative Cayley graph $\Gamma(G,\mathcal{S})$ is obtained from $\overline
{\Gamma(G:A)}$ by identifying the free endpoints of $(g,S)$ and $(h,S)$ if and
only if the cosets $gS$ and $hS$ coincide.

Now let $X$ be a $\mathcal{S}$--adapted $H$--almost invariant subset of $G$,
and let $\widehat{X}$ denote an enlargement of $X$ in $\Gamma(G,\mathcal{S})$
such that $\delta\widehat{X}$ is $H$--finite. The pre-image of $\delta
\widehat{X}$ in $\overline{\Gamma(G:A)}$ is a $H$--finite collection $E$ of
edges. For each cone point $gS_{j}$ in $\Gamma(G,\mathcal{S})$, the incident
cone edges have pre-image in $\overline{\Gamma(G:A)}$ which maps to a single
edge of $T$. If the cone point $gS_{j}$ lies in $\widehat{X}$, we label this
edge by $X$. Otherwise we label this edge by $X^{\ast}$. By repeating for all
cone points, we will label every edge of $T$ incident to $V_{0}$ in a
$G$--invariant way. We will colour all other edges of $T$ so that each
component of $T-V_{0}$ is monochrome. We now define a partition of
$\overline{G}$ into two sets $\overline{X}$ and its complement, such that
$\overline{X}\cap G=X$, by requiring that $g$ in $\overline{G}-G$ lies in
$\overline{X}$ if and only if $gV_{0}$ lies in a component of $T-V_{0}$ which
is labelled $X$. As the entire picture is $G$--invariant, it follows that
$H\overline{X}=\overline{X}$. It remains to show that $\overline{X}$ is
$H$--almost invariant. For then it will follow that $\overline{X}$ is an
extension of $X$ to $\overline{G}$, as required. Note that $\overline{\Gamma}$
is not the Cayley graph of $\overline{G}$, but we can use essentially the same
argument as in the proof of Lemma \ref{characterizingdX}.

Given $g\in\overline{G}$, we need to show that $\overline{X}g+\overline{X}$ is
$H$--finite. Let $v$ denote the vertex of $\overline{\Gamma}$ above $V_{0}$
which is the vertex of $\Gamma(G:A)$ identified with the identity element of
$G$. Let $\lambda=e_{1}\ldots e_{n}$ be a path in $\overline{\Gamma}$ which
joins $v$ to $gv$. Thus $x\lambda$ joins $xv$ to $xgv$, for any $x\in
\overline{G}$. Suppose that $x$ lies in $\overline{X}g+\overline{X}$. Thus one
of $x$ and $xg$ lies in $\overline{X}$ and the other does not. We will
consider various cases depending on whether $g$ and $x$ lie in $G$, and will
show that in all cases the path $x\lambda$ must contain some edge of $E$. Thus
$xe_{i}$ lies in $E$, for some $i$. As $\overline{G}$ acts freely on the edges
of $\overline{\Gamma}$, and $E$ is $H$--finite, the set of all elements $x$ in
$\overline{G}$ such that $xe_{i}$ lies in $E$ is also $H$--finite, for each
$i$. Hence $\overline{X}g+\overline{X}$ is $H$--finite, as required.

Now we consider the cases.

If $g$ and $x$ both lie in $G$, so does $xg$. Thus one of $x$ and $xg$ lies in
$X$ and the other does not. This implies that the path $x\lambda$ must contain
some edge of $E$.

If $g$ lies in $G$, but $x$ does not, then neither does $xg$. Hence one of
$xv$ and $xgv$ has image in a component of $T-V_{0}$ which is labelled $X$,
and the other has image in a component of $T-V_{0}$ which is labelled
$X^{\ast}$. Thus the path $x\lambda$ has image in $T$ which joins components
of $T-V_{0}$, one component labelled $X$ and the other labelled $X^{\ast}$. By
construction, $x\lambda$ must contain some edge of $E$.

If $g$ does not lie in $G$, then we cannot have both $x$ and $xg$ in $G$. Thus
either the path $x\lambda$ has image in $T$ which joins a component of
$T-V_{0}$ to $V_{0}$, or it has image in $T$ which joins components of
$T-V_{0}$, one component labelled $X$ and the other labelled $X^{\ast}$. In
either case, our construction implies that $x\lambda$ must contain some edge
of $E$.
\end{proof}

\begin{remark}
\label{extensioniscanonical}The above construction of an extension
$\overline{X}$ of $X$ depends on various choices. Firstly there is the choice
of the finite set $A$ which together with the $S_{j}$'s generates $G$.
Secondly there is the choice of an enlargement $\widehat{X}$ for $X$. The
extension is independent of the choice of $A$, but is clearly very dependent
on the choice of $\widehat{X}$. Thus there may be several extensions of $X$
which can be constructed in this way. However part 4) of Remark
\ref{enlargements} shows that this can only happen when there is $g$ in $G$,
and $S_{j}$ in $\mathcal{S}$, such that $gS_{j}$ is $H$--finite.
\end{remark}

\section{Extensions and Crossing\label{extensionsandcrossing}}

In this section, we discuss crossing of almost invariant subsets of a group
which need not be finitely generated. We recall the usual definition of
crossing. If $X$ and $Y$ are sets, the four intersection sets $X\cap Y$,
$X^{\ast}\cap Y$, $X\cap Y^{\ast}$, $X^{\ast}\cap Y^{\ast}$ are called the
\textit{corners} of the pair $(X,Y)$.

\begin{definition}
\label{defnofcrossing}Let $G$ be a group with subgroups $H\ $and $K$. Let $X$
be a $H$--almost invariant subset of $G$, and let $Y$ be a $K$--almost
invariant subset of $G$.

Then $X$ \textsl{crosses} $Y$ if all four corners of the pair $(X,Y)$ are $K$--infinite.
\end{definition}

Recall that if $G$ is finitely generated, but $H$ and $K$ need not be, then
this definition is symmetric for nontrivial almost invariant subsets. This is
Lemma 2.3 of \cite{Scott symmint}. Thus if $X$ and $Y\ $are each nontrivial,
then $X$ crosses $Y$ if and only $Y$ crosses $X$. The proof of this depends on
the coboundary, and shows that if a corner of the pair $(X,Y)$ is $K$--finite,
then it is also $H$--finite. For later use, we give a name to this property as follows.

\begin{definition}
\label{defnofcornersymmetric}If $F$ is a family of nontrivial almost invariant
subsets of a group $G$, we will say that crossing of elements of $F$ is
\emph{corner symmetric} if for any $H$-almost invariant set $X$ in $F$, and
$K$-almost invariant set $Y$ in $F$, a corner of the pair $(X,Y)$ is
$H$--finite if and only if it is $K$--finite.
\end{definition}

We have just seen that if crossing of elements of $F$ is corner symmetric,
then it is symmetric. We do not know if the converse is always true. Thus it
seems conceivable that there are situations where crossing of elements of $F$
could be symmetric but not corner symmetric.

The symmetry of crossing when $G$ is finitely generated allows us to make the
following very convenient definition.

\begin{definition}
\label{defnofsmallcorner} Let $G$ be a finitely generated group with subgroups
$H\ $and $K$. Let $X$ be a nontrivial $H$--almost invariant subset of $G$, and
let $Y$ be a nontrivial $K$--almost invariant subset of $G$. Then $X\cap Y$ is
\emph{small, if it} is $K$--finite.
\end{definition}

It seems unlikely that the symmetry of crossing can hold for arbitrary groups
$G$. But we will show that this symmetry holds when one has a group system
$(G,\mathcal{S})$ of finite type, and considers only almost invariant subsets
of $G$ which are $\mathcal{S}$--adapted. As each $S_{j}$ is countable, there
is a finitely generated group $G_{j}$ which contains $S_{j}$. We form a new
group $\overline{G}$ by amalgamating $G$ with $G_{j}$ over $S_{j}$, for each
value of $j$. Thus $\overline{G}$ is the fundamental group of a graph $\Gamma$
of groups which has one vertex $V$ with associated group $G$, and has edges
each with one end at $V$ and the other end at a vertex $V_{j}$ with associated
group $G_{j}$. The edge joining $V$ to $V_{j}$ has associated group $S_{j}$.
As $G$ is finitely generated relative to $\mathcal{S}$, it follows that
$\overline{G}$ is finitely generated. We will use this construction frequently
in what follows, and will refer to any group $\overline{G}$ constructed in
this way as a \textit{standard extension} of $G$.

The following result is what we need about the symmetry of crossing.

\begin{lemma}
\label{crossingissymmetric}Let $(G,\mathcal{S})$ be a group system of finite
type, and let $H$ and $K$ be subgroups of $G$. Let $X$ be a nontrivial
$H$--almost invariant subset of $G$, and let $Y$ be a nontrivial $K$--almost
invariant subset of $G$, and suppose that $X$ and $Y$ are $\mathcal{S}%
$--adapted. Then $X$ crosses $Y$ if and only if $Y$ crosses $X$.
\end{lemma}

\begin{remark}
In this result, none of $G$, $H$ or $K$ is assumed to be finitely generated.
As in the case when $G\ $is finitely generated, the argument below shows that
crossing is symmetric by showing it is corner symmetric.
\end{remark}

\begin{proof}
We will show that if $Y$ does not cross $X$, then $X$ does not cross $Y$. By
reversing the roles of $X$ and $Y$, the same argument will show that if $X$
does not cross $Y$, then $Y$ does not cross $X$. Taken together, these results
prove the lemma.

Let $\Gamma(G,\mathcal{S})$ be a relative Cayley graph for the group system
$(G,\mathcal{S})$, and let $\overline{G}$ be a standard extension of $G$, as
described above. Thus $\overline{G}$ is finitely generated. If $\widehat{X}$
is an enlargement for $X$ in $\Gamma(G,\mathcal{S})$ such that $\delta
\widehat{X}$ is $H$--finite, then Lemma \ref{extensionsexist} shows how to
construct an extension $\overline{X}$ of $X$ to $\overline{G}$, and the same
applies to finding an extension $\overline{Y}$ of $Y$.

Now suppose that $Y$ does not cross $X$. Thus (at least) one corner of the
pair $(X,Y)$ is $H$--finite. Without loss of generality we can assume that
$X\cap Y$ is $H$--finite. Now we claim

\begin{claim}
There are extensions $\overline{X}$ and $\overline{Y}$ of $X$ and $Y$
respectively such that $\overline{X}\cap\overline{Y}$ is equal to $X\cap Y$,
and so is also $H$--finite.
\end{claim}

Assuming this, the symmetry of crossing of almost invariant subsets of the
finitely generated group $\overline{G}$ implies that $\overline{X}%
\cap\overline{Y}$ is $K$--finite. It follows that $X\cap Y$ is also
$K$--finite, so that $X$ does not cross $Y$, as required.

It remains to prove our claim. We will make careful choices of enlargements
$\widehat{X}$ and $\widehat{Y}$ for $X$ and $Y$ respectively.

First we choose the enlargement $\widehat{Y}$. Recall from Remark
\ref{enlargements} that if we want $\delta\widehat{Y}$ to be $K$--finite, the
only choices available occur when there is $g$ in $G\ $and $S_{j}$ in
$\mathcal{S}$ such that the coset $gS_{j}$ is $K$--finite.

If $gS_{j}$ is $K$--finite and contained in $Y$, we choose the cone point
$gS_{j}$ to lie in $\widehat{Y}$. Thus $\delta\widehat{Y}$ contains no cone
edge incident to $gS_{j}$.

If $gS_{j}$ is $K$--finite and contained in $Y^{\ast}$, we choose the cone
point $gS_{j}$ not to lie in $\widehat{Y}$. Again this ensures that
$\delta\widehat{Y}$ contains no cone edge incident to $gS_{j}$.

If $gS_{j}$ is $K$--finite and meets both $Y$ and $Y^{\ast}$, we choose the
cone point $gS_{j}$ not to lie in $\widehat{Y}$. As the number of such
$gS_{j}$ is itself $K$--finite, it follows that $\delta\widehat{Y}$ is
$K$--finite, as required.

Next we choose the enlargement $\widehat{X}$. As above, the only choices
available occur when there is $g$ in $G\ $and $S_{j}$ in $\mathcal{S}$ such
that the coset $gS_{j}$ is $H$--finite.

If $gS_{j}$ is $H$--finite and contained in $X^{\ast}$, we choose the cone
point $gS_{j}$ not to lie in $\widehat{X}$. Thus $\delta\widehat{X}$ contains
no cone edge incident to $gS_{j}$.

If $gS_{j}$ is $H$--finite and meets both $X$ and $X^{\ast}$, we choose the
cone point $gS_{j}$ not to lie in $\widehat{X}$. Such cones contribute only
$H$--finitely many edges to $\delta\widehat{X}$.

If $gS_{j}$ is $H$--finite, is contained in $X$, and does not meet $Y$, we
choose the cone point $gS_{j}$ to lie in $\widehat{X}$. Thus $\delta
\widehat{X}$ contains no cone edge incident to $gS_{j}$.

Finally, if $gS_{j}$ is $H$--finite, is contained in $X$, and meets $Y$, we
choose the cone point $gS_{j}$ not to lie in $\widehat{X}$. For each such
coset $gS_{j}$, all the cone edges incident to $gS_{j}$ lie in $\delta
\widehat{X}$. There are $H$--finitely many such cone edges. As $X\cap Y$ is
assumed to be $H$--finite, the number of such cosets is also $H$--finite, so
that in total such cosets contribute only $H$--finitely many cone edges to
$\delta\widehat{X}$. Hence $\delta\widehat{X}$ is $H$--finite, as required.

Now consider the construction of extensions given in the proof of Lemma
\ref{extensionsexist}. Our choice of the enlargement $\widehat{X}$ implies
that if $gS_{j}$ is $H$--finite and meets both $X$ and $X^{\ast}$, then the
corresponding edge of $T$ incident to $V_{0}$ will be labelled by $X^{\ast}$.
Similarly our choice of $\widehat{Y}$ implies that if $gS_{j}$ is $K$--finite
and meets both $Y$ and $Y^{\ast}$, then the corresponding edge of $T$ incident
to $V_{0}$ will be labelled by $Y^{\ast}$. The proof of Lemma
\ref{extensionsexist} shows that $\overline{X}\cap\overline{Y}$ is equal to
$X\cap Y$ if and only if, after labelling edges of the $\overline{G}$--tree
$T$ by $X$ or $X^{\ast}$ and by $Y$ or $Y^{\ast}$, no edge of $T$ incident to
$V_{0}$ is labelled by $X$ and by $Y$. We will suppose that an edge $e$ of $T$
is so labelled and will obtain a contradiction. Suppose that $e$ corresponds
to the cone with cone point $gS_{j}$.

Our choice of $\widehat{X}$ combined with the fact that $e$ is labelled $X$
means that one of the following two conditions must hold:

\begin{enumerate}
\item (1X) $gS_{j}$ is $H$--infinite, and $gS_{j}\cap X^{\ast}$ is
$H$--finite, or

\item (2X) $gS_{j}$ is $H$--finite and contained in $X$.
\end{enumerate}

Similarly our choice of $\widehat{Y}$ combined with the fact that $e$ is
labelled $Y$ mean that one of the following two conditions must hold:

\begin{enumerate}
\item (1Y) $gS_{j}$ is $K$--infinite, and $gS_{j}\cap Y^{\ast}$ is
$K$--finite, or

\item (2Y) $gS_{j}$ is $K$--finite and contained in $Y$.
\end{enumerate}

We now examine the four cases arising from these conditions. In all cases, we
have $gS_{j}\cap X^{\ast}$ is $H$--finite and $gS_{j}\cap Y^{\ast}$ is
$K$--finite. As $X\cap Y$ is $H$--finite, it is trivial that $gS_{j}\cap(X\cap
Y)$ is also $H$--finite. Hence $gS_{j}$ is the union of a $H$--finite set with
a $K$--finite set. Now Lemma \ref{Neumannlemma} implies that $gS_{j}$ must be
either $H$--finite or $K$--finite. Thus conditions (1X)\ and (1Y) cannot both hold.

Suppose that conditions (1X) and (2Y) hold. Condition (1X) tells us that
$gS_{j}$ is $H$--infinite, and $gS_{j}\cap X^{\ast}$ is $H$--finite, so that
$gS_{j}\cap X$ must be $H$--infinite. Condition (2Y) tells us that $gS_{j}$ is
contained in $Y$. But then $gS_{j}\cap X=gS_{j}\cap(X\cap Y)$ is
$H$--infinite, whereas $X\cap Y$ is $H$--finite. This contradiction shows that
this case cannot occur.

Suppose that conditions (1Y) and (2X) hold. Condition (1Y) implies that
$gS_{j}$ is $K$--infinite, and $gS_{j}\cap Y^{\ast}$ is $K$--finite. In
particular $gS_{j}$ must meet $Y$. Condition (2X) tells us that $gS_{j}$ is
$H$--finite and contained in $X$. But now our choice of $\widehat{X}$ means
that the cone point $gS_{j}$ does not lie in $\widehat{X}$, so that the edge
$e$ must be labeled by $X^{\ast}$. This contradiction shows that this case
cannot occur.

Finally suppose that conditions (2X) and (2Y) hold. This implies that $gS_{j}$
is $H$--finite and is contained in $X$ and in $Y$. Again our choice of
$\widehat{X}$ implies that $e$ is labelled by $X^{\ast}$. This contradiction
shows that this case cannot occur.

These contradictions complete the proof of the claim, and hence complete the
proof of the lemma.
\end{proof}

We remark that the construction in the above proof of Lemma
\ref{crossingissymmetric} is so special, that it would not be possible to
apply it when one has a family of several adapted almost invariant subsets of
$G$ and wishes to extend them all simultaneously. In particular if the family
is invariant under some group, we would like the family of extensions to be
invariant under the same group. In order to obtain such results, we will now
restrict attention to almost invariant sets over finitely generated groups.

Suppose again that $(G,\mathcal{S})$ is a group system of finite type. If a
group $S$ in $\mathcal{S}$ is finitely generated, then $G$ must be finitely
generated relative to the subfamily of $\mathcal{S}$ obtained by removing $S$.
In fact the same condition holds if $S$ is contained in a finitely generated
subgroup of $G$. Let $\mathcal{S}^{\infty}$ denote the subfamily of
$\mathcal{S}$ which consists of all those groups in $\mathcal{S}$ which are
not contained in a finitely generated subgroup of $G$, and let $\mathcal{S}%
^{\prime}$ denote any subfamily of $\mathcal{S}$ which contains $\mathcal{S}%
^{\infty}$. Then $G$ is finitely generated relative to $\mathcal{S}^{\prime}$.
(In particular, if the family $\mathcal{S}^{\infty}$ is empty, then $G$ is
already finitely generated.) We will want to consider finitely generated
standard extensions $\overline{G}$ of $G$, using families which contain
$\mathcal{S}^{\infty}$. For us the two most important examples of such
families will be $\mathcal{S}^{\infty}$ itself and the family $\mathcal{S}%
_{nonfg}$ of all non-finitely generated groups in $\mathcal{S}$. As such
extensions will be used often in what follows, we set out the following definition.

\begin{definition}
\label{defnofSinfinitystandardextension} If $(G,\mathcal{S})$ is a group
system of finite type, then $\mathcal{S}^{\infty}$ denotes the subfamily of
$\mathcal{S}$ which consists of all those groups in $\mathcal{S}$ which are
not contained in a finitely generated subgroup of $G$, and $\mathcal{S}%
_{nonfg}$ denotes the subfamily of all non-finitely generated groups in
$\mathcal{S}$. If $\mathcal{S}^{\prime}$ is any subfamily of $\mathcal{S}$
which contains $\mathcal{S}^{\infty}$, then a \emph{standard }$S^{\prime}%
$\emph{--extension of }$G$ is a finitely generated standard extension
$\overline{G}$ of $G$, defined using the family $\mathcal{S}^{\prime}$. Thus
$\overline{G}$ is constructed by amalgamating $G$ with some finitely generated
group $A_{i}$ over $S_{i}$, for each $S_{i}\in\mathcal{S}^{\prime}$.
\end{definition}

\begin{remark}
If $(G,\mathcal{S})$ is a group system of finite type, then so is
$(G,\mathcal{S}^{\prime})$.
\end{remark}

The reason why such standard extensions are important, is contained in the
following result.

\begin{lemma}
\label{uniqueenlargement}Let $(G,\mathcal{S})$ be a group system of finite
type. Let $H$ be a finitely generated subgroup of $G$, and let $X$ be a
nontrivial $\mathcal{S}$--adapted $H$--almost invariant subset of $G$. Let
$\mathcal{S}^{\prime}$ denote a subfamily of $\mathcal{S}$ which contains
$\mathcal{S}^{\infty}$, such that no $S_{i}$ in $\mathcal{S}^{\prime}$ is
conjugate commensurable with a subgroup of $H$. Then there is only one
enlargement $\widehat{X}$ for $X$ in $\Gamma(G,\mathcal{S}^{\prime})$ such
that $\delta\widehat{X}$ is $H$--finite.
\end{lemma}

\begin{remark}
If $\mathcal{S}^{\prime}$ is equal to $\mathcal{S}^{\infty}$, it is automatic
that no $S_{i}$ in $\mathcal{S}^{\prime}$ is conjugate commensurable with a
subgroup of $H$, as $H$ is finitely generated.

If $H$ is Noetherian (or slender), i.e. every subgroup is finitely generated,
and if $\mathcal{S}^{\prime}$ is equal to $\mathcal{S}_{nonfg}$, it is
automatic that no $S_{i}$ in $\mathcal{S}^{\prime}$ is conjugate commensurable
with a subgroup of $H$.

The most important examples of Noetherian groups in this paper are $VPC$ groups.
\end{remark}

\begin{proof}
Trivially $X$ is adapted to $\mathcal{S}^{\prime}$. As $(G,\mathcal{S}%
^{\prime})$ is also a group system of finite type, it makes sense to consider
enlargements of $X$ in the graph $\Gamma(G,\mathcal{S}^{\prime})$. Recall from
part 4) of Remark \ref{enlargements} that, if there is no $g$ in $G$ and
$S_{i}$ in $\mathcal{S}^{\prime}$ such that $gS_{i}$ is $H$--finite, then
there is only one enlargement $\widehat{X}$ for $X$ in $\Gamma(G,\mathcal{S}%
^{\prime})$ such that $\delta\widehat{X}$ is $H$--finite.

Suppose there is $g$ in $G$ and $S_{i}$ in $\mathcal{S}^{\prime}$ such that
$gS_{i}$ is $H$--finite. Then $gS_{i}g^{-1}$ is also $H$--finite. Hence Lemma
\ref{Neumannlemma} tells us that $gS_{i}g^{-1}$ is commensurable with a
subgroup of $H$. As this contradicts our hypothesis about groups in
$\mathcal{S}^{\prime}$, we conclude that for any $g$ in $G\ $and $S_{i}$ in
$\mathcal{S}^{\prime}$, the coset $gS_{i}$ cannot be $H$--finite, which
completes the proof of the lemma.
\end{proof}

Using the notation of the above lemma, if $\overline{G}$ is a standard
$\mathcal{S}^{\prime}$--extension of $G$, then Lemma \ref{extensionsexist}
combined with Remark \ref{extensioniscanonical} implies that $X$ has a
canonical extension $\overline{X}$ in $\overline{G}$. We will prove several
important properties of adapted almost invariant sets over finitely generated
groups by using such standard extensions of $G$ and canonical extensions of
almost invariant subsets of $G$. First we state the following basic lemma.

\begin{lemma}
\label{extensionshavesamesmallcorner}Let $(G,\mathcal{S})$ be a group system
of finite type. Let $H$ and $K$ be finitely generated subgroups of $G$, and
let $X$ and $Y$ be nontrivial $\mathcal{S}$--adapted almost invariant subsets
of $G$ over $H$ and $K$ respectively. Let $\mathcal{S}^{\prime}$ denote a
subfamily of $\mathcal{S}$ which contains $\mathcal{S}^{\infty}$, such that no
$S_{i}$ in $\mathcal{S}^{\prime}$ is conjugate commensurable with a subgroup
of $H$ or of $K$. Let $\overline{G}$ be a standard $\mathcal{S}^{\prime}%
$--extension of $G$.

If a corner $U$ of the pair $(X,Y)$ is $H$--finite, then the corresponding
corner of the pair $(\overline{X},\overline{Y})$ is equal to $U$ (and so is
also $H$--finite).
\end{lemma}

\begin{proof}
As $\overline{G}$ is a standard $\mathcal{S}^{\prime}$--extension of $G$,
Lemma \ref{uniqueenlargement} tells us there is only one enlargement for each
of $X$ and $Y$. Now it follows from the proof of Lemma
\ref{crossingissymmetric} that if a corner $U$ of the pair $(X,Y)$ is
$H$--finite, then the corresponding corner of the pair $(\overline
{X},\overline{Y})$ is equal to $U$.

Of course the situation here is much simpler than that in Lemma
\ref{crossingissymmetric}, so we now give a simpler direct argument that if
the corner $X\cap Y$ is $H$--finite, then $\overline{X}\cap\overline{Y}$
equals $X\cap Y$. As in the proof of Lemma \ref{crossingissymmetric},
inequality can only occur if there is an edge $e$ incident to $V_{0}$ which is
labelled by $X$ and by $Y$. We will suppose such an edge $e$ exists and will
obtain a contradiction. Let $gS_{i}$ denote the coset corresponding to $e$.
The fact that $e$ is labelled by $X$ means that $gS_{i}\cap X^{\ast}$ is
$H$--finite. Similarly, the fact that $e$ is labelled by $Y$ means that
$gS_{i}\cap Y^{\ast}$ is $K$--finite. But $gS_{i}\cap(X\cap Y)$ is trivially
$H$--finite, so it follows that $gS_{i}$ is the union of a $H$--finite set and
a $K$--finite set. Hence the same holds for $gS_{i}g^{-1}$. Now Lemma
\ref{Neumannlemma} implies that $gS_{i}g^{-1}$ is commensurable with a
subgroup of $H$ or of $K$, which contradicts our hypotheses on $\mathcal{S}%
^{\prime}$. This contradiction completes the proof that $\overline{X}%
\cap\overline{Y}$ equals $X\cap Y$.
\end{proof}

The following basic property makes clear the importance of such standard extensions.

\begin{lemma}
\label{extensionsareequivalent}Let $(G,\mathcal{S})$ be a group system of
finite type. Let $H$ be a finitely generated subgroup of $G$, and let $X$ be a
nontrivial $\mathcal{S}$--adapted $H$--almost invariant subset of $G$. Let
$\mathcal{S}^{\prime}$ denote a subfamily of $\mathcal{S}$ which contains
$\mathcal{S}^{\infty}$, such that no $S_{i}$ in $\mathcal{S}^{\prime}$ is
conjugate commensurable with a subgroup of $H$. Let $\overline{G}$ be a
standard $\mathcal{S}^{\prime}$--extension of $G$.

If $X$ is equivalent to a $K$--almost invariant subset $Y$ of $G$, then their
canonical extensions $\overline{X}$ and $\overline{Y}$ are also equivalent.
\end{lemma}

\begin{proof}
As $X$ is equivalent to $Y$, we know that $X\cap Y^{\ast}$ and $X^{\ast}\cap
Y$ are both $H$--finite. Now Lemma \ref{extensionshavesamesmallcorner} tells
us that $\overline{X}\cap\overline{Y}^{\ast}$ and $\overline{X}^{\ast}%
\cap\overline{Y}$ are also both $H$--finite, so that $\overline{X}$ and
$\overline{Y}$ are equivalent, as required.
\end{proof}

Another basic property is that the canonical extensions of $X$ constructed
using such extensions of $G$ are automatically adapted to $\mathcal{S}$, but
in fact much more is true.

\begin{lemma}
\label{extensionsareadaptedtoAunionS}Let $(G,\mathcal{S})$ be a group system
of finite type. Let $H$ be a finitely generated subgroup of $G$, and let $X$
be a nontrivial $\mathcal{S}$--adapted $H$--almost invariant subset of $G$.
Let $\mathcal{S}^{\prime}$ denote a subfamily of $\mathcal{S}$ which contains
$\mathcal{S}^{\infty}$, such that no $S_{i}$ in $\mathcal{S}^{\prime}$ is
conjugate commensurable with a subgroup of $H$. Let $\overline{G}$ be a
standard $\mathcal{S}^{\prime}$--extension of $G$, and let $\mathcal{A}$
denote the family of all the $A_{j}$'s.

If $\overline{X}$ denotes the canonical extension of $X$ to $\overline{G}$,
then $\overline{X}$ is adapted to $\mathcal{A}\cup\mathcal{S}$.
\end{lemma}

\begin{proof}
For each $S_{i}\in\mathcal{S}$, there is a $H_{i}$--almost invariant subset
$X_{i}$ of $G$ which is strictly adapted to $S_{i}$ and is equivalent to $X$.
We will show that the extension $\overline{X_{i}}$ of $X_{i}$ is also strictly
adapted to $S_{i}$, so that, for all $g$ in $\overline{G}$, we have $gS_{i}$
contained in $\overline{X_{i}}$ or its complement. If $g$ lies in $G$, this is
clear, as we know that $gS_{i}$ is contained in $X_{i}$ or its complement.
Otherwise $g$ acts on the tree $T$ so that $gV_{0}$ is some other vertex $W$,
and $g$ lies in $\overline{X_{i}}$ or its complement according to the color of
the subtree which contains $W$. As $gsV_{0}=W$, for all $s$ in $S_{i}$, it
follows that $gS_{i}$ itself is contained in $\overline{X_{i}}$ or its complement.

As $A_{i}$ fixes a vertex of $T$ not in the orbit of $V_{0}$, it follows that
for all $g$ in $\overline{G}$, we have $gA_{i}$ lies in $\overline{X_{i}}$ or
its complement, so that $\overline{X_{i}}$ is strictly adapted to $A_{i}$. The
preceding lemma tells us that $\overline{X}$ is equivalent to each
$\overline{X_{i}}$. It follows that $\overline{X}$ is adapted to each $A_{i}$
and each $S_{i}$, so that $\overline{X}$ is adapted to $\mathcal{A}%
\cup\mathcal{S}$, as required.
\end{proof}

We recall that two subsets $U\ $and $V$ of $G$ are said to be \textit{nested}
if one of the four corners of the pair $(U,V)$ is empty.

\begin{lemma}
\label{finitenumberofdoublecosets}Let $(G,\mathcal{S})$ be a group system of
finite type, and let $H$ and $K$ be finitely generated subgroups of $G$. Let
$X$ be a nontrivial $H$--almost invariant subset of $G$, and let $Y$ be a
nontrivial $K$--almost invariant subset of $G$, and suppose that $X$ and $Y$
are $\mathcal{S}$--adapted. Then $\{g\in G:gX$ and $Y$ are not nested$\}$ is a
finite union of double cosets $KgH$, $g\in G$.
\end{lemma}

\begin{proof}
We consider a finitely generated standard $\mathcal{S}^{\infty}$--extension
$\overline{G}$ of $G$, and let $\overline{X}$ and $\overline{Y}$ denote the
canonical extensions of $X$ and $Y$ to almost invariant subsets of
$\overline{G}$. Then, for each $g$ in $G$, the canonical extension of $gX$
must be $g\overline{X}$. It is trivial that if $gX$ and $Y$ are not nested,
then their extensions $g\overline{X}$ and $\overline{Y}$ are also not nested.
As $\overline{G}$, $H\ $and $K$ are all finitely generated, Lemma 1.11 of
\cite{SS1} tells us that $\{g\in\overline{G}:g\overline{X}$ and $\overline{Y}$
are not nested$\}$ is a finite union of double cosets $KgH$, $g\in\overline
{G}$. Thus $\{g\in G:gX$ and $Y$ are not nested$\}$ is contained in a finite
union of double cosets $KgH$, $g\in\overline{G}$. As $H\ $and $K$ are
subgroups of $G$, such a double coset can meet $G$ if and only if $g$ lies in
$G$. It follows that $\{g\in G:gX$ and $Y$ are not nested$\}$ is contained in
a finite union of double cosets $KgH$, $g\in G$, and hence equals such a
union, as required.
\end{proof}

The following is an immediate consequence of Lemma
\ref{finitenumberofdoublecosets}.

\begin{corollary}
\label{finiteintersectionnumber}Let $(G,\mathcal{S})$ be a group system of
finite type, and let $H$ and $K$ be finitely generated subgroups of $G$. Let
$X$ be a nontrivial $H$--almost invariant subset of $G$, and let $Y$ be a
nontrivial $K$--almost invariant subset of $G$, and suppose that $X$ and $Y$
are $\mathcal{S}$--adapted. Then the intersection number $i(H\backslash
X,K\backslash Y)$ is finite.
\end{corollary}

However, we need to refine the argument of Lemma
\ref{finitenumberofdoublecosets} for later applications. Notice that the proof
of Lemma \ref{finitenumberofdoublecosets} shows that $i(H\backslash
X,K\backslash Y)\leq i(H\backslash\overline{X},K\backslash\overline{Y})$, but
does not show that they are equal. As this equality will be a crucial fact in
what comes later, we prove the following result.

\begin{lemma}
\label{crossingofextensions}Let $(G,\mathcal{S})$ be a group system of finite
type, and let $H$ and $K$ be finitely generated subgroups of $G$. Let $X$ be a
nontrivial $H$--almost invariant subset of $G$, and let $Y$ be a nontrivial
$K$--almost invariant subset of $G$, and suppose that $X$ and $Y$ are
$\mathcal{S}$--adapted. Let $\mathcal{S}^{\prime}$ denote a subfamily of
$\mathcal{S}$ which contains $\mathcal{S}^{\infty}$, such that no $S_{i}$ in
$\mathcal{S}^{\prime}$ is conjugate commensurable with a subgroup of $H$ or of
$K$. Let $\overline{G}$ be a standard $\mathcal{S}^{\prime}$--extension of
$G$, and let $\overline{X}$ and $\overline{Y}$ denote the canonical extensions
of $X$ and $Y$ to $\overline{G}$.

Then the following statements hold.

\begin{enumerate}
\item If $g$ is an element of $G$, then $X$ and $gY$ cross if and only if
$\overline{X}$ and $g\overline{Y}$ also cross.

\item If $g$ is an element of $\overline{G}-G$, then $\overline{X}$ and
$g\overline{Y}$ do not cross.
\end{enumerate}
\end{lemma}

\begin{proof}
1) If $g$ is an element of $G$ such that $X\ $and $gY$ cross, then each of the
four corners of the pair $(X,gY)$ must be $H$--infinite (and also
$K$--infinite). As $\overline{X}$ contains $X$ and $g\overline{Y}$ contains
$gY$, it is immediate that each of the four corners of the pair $(\overline
{X},g\overline{Y})$ must be $H$--infinite. Thus $\overline{X}$ and
$g\overline{Y}$ do cross, as required.

If $X\ $and $gY$ do not cross, then one of the four corners of the pair
$(X,gY)$ must be $H$--finite (and also $K$--finite). As in the proof of Lemma
\ref{extensionsareequivalent}, it follows that the corresponding corner of the
pair $(\overline{X},g\overline{Y})$ is also $H$--finite, so that $\overline
{X}$ and $g\overline{Y}$ do not cross.

2) Recall that $\overline{G}$ acts on the tree $T$ and that the vertex $V_{0}$
has stabilizer $G$ and encloses $\overline{X}$ and $\overline{Y}$. Thus
$g\overline{Y}$ is enclosed by $gV_{0}$. If $g$ is an element of $\overline
{G}-G$, then $gV_{0}$ and $V_{0}$ are distinct vertices of $T$, so that
$\overline{X}$ and $g\overline{Y}$ cannot cross.
\end{proof}

The following is an immediate consequence of Lemma \ref{crossingofextensions}.

\begin{corollary}
\label{intersectionnumbersofextensionsareequal}Let $(G,\mathcal{S})$ be a
group system of finite type, and let $H$ and $K$ be finitely generated
subgroups of $G$. Let $X$ be a nontrivial $H$--almost invariant subset of $G$,
and let $Y$ be a nontrivial $K$--almost invariant subset of $G$, and suppose
that $X$ and $Y$ are $\mathcal{S}$--adapted. Let $\mathcal{S}^{\prime}$ denote
a subfamily of $\mathcal{S}$ which contains $\mathcal{S}^{\infty}$, such that
no $S_{i}$ in $\mathcal{S}^{\prime}$ is conjugate commensurable with a
subgroup of $H$ or of $K$. Let $\overline{G}$ be a standard $\mathcal{S}%
^{\prime}$--extension of $G$, and let $\overline{X}$ and $\overline{Y}$ denote
the canonical extensions of $X$ and $Y$ to $\overline{G}$. Then the
intersection number $i(H\backslash X,K\backslash Y)$ equals $i(H\backslash
\overline{X},K\backslash\overline{Y})$. Further if $X$ is equivalent to a
$H^{\prime}$--almost invariant set $X^{\prime}$, and $Y\ $is equivalent to a
$K^{\prime}$--almost invariant set $\allowbreak Y^{\prime}$, then
$i(H\backslash X,K\backslash Y)=i(H^{\prime}\backslash X^{\prime},K^{\prime
}\backslash Y^{\prime})$.
\end{corollary}

We will now give a couple of applications of this corollary to demonstrate
that $\mathcal{S}^{\infty}$--extensions are a powerful tool for reducing
questions about group systems of finite type to questions about finitely
generated groups. We will use this idea on several occasions in this paper. In
Theorem 2.8 of \cite{SS1}, Scott and Swarup proved that if $G$ is a finitely
generated group with a finitely generated subgroup $H$, and if $X$ is a
nontrivial $H$--almost invariant subset of $G$ such that the self-intersection
number $i(H\backslash X,H\backslash X)$ equals $0$, then $X$ is equivalent to
an almost invariant set $Y$ which is associated to a splitting. We can
generalise this result to the setting of group systems of finite type, as follows.

\begin{theorem}
\label{splittingsexist}Let $(G,\mathcal{S})$ be a group system of finite type,
and let $H$ be a finitely generated subgroup of $G$. Let $X$ be a nontrivial
$\mathcal{S}$--adapted $H$--almost invariant subset of $G$. If the
self-intersection number $i(H\backslash X,H\backslash X)$ equals $0$, then $G$
has a $\mathcal{S}$--adapted splitting over some subgroup $H^{\prime}$
commensurable with $H$. Further each of the $H^{\prime}$--almost invariant
sets associated to the splitting is equivalent to $X$.
\end{theorem}

\begin{proof}
We consider a standard $\mathcal{S}^{\infty}$--extension $\overline{G}$ of
$G$, and consider the canonical extension $\overline{X}$ of $X$ to an almost
invariant subset of $\overline{G}$. Corollary
\ref{intersectionnumbersofextensionsareequal} tells us that $i(H\backslash
\overline{X},H\backslash\overline{X})=i(H\backslash X,H\backslash X)$ which
equals zero. As $\overline{G}$ is finitely generated, Theorem 2.8 of
\cite{SS1} tells us that $\overline{X}$ is equivalent to a $H^{\prime\prime}%
$--almost invariant set $\overline{Y}$ which is associated to a splitting of
$\overline{G}$ over $H^{\prime\prime}$. As $\overline{X}$ and $\overline{Y}$
are equivalent, $H^{\prime\prime}$ is commensurable with $H$. Let $Y$ denote
$\overline{Y}\cap G$ and let $H^{\prime}$ denote $H^{\prime\prime}\cap G$.
Then $Y$ is a $H^{\prime}$--almost invariant subset of $G$ which is equivalent
to $X$. As $\overline{Y}$ is associated to a splitting, its translates by
$\overline{G}$ are nested. It follows that the translates of $Y$ by $G$ are
also nested, so that $Y$ also is associated to a splitting of $G$. As $Y$ is
equivalent to $X$, it must be $\mathcal{S}$--adapted. It follows that the
splitting of $G$ to which $Y\ $is associated is also $\mathcal{S}$--adapted,
as required.
\end{proof}

In the same way, we can also generalise Theorem 2.5 of \cite{SS1} to the
setting of group systems of finite type, as follows.

\begin{theorem}
\label{splittingsarecompatible}Let $(G,\mathcal{S})$ be a group system of
finite type, such that $G$ has a finite number of $\mathcal{S}$--adapted
splittings over finitely generated subgroups, any two of which have
intersection number zero. Then this family of splittings is compatible up to
conjugacy. In particular, the almost invariant subsets of $G$ associated to
the splittings can be chosen so that the family of all translates is nested.
\end{theorem}

\begin{proof}
For $1\leq i\leq n$, we let $X_{i}$ denote a $H_{i}$--almost invariant set
associated to the $i$--th given splitting. Now consider a standard
$\mathcal{S}^{\infty}$--extension $\overline{G}$ of $G$, and, for each $i$,
consider the canonical extension $\overline{X_{i}}$ of $X_{i}$ to an almost
invariant subset of $\overline{G}$. As $X_{i}$ is associated to a splitting of
$G$, the family of translates by $G$ of $X_{i}$ is nested. By Lemma
\ref{extensionshavesamesmallcorner}, the family of translates by $\overline
{G}$ of $\overline{X_{i}}$ is also nested, so that $\overline{X_{i}}$ is
associated to a splitting of $\overline{G}$. Corollary
\ref{intersectionnumbersofextensionsareequal} tells us that any pair of the
$\overline{X_{i}}$'s has intersection number zero. As $\overline{G}$ is
finitely generated, Theorem 2.5 of \cite{SS1} tells us that $\overline{X_{i}}$
is equivalent to an almost invariant set $\overline{Y_{i}}$, such that the
family of all translates by $\overline{G}$ of all of the $\overline{Y_{i}}$'s
is nested. Now let $Y_{i}$ denote the intersection $\overline{Y_{i}}\cap G$.
Then the family of all translates by $G$ of all of the $Y_{i}$'s is also
nested. As $\overline{X_{i}}$ and $\overline{Y_{i}}$ are equivalent, so are
$X_{i}$ and $Y_{i}$. Thus $Y_{i}$ is associated to the $i$--th given splitting
of $G$, proving the required result.
\end{proof}

For later use we will also need to define strong and weak crossing of almost
invariant sets in the setting of group systems of finite type. We recall the
definition given in section 3 of \cite{SS1} in the setting of a finitely
generated group $G$. For a subset $X$ of $G$, recall that $\delta X$ denotes
the coboundary of $X$ in the Cayley graph $\Gamma$ of $G$ with respect to some
finite generating set. Thus $\delta X$ is a collection of edges of $\Gamma$.
If $Y$ is a set of vertices of $\Gamma$, we use the notation $\delta X\cap Y$
to denote the collection of endpoints of edges of $\delta X$ which lie in $Y$.
Let $H$ and $K$ be subgroups of $G$, and let $X$ and $Y$ be nontrivial almost
invariant subsets of $G$ over $H$ and $K$ respectively. Then $X$ crosses $Y$
strongly if $\delta X\cap Y$ and $\delta X\cap Y^{\ast}$ are both
$K$--infinite. If $G$ is not finitely generated, then it no longer makes sense
to use $\delta X$ in any definitions. Instead we will give the following new definition.

\begin{definition}
\label{defnofstrongandweakcrossing}Let $(G,\mathcal{S})$ be a group system of
finite type, and let $H$ and $K$ be subgroups of $G$. Let $X$ be a nontrivial
$H$--almost invariant subset of $G$, and let $Y$ be a nontrivial $K$--almost
invariant subset of $G$, and suppose that $X$ and $Y$ are $\mathcal{S}%
$--adapted. Then $X$\textsl{ crosses }$Y$\textsl{ strongly} if $H\cap Y$ and
$H\cap Y^{\ast}$ are both $K$--infinite. And $X$\textsl{ crosses }$Y$\textsl{
weakly} if $X$ crosses $Y$, but $X\ $does not cross $Y$ strongly.
\end{definition}

If $G$ is finitely generated, which is automatic if the family $\mathcal{S}$
is empty, then this definition is equivalent to the one given in \cite{SS1}.
But it has the advantage of not requiring one to choose a generating set for
$G$. In particular, as discussed in \cite{SS1}, strong and weak crossing need
not be symmetric, i.e. if $X$ crosses $Y$ strongly, it need not be the case
that $Y$ crosses $X$ strongly. As in the case when $G$ is finitely generated,
strong crossing implies crossing.

\begin{lemma}
Let $(G,\mathcal{S})$ be a group system of finite type, and let $H$ and $K$ be
subgroups of $G$. Let $X$ be a nontrivial $H$--almost invariant subset of $G$,
and let $Y$ be a nontrivial $K$--almost invariant subset of $G$, and suppose
that $X$ and $Y$ are $\mathcal{S}$--adapted.

If $X$ crosses $Y$ strongly, then $X$ crosses $Y$.
\end{lemma}

\begin{proof}
As $X$ crosses $Y$ strongly, we know that $H\cap Y$ and $H\cap Y^{\ast}$ are
both $K$--infinite.

Let $x$ denote some element of $X$. As $HX=X$, it follows that $X$ contains
$Hx$. As $Y$ is $K$--almost invariant, $Yx^{-1}$ is $K$--almost equal to $Y$.
Hence $H\cap Yx^{-1}$ is $K$--almost equal to $H\cap Y$, and so is $K$--infinite.

Now $Hx\cap Y=(H\cap Yx^{-1})x$, so it follows that $Hx\cap Y$ is also
$K$--infinite. As $X$ contains $Hx$, this implies that $X\cap Y$ is $K$--infinite.

Similar arguments show that all the other corners of the pair $(X,Y)$ are
$K$--infinite, so that $X$ crosses $Y$, as claimed.
\end{proof}

We also note that strong crossing is preserved when taking extensions.

\begin{lemma}
Let $(G,\mathcal{S})$ be a group system of finite type, and let $H$ and $K$ be
subgroups of $G$. Let $X$ be a nontrivial $H$--almost invariant subset of $G$,
and let $Y$ be a nontrivial $K$--almost invariant subset of $G$, and suppose
that $X$ and $Y$ are $\mathcal{S}$--adapted. Let $\overline{G}$ be a standard
extension of $G$, and let $\overline{X}$ and $\overline{Y}$ denote extensions
of $X$ and $Y$ to $\overline{G}$. Then $X$ crosses $Y$ strongly if and only if
$\overline{X}$ crosses $\overline{Y}$ strongly.
\end{lemma}

\begin{proof}
If $X$ crosses $Y$ strongly, then $H\cap Y$ and $H\cap Y^{\ast}$ are both
$K$--infinite. It follows immediately that $H\cap\overline{Y}$ and
$H\cap\overline{Y}^{\ast}$ are both $K$--infinite, so that $\overline{X}$
crosses $\overline{Y}$ strongly.

Conversely, if $\overline{X}$ crosses $\overline{Y}$ strongly, then
$H\cap\overline{Y}$ and $H\cap\overline{Y}^{\ast}$ are both $K$--infinite. As
$H$ is contained in $G$, and $Y=\overline{Y}\cap G$, it follows that $H\cap Y$
and $H\cap Y^{\ast}$ are both $K$--infinite, so that $X$ crosses $Y$ strongly.
\end{proof}

\section{Good position and very good position\label{goodposition}}

In \cite{SS2}, Scott and Swarup introduced the terminology of good position
for a family of almost invariant subsets of a finitely generated group $G$,
but the idea can be traced back to Scott's paper \cite{Scott:torus}. Let
$\{H_{i}\}_{i\in I}$ be a family of subgroups of $G$, let $X_{i}$ be a $H_{i}%
$--almost invariant subset of $G$, and let $\mathcal{X}=\{X_{i}\}_{i\in I}$.
Finally let $E(\mathcal{X})$ denote $\{gX_{i},gX_{i}^{\ast}:g\in G,i\in I\}$.
Then the family $E(\mathcal{X})$ is in\textit{ good position} if whenever $U$
and $V$ are elements of $E(\mathcal{X})$, such that the pair $(U,V)\ $has two
small corners, then one of the corners is empty. (See Definition
\ref{defnofsmallcorner}.) The point of this condition is that it makes
possible the definition of a partial order of "almost inclusion" on
$E(\mathcal{X})$. If $U$ and $V$ are elements of $E(\mathcal{X})$, we would
like to say that $U\leq V$ if $U\cap V^{\ast}$ is small, i.e. that $U$ is
almost contained in $V$. Clearly there is a problem if the pair $(U,V)$ has
two small corners. But if $E(\mathcal{X})$ is in good position, and $U$ and
$V$ are elements of $E(\mathcal{X})$, we can define $U\leq V$ if $U\cap
V^{\ast}$ is either empty or the only small corner of the pair $(U,V)$. Lemma
1.14 of \cite{SS1} shows that the relation $\leq$ is a partial order on
$E(\mathcal{X})$. Note that the proof of Lemma 1.14 is valid even if some
$H_{i}$'s are not finitely generated. But it is crucial that $G$ be finitely generated.

As in the proof of Theorem \ref{splittingsexist}, we can use a standard
$\mathcal{S}^{\infty}$--extension $\overline{G}$ of $G$ to generalise Lemma
1.14 of \cite{SS1} to the setting of group systems of finite type. The result
is the following.

\begin{lemma}
\label{partialorderforsystems}Let $(G,\mathcal{S})$ be a group system of
finite type, let $\{H_{i}\}_{i\in I}$ be a family of subgroups of $G$, let
$X_{i}$ be a $\mathcal{S}$--adapted $H_{i}$--almost invariant subset of $G$,
and let $\mathcal{X}=\{X_{i}\}_{i\in I}$. Let $E(\mathcal{X})$ denote
$\{gX_{i},gX_{i}^{\ast}:g\in G,i\in I\}$, and suppose that $E(\mathcal{X})$ is
in good position. Define the relation $\leq$ on $E(\mathcal{X})$ by $U\leq V$
if $U\cap V^{\ast}$ is either empty or the only small corner of the pair
$(U,V)$. Then $\leq$ is a partial order on $E(\mathcal{X})$.
\end{lemma}

\begin{proof}
We consider a standard $\mathcal{S}^{\infty}$--extension $\overline{G}$ of
$G$, and consider the canonical extension $\overline{U}$ of $U$, for each
$U\in E(\mathcal{X})$. Let $U\ $and $V$ denote elements of $E(\mathcal{X})$.

If a corner of the pair $(U,V)$ is not small, it is trivial that the
corresponding corner of the pair $(\overline{U},\overline{V})$ is also not small.

If a corner $C$ of the pair $(U,V)$ is small, Lemma
\ref{extensionshavesamesmallcorner} tells us that the corresponding corner of
the pair $(\overline{U},\overline{V})$ is equal to $C$, and so also small.

It follows that, as $E(\mathcal{X})$ is in good position, the family
$\overline{E}=\{\overline{U}:U\in E(\mathcal{X})\}$ is also in good position.
Further the given relation $\leq$ on $E(\mathcal{X})$ is equal to the relation
$\leq$ on $\overline{E}$. As $\overline{G}$ is finitely generated, Lemma 1.14
of \cite{SS1} tells us that this last relation is a partial order on
$\overline{E}$, so that the given relation $\leq$ on $E(\mathcal{X})$ is also
a partial order, as required.
\end{proof}

In section 2 of \cite{NibloSageevScottSwarup}, the authors give a good
discussion of this partial order. In Lemma 2.10 of that paper, the authors
showed that if $G$ is a finitely generated group, if $H_{i}$ is a finitely
generated subgroup, for $1\leq i\leq n$, and if $X_{i}$ is a nontrivial
$H_{i}$--almost invariant subset of $G$, for $1\leq i\leq n$, then each
$X_{i}$ is equivalent to $Y_{i}$ such that the family $E(\mathcal{Y}%
)=\{gY_{i},gY_{i}^{\ast}:g\in G,i\in I\}$ is in good position. Thus the
partial order $\leq$ can be defined on $E(\mathcal{Y})$. The crucial step is
when $n=1$. The result in this case follows easily from Proposition 2.14 of
\cite{SS1}. If $n>1$, we can now assume that, for each $i$, the family of all
translates of $X_{i}$ and $X_{i}^{\ast}$ is in good position. If
$E(\mathcal{X})$ is not in good position, it must contain distinct elements
$U$ and $V$ such that the pair $(U,V)$ has two small (and non-empty) corners,
and $U\ $and $V$ are translates of distinct $X_{i}$'s. Thus there are distinct
$i$ and $j$ such that $X_{j}$ is equivalent to a translate of $X_{i}$ or to
$X_{i}^{\ast}$. In this case we simply, replace $X_{j}$ by that translate. By
repeating, we will eventually arrange that the new $E(\mathcal{X})$ is in good
position, as required. It may seem that we took the easy way out in this last
step, but in Example 2.11 of \cite{NibloSageevScottSwarup}, the authors showed
that in some cases, this is the only way to obtain a family in good position.

The partial order obtained in the preceding paragraph appears to depend on the
choices of the $Y_{i}$'s. In Lemma 2.8 of \cite{NibloSageevScottSwarup} the
authors discussed the uniqueness of this partial order in the case $n=1$. They
showed that in most cases, if $X_{1}$ is equivalent to $Y_{1}$, and is also
equivalent to $Z_{1}$ such that the families $E(\mathcal{Y})$ and
$E(\mathcal{Z})$ are each in good position, then there is an equivariant order
preserving bijection between $E(\mathcal{Y})$ and $E(\mathcal{Z})$. This will
be discussed in more detail in the proof of Lemma
\ref{algregnbhddependsonlyonequivalenceclasses}.

It will be helpful for the arguments of sections \ref{algregnbhds} and
\ref{CCC'swithstrongcrossing} to have a slight refinement of good position.
Recall that in order to prove Lemma 2.10 of \cite{NibloSageevScottSwarup}, the
crucial step was to prove it for a single $H$--almost invariant subset $X$ of
$G$. We consider the subgroups $\mathcal{K}=\{g\in G:gX$ is equivalent to $X$
or $X^{\ast}\}$ and $\mathcal{K}_{0}=\{g\in G:gX$ is equivalent to $X\}$. It
is trivial that $H\subset\mathcal{K}_{0}$, and that $\mathcal{K}_{0}$ is a
subgroup of $\mathcal{K}$ of index at most $2$. In Lemma 2.10 of \cite{SS1},
it is proved that $\mathcal{K}\subset Comm_{G}(H)$. The proof of Proposition
2.14 of \cite{SS1} has two cases, depending on whether $H$ has finite or
infinite index in $Comm_{G}(H)$. In the first case, $H$ must have finite index
in $\mathcal{K}$. Part 1) of Lemma 2.3 of \cite{NibloSageevScottSwarup} states
that $X$ is equivalent to an almost invariant set $W$ over $\mathcal{K}_{0}$
such that $W$ is in good position. Our refinement of this result is the following.

We will be interested in the case when $H$ has finite index in $Comm_{G}(H)$.
In this situation, we will say that $H\ $has \textit{small commensuriser in
}$G$, as in \cite{SS2}.

\begin{lemma}
\label{refinedgoodposition}Let $G$ be a finitely generated group with a
finitely generated subgroup $H$, and let $X$ be a nontrivial $H$--almost
invariant subset of $G$. If $H$ has small commensuriser in $G$, then $X$ is
equivalent to an almost invariant set $W$ over $\mathcal{K}_{0}$, such that if
$U$ and $V$ are equivalent elements of $E(W)=\{gW,gW^{\ast}:g\in G\}$, then
$U\ $and $V$ are equal. If there is $g\in G$ such that $gX$ is equivalent to
$X^{\ast}$, then $W$ is invertible.
\end{lemma}

\begin{remark}
It is automatic that $W$ is in good position.
\end{remark}

\begin{proof}
By applying part 1) of Lemma 2.3 of \cite{NibloSageevScottSwarup}, we can
arrange that $X$ is an almost invariant set over $\mathcal{K}_{0}$, and is in
good position.

If $\mathcal{K}=\mathcal{K}_{0}$, we can simply take $W$ equal to $X$. Note
that for any element $k$ of $\mathcal{K}_{0}$, we have $kX=X$ and $kX^{\ast
}=X^{\ast}$. If $X$ and $Y$ are equivalent elements of $E(\mathcal{X})$, then
$Y$ is equal to $kX$ or to $kX^{\ast}$, for some $k$ in $\mathcal{K}%
=\mathcal{K}_{0}$. It follows immediately that $X$ and $Y$ are equal, which
proves the lemma in this case.

If $\mathcal{K}$ and $\mathcal{K}_{0}$ are not equal, we let $k$ be an element
of $\mathcal{K}-\mathcal{K}_{0}$, so that $kX$ is equivalent to $X^{\ast}$.
But it may be that $kX\ $and $X^{\ast}$ are not equal. In this case we will
alter $X$ as follows. As $X$ is in good position, and $kX$ is equivalent to
$X^{\ast}$, one of the corners $X\cap kX$ and $X^{\ast}\cap kX^{\ast}$ is
empty. By replacing $X$ by $X^{\ast}$ if needed, we can assume that the corner
$X\cap kX$ is empty. As $\mathcal{K}_{0}$ is normal in $\mathcal{K}$, there is
a natural action of $k$ as an involution on the quotient set $\mathcal{K}%
_{0}\backslash G$. Let $P$ denote $\mathcal{K}_{0}\backslash X$, so that the
corner $P\cap kP$ is empty, hence $P$ and $kP$ are disjoint. As $kX$ is
equivalent to $X^{\ast}$, we have $kP$ is almost equal to $P^{\ast}$. We
consider the action of $k$ on the finite set $\mathcal{K}_{0}\backslash
G-(P\cup kP)$. The action of $k$ on $\mathcal{K}_{0}\backslash G$ is free, so
this finite set consists of pairs of points interchanged by $k$. For each such
pair, we add exactly one point of the pair to $P$, to obtain a new subset $Q$
of $\mathcal{K}_{0}\backslash G$ so that $Q$ is almost equal to $P$, and
$kQ=Q^{\ast}$. Now let $W$ denote the pre-image in $G$ of $Q$. Then $W$ is an
almost invariant subset of $G$ over $\mathcal{K}_{0}$ which is equivalent to
$X$, such that $kW=W^{\ast}$. Combined with the first part of the proof, this
implies that if $U$ and $V$ are equivalent elements of $E(W)$, then $U\ $and
$V$ are equal, which completes the proof of the lemma.
\end{proof}

As in the proof of Theorem \ref{splittingsexist}, we can use a standard
$\mathcal{S}^{\infty}$--extension $\overline{G}$ of $G$ to generalise Lemma
\ref{refinedgoodposition} to the setting of group systems of finite type. The
result is the following.

\begin{lemma}
Let $(G,\mathcal{S})$ be a group system of finite type, let $H$ be a finitely
generated subgroup of $G$, and let $X$ be a nontrivial $\mathcal{S}$--adapted
$H$--almost invariant subset of $G$. If $H$ has small commensuriser in $G$,
then $X$ is equivalent to an almost invariant set $Y$ such that if $U$ and $V$
are equivalent elements of $E(Y)=\{gY,gY^{\ast}:g\in G\}$, then $U\ $and $V$
are equal.
\end{lemma}

\begin{remark}
It is again automatic that $Y$ is in good position.
\end{remark}

As in the absolute case, it is easy to extend this to finite families of
almost invariant subsets of $G$. The result we obtain is the following.

\begin{lemma}
\label{refinedgoodpositioningroupsystems}Let $(G,\mathcal{S})$ be a group
system of finite type, let $H_{i}$ be a finitely generated subgroup of $G$,
for $1\leq i\leq n$, and let $X_{i}$ be a nontrivial $\mathcal{S}$--adapted
$H_{i}$--almost invariant subset of $G$, for $1\leq i\leq n$. Let
$\mathcal{X}$ denote the family $\{X_{i}\}_{i=1}^{n}$, and let $E(\mathcal{X}%
)$ denote $\{gX_{i},gX_{i}^{\ast}:g\in G,1\leq i\leq n\}$. Then each $X_{i}$
is equivalent to $Y_{i}$ such that the following conditions hold:

\begin{enumerate}
\item The $Y_{i}$'s are in good position.

\item If $\mathcal{Y}=\{Y_{i}\}_{i=1}^{n}$, and $U$ and $V$ are equivalent

elements of $E(\mathcal{Y})$, then either $U\ $and $V$ are equal, or $U$ is a
translate of $V$ or $V^{\ast}$ such that the stabilizer of $U$ has large commensuriser.
\end{enumerate}
\end{lemma}

\begin{remark}
Although the restriction that the $H_{i}$'s be finitely generated is needed in
general, there is one situation where good position is easy to achieve even if
some $H_{i}$ is not finitely generated. One needs to add the condition that if
$H_{i}$ is not finitely generated, then $X_{i}$ is associated to a splitting
of $G$ over $H_{i}$. For then $X_{i}$ is equivalent to $Y_{i}$ whose
translates are nested, and so it is trivial that $Y_{i}$ is in good position.
\end{remark}

In \cite{NibloSageevScottSwarup}, the authors introduced the idea of very good
position, which is a much stronger condition than good position. The family
$\mathcal{X}=\{X_{i}\}_{i\in I}$ is in \textit{very good position} if whenever
$U$ and $V$ are elements of $E(\mathcal{X})$, then either the pair
$(U,V)\ $has no small corner or one of the corners is empty. They showed that
if the family $\mathcal{X}$ is finite, and $G$ and each $H_{i}$ is finitely
generated, it is possible to replace each $X_{i}$ by an equivalent set $Y_{i}$
such that the $Y_{i}$'s are in very good position. Again $\mathcal{S}^{\infty
}$--extensions allow us to generalise this to the setting of group systems of
finite type. The result is the following.

\begin{lemma}
\label{cangetvgpinsystems}Let $(G,\mathcal{S})$ be a group system of finite
type, let $H_{i}$ be a finitely generated subgroup of $G$, for $1\leq i\leq
n$, and let $X_{i}$ be a nontrivial $\mathcal{S}$--adapted $H_{i}$--almost
invariant subset of $G$, for $1\leq i\leq n$. Let $\mathcal{X}$ denote the
family $\{X_{i}\}_{i=1}^{n}$, and let $E(\mathcal{X})$ denote $\{gX_{i}%
,gX_{i}^{\ast}:g\in G,1\leq i\leq n\}$. Then each $X_{i}$ is equivalent to
$Y_{i}$ such that the following conditions hold:

\begin{enumerate}
\item The $Y_{i}$'s are in very good position.

\item If $\mathcal{Y}=\{Y_{i}\}_{i=1}^{n}$, and $U$ and $V$ are equivalent
elements of $E(\mathcal{Y})$, then either $U\ $and $V$ are equal, or $U$ is a
translate of $V$ or $V^{\ast}$ such that the stabilizer of $U$ has large commensuriser.
\end{enumerate}
\end{lemma}

\begin{proof}
By applying Lemma \ref{refinedgoodpositioningroupsystems}, we can arrange that
the $X_{i}$'s satisfy conclusions 1) and 2) of that lemma.

Now consider a finitely generated standard $\mathcal{S}^{\infty}$--extension
$\overline{G}$ of $G$. We have a canonical extension $\overline{U}$ for each
element $U$ of $E(\mathcal{X})$ and these extensions are equivariant under the
action of $G$. Lemma 4.1 of \cite{NibloSageevScottSwarup} tells us that there
are subgroups $\overline{K_{i}}$ of $\overline{G}$ and $\overline{K_{i}}%
$--almost invariant subsets $\overline{Y_{i}}$ of $\overline{G}$ such that
$\overline{Y_{i}}$ is equivalent to $\overline{X_{i}}$, and the $\overline
{Y_{i}}$'s are in very good position in $\overline{G}$. If we let $Y_{i}$
denote the intersection $\overline{Y_{i}}\cap G$, and let $K_{i}$ denote the
intersection $\overline{K_{i}}\cap G$, then $Y_{i}$ is a $K_{i}$--almost
invariant subset of $G$ which is equivalent to $X_{i}$. Let $\mathcal{Y}%
=\{Y_{i}\}_{i=1}^{n}$, and let $E(\mathcal{Y})$ denote $\{gY_{i},gY_{i}^{\ast
}:g\in G,1\leq i\leq n\}$. Let $U$ and $V$ be elements of $E(\mathcal{Y})$,
and let $\overline{U}$ and $\overline{V}$ denote their canonical extensions to
$\overline{G}$. As $\overline{U}$ and $\overline{V}$ are in very good position
in $\overline{G}$, either the pair $(\overline{U},\overline{V})\ $has no small
corner or one of the corners is empty. In the first case, $\overline{U}$ and
$\overline{V}$ cross, so Lemma \ref{crossingofextensions} tells us that
$U\ $and $V$ also cross and hence the pair $(U,V)\ $has no small corner. In
the second case it is trivial that the corresponding corner of the pair
$(U,V)$ is also empty. As this holds for any $U$ and $V$ in $E(\mathcal{Y})$,
we have shown that the $Y_{i}$'s are in very good position. Thus the $Y_{i}$'s
satisfy conclusion 1) of Lemma \ref{cangetvgpinsystems}.

Next we show that the $Y_{i}$'s also satisfy conclusion 2) of Lemma
\ref{cangetvgpinsystems}.

Let $U$ and $V$ be equivalent elements of $E(\mathcal{Y})$ such that $U$ is
not a translate of $V$ or of $V^{\ast}$. Without loss of generality, we can
suppose that $U=gY_{i}$ and $V=kY_{j}$, with $i\neq j$. Thus $U$ is equivalent
to $gX_{i}$, and $V$ is equivalent to $kX_{j}$, so that $gX_{i}$ is equivalent
to $kX_{j}$. As the $X_{i}$'s satisfy conclusion 2) of Lemma
\ref{refinedgoodpositioningroupsystems}, it follows that $gX_{i}$ and $kX_{j}$
are equal, so that $U\ $and $V$ must also be equal. Hence the $Y_{i}$'s
satisfy conclusion 2) of Lemma \ref{cangetvgpinsystems}.

It remains to consider the case when $U$ is a translate of $V$ or of $V^{\ast
}$. Thus $U$ and $V$ are each translates of some $Y_{i}$ or $Y_{i}^{\ast}$.
Without loss of generality, we can suppose that $U$ is a translate of $Y_{i}$
and that $V$ is equal to $Y_{i}$ or to $Y_{i}^{\ast}$.

Suppose that $H_{i}$ has small commensuriser in $G$, and that there is $k$ in
$G$ such that $kY_{i}$ is equivalent to $Y_{i}$. The set $\mathcal{K}=\{g\in
G:gY_{i}$ is equivalent to $Y_{i}$ or $Y_{i}^{\ast}\}$ is contained in
$Comm_{G}(K_{i})=Comm_{G}(H_{i})$, and contains $k$. As we are assuming that
$H_{i}$ has small commensuriser in $G$, it follows that $\mathcal{K}$ contains
$K_{i}$ with finite index. Hence there is some positive power $k^{n}$ of $k$
such that $k^{n}$ lies in $K_{i}$, so that $k^{n}Y_{i}=Y_{i}$. As $kY_{i}$ is
equivalent to $Y_{i}$ and the $Y_{i}$'s are in very good position, we must
have $kY_{i}\subset Y_{i}$ or $Y_{i}\subset kY_{i}$. In the first case the
inclusions $Y_{i}=k^{n}Y_{i}\subset k^{n-1}Y_{i}\subset\ldots\subset$
$kY_{i}\subset Y_{i}$ show that $kY_{i}=Y_{i}$, and a similar argument applies
if $Y_{i}\subset kY_{i}$. Thus $K_{i}$ is equal to $\mathcal{K}_{0}=\{g\in
G:gY_{i}$ is equivalent to $Y_{i}\}$. As the $X_{i}$'s are also in very good
position, the same arguments show that $H_{i}$ is also equal to $\mathcal{K}%
_{0}$, so the natural map from $E(X_{i})$ to $E(Y_{i})$ which sends $gX_{i}$
to $gY_{i}$ is a bijection. It follows that $U\ $and $V$ must be equal in this case.

This completes the proof of Lemma \ref{cangetvgpinsystems}.
\end{proof}

Recently Lassonde \cite{Lassonde} extended Lemma 4.1 of
\cite{NibloSageevScottSwarup} to cover the case where some of the almost
invariant sets correspond to splittings over groups which need not be finitely
generated. Her result requires a sandwiching condition, as does the corrected
definition of an algebraic regular neighbourhood discussed in subsection
\ref{invertibleaisets}. However she uses a slightly more general version of
the condition which we state here.

\begin{definition}
\label{defnofsandwiching}Let $\mathcal{X}=\{X_{j}:j\in J\}$ be a family of
almost invariant subsets of a group $G$, which are in good position, and let
$\leq$ denote the partial order of almost inclusion on $E(\mathcal{X})$, the
collection of all translates of the $X_{j}$'s and their complements. We let
$E(X_{k})$ denote the collection of all translates of $X_{k}$ and of
$X_{k}^{\ast}$.

\begin{enumerate}
\item $X_{j}$ is \textsl{sandwiched by} $X_{k}$, if either $X_{j}$ crosses
every element of $E(X_{k})$, or there exist $U\ $and $V$ in $E(X_{k})$ such
that $U\leq X_{j}\leq V$.

\item The family $\mathcal{X}$ \textsl{satisfies sandwiching} if for each $j$
and $k$ in $J$, with $j\neq k$, we have $X_{j}$ is sandwiched by $X_{k}$.
\end{enumerate}
\end{definition}

\begin{remark}
\label{sandwichingholdsifallisfg}If $X_{j}$ and $X_{k}$ are almost invariant
subsets of a finitely generated group $G$ over finitely generated subgroups,
then it is always the case that $X_{j}$ is sandwiched by $X_{k}$, and $X_{k}$
is sandwiched by $X_{j}$. This was proved on page 72 of \cite{SS2} as part of
the proof of Proposition 5.7.
\end{remark}

Now we have the following generalization of the first part of Lemma
\ref{cangetvgpinsystems}, which is proved in a similar way.

\begin{lemma}
\label{canobtainverygoodposition}Let $(G,\mathcal{S})$ be a group system of
finite type and, for $1\leq i\leq n$, let $X_{i}$ be a nontrivial
$\mathcal{S}$--adapted $H_{i}$--almost invariant subset of $G$. Suppose
further that the collection $\mathcal{X}$ of all the $X_{i}$'s is in good
position and satisfies sandwiching, and that if $H_{i}$ is not finitely
generated, then $X_{i}$ is associated to a splitting of $G$ over $H_{i}$.

Let $E(\mathcal{X})$ denote the family of all translates by $G$ of the $X_{i}%
$'s and their complements. Then each $X_{i}$ is equivalent to a $K_{i}%
$--almost invariant subset $Y_{i}$ of $G$ such that the $Y_{i}$'s are in very
good position, i.e. the set $E(\mathcal{Y})$ of all translates by $G$ of all
the $Y_{i}$'s and their complements is in very good position.
\end{lemma}

\begin{remark}
The condition of good position for the $X_{i}$'s is needed in order to be able
to define sandwiching. But the discussion of good position above shows that
this condition is easy to achieve so long as, for each $i$, either $H_{i}$ is
finitely generated, or $X_{i}$ is associated to a splitting.
\end{remark}

\begin{proof}
Consider a finitely generated standard $\mathcal{S}^{\infty}$--extension
$\overline{G}$ of $G$. As before, we have a canonical extension $\overline{U}$
for each element $U$ of $E(\mathcal{X})$ and these extensions are equivariant
under the action of $G$. We let $\overline{E}(\mathcal{X})$ denote the family
of all these extensions. Lemma \ref{extensionshavesamesmallcorner} tells us
that if $U$ and $V$ are sets in $E(\mathcal{X})$ with a small corner $C$, then
the corresponding corner of the pair $(\overline{U},\overline{V})$ is equal to
$C$. In particular, if a corner of the pair $(U,V)$ is empty, then the
corresponding corner of the pair $(\overline{U},\overline{V})$ is also empty.
This immediately implies that the $\overline{X_{i}}$'s are also in good
position. By re-indexing the $X_{i}$'s if needed, we can arrange that there is
an integer $k$ such that $H_{i}$ is finitely generated whenever $i\leq k$, and
$H_{i}$ is not finitely generated whenever $i>k$. By the first part of Lemma
\ref{cangetvgpinsystems}, after replacing each $\overline{X_{i}}$, $1\leq
i\leq k$, by an equivalent almost invariant set, we can arrange that
$\overline{X_{1}},\ldots,\overline{X_{k}}$ are in very good position. Now
Corollary 11.4 of Lassonde's paper \cite{Lassonde} tells us that by replacing
each $\overline{X_{i}}$, $1\leq i\leq n$, by an equivalent almost invariant
set $\overline{Y_{i}}$, we can arrange that $\overline{E}(\mathcal{X})$ is in
very good position. It follows that each $X_{i}$ is equivalent to the almost
invariant subset $Y_{i}=\overline{Y_{i}}\cap G$ of $G$ and that the $Y_{i}$'s
are in very good position, as required.
\end{proof}

The following relative version of Lemma \ref{canobtainverygoodposition} will
also be useful.

\begin{lemma}
\label{canobtainvgprelalreadyvgpforsystems}Let $(G,\mathcal{S})$ be a group
system of finite type and, for $1\leq i\leq n$, let $X_{i}$ be a nontrivial
$\mathcal{S}$--adapted $H_{i}$--almost invariant subset of $G$. Suppose
further that the collection of all the $X_{i}$'s is in good position and
satisfies sandwiching, and that if $H_{i}$ is not finitely generated, then
$X_{i}$ is associated to a splitting of $G$ over $H_{i}$.

If the family $X_{1},\ldots,X_{n-1}$ is in very good position, then $X_{n}$ is
equivalent to an almost invariant subset $Y_{n}$ of $G$ such that the family
$X_{1},\ldots,X_{n-1},Y_{n}$ is in very good position.
\end{lemma}

As for the previous lemmas, one can use a finitely generated standard
$\mathcal{S}^{\infty}$--extension of $G$, to reduce the proof of Lemma
\ref{canobtainvgprelalreadyvgp} to the proof of the following relative version
of Corollary 11.4 of \cite{Lassonde}.

\begin{lemma}
\label{canobtainvgprelalreadyvgp}Let $G$ be a finitely generated group and,
for $1\leq i\leq n$, let $X_{i}$ be a nontrivial $H_{i}$--almost invariant
subset of $G$. Suppose further that the collection of all the $X_{i}$'s is in
good position and satisfies sandwiching, and that if $H_{i}$ is not finitely
generated, then $X_{i}$ is associated to a splitting of $G$ over $H_{i}$.

If the family $X_{1},\ldots,X_{n-1}$ is in very good position, then $X_{n}$ is
equivalent to an almost invariant subset $Y_{n}$ of $G$ such that the family
$X_{1},\ldots,X_{n-1},Y_{n}$ is in very good position.
\end{lemma}

\begin{proof}
For the first part of our argument, we will use only the fact that the family
$X_{1},\ldots,X_{n}$ is in good position. For each $m\leq n$, let $E_{m}$
denote the set of all translates of the $X_{i}$'s and their complements, for
$1\leq i\leq m$. As the $X_{i}$'s are in good position, there are two partial
orders on $E_{m}$. One is the partial order of inclusion, the other is the
partial order $\leq$ of almost inclusion. Let $K_{m}$ denote the cubed complex
constructed by Sageev from the partially ordered set $(E_{m},\subset)$. The
vertices of $K_{m}$ are the ultrafilters of $(E_{m},\subset)$. Let $C_{m}$
denote the component of $K_{m}$ which contains all the basic vertices of
$K_{m}$. Sageev showed that $C_{m}$ is $CAT(0)$. Also let $\Lambda_{m}$ denote
the cubed complex constructed in \cite{Lassonde} from the partially ordered
set $(E_{m},\leq)$. The vertices of $\Lambda_{m}$ are the ultrafilters of
$(E_{m},\leq)$. There is a natural embedding of $\Lambda_{m}$ in $K_{m}$, and
we let $L_{m}$ denote the component of $\Lambda_{m}$ which equals $C_{m}%
\cap\Lambda_{m}$.

We now compare the cubings $C_{n-1}$ and $C_{n}$. Let $V$ be an ultrafilter on
$(E_{n},\subset)$. Then $V\cap E_{n-1}$ is clearly an ultrafilter on
$(E_{n-1},\subset)$. This defines a map $\varphi$ from the vertices of $K_{n}$
to those of $K_{n-1}$. Now consider an edge $e$ of $K_{n}$, which joins
vertices $V$ and $V^{\prime}$. There is an element $W$ of $E_{n}$ such that
$V^{\prime}=V-\{W\}+\{W^{\ast}\}$. If $W$ is a translate of $X_{n}$ or of
$X_{n}^{\ast}$, then $\varphi(V)=\varphi(V^{\prime})$ and we define
$\varphi(e)$ to equal the vertex $\varphi(V)$. If $W$ is not a translate of
$X_{n}$ or of $X_{n}^{\ast}$, then $\varphi(V)$ and $\varphi(V^{\prime})$ are
distinct and $\varphi(V^{\prime})=\varphi(V)-\{W\}+\{W^{\ast}\}$. Thus
$\varphi(V)$ and $\varphi(V^{\prime})$ are joined by an edge $f$. We define
$\varphi(e)$ to equal $f$. Having extended $\varphi$ to the $1$--skeleton of
$K_{n}$, it is now trivial to extend $\varphi$ to a map $K_{n}\rightarrow
K_{n-1}$. This map sends $C_{n}$ to $C_{n-1}$, as it sends basic vertices of
$K_{n}$ to basic vertices of $K_{n-1}$. Clearly the induced map is $G$--equivariant.

As $L_{m}$ is naturally contained in $C_{m}$, for each $m$, the map
$C_{n}\rightarrow C_{n-1}$ induces a map $L_{n}\rightarrow L_{n-1}$ which is
also $G$--equivariant.

So far our discussion has just used the assumption that the $X_{i}$'s are in
good position. As the family $X_{1},\ldots,X_{n-1}$ is assumed to be in very
good position, it follows that $L_{n-1}=C_{n-1}$.

Now we consider the cubing $L_{n}$. As in \cite{NibloSageevScottSwarup}, there
are hyperplanes $\mathcal{H}_{1},\ldots,\mathcal{H}_{n}$ in $L_{n}$ which,
after choosing a base vertex $w$ of $L_{n}$, determine almost invariant sets
$Y_{1},\ldots,Y_{n}$ in $G$ such that $Y_{i}$ is equivalent to $X_{i}$, and
the $Y_{i}$'s are in very good position. The construction is to pick one side
$\mathcal{H}_{i}^{+}$ of $\mathcal{H}_{i}$, and let $Y_{i}=\{g\in
G:gw\in\mathcal{H}_{i}^{+}\}$. The $Y_{i}$'s are in very good position for any
choice of $w$, and $Y_{i}$ is equivalent to $X_{i}$ or $X_{i}^{\ast}$. By
replacing $\mathcal{H}_{i}^{+}$ by the other side of $\mathcal{H}_{i}$ if
needed, we can arrange that $Y_{i}$ is equivalent to $X_{i}$, for each $i$. We
claim that we can choose $w$ to arrange that $Y_{i}=X_{i}$, for $1\leq i\leq
n-1$.

Recall that $\mathcal{H}_{i}$ is defined as dual to the set of all edges of
$L_{n}$ which exit $X_{i}$. Thus if $i<n$, there is a corresponding hyperplane
$\mathcal{K}_{i}$ in $L_{n-1}$ which must also determine $Y_{i}$ in the same
way. But $L_{n-1}=C_{n-1}$, and in $C_{n-1}$ we can choose our base vertex to
be the basic vertex $V_{e}$ determined by the identity element $e$ of $G$.
Then $\{g\in G:gV_{e}\in\mathcal{K}_{i}^{+}\}$ equals $X_{i}$ or $X_{i}^{\ast
}$. By replacing $\mathcal{K}_{i}^{+}$ by the other side of $\mathcal{K}_{i}$
if needed, we can arrange that $\{g\in G:gV_{e}\in\mathcal{K}_{i}^{+}\}=X_{i}%
$, for each $i<n$. In the proof of Theorem 8.11 of \cite{Lassonde}, Lassonde
showed how a basic ultrafilter in $(E_{n-1},\leq)$ can be extended to an
ultrafilter in $(E_{n},\leq)$. It follows that there is a vertex $w$ of
$L_{n}$ which maps to the vertex $V_{e}$ of $L_{n-1}$. Then $\{g\in
G:gw\in\mathcal{H}_{i}^{+}\}=X_{i}$, for each $i<n$. Thus, for this choice of
$w$, we have $Y_{i}=X_{i}$, for $1\leq i\leq n-1$, as claimed.
\end{proof}

An immediate and surprising consequence of the preceding lemma is that one can
put infinite families of almost invariant sets in very good position. For
simplicity, we state this in the case when $G\ $is finitely generated, but the
natural analogue holds if $(G,\mathcal{S})$ is a group system of finite type,
and we consider only almost invariant subsets of $G\ $which are $\mathcal{S}$--adapted.

\begin{corollary}
Let $G$ be a finitely generated group with finitely generated subgroups
$H_{i}$, $i\geq1$. For each $i\geq1$, let $Y_{i}$ be a nontrivial $H_{i}%
$--almost invariant subset of $G$. Then $Y_{i}$ is equivalent to an almost
invariant subset $Z_{i}$ of $G$ such that the family of all the $Z_{i}$'s,
$i\geq1$, is in very good position.
\end{corollary}

Let $H$ denote a finitely generated subgroup of a finitely generated group
$G$, and let $\Gamma(G)$ denote the Cayley graph of $G$ with respect to some
finite generating set. If $X$ is a $H$--almost invariant subset of $G$, then
$H\backslash X$ is an almost invariant subset of $H\backslash G$, and hence is
a vertex set of $H\backslash\Gamma(G)\ $with finite coboundary. As
$H\backslash\Gamma(G)$ has only countably many finite collections of edges, it
follows that $G$ has only countably many $H$--almost invariant subsets. Now a
finitely generated group $G$ has only countably many finitely generated
subgroups. Combining this information tells us that $G$ has only countably
many almost invariant subsets which are over finitely generated subgroups.
Thus we have the following consequence.

\begin{corollary}
Let $G$ be a finitely generated group. For each equivalence class of
nontrivial almost invariant subsets of $G$ which are over finitely generated
subgroups of $G$ one can choose a representative in such a way that the entire
family of chosen sets is in very good position.
\end{corollary}

\section{Algebraic regular neighbourhoods\label{algregnbhds}}

The theory of algebraic regular neighbourhoods of families of almost invariant
subsets of a group was first developed by Scott and Swarup in \cite{SS2}. In
this section, we will give a slightly modified approach to developing the
theory, which helps in the extension of the ideas of \cite{SS2} to the setting
of group systems. The main difference is that the original construction in
\cite{SS2} becomes the new definition, and the original definition in
\cite{SS2} becomes a statement of the basic properties of algebraic regular
neighbourhoods. To simplify things for the reader, we will start at the
beginning, and use results from \cite{SS2} as little as possible. We think
this gives a clearer and simpler approach to the theory of algebraic regular
neighbourhoods. Later in the section, we will use very good position to
simplify some of the original proofs of the basic results. In particular, we
state Proposition \ref{verygoodpositioncubingalsoyiledsalgregnbhd} which shows
how to construct algebraic regular neighbourhoods using very good position.

Before we go any further, we need to discuss the idea of a pretree. The
vertices of a tree possess a natural idea of betweenness. The idea of a
pretree formalises this. A \textit{pretree} consists of a set $P$ together
with a ternary relation on $P$ denoted $xyz$ which one should think of as
meaning that $y$ is strictly between $x$ and $z$, so that $y$ is distinct from
$x$ and $y$. The relation should satisfy the following four axioms:

\begin{itemize}
\item (T0) If $xyz$, then $x\neq z$.

\item (T1) $xyz$ implies $zyx$.

\item (T2) $xyz$ implies not $xzy$.

\item (T3) If $xyz$ and $w\neq y$, then $xyw$ or $wyz$.
\end{itemize}

A pretree is said to be \textit{discrete}, if, for any pair $x$ and $z$ of
elements of $P$, the set $\{y\in P:xyz\}$ is finite. Clearly, the vertex set
of any simplicial tree forms a discrete pretree with the induced idea of
betweenness. It is a standard result that a discrete pretree $P$ can be
embedded in a natural way into the vertex set of a simplicial tree $T$ so that
the notion of betweenness is preserved. We briefly describe the construction
of $T$ following Bowditch's papers \cite{B1} and \cite{B3}. For the proofs,
see section 2 of \cite{B1} and section 2 of \cite{B3}. These are discussed in
more detail in \cite{AN} and \cite{B2}. For any pretree $P$, we say that two
distinct elements of $P$ are \textit{adjacent} if there are no elements of $P$
between them. We define a \textit{star} in $P$ to be a maximal subset of $P$
which consists of mutually adjacent elements. (This means that any pair of
elements of a star are adjacent. It also means that any star has at least $2$
elements.) We now enlarge the set $P$ by adding in all the stars of $P$ to
obtain a new set $V$. One can define a pretree structure on $V$ which induces
the original pretree structure on $P$. In the pretree structure on $V$, a star
is adjacent in $V$ to each element of $P$ that it contains, and no two
elements of $P$ are adjacent in $V$. Next we give $V$ the discrete topology
and add edges to $V$ to obtain a graph $T$ with $V$ as its vertex set. For
each pair of adjacent elements of $V$, as just defined, we simply add an edge
which joins them. If $P$ is discrete, then it can be shown that $T$ is a tree
with vertex set $V$. It follows easily from this construction that if a group
$G$ acts on the original pretree $P$, this action extends naturally to an
action of $G$ on the simplicial tree $T$. Moreover, $G$ will act without
inversions on $T$. This will then give a graph of groups decomposition for
$G$, though this decomposition would be trivial if $G$ fixed a vertex of $T$.
The tree $T$ is clearly bipartite with vertex set $V(T)$ expressed as the
disjoint union of $V_{0}(T)$ and $V_{1}(T)$, where $V_{0}(T)$ equals $P$ and
$V_{1}(T)$ equals the collection of stars in $P$. This is why $G$ acts without
inversions on $T$. Note that each $V_{1}$--vertex of $T$ has valence at least
$2$, because a star in $P$ always has at least $2$ elements.

Note also that if we start with a tree $T^{\prime}$, let $P$ denote its vertex
set with the induced idea of betweenness and then construct the tree $T$ as
above, then $T$ is not the same as $T^{\prime}$. In fact, $T$ is obtained from
$T^{\prime}$ by subdividing every edge into two edges.

In order to introduce our ideas on algebraic regular neighbourhoods, we
consider a closed orientable surface $M$ with a finite family $\{C_{\lambda
}\}_{\lambda\in\Lambda}$ of immersed circles, and let $N$ denote a regular
neighbourhood of the union of the $C_{\lambda}$'s. Note that $N$ need not be
connected. For simplicity, we assume that no component of the closure of $M-N$
is a disc, so that $\partial N$ consists of finitely many simple closed curves
in $M$, each not bounding a disc in $M$. Denote $\pi_{1}(M)$ by $G$.

We can associate to this setup the graph of groups decomposition $\Gamma$ of
$G$ determined by $\partial N$. The vertices of $\Gamma$ correspond to
components of $N$ and of $M-N$, the edges of $\Gamma$ correspond to the
components of $\partial N$, and each edge of $\Gamma$ joins a vertex of one
type to a vertex of the other type. Thus $\Gamma$ is naturally a bipartite
graph. We will denote the collection of vertices of $\Gamma$ which correspond
to components of $N$ by $V_{0}(\Gamma)$, or simply $V_{0}$ if the context is
clear. The remaining vertices will be denoted by $V_{1}(\Gamma)$ or simply
$V_{1}$. If we consider the pre-image $\widetilde{N}$ of $N$ in the universal
cover $\widetilde{M}$ of $M$, then $\partial\widetilde{N}$ is a family of
disjoint lines, and the dual graph to $\partial\widetilde{N}$ in
$\widetilde{M}$ is a tree $T$ on which $G$ acts with quotient $\Gamma$. The
vertex and edge groups of $\Gamma$ are simply the vertex and edge stabilisers
for the action of $G$ on $T$. Again $T$ is naturally bipartite with some
vertices corresponding to components of $\widetilde{N}$ and some vertices
corresponding to components of $\widetilde{M}-\widetilde{N}$.

For simplicity, suppose that each $C_{\lambda}$ lifts to an embedded line
$S_{\lambda}$ in $\widetilde{M}$. Let $\Sigma$ denote the collection of all
the translates of all the $S_{\lambda}$'s, and let $\left\vert \Sigma
\right\vert $ denote the union of all the elements of $\Sigma$. Thus
$\left\vert \Sigma\right\vert $ is the complete pre-image in $\widetilde{M}$
of the union of the images of all the $C_{\lambda}$'s and $\widetilde{N}$ is a
regular neighbourhood of $\left\vert \Sigma\right\vert $. The fact that $N$ is
a regular neighbourhood of the $C_{\lambda}$'s implies that the inclusion of
the union of the $C_{\lambda}$'s into $N$ induces a bijection between
components, and that the inclusion of $\left\vert \Sigma\right\vert $ into
$\widetilde{N}$ also induces a bijection between components. Thus the $V_{0}%
$--vertices of $T$ correspond to the components of $\left\vert \Sigma
\right\vert $, and the $V_{1}$--vertices of $T$ correspond to the components
of $\widetilde{M}-\widetilde{N}$. If two elements $S$ and $S^{\prime}$ of
$\Sigma$ belong to the same component of $\left\vert \Sigma\right\vert $,
there must be a finite chain $S=S_{0},S_{1},...,S_{n}=S^{\prime}$ of elements
of $\Sigma$ such that $S_{i}$ intersects $S_{i+1}$, for each $i$. Thus the
elements of $\Sigma$ which form a component of $\left\vert \Sigma\right\vert $
are an equivalence class of the equivalence relation on $\Sigma$ generated by
saying that two elements of $\Sigma$ are related if they intersect.

The simplest example is when the family of immersed circles in $M$ consists of
a single simple closed curve $C$. In this case the regular neighbourhood $N$
of $C$ is an annulus, and $\Gamma$ has a single $V_{0}$--vertex $v$
corresponding to $N$, and two edges, each incident to $v$, corresponding to
the boundary components of $N$. The other end of each edge is a $V_{1}%
$--vertex. If $C$ separates $M$, there are two $V_{1}$--vertices corresponding
to the components of $M-N$. Otherwise $M-N$ is connected, and $\Gamma$ has a
single $V_{1}$--vertex, so that $\Gamma$ is a loop.

The above is what we want to encode in the algebraic setting, except that as
we will be dealing with almost invariant sets, we will consider crossing
rather than intersection. This corresponds to ignoring "inessential"
intersections in the topological theory. Thus given a group $G$ and a family
$\mathcal{X}$ of almost invariant subsets of $G$, we want to construct a
pretree out of "cross-connected components" of almost invariant sets, then use
the star construction to define a $G$--tree $T$, and finally define the
algebraic regular neighbourhood of $\mathcal{X}$ in $G$ to be the graph of
groups structure $\Gamma=G\backslash T$ of $G$. The details of our
construction are not the same as in \cite{SS2}. We have modified them to
remove any direct reference to the idea of good position. This uses ideas from
\cite{Guirardel-Levitt4}, but their ideas also need modifying to work in the
more general setting of this paper. The reason for removing good position from
our discussion is that this is a stronger condition than needed, as shown by
the need to introduce the idea of good enough position on page 46
of \cite{SS2}. But as in the original development of the theory of algebraic
regular neighbourhoods, a crucial role will be played by Theorems 2.5 and 2.8
of \cite{SS1} and their generalizations to the setting of group systems,
Theorems \ref{splittingsexist} and \ref{splittingsarecompatible}.

For later reference, we make the following definition.

\begin{definition}
\label{defnofE(X)andF(X)}Let $G$ be a group with a family of subgroups
$\{H_{\lambda}\}_{\lambda\in\Lambda}$. For each $\lambda\in\Lambda$, let
$X_{\lambda}$ denote a nontrivial $H_{\lambda}$--almost invariant subset of
$G$, and let $\mathcal{X}$ denote the family $\{X_{\lambda}\}_{\lambda
\in\Lambda}$. We enlarge $\mathcal{X}$ to a $G$--invariant family, by defining
$E(\mathcal{X})$ to be the family $\{gX_{\lambda},gX_{\lambda}^{\ast}:g\in
G,\lambda\in\Lambda\}$. We further enlarge $E(\mathcal{X})$ as follows. If $X$
is an element of $E(\mathcal{X})$ which crosses some other element of
$E(\mathcal{X})$, then we add to $E(\mathcal{X})$ every almost invariant
subset of $G$ equivalent to $X$. The resulting family is denoted by
$F(\mathcal{X})$.
\end{definition}

\begin{remark}
We do not assume that the family $\Lambda$ is finite, nor that $G$ or the
$H_{\lambda}$'s are finitely generated. We will explain below that when
constructing the family $F(\mathcal{X})$, it is important not to add elements
equivalent to $X$ unless $X$ crosses some other element of $E(\mathcal{X})$.
\end{remark}

We will outline how to define the algebraic regular neighbourhood of the
family $\mathcal{X}$ in $G$, but whether this is possible depends on
$\mathcal{X}$ and $G$. We suspect that there can be no theory of algebraic
regular neighbourhoods without some strong finiteness conditions. Our basic
definition is also a construction, and it requires some key conditions to be
satisfied. After giving our definition, it is simple to extend it to handle
some situations where these conditions may not be satisfied.

For our basic definition, we will insist on the following two conditions. Note
that if $G$ is finitely generated, the first condition always holds.

\begin{itemize}
\item Crossing of almost invariant sets in $F(\mathcal{X})$ is corner
symmetric. (See Definition \ref{defnofcornersymmetric}.)
\end{itemize}

This means that if $U$ is a $H$--almost invariant element of $F(\mathcal{X})$,
and $V$ is a $K$--almost invariant element of $F(\mathcal{X})$, then a corner
of the pair $(U,V)$ is $H$--finite if and only if that same corner is
$K$--finite. (Recall from Definition \ref{defnofsmallcorner} that such a
corner is called small.) This implies that crossing of almost invariant sets
in $F(\mathcal{X})$ is symmetric, i.e. if $U\ $and $V$ are elements of
$F(\mathcal{X})$ such that $U$ crosses $V$, then also $V$ crosses $U$. Given
this symmetry, we can consider the collection $P$ of equivalence classes of
the almost invariant sets in $F(\mathcal{X})$ under the relation generated by
crossing and complementation. These equivalence classes are called
cross-connected components (CCCs). We want $P$ to constitute the $V_{0}%
$--vertices of a bipartite $G$--tree $T$.

Note that enlarging $E(\mathcal{X})$ to $F(\mathcal{X})$ in the way prescribed
in Definition \ref{defnofE(X)andF(X)} did not change the number of CCCs.
A\ special role is played by elements of $F(\mathcal{X})$ which cross no
element of $F(\mathcal{X})$. We call such elements \emph{isolated in}
$F(\mathcal{X})$, or simply \emph{isolated} if the context is clear. If $X$ is
isolated in $F(\mathcal{X})$, then the CCC which contains $X$ consists solely
of $X$ and $X^{\ast}$. We call such a CCC \emph{isolated}. If $X$ is an
isolated element of $E(\mathcal{X})$, adding elements to $E(\mathcal{X})$
which are equivalent to $X$ would certainly change the number of CCCs. Thus we
do not add such elements when forming $F(\mathcal{X})$. Further we need the
following condition on isolated elements in $F(\mathcal{X})$. This condition
replaces the idea of good position which was used in \cite{SS2}.

\begin{itemize}
\item The family of all isolated elements of $F(\mathcal{X})$ is nested.
\end{itemize}

Next we need a couple of technical results about almost invariant sets in
$F(\mathcal{X})$. By assumption, crossing of such sets is corner symmetric.
Recall the partial order $\leq$ defined at the start of section
\ref{goodposition}. In order to ensure that $\leq$ was a partial order we
needed to assume that our almost invariant sets were in good position. In our
situation, this need not be the case, so our relation $\leq$ may not be a
partial order. Given elements $U\ $and $V$ of $F(\mathcal{X})$, we will define
$U\leq V$ to mean that the corner $U\cap V^{\ast}$ is small, as in Definition
\ref{defnofsmallcorner}. Note that if $U\leq V$, there is $U^{\prime}$
equivalent to $U$ such that $U^{\prime}\subset V$, and there is $V^{\prime}$
equivalent to $V$ sch that $U\subset V^{\prime}$. We obtain $U^{\prime}$ from
$U$ by removing the corner $U\cap V^{\ast}$, and obtain $V^{\prime}$ from $V$
by adding the corner $U\cap V^{\ast}$. It follows that $\leq$ is transitive.
For if $U\leq V\leq W$, there is $U^{\prime}$ equivalent to $U$, and
$W^{\prime}$ equivalent to $W$ such that $U^{\prime}\subset V\subset
W^{\prime}$. Thus $U^{\prime}\subset W^{\prime}$, so that we must have $U\leq
W$, as required. The following result is Lemma 3.5 of \cite{SS2}.

\begin{lemma}
\label{Lemma3.5ofmomster}If $U$, $V$ and $Z$ are elements of $F(\mathcal{X})$
such that $U\leq V$, and $Z$ crosses $U$ but not $V$, then either $Z\leq V$ or
$Z^{\ast}\leq V$.
\end{lemma}

\begin{proof}
As $Z$ crosses $U$, none of the four corners $Z^{(\symbol{94})}\cap U^{(\ast
)}$ is small. Since $U\leq V$, it follows that $Z^{\symbol{94}}\cap V$ and
$Z\cap V$ are also not small. As $Z$ does not cross $V$, either
$Z^{\symbol{94}}\cap V^{\ast}$ or $Z\cap V^{\ast}$ is small. Thus either
$Z\leq V$ or $Z^{\ast}\leq V$, as required.
\end{proof}

We will also need the following consequence.

\begin{corollary}
\label{corollarytoL3.5ofmonster}If $u$ and $v$ are distinct CCCs of
$F(\mathcal{X})$, if $U$ and $V$ are elements of $u$ and $v$ respectively such
that $U\leq V$, and if $Z$ lies in $u$, then either $Z\leq V$ or $Z^{\ast}\leq
V$.
\end{corollary}

\begin{proof}
First suppose that $u$ is isolated. Then $Z$ must be equal to $U$ or to
$U^{\ast}$, so the result is immediate.

Otherwise, $u$ is not isolated, so there is a sequence $Z_{1},Z_{2}%
,\ldots,Z_{n}=Z$ of elements of $u$ such that $Z_{1}$ crosses $U$, and $Z_{i}$
crosses $Z_{i-1}$, for $2\leq i\leq n$. Note that as $u$ and $v$ are distinct
CCCs, each $Z_{i}$ cannot cross $V$. Now the proof is by induction on $n$. If
$n=1$, this is just Lemma \ref{Lemma3.5ofmomster}. If $n>1$, we prove the
induction step by applying Lemma \ref{Lemma3.5ofmomster} with $U$ replaced by
$Z_{n-1}$ or $Z_{n-1}^{\ast}$, as appropriate.
\end{proof}

Now we define a relation of (strict) betweenness on the collection $P$ of all
CCCs of $F(\mathcal{X})$, by defining $y$ to be \textit{between} $x$ and $z$
if $y$ is distinct from $x$ and $z$, and there are elements $X$, $Y$ and $Z$
of $x$, $y$ and $z$ respectively, such that $X\subset Y\subset Z$.

\begin{lemma}
\label{Pisapretree}The set $P$ with this relation of betweenness is a pretree.
\end{lemma}

\begin{proof}
We need to verify that $P$ satisfies the four axioms (T0)--(T3) for a pretree,
which we give below.

\begin{itemize}
\item (T0) If $xyz$, then $x\neq z$.

\item (T1) $xyz$ implies $zyx$.

\item (T2) $xyz$ implies not $xzy$.

\item (T3) If $xyz$ and $w\neq y$, then $xyw$ or $wyz$.
\end{itemize}

To prove that T(0) holds, suppose that $xyz$ and $x=z$. Thus there are
elements $X$, $Y$ and $Z$ of $x$, $y$ and $z$ respectively, such that
$X\subset Y\subset Z$. If $Z=X$, then also $Y=X$, contradicting the assumption
that $x$ and $y$ are distinct. Also $Z$ cannot equal $X^{\ast}$. Thus the CCC
$x$ is not isolated. As $X\leq Y\ $and $Z\in x$, Corollary
\ref{corollarytoL3.5ofmonster} tells us that either $Z\leq Y$ or $Z^{\ast}\leq
Y$. As $Y\subset Z$, it follows that we must have $Z\leq Y$, so that $Y$ and
$Z\ $are equivalent. As $x$ is not isolated, that implies that $Y\ $also lies
in $x$, again contradicting the assumption that $x$ and $y$ are distinct.

Note that (T0) implies that if $xyz$, then $x$, $y$ and $z$ are all distinct.

T(1) is clear as $X\subset Y\subset Z$ implies that $Z^{\ast}\subset Y^{\ast
}\subset X^{\ast}$.

To prove that (T2) holds, suppose that $xyz$ and $xzy$ both hold. Thus there
are elements $X,U$ of $x$, $Y,V$ of $y$, and $Z,W$ of $z$ such that $X\subset
Y\subset Z$ and $U\subset W\subset V$.

We start by handling the special cases when some of $X$, $Y$ and $Z$ are equivalent.

Suppose first that all three are equivalent. Now equivalent elements of
$F(\mathcal{X})$ cannot lie in distinct CCCs unless they are isolated. Thus
each of $x$, $y$ and $z$ is isolated. In particular, it follows that
$U=X^{(\ast)}$, $V=Y^{(\ast)}$ and $W=Z^{(\ast)}$, where $X^{(\ast)}$ denotes
one of $X$ or $X^{\ast}$. Now the inclusions $X\subset Y$ and $U\subset V$
imply that $U=X$ and $V=Y$. Similarly we must have $W=Z$. Now the inclusions
$Y\subset Z$ and $W\subset V$ imply that $Y=Z$, so that $y=z$, contradicting
the hypothesis that $y$ and $z$ are distinct CCCs. Similarly the assumption
that $U$, $V\ $and $W$ are equivalent also yields a contradiction.

Next suppose that $X$ and $Y$ are equivalent, so that $x$ and $y$ are both
isolated, and we must have $U=X^{(\ast)}$ and $V=Y^{(\ast)}$. The inclusions
$X\subset Y$ and $U\subset V$ imply that $U=X$ and $V=Y$. In particular,
$U$\ and $V$\ are equivalent, so the inclusions $U\subset W\subset V$ imply
that $U$, $V$ and $W$ are all equivalent, which yields a contradiction by the
preceding paragraph.

Similar arguments yield a contradiction if any two of $X$, $Y\ $and $Z$ are
equivalent, or if any two of $U$, $V$ and $W$ are equivalent. So from now on
we will assume that no two of $X$, $Y\ $and $Z$ are equivalent, and that no
two of $U$, $V\ $and $W$ are equivalent.

Now we will use the relation $\leq$ on elements of $F(\mathcal{X})$, discussed
above, which is given by $U\leq V$ if the corner $U\cap V^{\ast}$ is small.
(See Definition \ref{defnofsmallcorner}.)

As $X\subset Y$ and $U\in x$, Corollary \ref{corollarytoL3.5ofmonster} tells
us that $U^{(\ast)}\leq Y$, where $U^{(\ast)}$ denotes $U\ $or $U^{\ast}$.
Similarly, as $Y\subset Z$ and $W\in z$, we have $Y\leq W^{(\ast)}$. As our
relation $\leq$ is transitive, we have one of the four relations $U^{(\ast
)}\leq W^{(\ast)}$. As $U\subset W$, it is not possible to have $U\leq
W^{\ast}$, or $U^{\ast}\leq W$. If $U^{\ast}\leq W^{\ast}$, then $U$ and $W$
must be equivalent, which contradicts our assumption.

We conclude that we must have $U\leq W$, so that $U\leq Y\leq W$.

Now we make a similar argument interchanging the roles of $X$, $Y\ $and $Z$
with those of $U$, $V$ and $W$. We conclude that $X\leq W\leq Y$.

As $Y\leq W$, and $W\leq Y$, it follows that $W\ $and $Y$ are equivalent,
contradicting our assumption.

To prove that (T3) holds, let $X$, $Y$ and $Z$ be elements of $x$, $y$ and $z$
respectively such that $X\subset Y\subset Z$, and let $W$ be an element of
$w$. The required result is trivial if $w$ is equal to $x$ or to $y$, so we
can suppose that $w$ is distinct from all of $x$, $y$ and $z$. In particular,
$W$ does not cross any of $X$, $Y$ and $Z$. Thus the pair $(W,X)$ has a small
corner. After replacing each of $X$ and $W$ by its complement if needed, we
can arrange that $W\leq X$ or $X\leq W$.

If $W\leq X$, we have $W\leq Y\subset Z$. If $W$ is not isolated, there is
$W^{\prime}$ equivalent to $W$ such that $W^{\prime}\subset Y\subset Z$, so
that $wyz$, as required. If $W$ is isolated, but $Y\ $and $Z\ $are not, then
$Y\ $and $Z\ $are equivalent to $Y^{\prime}$ and $Z^{\prime}$ such that
$W\subset Y^{\prime}\subset Z^{\prime}$, and again we have $wyz$. If $W$ and
$Y$ are both isolated, they must be nested by hypothesis, so again we have
$wyz$. Finally, consider the case when $W$ and $Z$ are isolated, but $Y$ is
not isolated. As $W\ $and $Z$ are isolated, we must have $W\subset Z$. For
otherwise, we would have $W\leq Y\subset Z\subset W$, which would imply that
$W$, $Y\ $and $Z$ are all equivalent. In particular, $Y$ would be isolated, a
contradiction. Now adding the corner $W\cap Y^{\ast}$ to $Y$ yields
$Y^{\prime}$ in the CCC $y$ which is equivalent to $Y$ such that $W\subset
Y^{\prime}$. As $W$ and $Y$ are each contained in $Z$, it follows that we also
have $Y^{\prime}\subset Z$, so that we have $wyz$, as required.

If $X\leq W$, we compare $W\ $and $Y$. Again one corner of the pair $(W,Y)$
must be small.

If $W\leq Y$, then as in the case when $W\leq X$, we have $wyz$.

If $Y\leq W$, we have $X\subset Y\leq W$. Thus $W^{\ast}\leq Y^{\ast}\subset
X^{\ast}$. Now we use the argument in the case $W\leq X$ but applied to
$W^{\ast}$, $Y^{\ast}$ and $X^{\ast}$ instead of to $W$, $Y$ and $Z$, to show
that $wyx$, and so $xyw$, as required.

The relation $W\leq Y^{\ast}$ is not possible, as $X\subset Y$ and $X\leq W$.

Finally the relation $Y^{\ast}\leq W$ implies that $W^{\ast}\leq Y$, so that
$W^{\ast}\leq Y\subset Z$. Now we use the argument in the case $W\leq X$ but
applied with $W^{\ast}$ in place of $W$, to show that $wyz$, as required.
\end{proof}

If $P$ is a discrete pretree, then the star construction yields a bipartite
$G$--tree $T$, and we want to define the quotient $G\backslash T$ to be the
algebraic regular neighbourhood of $\mathcal{X}$ in $G$. However there is an
exceptional case, when this definition needs slight modification. The reason
for this modification is to ensure that the resulting graph of groups is
unaltered if we replace each $X_{\lambda}$ by an equivalent almost invariant
subset of $G$.

To see the difficulty, let $M$ be a closed and non-orientable surface, but not
the projective plane, and let $C$ be a simple closed curve on $M$ which bounds
a Moebius band $B$. As above, a regular neighbourhood of $C$ yields a graph of
groups structure $\Gamma$ on $G=\pi_{1}(M)$ which has a single $V_{0}%
$--vertex, labeled $H=\pi_{1}(C)$ and two $V_{1}$--vertices, one of which is
labeled by $K=\pi_{1}(B)$. But $C$ is homotopic to a double cover of a
$1$--sided curve $D$, whose regular neighbourhood is $B$. If we use the
regular neighbourhood of $D$ to construct the graph of groups $\Gamma$, then
$\Gamma$ will have two vertices, and one edge. The edge group is $H$ and the
$V_{0}$--vertex of $\Gamma$ is labeled by $K$. To translate this into the
algebraic setting, let $k$ denote a generator of $K$. Then $C$ determines a
$H$--almost invariant subset $X$ of $G$ such that $kX$ is equivalent to
$X^{\ast}$, and $D$ determines a $H$--almost invariant subset $Y$ of $G$ such
that $kY=Y^{\ast}$. Further $X$ and $Y$ are equivalent. The problem is that if
we make the above construction using $X$, assuming that $kX$ is not equal to
$X^{\ast}$, we obtain a graph of groups structure for $G$ with two edges and
two $V_{1}$--vertices, whereas, if we use $Y$, we obtain a graph of groups
structure for $G$ with one edge and one $V_{1}$--vertex. The ``correct" graph
of groups structure is the one with two edges. This can easily be obtained
from the ``incorrect" one by inserting a $V_{0}$--vertex in the centre of the
edge and changing the $V_{0}$--vertex at one end of that edge to be a $V_{1}%
$--vertex. Clearly the same problem (and solution) will occur if we have a
family of many curves on $M$ which includes $C$ such that none of the other
curves cross $C$. In \cite{SS2}, the authors had hoped to avoid this problem
by insisting that isolated elements be non-invertible. However, in Example
\ref{examplewheregoodpositionimpliesinvertible}, we give an example to show
that this cannot always be done.

Now we make the following formal definitions.

\begin{definition}
\label{defnofinvertible}An almost invariant subset $X$ of a group $G$ is
\emph{invertible} if there is $k$ in $G$ such that $kX=X^{\ast}$.
\end{definition}

\begin{definition}
\label{defnofARN-ready}Let $G$ be a group with a family of subgroups
$\{H_{\lambda}\}_{\lambda\in\Lambda}$. For each $\lambda\in\Lambda$, let
$X_{\lambda}$ denote a nontrivial $H_{\lambda}$--almost invariant subset of
$G$, and let $\mathcal{X}$ denote the family $\{X_{\lambda}\}_{\lambda
\in\Lambda}$. Also let $F(\mathcal{X})$ denote the family defined above.

The family $\mathcal{X}$ is \emph{regular neighbourhood ready} ( or
\emph{RN--ready}) if the following conditions hold:

\begin{enumerate}
\item Crossing of almost invariant sets in $F(\mathcal{X})$ is corner
symmetric. (See Definition \ref{defnofcornersymmetric}.)

\item The family of all isolated elements of $F(\mathcal{X})$ is nested.
\end{enumerate}
\end{definition}

\begin{definition}\label{defn9.8}
\label{defnofalgregnbhdifRN-ready}Let $G$ be a group with a family of
subgroups $\{H_{\lambda}\}_{\lambda\in\Lambda}$. For each $\lambda\in\Lambda$,
let $X_{\lambda}$ denote a nontrivial $H_{\lambda}$--almost invariant subset
of $G$, and let $\mathcal{X}$ denote the family $\{X_{\lambda}\}_{\lambda
\in\Lambda}$. Also let $F(\mathcal{X})$ denote the family defined above.
Suppose that $\mathcal{X}$ is RN--ready, so that the CCCs of $F(\mathcal{X})$
form a pretree $P$.

If $P$ is discrete, the star construction yields a bipartite $G$--tree $T$.
\emph{The algebraic regular neighbourhood }$\Gamma\mathcal{(X}:G)$\emph{ of
the family }$\mathcal{X}$\emph{ in} $G$, is the quotient graph of groups
$G\backslash T$, unless some $X_{\lambda}$'s are isolated in $F(\mathcal{X})$
and invertible. In that case, we construct $\Gamma(\mathcal{X}:G)$ from
$G\backslash T$ by inserting a $V_{0}$--vertex in the centre of each of the
edges of $G\backslash T$ whose edge splittings are given by an isolated and
invertible $X_{\lambda}$, and changing the adjacent $V_{0}$--vertex to a
$V_{1}$--vertex.
\end{definition}

\begin{remark}
\label{remarkGuirardelexample}Even if $G$ is finitely generated and $\Lambda$
is finite, it is possible that $P$ is not discrete. In Example
\ref{incompatiblebutintersectionnozero}, due to Guirardel, $G$ is the free
group of rank $2$ with two almost invariant subsets $Z_{s}$ and $Z_{t}$, each
associated to a splitting of $G\ $over a non-finitely generated subgroup. Let
$\mathcal{Z}$ denote the family $\{Z_{s},Z_{t}\}$. In this example, the family
$\mathcal{Z}$ is in RN--ready, but the pretree determined by $F(\mathcal{Z})$
is not discrete.
\end{remark}

\begin{definition}
\label{defnofalgregnbhd}Let $G$ be a group with a family of subgroups
$\{H_{\lambda}\}_{\lambda\in\Lambda}$. For each $\lambda\in\Lambda$, let
$X_{\lambda}$ denote a nontrivial $H_{\lambda}$--almost invariant subset of
$G$, and let $\mathcal{X}$ denote the family $\{X_{\lambda}\}_{\lambda
\in\Lambda}$. Suppose that each $X_{\lambda}$ is equivalent to an almost
invariant subset $Y_{\lambda}$ of $G$ such that the family $\mathcal{Y=}%
\{Y_{\lambda}\}_{\lambda\in\Lambda}$ is RN--ready. If the algebraic regular
neighbourhood $\Gamma(\mathcal{Y}:G)$ exists, then the algebraic regular
neighbourhood $\Gamma(\mathcal{X}:G)$ is defined to be $\Gamma(\mathcal{Y}:G)$.
\end{definition}

The following lemma shows that this last definition is independent of the
choices of the $Y_{\lambda}$'s. Thus the algebraic regular neighbourhood of a
family $\{X_{\lambda}\}$ depends only on the equivalence classes of the
$X_{\lambda}$'s.

\begin{lemma}
\label{algregnbhddependsonlyonequivalenceclasses}Let $G$ be a group with a
family of subgroups $\{H_{\lambda}\}_{\lambda\in\Lambda}$. For each
$\lambda\in\Lambda$, let $X_{\lambda}$ and $Y_{\lambda}$ denote equivalent
nontrivial $H_{\lambda}$--almost invariant subsets of $G$. Let $\mathcal{X}$
and $\mathcal{Y}$ denote $\{X_{\lambda}:\lambda\in\Lambda\}$ and
$\{Y_{\lambda}:\lambda\in\Lambda\}$ respectively. Let $F(\mathcal{X})$ and
$F(\mathcal{Y})$ be as defined in Definition \ref{defnofE(X)andF(X)}, and
suppose that each of $\mathcal{X}$ and $\mathcal{Y}$ is RN--ready. Further
suppose that if $U$ and $V\ $are equivalent isolated elements of
$F(\mathcal{X})$, then $V$ is a translate of $U$.

If the pretree $P$ determined by $F(\mathcal{X})$ is discrete, then the
pretree $Q$ determined by $F(\mathcal{Y})$ is also discrete, and the algebraic
regular neighbourhoods $\Gamma(\mathcal{X}:G)$ and $\Gamma(\mathcal{Y}:G)$ are isomorphic.
\end{lemma}

\begin{remark}
The restriction on equivalent isolated elements of $F(\mathcal{X})$ is to
avoid the following situation. Suppose that the family $\mathcal{X}$ consists
of just two equivalent isolated elements $X_{1}$ and $X_{2}$, and that the
family $\mathcal{Y}$ consists of the single element $Y_{1}$ which is equal to
$X_{1}$. Thus the families $\mathcal{X}$ and $\mathcal{Y}$ satisfy all the
hypotheses of the lemma, 
but the result is plainly false.
\end{remark}

\begin{proof}
This is proved by an indirect argument in \cite{SS2}, but we give here a
direct proof which uses some ideas of \cite{NibloSageevScottSwarup}.

We will see that, in most cases, the correspondence between $X_{\lambda}$ and
$Y_{\lambda}$ induces a $G$--equivariant bijection between $P$ and $Q$, which
preserves betweenness. This immediately induces an isomorphism of pretrees $P$
and $Q$. Thus if $P$ is discrete, so is $Q$, and the resulting $G$--trees
$T_{\mathcal{X}}$ and $T_{\mathcal{Y}}$ are $G$--equivariantly isomorphic,
implying that $\Gamma(\mathcal{X}:G)$ and $\Gamma(\mathcal{Y}:G)$ are isomorphic.

Consider a non-isolated CCC $x$ of $F(\mathcal{X})$. Replacing each
$gX_{\lambda}$ by the equivalent set $gY_{\lambda}$ leaves $x$ unchanged, so
that the correspondence between $X_{\lambda}$ and $Y_{\lambda}$ induces the
identity map on the family of non-isolated CCCs of $F(\mathcal{X})$. For
isolated CCCs, the situation is a little different. First we need to know that
if $X_{\lambda}$ is isolated, then the map sending $gX_{\lambda}$ to
$gY_{\lambda}$ is well defined. We need to know that if $gX_{\lambda
}=X_{\lambda}$, then $gY_{\lambda}=Y_{\lambda}$. This is the case, because of
our assumption that isolated elements of $F(\mathcal{X})$ and of
$F(\mathcal{Y})$ are nested. Let $K\ $and $L$ denote the stabilizers of
$X_{\lambda}$ and $Y_{\lambda}$ respectively. Note that $K$ and $L$ must be
commensurable. Now suppose that $gX_{\lambda}=X_{\lambda}$. Then $gY_{\lambda
}$ is equivalent to $Y_{\lambda}$. The nestedness assumption implies that
$gY_{\lambda}=Y_{\lambda}$, or $gY_{\lambda}\subset Y_{\lambda}$ or
$Y_{\lambda}\subset gY_{\lambda}$. As $K\ $and $L$ are commensurable, there
must be $n$ such that $g^{n}Y_{\lambda}=Y_{\lambda}$. It follows that
$gY_{\lambda}=Y_{\lambda}$, so that $K\subset L$. Reversing the roles of
$X_{\lambda}$ and $Y_{\lambda}$ shows that also $L\subset K$, so that $L$ and
$K$ are equal, as required.

Hence if no $X_{\lambda}$ or $Y_{\lambda}$ is isolated and invertible, it is
immediate that we have a well defined $G$--equivariant bijection $\varphi$
from $P$ to $Q$.

Suppose that $\varphi$ does not preserve betweenness. Thus there are CCCs $x$,
$y$ and $z$ of $F(\mathcal{X})$ such that $xyz$, and that $(\varphi x)(\varphi
z)(\varphi y)$. Thus there are elements $X$, $Y$ and $Z$ of $x$, $y$ and $z$
respectively such that $X\subset Y\subset Z$, and there are elements $U$, $V$
and $W$ of $\varphi x$, $\varphi y$ and $\varphi z$ respectively such that
$U\subset W\subset V$. We will argue much as in the proof of axiom (T2) in
Lemma \ref{Pisapretree}. Suppose that no two of $X$, $Y$ and $Z$ are
equivalent, and that no two of $U$, $V$ and $W$ are equivalent. Recall that
either $\varphi x=x$, or $x$ is isolated. In the first case, as $X\subset Y$
and $U\in x$, we apply Corollary \ref{corollarytoL3.5ofmonster} which tells us
that $U^{(\ast)}\leq Y$. In the second case, when $x$ is isolated, then $U$
must be equivalent to $X$ or to $X^{\ast}$. As $X\subset Y$, it follows
immediately that $U^{(\ast)}\leq Y$, in this case also. Similarly, as
$Y\subset Z$ and $W\in\varphi z$, we have $Y\leq W^{(\ast)}$. As our relation
$\leq$ is transitive, we have one of the four relations $U^{(\ast)}\leq
W^{(\ast)}$. As $U\subset W$, it is not possible to have $U\leq W^{\ast}$, or
$U^{\ast}\leq W$. If $U^{\ast}\leq W^{\ast}$, then $U$ and $W$ must be
equivalent, which contradicts our assumption.

We conclude that we must have $U\leq W$, so that $U\leq Y\leq W$.

Now we make a similar argument interchanging the roles of $X$, $Y\ $and $Z$
with those of $U$, $V$ and $W$. We conclude that $X\leq W\leq Y$. As $Y\leq
W$, and $W\leq Y$, it follows that $W\ $and $Y$ are equivalent, again
contradicting our assumption.

We conclude that $\varphi$ preserves betweenness, except possibly in the
situation where two or more of $X$, $Y\ $and $Z$ are equivalent. Recall that
if two of $X$, $Y\ $and $Z$ are equivalent, then the corresponding CCCs are
isolated. Further, if $X$ and $Z$ are equivalent, the fact that $X\subset
Y\subset Z$ implies that all three of $X$, $Y\ $and $Z$ are equivalent. Thus
$\varphi$ preserves betweenness except possibly in the cases when $X$ and $Y$
are equivalent, or $Y$ and $Z\ $are equivalent, or all three are equivalent.

At this point, we use the hypothesis that if $U$ and $V\ $are equivalent
isolated elements of $F(\mathcal{X})$, then $V$ is a translate of $U$.
It follows that if the pretree $P$ determined by $F(\mathcal{X})$ is discrete,
then the pretree $Q$ determined by $F(\mathcal{Y})$ is also discrete. Thus we
can construct the corresponding trees $T_{\mathcal{X}}$ and $T_{\mathcal{Y}}$,
and we can consider the quotient graphs of groups $G\backslash T_{\mathcal{X}%
}$ and $G\backslash T_{\mathcal{Y}}$. The bijection $\varphi$ between $P$ and
$Q$ shows that these graphs of groups are isomorphic, although this
isomorphism may not be induced by $\varphi$, so that the algebraic regular
neighbourhoods $\Gamma(\mathcal{X}:G)$ and $\Gamma(\mathcal{Y}:G)$ are isomorphic.

It follows that, if no $X_{\lambda}$ or $Y_{\lambda}$ is isolated and
invertible, then $\Gamma(\mathcal{X}:G)$ and $\Gamma(\mathcal{Y}:G)$ are
isomorphic, as claimed. If $X_{\lambda}$ is isolated and invertible, so there
is $g$ in $G$ such that $gX_{\lambda}=X_{\lambda}^{\ast}$, and if the
corresponding set $Y_{\lambda}$ satisfies $gY_{\lambda}=Y_{\lambda}^{\ast}$,
then the same proof as above applies to show that $\Gamma(\mathcal{X}:G)$ and
$\Gamma(\mathcal{Y}:G)$ are isomorphic.

Finally if some $X_{\lambda}$ is isolated in $E(\mathcal{X})$ and invertible,
and if the corresponding $Y_{\lambda}$ is not invertible, then the pretrees
$P$ and $Q$ are non-isomorphic, and so are the corresponding $G$--trees
$T_{\mathcal{X}}$ and $T_{\mathcal{Y}}$, and their quotients by $G$. In this
case, for each such $\lambda$, we subdivide the edge of $G\backslash
T_{\mathcal{X}}$ associated to $X_{\lambda}$ as in Definition
\ref{defnofalgregnbhdifRN-ready}. Similarly if some $Y_{\lambda}$ is isolated
in $E(\mathcal{Y})$ and invertible, and if the corresponding $X_{\lambda}$ is
not invertible, for each such $\lambda$, we subdivide the edge of $G\backslash
T_{\mathcal{Y}}$ associated to $Y_{\lambda}$ as in Definition
\ref{defnofalgregnbhdifRN-ready}. The resulting graphs of groups structures
are isomorphic.
\end{proof}

It remains to describe some cases where we know that $\mathcal{X}$ has an
algebraic regular neighbourhood. Part 1) of the following theorem was proved
in \cite{SS2}.

\begin{theorem}
\label{algregnbhdexists}Let $G$ be a group, and let $H_{1},\ldots,H_{n}$ be
finitely generated subgroups of $G$. For $i=1,\ldots,n$, let $X_{i}$ be a
nontrivial $H_{i}$--almost invariant subset of $G$, and let $\mathcal{X}$
denote the family $\{X_{i}\}_{1\leq i\leq n}$.

\begin{enumerate}
\item If $G$ is finitely generated, then the algebraic regular neighbourhood
$\Gamma(\mathcal{X}:G)$ exists.

\item If $(G,\mathcal{S})$ is a group system of finite type, and each $X_{i}$
is $\mathcal{S}$--adapted, then the algebraic regular neighbourhood
$\Gamma(\mathcal{X}:G)$ exists.
\end{enumerate}
\end{theorem}

\begin{remark}
Note that part 1) is the special case of part 2) when the family $\mathcal{S}$
is empty.

Using Lassonde's results in \cite{Lassonde}, one can extend parts 1) and 2) of
the above existence result to allow some $H_{i}$'s to be non-finitely
generated so long as the corresponding $X_{i}$ is associated to a splitting,
and the family $\mathcal{X}$ satisfies sandwiching (see Definition
\ref{defnofsandwiching} and Remark \ref{sandwichingholdsifallisfg}). Note also
that this sandwiching condition on the family $\mathcal{X}$ cannot simply be
omitted, as shown by Example \ref{incompatiblebutintersectionnozero}, and
Remark \ref{remarkGuirardelexample}.
\end{remark}

\begin{proof}
As usual, we define $E(\mathcal{X})$ and $F(\mathcal{X})$ as in Definition
\ref{defnofE(X)andF(X)}.

1) As $G$ is finitely generated, Lemma 2.3 of \cite{Scott symmint} tells us
that crossing of almost invariant sets in $F(\mathcal{X})$ is corner
symmetric. If $\mathcal{X}$ is not RN-ready, we need to show that each $X_{i}$
is equivalent to $Y_{i}$ such that the family $\mathcal{Y=}\{Y_{i}\}_{1\leq
i\leq n}$ is RN-ready. Rather than changing every $X_{i}$ to arrange that the
family $\mathcal{Y}$ is in good position, we need only concern ourselves with
the isolated $X_{i}$'s. Note that this lemma depends crucially on the finite
generation of the $H_{i}$'s. Now we can apply results from \cite{SS1}. For
each isolated element $X$ of $E(\mathcal{X})$, the fact that $X$ crosses none
of its translates means we can apply Theorem 2.8 of \cite{SS1} which tells us
that $X$ is equivalent to $Y\ $such that the translates of $Y$ are nested.
Further Theorem 2.5 of \cite{SS1} shows that these new sets can be chosen so
that the family of all translates of all of them is nested. After replacing
each isolated $X_{i}$ in this way, the new family $\mathcal{Y=}\{Y_{i}%
\}_{1\leq i\leq n}$ is RN-ready, as required. Thus $\Gamma(\mathcal{Y}:G)$,
and hence $\Gamma(\mathcal{X}:G)$, exists.

2) As $(G,\mathcal{S})$ is a group system of finite type, and each $X_{i}$ is
$\mathcal{S}$--adapted, Lemma \ref{crossingissymmetric} tells us that crossing
of almost invariant sets in $E(\mathcal{X})$ is corner symmetric. If the
family $\mathcal{X}$ is not RN-ready, we again replace the isolated elements
of the family $E(\mathcal{X})$, as in part 1). This time, we apply Theorems
\ref{splittingsexist} and \ref{splittingsarecompatible}, which are the
generalisations of Theorem 2.8 and 2.5 of \cite{SS1} to the setting of group
systems. As in part 1), after replacing each isolated $X_{i}$ in this way, the
new family $\mathcal{Y=}\{Y_{i}\}_{1\leq i\leq n}$ is RN-ready, as required.
Thus $\Gamma(\mathcal{Y}:G)$, and hence $\Gamma(\mathcal{X}:G)$, exists.
\end{proof}

\begin{remark}
When forming the algebraic regular neighbourhood of a family $\mathcal{X}$ of
almost invariant subsets $X_{\lambda}$, $\lambda\in\Lambda$, of a group $G$,
we allow the possibility of distinct equivalent isolated almost invariant
sets. For example, if $X_{\lambda}$ is not invertible, but is equivalent to an
invertible almost invariant set $Y_{\lambda}$, there is $k$ in $G$ such that
$kY_{\lambda}=Y_{\lambda}^{\ast}$, and $kX_{\lambda}$ is equivalent to
$X_{\lambda}^{\ast}$, but is not equal to $X_{\lambda}^{\ast}$. It is also
possible that there are infinitely many equivalent, but distinct, elements of
$E(\mathcal{X})$.The simplest example of this occurs when the family
$\mathcal{X}$ consists of a single $H$--almost invariant set $X$, and $X$ is
associated to a HNN splitting of $G$ of the form $H\ast_{H}$, with both
inclusions being isomorphisms. Thus $H$ is normal in $G$ with infinite cyclic
quotient. If $k$ denotes an element of $G$ which projects to a generator of
$G/H$, then $k^{r}X$ is equivalent to $X$, but cannot be equal to $X$, for
each non-zero integer $r$. In fact, after replacing $k$ by its inverse if
needed, we have $%
{\displaystyle\bigcup\nolimits_{r\in\mathbb{Z}}}
k^{r}X=G$, and $%
{\displaystyle\bigcap\nolimits_{r\in\mathbb{Z}}}
k^{r}X$ is empty. For a topological example, let $C$ be a simple closed curve
on the $2$--torus. For a detailed discussion of all this, see the proof of
Lemma 2.7 of \cite{NibloSageevScottSwarup}.
\end{remark}

As already discussed, this idea of an algebraic regular neighbourhood is
analogous to the idea of a topological regular neighbourhood, but a slight
modification of the definition is also useful.

We will say that a bipartite graph of groups $\Gamma$ is \textit{reduced} if
whenever $\Gamma$ has an isolated vertex $v$ with two distinct adjacent
vertices, then neither of these adjacent vertices is isolated. If $\Gamma$ is
a finite bipartite graph of groups which is not reduced, we can reduce it as
follows. If $v$ and $w$ are adjacent isolated vertices of $\Gamma$, each with
two distinct adjacent vertices, then we collapse the three distinct edges
incident to at least one of $v$ and $w$ to a single edge. Repeating this
process will eventually yield a reduced bipartite graph of groups, which is
clearly unique.

\begin{definition}
If $\Gamma$ is the algebraic regular neighbourhood of a family of almost
invariant subsets of a group $G$, and if $\Gamma$ is finite, then \emph{the
reduced algebraic regular neighbourhood} of the same family of almost invariant
subsets of $G$ is the result of the above reduction process applied to
$\Gamma$.
\end{definition}

To explain the need for this definition, consider the following example. We
return to the example of closed curves $C_{\lambda}$ on a closed orientable
surface $M$. Let $F$ be a compact subsurface of $M$, such that no component of
$\partial F$ bounds a disc in $M$, and $F$ is not an annulus. Let the family
$\{C_{\lambda}\}_{\lambda\in\Lambda}$ consist of closed curves in $M$ which
are homotopic into $F$. Once this family is large enough, we expect that $F$
will be a component of the regular neighbourhood $N$ of the family. But if the
family includes a boundary component $S$ of $F$, there will be an additional
annulus component of $N$. In the purely algebraic setting, let $G\ $denote
$\pi_{1}(M)$, let $\{X_{\lambda}\}_{\lambda\in\Lambda}$ be the family of all
almost invariant subsets of $G$ over infinite cyclic subgroups of $\pi_{1}%
(F)$, and let $S_{1},\ldots,S_{k}$ denote the components of $\partial F$. Then
the algebraic regular neighbourhood $\Gamma$ of the $X_{\lambda}$'s in $G$ has
one $V_{0}$--vertex which carries $\pi_{1}(F)$, and $k$ additional isolated
$V_{0}$--vertices which carry the groups $\pi_{1}(S_{i})$, $1\leq i\leq k$.
The reduced version $\Gamma_{red}$ of $\Gamma$ has only one $V_{0}$--vertex,
which carries $\pi_{1}(F)$, so the extra isolated $V_{0}$--vertices have been
removed. In this example, $\Gamma_{red}$ has the property that each
$X_{\lambda}$ is enclosed by a unique $V_{0}$--vertex, but this need not
always be the case. For example, suppose that the subsurface $F$ of $M$ has
two components $F_{1}$ and $F_{2}$, neither of which is an annulus, and that a
component $S_{1}$ of $\partial F_{1}$ is parallel to a component $S_{2}$ of
$\partial F_{2}$. Let $\{X_{\lambda}\}_{\lambda\in\Lambda}$ be the family of
all almost invariant subsets of $G$ over infinite cyclic subgroups of $\pi
_{1}(F_{i})$, $i=1,2$. Then the reduced algebraic regular neighbourhood
$\Gamma$ of the $X_{\lambda}$'s in $G$ has two $V_{0}$--vertices which carry
$\pi_{1}(F_{1})$ and $\pi_{1}(F_{2})$. Thus the almost invariant subset $X$
corresponding to $S_{1}$ and $S_{2}$ is enclosed by both $V_{0}$--vertices.
Note that there is an isolated $V_{1}$--vertex which also encloses $X$.

At this point, we have defined reduced and unreduced algebraic regular
neighbourhoods, and shown they exist in certain cases, but have not discussed
the properties of our construction. We can substantially simplify some of the
arguments given in \cite{SS2} by using very good position, and we will now
explain how this works. This is described in Lassonde's paper \cite{Lassonde},
but we repeat it for completeness.

Let $G$ be a finitely generated group, let $H_{1},\ldots,H_{n}$ be finitely
generated subgroups of $G$, and, for $i=1,\ldots,n$, let $X_{i}$ be a
nontrivial $H_{i}$--almost invariant subset of $G$. Recall that in
\cite{NibloSageevScottSwarup}, the authors showed that we can replace each
$X_{i}$ by an equivalent set $Y_{i}$ such that the $Y_{i}$'s are in very good
position. In section \ref{goodposition}, this result was extended to the
setting of group systems of finite type.

Suppose that $Y_{1},\ldots,Y_{n}$ are in very good position, and let
$\mathcal{Y}$ denote the family $\{Y_{1},\ldots,Y_{n}\}$. Now we consider the
cubing $C=C(\mathcal{Y})$ constructed by Sageev in \cite{Sageev-cubings} from
the set $E=E(\mathcal{Y})\ $partially ordered by inclusion. Each element of
$E$ determines a hyperplane in $C$, and nested elements of $E$, which are not
equal or complementary, determine disjoint hyperplanes. As $Y_{1},\ldots
,Y_{n}$ are in very good position, non-nested elements of $E$ cross, so we see
that two distinct hyperplanes of $C$ fail to be disjoint if and only if the
corresponding elements of $E$ cross. In \cite{Niblo}, Niblo considered the
space $W$ obtained from $C$ by removing all separating vertices in \ $C$, and
then replacing each component of the result by its closure. Each component of
$W$ is itself a cubing, and Niblo showed that in each component the
hyperplanes are cross-connected. Thus the components of $W$ correspond
precisely to the CCCs of $\overline{E}$. Recall that these CCCs form the
$V_{0}$--vertices of a bipartite $G$--tree $T$ whose quotient by $G$ yields
the algebraic regular neighbourhood $\Gamma(\mathcal{Y}:G)$. Instead of using
pretrees to construct $T$, Niblo gives a natural way to directly construct $T$
from the cubing $C$. The $V_{0}$--vertices of $T$ are simply the components of
$W$, and the $V_{1}$--vertices of $T$ are the separating vertices of $C$. An
edge joins a component $w$ of $W$ to a separating vertex $v$ of $C$ if and
only if $v$ lies in $w$. If some $Y_{i}$ is isolated and invertible, we again
subdivide the corresponding edge of $G\backslash T$ in order to construct
$\Gamma(\mathcal{Y}:G)$. It is easy to see that this construction gives rise
to the same tree $T\ $as the construction in Definition \ref{defnofalgregnbhd}%
. Thus the cubing $C$ encodes all the information in $\Gamma(\mathcal{Y}:G)$
and a great deal more information about how all the elements of $E$ cross. For
later reference, we set this out as a proposition.

\begin{proposition}
\label{verygoodpositioncubingalsoyiledsalgregnbhd}Let $G$ be a finitely
generated group, let $H_{1},\ldots,H_{n}$ be finitely generated subgroups of
$G$, and, for $i=1,\ldots,n$, let $Y_{i}$ be a nontrivial $H_{i}$--almost
invariant subset of $G$. Suppose that $Y_{1},\ldots,Y_{n}$ are in very good
position, let $\mathcal{Y}$ denote the family $\{Y_{1},\ldots,Y_{n}\}$, let
$E(\mathcal{Y})=\{gY_{i},gY_{i}^{\ast}:g\in G,1\leq i\leq n\}$, and let $C$
denote the cubing constructed by Sageev in \cite{Sageev-cubings} from the set
$E(\mathcal{Y})\ $partially ordered by inclusion. Let $W$ denote the space
obtained from $C$ by removing all separating vertices, and then replacing each
component of the result by its closure. Let $T$ be the bipartite $G$--tree
whose $V_{0}$--vertices are the components of $W$, and whose $V_{1}$--vertices
are the separating vertices of $C$. An edge of $T$ joins a component $w$ of
$W$ to a separating vertex $v$ of $C$ if and only if $v$ lies in $w$.

Then\emph{ }the\emph{ algebraic regular neighbourhood }$\Gamma\mathcal{(Y}%
:G)$\emph{ of the family }$\mathcal{Y}$\emph{ in} $G$, is equal to the
quotient graph of groups $G\backslash T$, unless some $Y_{i}$'s are isolated
in $E(\mathcal{Y})$ and invertible. In that case, $\Gamma(\mathcal{Y}:G)$ is
obtained from $G\backslash T$ by inserting a $V_{0}$--vertex in the centre of
each of the edges of $G\backslash T$ whose edge splittings are given by an
isolated and invertible $Y_{i}$, and changing the adjacent $V_{0}$--vertex to
a $V_{1}$--vertex.
\end{proposition}

Now the following results summarise the key properties of $\Gamma
(\mathcal{X}:G)$, when it exists. For simplicity, and to help the reader
follow the history, we continue to treat separately the cases of group systems
$(G,\mathcal{S})$ of finite type, and the special case when the family
$\mathcal{S}$ is empty and so $G$ is finitely generated. Note that the
properties listed in the next result formed the original definition of
$\Gamma(\mathcal{X}:G)$, Definition 6.1 of \cite{SS2}, after correcting that
definition as in part 2 of the theorem below.

\begin{theorem}
\label{keypropertiesofalgregnbhd}Let $G$ be a finitely generated group with a
family of finitely generated subgroups $\{H_{\lambda}\}_{\lambda\in\Lambda}$.
For each $\lambda\in\Lambda$, let $X_{\lambda}$ denote a nontrivial
$H_{\lambda}$--almost invariant subset of $G$, and let $\mathcal{X}$ denote
$\{X_{\lambda}:\lambda\in\Lambda\}$. Let $E(\mathcal{X})$ denote
$\{gX_{\lambda},gX_{\lambda}^{\ast}:g\in G,\lambda\in\Lambda\}$. If the
algebraic regular neighbourhood $\Gamma(\mathcal{X}:G)$ of $\mathcal{X}$ in $G$
exists, then $\Gamma(\mathcal{X}:G)$ is the unique bipartite graph of groups
structure $\Gamma$ for $G$ such that the following conditions hold:

\begin{enumerate}
\item Each $X_{\lambda}$ is enclosed by some $V_{0}$--vertex of $\Gamma$, and
each $V_{0}$--vertex of $\Gamma$ encloses some $X_{\lambda}$.

\item If $X$ is a $H$--almost invariant subset of $G$ which crosses no element
of $E(\mathcal{X})$, if $X$ is associated to a splitting of $G$ over $H$, and
if $X$ is sandwiched by $E(\mathcal{X})$, then $X$ is enclosed by some $V_{1}%
$--vertex of $\Gamma$.

\item $\Gamma$ is minimal.

\item There is a bijection $f$ from the $G$--orbits of isolated elements of
$\overline{E}(\mathcal{X})$ to the isolated $V_{0}$--vertices of $\Gamma$,
such that $f(\overline{X})$ encloses $X$.
\end{enumerate}
\end{theorem}

\begin{remark}
In part 2), it is not assumed that $H$ is finitely generated.

As $G$ is finitely generated, and $\Gamma$ is minimal, it follows that
$\Gamma$ is finite.
\end{remark}

\begin{proof}
We start by showing that $\Gamma(\mathcal{X}:G)$ satisfies conditions 1)-4).
Condition 4) is clear from Definition \ref{defnofalgregnbhd}. Conditions 1)-2)
were proved in Lemma 5.1 and Proposition 5.7 of \cite{SS2}, using good
position. Finally condition 3) is just Proposition 5.2 of \cite{SS2}.

Conversely, if $\Gamma$ is a bipartite graph of groups structure for $G$ such
that conditions 1)-4) hold, the uniqueness result in Theorem 6.7 of \cite{SS2}
shows that $\Gamma$ must equal $\Gamma(\mathcal{X}:G)$.
\end{proof}

\begin{remark}
\label{XnotcrossinganyXiisenclosedbysomeV1-vertex}In condition 2), if $H$ is
assumed to be finitely generated, the argument of Proposition 5.7 of
\cite{SS2} applies to any $H$--almost invariant subset $X$ of $G$ which
crosses no element of $E(\mathcal{X})$, whether or not it is associated to a
splitting. The conclusion is that such $X$ must be enclosed by some $V_{1}%
$--vertex, which was first proved in Proposition 5.7 of \cite{SS2}.
\end{remark}

Before proving the corresponding result in the setting of group systems of
finite type, we need the following technical lemma and corollary.

\begin{lemma}
\label{UcannotbeK-adapted}Let $G$ be a finitely generated group with a
finitely generated subgroup $H$, and let $U$ be a nontrivial $H$--almost
invariant subset of $G$. Let $k$ be an element of $G$, and let $K$ be the
cyclic subgroup generated by $k$. If $kU<U$, then $U\ $cannot be
$K$--adapted.{}
\end{lemma}

\begin{remark}
The notation $kU<U$ means that $kU\leq U$ and $kU\neq U$.
\end{remark}

\begin{proof}
We will work in the Cayley graph $\Gamma$ of $G$ with respect to some finite
generating set. As $k^{n}U<U$, for all $n\geq1$, it follows that $K$ is
infinite cyclic. We let $K_{+}$ denote $\{k^{n}:n\geq1\}$, and let $K_{-}$
denote $\{k^{n}:n\leq-1\}$.

Lemma 2.31 of \cite{SS2} tells us that $(g\in G:gU\leq U\}$ is contained in a
bounded neighbourhood of $U$ in $\Gamma$. It follows that $K_{+}$ is contained
in a bounded neighbourhood of $U$. Further $K_{+}$ cannot be contained in a
bounded neighbourhood of $\delta U$. For this would imply that $K_{+}$ is
$H$--finite, and hence that there is $n\geq1$ such that $k^{n}\in H$, which is
impossible as $k^{n}U<U$ but $hU=U$, for all $h\in H$. We conclude that
$K_{+}$ is $H$--infinite, and that $K\cap U$ is also $H$--infinite.

Now the fact that $kU<U$ implies that $k^{-1}U^{\ast}<U^{\ast}$. Thus we can
apply the above arguments with $K_{-}$ replacing $K_{+}$ and $U^{\ast}$
replacing $U$. We conclude that $K\cap U^{\ast}$ is also $H$--infinite.

As $K\cap U$ and $K\cap U^{\ast}$ are both $H$--infinite, it follows that
$U\ $cannot be $K$--adapted, as required.
\end{proof}

\begin{corollary}
\label{Ucannot beK-adaptedinsystem}Let $(G,\mathcal{S})$ be a group system of
finite type, let $H$ be a finitely generated subgroup of $G$, and let $U$
denote a nontrivial $\mathcal{S}$--adapted $H$--almost invariant subset of
$G$. Let $k$ be an element of $G$, and let $K$ be the cyclic subgroup
generated by $k$. If $kU<U$, then $U\ $cannot be $K$--adapted.
\end{corollary}

\begin{proof}
We consider a $S^{\infty}$--extension $\overline{G}$ of $G$, and the unique
extension $\overline{U}$ of $U$. As $kU<U$, it follows that $k\overline
{U}<\overline{U}$. Now the preceding lemma tells us that $K\cap\overline{U}$
and $K\cap\overline{U}^{\ast}$ are both $H$--infinite. As $K\subset G$, it
follows that $K\cap U$ and $K\cap U^{\ast}$ are both $H$--infinite, so that
$U\ $cannot be $K$--adapted, as required.
\end{proof}

Our next result summarises the key properties of $\Gamma(\mathcal{X}:G)$, when
it exists, in the setting of group systems of finite type.

\begin{theorem}
\label{keypropertiesofalgregnbhdofpair}Let $(G,\mathcal{S})$ be a group system
of finite type, and let $\{H_{\lambda}\}_{\lambda\in\Lambda}$ be a family of
finitely generated subgroups of $G$. For each $\lambda\in\Lambda$, let
$X_{\lambda}$ denote a nontrivial $\mathcal{S}$--adapted $H_{\lambda}$--almost
invariant subset of $G$, and let $\mathcal{X}$ denote $\{X_{\lambda}%
:\lambda\in\Lambda\}$. Let $E(\mathcal{X})$ denote $\{gX_{\lambda}%
,gX_{\lambda}^{\ast}:g\in G,\lambda\in\Lambda\}$. If the algebraic regular
neighbourhood $\Gamma(\mathcal{X}:G)$ of $\mathcal{X}$ in $G$ exists, then
$\Gamma(\mathcal{X}:G)$ is the unique bipartite graph of groups structure
$\Gamma$ for $G$ such that the following conditions hold:

\begin{enumerate}
\item $\Gamma$ is $\mathcal{S}$--adapted.

\item Each $X_{\lambda}$ is enclosed by some $V_{0}$--vertex of $\Gamma$, and
each $V_{0}$--vertex of $\Gamma$ encloses some $X_{\lambda}$.

\item If $X$ is a $\mathcal{S}$--adapted $H$--almost invariant subset of $G$
which crosses no element of $E(\mathcal{X})$, if $X$ is associated to a
splitting of $G$ over $H$, and if $X$ is sandwiched by $E(\mathcal{X})$, then
$X$ is enclosed by some $V_{1}$--vertex of $\Gamma$.

\item $\Gamma$ is minimal.

\item There is a bijection $f$ from the $G$--orbits of isolated elements of
$\overline{E}(\mathcal{X})$ to the isolated $V_{0}$--vertices of $\Gamma$,
such that $f(\overline{X})$ encloses $X$.
\end{enumerate}
\end{theorem}

\begin{remark}
\label{remarkongroupsystems}In part 3), it is not assumed that $H$ is finitely generated.

Recall from Definition \ref{defnofadaptedgraphofgroups} that a graph of groups
structure $\Gamma$ for $G$ is $\mathcal{S}$--adapted if each $S_{i}$ in
$\mathcal{S}$ is conjugate into a vertex group of $\Gamma$.

Note that Theorem \ref{keypropertiesofalgregnbhd} is the special case of
Theorem \ref{keypropertiesofalgregnbhdofpair} obtained when the family
$\mathcal{S}$ is empty.

As $(G,\mathcal{S})$ is a group system of finite type, and $\Gamma
(\mathcal{X}:G)$ is $\mathcal{S}$--adapted, it follows that $G$ is finitely
generated relative to some finite collection of vertex groups of $\Gamma$. As
$\Gamma$ is minimal, this implies that $\Gamma$ must be finite.
\end{remark}

\begin{proof}
We start by showing that the algebraic regular neighbourhood $\Gamma
(\mathcal{X}:G)$ satisfies conditions 1)-5).

1) Let $T$ be the universal covering $G$--tree of $\Gamma$, and let $S$ denote
one of the groups in $\mathcal{S}$. We need to show that $S$ fixes some vertex
of $T$. In order to prove this result, we will follow the ideas in the proof
of Lemma \ref{splittingisadaptediffa.i.setis}. First suppose that $S$ is
finitely generated. If $S$ does not fix a vertex of $T$, then some element $k$
of $S$ also does not fix a vertex and so has an axis $l$ in $T$. Let $Z$ be a
$V_{0}$--vertex of $l$, and consider the vertices $k^{-1}Z$ and $kZ$. These
vertices correspond to points $k^{-1}A$, $A$ and $kA$ of the pretree $P$ used
to define $T$. As $Z$ is between $k^{-1}Z$ and $kZ$ on the axis $l$, it
follows that $A$ is between $k^{-1}A$ and $kA$ in the pretree $P$. Hence there
are almost invariant sets $V$, $U$ and $W$ in $k^{-1}A$, $A$ and $kA$
respectively such that $V<U<W$. Thus $kV$ and $kU$ are almost invariant sets
in $A$ and $kA$ respectively such that $kV<kU$. Now Lemma 3.5 of \cite{SS2}
and the fact that $U<W$ in $kA$ implies that for any almost invariant set $X$
in $kA$, we have $U<X$ or $U<X^{\ast}$. In particular $U<kU$ or $U<kU^{\ast}$.
Similarly the fact that $kV<kU$ implies that for any almost invariant set $Y$
in $A$, we have $Y<kU$ or $Y^{\ast}<kU$. In particular $U<kU$ or $U^{\ast}%
<kU$. It follows that we must have $U<kU$, so that $k^{-1}U<U$. Now we apply
Corollary \ref{Ucannot beK-adaptedinsystem} with $k^{-1}$ in place of $k$ to
tell us that $U\ $cannot be $K$--adapted, and hence cannot be $S$--adapted,
which is a contradiction. Thus $S$ must fix a vertex of $T$.

In the general case when $S$ need not be finitely generated, we can apply the
argument at the end of the proof of Lemma \ref{splittingisadaptediffa.i.setis}.

We conclude that $\Gamma(\mathcal{X}:G)$ is $\mathcal{S}$--adapted, which
completes the proof of part 1).

2) As for part 1) of Theorem \ref{keypropertiesofalgregnbhd}, we wish to apply
the proof of Lemma 5.1 of \cite{SS2}. There is no difficulty till near the end
of that proof where Lemma 2.31 of \cite{SS2} is quoted. To deal with this
difficulty we proceed as follows. As in the proof of Lemma 2.31 of \cite{SS2},
let $U$ be an element of $E(\mathcal{X})$ for which we have shown that
$gU^{(\ast)}\leq U$, for every $g\in Z_{s}^{\ast}$. Now consider a
$\mathcal{S}^{\infty}$ extension $\overline{G}$ of $G$, and let $\overline{U}$
be the unique extension of $U$ to $\overline{G}$. As the splitting associated
to $Z_{s}$ is $\mathcal{S}$--adapted, there is an associated splitting of
$\overline{G}$ refining the decomposition of $\overline{G}$ as a
$\mathcal{S}^{\infty}$ extension $\overline{G}$ of $G$. We let $\overline
{Z_{s}}$ denote an associated almost invariant subset of $\overline{G}$. Now
we apply Lemma 2.31 of \cite{SS2} to tell us that $\overline{Z_{s}}^{\ast}$
lies in a bounded neighbourhood of $U$ in the Cayley graph of $\overline{G}$.
Thus $\overline{U}^{\ast}\cap\overline{Z_{s}}^{\ast}$ is small, and hence
$U\cap Z_{s}^{\ast}$ is small. Now the rest of the proof of Lemma 5.1 of
\cite{SS2} applies to complete the proof of part 2).

3) We proceed as in the proof of part 2) but using the proof of Proposition
5.7 of \cite{SS2} instead of Lemma 5.1 of \cite{SS2}.

4) This follows from the argument of Proposition 5.2 of \cite{SS2}.

5) This is clear from Definition \ref{defnofalgregnbhd}.

This completes the proof that $\Gamma(\mathcal{X}:G)$ satisfies conditions
1)-5) of the theorem.

Conversely, suppose that $\Gamma$ is a bipartite graph of groups structure for
$G$ such that conditions 1)-5) hold. We note that the proof of the uniqueness
result in Theorem 6.7 of \cite{SS2} applies equally well to the setting of
group systems of finite type. From Remark \ref{remarkongroupsystems}, we know
that any such algebraic regular neighbourhood is finite. The crucial point to
note is that as $\Gamma$ is $\mathcal{S}$--adapted, Lemma
\ref{adaptedgraphofgroups} shows that each edge splitting of $\Gamma$ is
$\mathcal{S}$--adapted. Now condition 3) of the theorem leads to the
uniqueness result as in the proof of Theorem 6.7 of \cite{SS2}.
\end{proof}

\begin{remark}
In the proof of condition 2), if $H$ is assumed to be finitely generated, the
argument applies to any $\mathcal{S}$--adapted $H$--almost invariant subset
$X$ of $G$ which crosses no element of $E$, whether or not it is associated to
a splitting. The conclusion is that such $X$ must be enclosed by some $V_{1}%
$--vertex, as in the case when $\mathcal{S}$ is empty and $G$ is finitely generated.
\end{remark}

The conclusion in the above theorem that $\Gamma(X:G)$ is $\mathcal{S}%
$--adapted, implies that each $S$ in $\mathcal{S}$ must be conjugate into some
vertex group of $\Gamma(X:G)$. It seems natural to suppose that as each
$X_{i}$ is adapted to $S$, then $S$ must be conjugate into a $V_{1}$--vertex
group, but this is not the case. To see this consider the example of the
topological JSJ decomposition of the pair $(M,\varnothing)$ discussed in the
introduction. The almost invariant subsets of $\pi_{1}(M)$ which correspond to
essential tori in $(M,\varnothing)$ are adapted to the fundamental group of
each boundary component of $M$. Thus the graph of groups decomposition
$\Gamma$ of $\pi_{1}(M)$ determined by this JSJ decomposition is also adapted
to the fundamental group of each boundary component of $M$. Now suppose that
$M$ has a torus boundary component $T$ and that $V(M)$ has a component $W$
which contains $T$. Then $\pi_{1}(T)$ is conjugate into a $V_{0}$--vertex
group, namely $\pi_{1}(W)$, and is not conjugate into any other vertex group.
In particular, $\pi_{1}(T)$ is not conjugate into a $V_{1}$--vertex group,
although $\Gamma$ is adapted to $\pi_{1}(T)$.

Finally the analogous results on reduced algebraic regular neighbourhoods
follow easily from the above.

\begin{theorem}
\label{keypropertiesofreducedalgregnbhd} Let $G$ be a finitely generated group
with a family of finitely generated subgroups $\{H_{\lambda}\}_{\lambda
\in\Lambda}$. For each $\lambda\in\Lambda$, let $X_{\lambda}$ denote a
nontrivial $H_{\lambda}$--almost invariant subset of $G$, and let
$\mathcal{X}$ denote $\{X_{\lambda}:\lambda\in\Lambda\}$. Let $E(\mathcal{X})$
denote $\{gX_{\lambda},gX_{\lambda}^{\ast}:g\in G,\lambda\in\Lambda\}$. Assume
that the algebraic regular neighbourhood $\Gamma(\mathcal{X}:G)$ of
$\mathcal{X}$ in $G$ exists. Let $\Gamma_{red}(\mathcal{X}:G)$ denote the
reduced algebraic regular neighbourhood of $\mathcal{X}$ in $G$. Then
$\Gamma_{red}(\mathcal{X}:G)$ is the unique bipartite graph of groups
structure $\Gamma$ for $G$ such that the following conditions hold:

\begin{enumerate}
\item Each $X_{\lambda}$ is enclosed by some $V_{0}$--vertex of $\Gamma$, and
each $V_{0}$--vertex of $\Gamma$ encloses some $X_{\lambda}$.

\item If $\sigma$ is a splitting of $G$ over a subgroup $H$ (which need not be
finitely generated), and if $X$ is a $H$--almost invariant subset of $G$
associated to $\sigma$ such that $X$ crosses no element of $E(\mathcal{X})$,
and if $X$ is sandwiched by $E(\mathcal{X})$, then $\sigma$ is enclosed by
some $V_{1}$--vertex of $\Gamma$.

\item $\Gamma$ is minimal, and reduced bipartite.
\end{enumerate}
\end{theorem}

In the setting of group systems, we have the following result.

\begin{theorem}
\label{keypropertiesofreducedalgregnbhdofapair}Let $(G,\mathcal{S})$ be a
group system of finite type, and let $\{H_{\lambda}\}_{\lambda\in\Lambda}$ be
a family of finitely generated subgroups of $G$. For each $\lambda\in\Lambda$,
let $X_{\lambda}$ denote a nontrivial $\mathcal{S}$--adapted $H_{\lambda}%
$--almost invariant subset of $G$, and let $\mathcal{X}$ denote $\{X_{\lambda
}:\lambda\in\Lambda\}$. Let $E(\mathcal{X})$ denote $\{gX_{\lambda
},gX_{\lambda}^{\ast}:g\in G,\lambda\in\Lambda\}$. Assume that the algebraic
regular neighbourhood $\Gamma(\mathcal{X}:G)$ of $\mathcal{X}$ in $G$ exists.
Let $\Gamma_{red}(\mathcal{X}:G)$ denote the reduced algebraic regular
neighbourhood of $\mathcal{X}$ in $G$. Then $\Gamma_{red}(\mathcal{X}:G)$ is
the unique $\mathcal{S}$--adapted bipartite graph of groups structure $\Gamma$
for $G$ such that the following conditions hold:

\begin{enumerate}
\item Each $X_{\lambda}$ is enclosed by some $V_{0}$--vertex of $\Gamma$, and
each $V_{0}$--vertex of $\Gamma$ encloses some $X_{\lambda}$.

\item If $\sigma$ is a $\mathcal{S}$--adapted splitting of $G$ over a subgroup
$H$ (which need not be finitely generated), and if $X$ is a $H$--almost
invariant subset of $G$ associated to $\sigma$ such that $X$ crosses no
element of $E(\mathcal{X})$, and if $X$ is sandwiched by $E(\mathcal{X})$,
then $\sigma$ is enclosed by some $V_{1}$--vertex of $\Gamma$.

\item $\Gamma$ is minimal and reduced bipartite.
\end{enumerate}
\end{theorem}

Note that these results are identical to those for an unreduced regular
neighbourhood in Theorems \ref{keypropertiesofalgregnbhd} and
\ref{keypropertiesofalgregnbhdofpair} except that in each case, condition 4)
has been removed, and the fact that $\Gamma$ is reduced bipartite has been
added to condition 3).

\section{Coends\label{coends}}

Up to this point everything we have discussed involved the number of ends
$e(G,H)$ of a pair $(G,H)$ of groups, or the number $e(G,H:\mathcal{S})$ of
such ends which are adapted to a family $\mathcal{S}$ of subgroups of $G$. In
this section, we discuss some terminology which will be very useful later in
this paper.

\begin{definition}
\label{defnofendsof[H]}Let $(G,\mathcal{S})$ be a group system, let $H$ be a
subgroup of $G$, and let $[H]$ denote the family of all subgroups of $G$ which
are commensurable with $H$. Then \textsl{the number of ends of }$[H]$\textsl{
in the system }$(G,\mathcal{S})$ is

$\sup\{e(G,K:\mathcal{S}):K\in\lbrack H]\}$, and will be denoted
$e(G,[H]:\mathcal{S})$.
\end{definition}

In the rest of this section, we will discuss how this definition relates to
the concept of coends. This discussion is included for completeness and is not
needed to understand the results in this paper.

The concept of coends played an important role in \cite{SS2}. From the point
of view of that paper, a key property of coends is contained in the following result.

\begin{lemma}
\label{propertyofcoends}Let $G$ be a finitely generated group, and let $H$ be
a finitely generated subgroup of $G$. Further suppose that for any subgroup
$K$ of $H$ which has infinite index in $H$, the number of ends $e(G,K)$ equals
$1$. Then the number of coends of $H$ in $G$, denoted $\widetilde{e}(G,H)$,
equals $\sup\{e(G,K):K\subset H,\left\vert H:K\right\vert <\infty\}$. In
particular, if $\widetilde{e}(G,H)$ is finite, there is a subgroup $K$ of $H$
which has finite index in $H$ such that $\widetilde{e}(G,H)=e(G,K)$.
\end{lemma}

\begin{remark}
If $K^{\prime}$ is a subgroup of $K$ of finite index, then $e(G,K^{\prime
})\geq e(G,K)$. Thus $\sup\{e(G,K):K\subset H,\left\vert H:K\right\vert
<\infty\}=\sup\{e(G,K):K\in\lbrack H]\}$.
\end{remark}

In Corollary \ref{propertyofadaptedcoends}, at the end of this section, we
will prove the adapted analogue of this result which will connect the
invariant $e(G,[H]:\mathcal{S})$ with coends.

In \cite{SS2}, the only use made of coends was in a setting where the
hypotheses of Lemma \ref{propertyofcoends} held. Further all the relevant
arguments of that paper used the invariant $\sup\{e(G,K):K\subset H,\left\vert
H:K\right\vert <\infty\}$ directly rather than $\widetilde{e}(G,H)$. In this
paper, we follow a similar path, but using the invariant $e(G,[H]:\mathcal{S}%
)$ defined above, which is the adapted analogue of the invariant
$\sup\{e(G,K):K\subset H,\left\vert H:K\right\vert <\infty\}$.

The concept of the number of coends of $H$ in $G$ was introduced independently
by Kropholler and Roller \cite{KR}, by Bowditch \cite{B3} and by Geoghegan
\cite{Geoghegan}. Kropholler and Roller called this the number of relative
ends of the pair $(G,H)$, and used the notation $\widetilde{e}(G,H)$. Bowditch
talked of the number of coends of the pair, and Geoghegan used the term
`filtered coends'. Bowditch and Geoghegan made their definitions on the
assumption that $G$ was finitely generated, while Kropholler and Roller made
no such assumption. In this section we will use the definition and notation of
Kropholler and Roller, but will use Bowditch's terminology of coends.

We recall our notation from sections \ref{ends} and
\ref{adaptedalmostinvariantsets}. Let $G$ be a group and let $E$ be a set on
which $G$ acts on the right. Let $PE$ denote the power set of $E$, and let
$FE$ denote the family of all finite subsets of $E$. Two sets $A$ and $B$
whose symmetric difference lies in $FE$ are almost equal, and we write
$A\overset{a}{=}B$. This amounts to equality in the quotient group $PE/FE$.
Now define
\[
QE=\{A\subset E:\forall g\in G,\;A\overset{a}{=}Ag\}.
\]
The action of $G$ on $PE$ by right translation preserves the subgroups
$QE\ $and $FE$, and $QE/FE$ is the subgroup of elements of $PE/FE$ fixed under
the induced action. Elements of $QE$ are said to be \textit{almost invariant}.

If $H$ is a subgroup of $G$, and we take $E$ to be the coset space
$H\backslash G$ of all cosets $Hg$, with the action of $G$ being right
multiplication, then the number of ends of the pair $(G,H)$ is
\[
e(G,H)=\dim_{\mathbb{Z}_{2}}\;\left(  \frac{Q(H\backslash G)}{F(H\backslash
G)}\right)  .
\]

If $(G,\mathcal{S})$ is a group system, we let $Q(H\backslash G)(\mathcal{S})$
denote the subset of $P(H\backslash G)$ consisting of all almost invariant
subsets of $H\backslash G$ whose pre-image in $G$ is $\mathcal{S}$--adapted.
Again it is easy to see that $Q(H\backslash G)(\mathcal{S})$ is a subspace of
$P(H\backslash G)$. Thus we can define the number of $\mathcal{S}$--adapted
ends of the pair $(G,H)$ to be%

\[
e(G,H:\mathcal{S})=\dim_{\mathbb{Z}_{2}}\;\left(  \frac{Q(H\backslash
G)(\mathcal{S})}{F(H\backslash G)}\right)  .
\]

Now there is a natural bijection between almost invariant subsets of
$H\backslash G$ and $H$--almost invariant subsets of $G$, in which finite
subsets of $H\backslash G$ correspond to $H$--finite subsets of $G$. Thus
$e(G,H)$ can also be viewed as the dimension of the space of $H$--almost
invariant subsets of $G$ modulo $H$--finite subsets. Similarly
$e(G,H:\mathcal{S})$ can also be viewed as the dimension of the space of
$\mathcal{S}$--adapted $H$--almost invariant subsets of $G$ modulo $H$--finite subsets.

In \cite{KR}, Kropholler and Roller considered subsets $X$ of $G$ such that
$X$ and $Xg$ are $H$--almost equal for all $g$ in $G$. But they did not assume
that $X$ is invariant under the left action of $H$. Thus $X$ need not be
$H$--almost invariant. The number of coends of the pair $(G,H)$, denoted
$\widetilde{e}(G,H)$, is then the dimension of the space of such subsets of
$G$ modulo $H$--finite subsets. More formally, let $F_{H}(G)$ denote the
family of all $H$--finite subsets of $G$. Then Kropholler and Roller defined%
\[
\widetilde{e}(G,H)=\dim_{\mathbb{Z}_{2}}\;\left(  \frac{PG}{F_{H}(G)}\right)
^{G},
\]
where the exponent $G$ denotes the subspace of points fixed by the right
action of $G$. Note that the discussion in the preceding paragraph makes it
clear that $e(G,H)\leq\widetilde{e}(G,H)$. We will refer to $\widetilde{e}%
(G,H)$ as the number of coends of the pair $(G,H)$, or as the number of coends
of $H$ in $G$.

Now we will define the adapted analogue of $\widetilde{e}(G,H)$. This extends
Definition \ref{defnofadapted} to the present setting.

\begin{definition}
\label{defnofadaptedforcoends}Let $(G,\mathcal{S})$ be a group system, let $H$
be a subgroup of $G$, and let $S$ be a group in $\mathcal{S}$.

A subset $X$ of $G$ is \textsl{strictly adapted to the subgroup} $S$ if, for
all $g\in G$, the coset $gS$ is contained in $X$ or in $X^{\ast}$.

A subset $X$ of $G$ such that $X$ and $Xg$ are $H$--almost equal for all $g$
in $G$ is \textsl{adapted to} $S$ if it is $H$--almost equal to a subset
$X^{\prime}$ of $G$ such that $X^{\prime}$ is strictly adapted to $S$.

A subset $X$ of $G$ such that $X$ and $Xg$ are $H$--almost equal for all $g$
in $G$ is \textsl{adapted to the family} $\mathcal{S}$ if it is adapted to
each $S$ in $\mathcal{S}$.
\end{definition}

Let $Q_{H}(G,\mathcal{S})$ denote the collection of all subsets $X$ of $G$
such that $X$ and $Xg$ are $H$--almost equal for all $g$ in $G$, and $X$ is
adapted to the family $\mathcal{S}$. Let $F_{H}(G)$ denote the family of all
$H$--finite subsets of $G$. We define the number of coends of the pair $(G,H)$
adapted to $\mathcal{S}$ to be%
\[
\bar{e}(G,H:\mathcal{S})=\dim_{\mathbb{Z}_{2}}\;\left(  \frac
{Q_{H}(G,\mathcal{S})}{F_{H}(G)}\right)  .
\]

As in the non-adapted case, it is easy to see that $e(G,H:\mathcal{S}%
)\leq\bar{e}(G,H:\mathcal{S})$.

The connection between $e(G,H)$ and $\bar{e}(G,H)$ is discussed in
\cite{KR} and also at the end of chapter 2 of \cite{SS2}. We will extend the
discussion of \cite{SS2} to the adapted case, and will prove the appropriate
generalization of Lemma \ref{propertyofcoends}. First we generalise Lemma 2.40
of \cite{SS2}.

\begin{lemma}
\label{coendimpliessomeend}Let $(G,\mathcal{S})$ be a group system of finite
type, and let $H$ be a finitely generated subgroup of infinite index in $G$.
Then $\bar{e}(G,H:\mathcal{S})\geq2$ if and only if there is a subgroup
$K$ of $H$ such that $e(G,K:\mathcal{S})\geq2$.
\end{lemma}

\begin{remark}
The subgroup $K$ may have infinite index in $H$.
\end{remark}

\begin{proof}
If there is a subgroup $K$ of $H$ such that $e(G,K:\mathcal{S})\geq2$, then
there is a nontrivial $K$--almost invariant subset $X$ of $G$ which is adapted
to $\mathcal{S}$. Thus $X$ and $X^{\ast}$ are each $K$--infinite, and $X$ and
$Xg$ are $K$--almost equal for all $g$ in $G$. Thus it is trivial that $X$ and
$Xg$ are $H$--almost equal for all $g$ in $G$. As $H$ has infinite index in
$G$, Lemma 2.13 of \cite{SS2} implies that $X$ and $X^{\ast}$ are each
$H$--infinite. It follows that $\bar{e}(G,H:\mathcal{S})\geq2$ as required.

Now we suppose that $\bar{e}(G,H:\mathcal{S})\geq2$. For this proof we
will consider the relative Cayley graph $\Gamma(G,\mathcal{S})$ discussed in
section \ref{groupsystems}. Our assumption implies there is a subset $X$ of
$G$ such that $X$ and $Xg$ are $H$--almost equal for all $g$ in $G$, neither
$X$ nor $X^{\ast}$ is $H$--finite, and $X$ is adapted to $\mathcal{S}$. We are
not assuming that $hX=X$, for all $h\in H$, so that $X$ need not be
$H$--almost invariant. Let $S$ denote a group in $\mathcal{S}$, so that $X$ is
$H$--almost equal to a subset of $G\ $which is strictly adapted to $S$. Now
the proof of Lemma \ref{finitelymanydoublecosets} shows that there are only
$H$--finitely many cosets $gS$ such that $gS\cap X$ and $gS\cap X^{\ast}$ are
both non-empty. Hence we can apply the first paragraph of the proof of Lemma
\ref{characterizingdX} to show there is an enlargement $\widehat{X}$ of $X$
such that $\delta\widehat{X}$ is $H$--finite. Note that Lemmas
\ref{finitelymanydoublecosets} and \ref{characterizingdX} have the hypothesis
that $X$ is $H$--almost invariant, so that $HX=X$, but this is not used in the
parts of their proofs which we are applying.

Thus there is a finite collection $F$ of edges of $\Gamma(G,\mathcal{S})$ such
that $\delta\widehat{X}\subset HF$. Let $R$ be a finite generating set for
$H$. Then $RF$ is also a finite union of edges of $\Gamma(G,\mathcal{S})$ and
so is contained in some finite connected subcomplex $Z$ of $\Gamma
(G,\mathcal{S})$. Let $L$ denote $HZ$. Using the fact that $R$ generates $H$,
it is easy to check that $L$ is connected. As $Z$ contains $F$, it is trivial
that $L=HZ\supset HF\supset\delta\widehat{X}$. Thus $L$ is a connected,
$H$--finite and $H$--invariant subcomplex of $\Gamma(G,\mathcal{S})$ which
contains $\delta\widehat{X}$.

Now let $Y$ denote the closure of the complement $\Gamma(G,\mathcal{S})-L$, so
that $Y$ is a subcomplex of $\Gamma(G,\mathcal{S})$ and $HY=Y$. As $L$ is
$H$--finite, the components of $Y$ also form a $H$--finite family. As $G$ is
$H$--infinite, there must be a component $Y^{\prime}$ of $Y$ such that the set
$W$ of vertices of $Y^{\prime}$ which lie in $G$ is also $H$--infinite. Let
$K$ denote the subgroup of $H$ which stabilises $Y^{\prime}$, and hence
stabilizes $W$. Let $\widehat{W}$ denote the set of all vertices of
$Y^{\prime}$, so that $\widehat{W}$ is an enlargement of $W$, and
$\delta\widehat{W}$ is contained in $L$. Thus $\delta\widehat{W}$ is
$H$--finite and $K$--invariant, and so must be $K$--finite. Also $K$ must
equal the subgroup of $H$ which stabilises $\delta\widehat{W}$ or have index
$2$ in this subgroup.

Now each component of $Y$ lies in $\widehat{X}$ or its complement. It follows
that $W$ is contained in $X$ or in $X^{\ast}$, so that $W^{\ast}$ must be
$H$--infinite. It follows that $W$ is a nontrivial $K$--almost invariant
subset of $G$, so that $e(G,K:\mathcal{S})\geq2$, as required.
\end{proof}

In the special situation we consider in this paper, we can prove more.

\begin{lemma}
\label{coendimpliesend}Let $(G,\mathcal{S})$ be a group system of finite type,
and let $H$ be a finitely generated subgroup of $G$. Suppose that for any
subgroup $K$ of $H$ which has infinite index in $H$, the number of ends
$e(G,K:\mathcal{S})$ equals $1$. Let $X$ be a subset of $G$ such that $X$ and
$X^{\ast}$ are each $H$--infinite, $X$ and $Xg$ are $H$--almost equal for all
$g$ in $G$, and $X$ is adapted to the family $\mathcal{S}$. (Thus
$\widetilde{e}(G,H:\mathcal{S})\geq2$.)

Then there is a subgroup $H^{\prime}$ of $H$ which has finite index in $H$,
and $X$ is $H$--almost equal to a nontrivial $H^{\prime}$--almost invariant
subset $X^{\prime}$ of $G$.
\end{lemma}

\begin{proof}
As in the proof of Lemma \ref{coendimpliessomeend}, we will consider the
relative Cayley graph $\Gamma(G,\mathcal{S})$, and we will use the notation in
that proof. Thus $\widehat{X}$ is an enlargement of $X$ such that
$\delta\widehat{X}$ is $H$--finite, $Y$ denotes the closure of the complement
$\Gamma(G,\mathcal{S})-L$, and the components of $Y$ form a $H$--finite
family. Further we showed that if $Y^{\prime}$ is a component of $Y$ such that
the set $W$ of vertices of $Y^{\prime}$ which lie in $G$ is $H$--infinite,
then there is $K\subset H$ such that $W$ is a nontrivial $K$--almost invariant
subset of $G$, so that $e(G,K:\mathcal{S})\geq2$. Now the special assumptions
of this lemma imply that $K$ must be of finite index in $H$. It follows that
there can only be finitely many such components $Y^{\prime}$ of $Y$, and all
other components of $Y$ intersect $G$ in a $H$--finite set. We let $X^{\prime
}$ denote the union of all the sets $W$ obtained in this manner which are
contained in $X$. Thus $X^{\prime}$ is a finite union of such sets, so there
is a subgroup $H^{\prime}$ of finite index in $H$ such that $X^{\prime}$ is a
nontrivial $H^{\prime}$--almost invariant subset of $G$. Further $X^{\prime}$
is $H$--almost equal to $X$, which completes the proof of the lemma.
\end{proof}

Now the following generalisation of Lemma \ref{propertyofcoends} is immediate.

\begin{corollary}
\label{propertyofadaptedcoends}Let $(G,\mathcal{S})$ be a group system of
finite type, and let $H$ be a finitely generated subgroup of $G$. Suppose that
for any subgroup $K$ of $H$ which has infinite index in $H$, the number of
ends $e(G,K:\mathcal{S})$ equals $1$. Then the number of coends of the pair
$(G,H)$ adapted to $\mathcal{S}$, denoted $\widetilde{e}(G,H:\mathcal{S})$,
equals $\sup\{e(G,K:\mathcal{S}):K\in\lbrack H]\}$. In particular, if
$\widetilde{e}(G,H:\mathcal{S})$ is finite, there is a subgroup $K$ of $H$
which has finite index in $H$ such that

$\widetilde{e}(G,H:\mathcal{S})=e(G,K:\mathcal{S})$.
\end{corollary}

\section{Crossing of sets over $VPC$
groups\label{crossingofa.i.setsoverVPCgroups}}

In this section, we will generalise in two different ways some of the results
of chapter 7 of \cite{SS2}, and then give some important applications. In
chapter 7 of \cite{SS2}, the authors considered a finitely generated,
one-ended group $G$, and almost invariant subsets of $G$ over $VPC1$
subgroups. In this section, the first generalization is to consider group
systems of finite type. The second generalization is to consider almost
invariant sets over $VPCn$ groups, for any $n\geq1$.

As we will frequently need to refer to groups which are $VPCk$, for some
$k<n$, it will be convenient to introduce the notation that such a group is
$VPC(<n)$.

In this paper, we will generalise these results to the setting of a group
system $(G,\mathcal{S})$ of finite type. An important tool is our use of a
relative Cayley graph for $(G,\mathcal{S})$ rather than a Cayley graph for
$G$. As relative Cayley graphs are almost always not locally finite, even when
$G$ is finitely generated, this requires some care. The subgroups $H$ and $K$
will (most of the time) be assumed to be $VPCn$, and the almost invariant
subsets $X$ and $Y\ $will be assumed to be $\mathcal{S}$--adapted.

Recall from section \ref{groupsystems} that if $(G,\mathcal{S})$ is a group
system of finite type, then $G$ is generated by a finite subset relative to
$\mathcal{S}$. Let $A$ be such a finite set, and let $\Gamma(G:A)$ denote the
graph with vertex set $G$ and edges $(g,ga)$ which join $g$ to $ga$, for each
$a\in A$. Note that $\Gamma(G:A)$ is not connected unless $A$ generates $G$. A
relative Cayley graph $\Gamma(G,\mathcal{S})$ can be constructed from
$\Gamma(G:A)$ by attaching a cone on each coset $gS_{j}$ of each $S_{j}$ in
$\mathcal{S}$. Thus $\Gamma(G,\mathcal{S})$ is locally finite at all non-cone
points. Note that $\Gamma(G,\mathcal{S})$ is connected.

As usual we identify $G$ with the set of vertices of $\Gamma(G:A)$, and we
will think of a subset $X$ of $G$ as a subset of the vertex set of
$\Gamma(G:A)$, and also as a subset of the vertex set of $\Gamma
(G,\mathcal{S})$. Recall that, for a subset $X$ of the vertex set of a graph
$\Gamma$, the \textit{coboundary} $\delta X$ of $X$ in $\Gamma$ consists of
the collection of edges of $\Gamma$ with exactly one vertex in $X$. It will
sometimes be convenient not to use the coboundary of $X$ in a graph, but
instead to use the following closely related idea.

\begin{definition}
\label{defnofboundary} For a subset $X$ of the vertex set of a graph $\Gamma$,
the \emph{boundary} of $X$ in $\Gamma$ consists of those points of $X$ which
are endpoints of an edge of $\delta X$. We denote the boundary of $X$ by
$\partial X$.
\end{definition}

\begin{remark}
It is convenient that $\partial X$ is a subset of $X$. Thus, for example, if
$X$ and $Y$ are any vertex sets of $\Gamma$, we have $\partial(X\cap
Y)=(\partial X\cap Y)\cup(X\cap\partial Y)$. But note that if $X^{\ast}$
denotes the vertices of $\Gamma$ not in $X$, then $\partial X$ is not equal to
$\partial X^{\ast}$ (unless both are empty), whereas $\delta X$ and $\delta
X^{\ast}$ are always equal.
\end{remark}

We will also use the following additional definition.

\begin{definition}
\label{defnofcoboundaryinE}If $X\subset E$, and $E$ is a subset of the vertex
set of a graph $\Gamma$, the \emph{coboundary of }$X$\emph{ in }$E$ consists
of those edges of $\delta X$ whose endpoints both lie in $E$. This is denoted
by $\delta_{E}X$. The \emph{boundary} of $X$ in $E$ consists of those points
of $X$ which are endpoints of an edge of $\delta_{E}X$. We denote the boundary
of $X$ in $E$ by $\partial_{E}X$.
\end{definition}

We will need some combinatorial results. The first, rather trivial, such
result will be used many times in this paper.

\begin{lemma}
\label{L-finiteintersection}Let $G$ be a group with subgroups $H$ and $K$, and
let $L$ denote $H\cap K$. Let $Y$ be a $K$--almost invariant subset of $G$.
Then the following hold:

\begin{enumerate}
\item $Hg\cap Y$ is $L$--almost equal to $(H\cap Y)g$, for every $g\in G$.

\item $Hg\cap Y$ is $L$--finite if and only if $H\cap Y$ is $L$--finite, for
every $g\in G$.
\end{enumerate}
\end{lemma}

\begin{proof}
Let $g$ be an element of $G$. As $Y$ is $K$--almost invariant, $Yg^{-1}$ is
$K$--almost equal to $Y$. Hence $H\cap Yg^{-1}$ is $L$--almost equal to $H\cap
Y$.

Now $Hg\cap Y=(H\cap Yg^{-1})g$ is $L$--almost equal to $(H\cap Y)g$. The
lemma follows.
\end{proof}

Our next result is standard in the non-adapted case.

\begin{lemma}
\label{intersectionisL-a.i.}Let $(G,\mathcal{S})$ be a group system of finite
type, let $H$ and $K$ be subgroups of $G$, and let $X$ and $Y\ $be nontrivial
$\mathcal{S}$--adapted almost invariant subsets of $G$ over $H$ and $K$
respectively. Let $L$ denote the intersection $H\cap K$. Suppose that $H\cap
Y$ and $X\cap K$ are $L$--finite. Then $X\cap Y$ is a $\mathcal{S}$--adapted
almost invariant subset of $G$ over $L$.
\end{lemma}

\begin{remark}
The assumptions that $H\cap Y$ and $X\cap K$ are $L$--finite imply that $X$
does not cross $Y$ strongly, and that $Y$ does not cross $X$ strongly. But
note that this result does not assume that $X$ and $Y$ cross at all.
\end{remark}

\begin{proof}
We will work in a relative Cayley graph $\Gamma(G,\mathcal{S})$ of the given
group system. As $X$ is a $\mathcal{S}$--adapted $H$--almost invariant subset
of $G$, Lemma \ref{characterizingdX} tells us that there is an enlargement
$\widehat{X}$ of $X$ whose coboundary $\delta\widehat{X}$ is $H$--finite.
Recall that this means that $\widehat{X}$ consists of $X$ together with some
cone points of $\Gamma(G,\mathcal{S})$. Similarly, as $Y$ is a $\mathcal{S}%
$--adapted $K$--almost invariant subset of $G$, there is an enlargement
$\widehat{Y}$ of $Y$ whose coboundary $\delta\widehat{Y}$ is $K$--finite.

As $H\cap\widehat{Y}=H\cap Y$ is $L$--finite, and $\delta\widehat{X}$ is
$H$--finite, Lemma \ref{L-finiteintersection} implies that the collection of
edges of $\delta\widehat{X}$ with one end in $\widehat{X}\cap\widehat{Y}$ is
also $L$--finite. Similarly the collection of edges of $\delta\widehat{Y}$
with one end in $\widehat{X}\cap\widehat{Y}$ must be $L$--finite. Now the
coboundary $\delta(\widehat{X}\cap\widehat{Y})$ consists of the union of these
two collections of edges, and so $\delta(\widehat{X}\cap\widehat{Y})$ must be
$L$--finite. As $\widehat{X}\cap\widehat{Y}$ is an enlargement of $X\cap Y$,
and $L(X\cap Y)=(X\cap Y)$, Lemma \ref{characterizingdX} tells us that $X\cap
Y$ is $L$--almost invariant and $\mathcal{S}$--adapted. This completes the
proof of the lemma.
\end{proof}

Next we state some basic facts about $VPC$ groups which will be needed. Our
statement incorporates parts of Theorems 9, 10 and 13 of \cite{Houghton}.

\begin{lemma}
\label{Houghtonlemma}(Houghton) Let $G$ be a $VPCn$ group and let $H$ be a
$VPCk$ subgroup, for some $k<n$. Then the following statements hold:

\begin{enumerate}
\item $e(G,H)$ equals $1$ or $2$.

\item If $e(G,H)=2$, then $k=n-1$, and $G$ has a finite index subgroup $J$
which contains $H$ as a normal subgroup, such that $J/H$ is infinite cyclic.

\item If $e(G,H)=1$, and $k=n-1$, then $G$ has a finite index subgroup $J$,
and $H$ has an index $2$ subgroup $M$, such that $J$ contains $M$ as a normal
subgroup, and $J/M$ is infinite cyclic.
\end{enumerate}
\end{lemma}

\begin{remark}
The reader should note that case 3) of the above lemma can occur. A simple
example occurs when $G$ is the fundamental group of the Klein bottle $K$, and
$H$ is the fundamental group of a Moebius band embedded in $K$.
\end{remark}

Now we can prove the following result, which is closely connected with
Proposition 7.3 of \cite{SS2}.

\begin{lemma}
\label{weak-weakcrossing}Let $(G,\mathcal{S})$ be a group system of finite
type, such that $e(G:\mathcal{S})=1$. Let $H\ $and $K$ be $VPC(n+1)$ subgroups
of $G$, and let $L$ denote $H\cap K$. Let $X\ $and $Y$ be nontrivial
$\mathcal{S}$--adapted almost invariant subsets of $G$ over $H$ and $K$
respectively. Suppose that, for any subgroup $M$ of $H$ or of $K$ with
infinite index, $e(G,M:\mathcal{S})=1$. Suppose further that $X$ crosses $Y$
weakly and $Y$ crosses $X$ weakly. Then $H$ and $K$ must be commensurable, and
$e(G,[H]:\mathcal{S})\geq4$.
\end{lemma}

\begin{proof}
As $X$ crosses $Y$ weakly, Definition \ref{defnofstrongandweakcrossing} tells
us that one of $H\cap Y$ or $H\cap Y^{\ast}$ is $K$--finite. By replacing $Y$
by its complement if needed, we can assume that $H\cap Y$ is $K$--finite. As
$H$ is certainly $H$--finite, we see that $H\cap Y$ must be $L$--finite.
Similarly, as $Y$ crosses $X$ weakly, we can assume that $X\cap K$ is
$L$--finite. Now Lemma \ref{intersectionisL-a.i.} shows that $X\cap Y$ is a
$\mathcal{S}$--adapted almost invariant subset of $G$ over $L$. As $X\ $and
$Y$ are assumed to cross, each corner of the pair is $H$--infinite, and hence
certainly $L$--infinite. It follows that $X\cap Y$ and its complement are both
$L$--infinite, so that $X\cap Y$ is a nontrivial $\mathcal{S}$--adapted almost
invariant subset of $G$ over $L$. Thus $e(G,L:\mathcal{S})\geq2$. Now the
hypotheses of the lemma imply that $L$ must have finite index in $H$ and in
$K$. Hence $H$ and $K$ are commensurable, as required.

As $L$ has finite index in $H$ and in $K$, each of $X$ and $Y$ is $L$--almost
invariant. Further each of $H\cap Y$ and $H\cap Y^{\ast}$ is $K$--finite, and
each of $X\cap K$ and $X^{\ast}\cap K$ is $H$--finite. Thus Lemma
\ref{intersectionisL-a.i.} implies that each of the four corners of the pair
$(X,Y)$ is $\mathcal{S}$--adapted and almost invariant over $L$. As $X$ and
$Y$ cross, each of these corners must be $L$--infinite. Thus $G\ $has four
disjoint nontrivial $\mathcal{S}$--adapted almost invariant subsets over $L$.
It follows that $e(G,L:\mathcal{S})\geq4$, so that we must have
$e(G,[H]:\mathcal{S})\geq4$, as required.
\end{proof}

\subsection{Adapted ends of subsets of $G$}

In this and the next section, it will be often be important to replace an
almost invariant subset $X$ of a group $G$, regarded as a subset of a relative
Cayley graph $\Gamma(G,\mathcal{S})$, by a connected subgraph which also has
connected "boundary". If $Z$ is a vertex set of a graph $\Gamma$, we will say
that $Z$ is \textit{connected in }$\Gamma$, if the maximal subgraph of
$\Gamma$ with vertex set equal to $Z$ is connected. We will denote this
maximal subgraph of $\Gamma$ by $\Gamma\lbrack Z]$. If the context leaves no
doubt we may simply say that $Z$ is connected. For a set $X$ of vertices of
$\Gamma$, we let $N_{R}X$ denote the collection of vertices of $\Gamma$ with
distance at most $R$ from $X$, i.e. $N_{R}X$ consists of those vertices of
$\Gamma$ which can be joined to a point of $X$ by an edge path with at most
$R$ edges. If $G\ $is finitely generated, so that a Cayley graph for $G$ with
respect to a finite generating set is locally finite, we have the following
result which guides our thinking.

\begin{lemma}
\label{R-nbhdisconnected}Let $G$ be a finitely generated group with Cayley
graph $\Gamma$, let $H$ be a finitely generated subgroup of $G$, and let $X$
be a $H$--almost invariant subset of $G$. Then there is $R$ such that
$N_{R}\partial X$ and $N_{R}X$ are connected.
\end{lemma}

\begin{remark}
It is obvious that if $N_{R}\partial X$ and $N_{R}X$ are connected, then
$N_{U}\partial X$ and $N_{U}X$ are also connected, for all $U\geq R$. It
follows that the set of $R$ such that $N_{R}\partial X$ and $N_{R}X$ are
connected consists of all integers $\geq R_{0}$, for some $R_{0}$.
\end{remark}

\begin{proof}
As $X$ is $H$--almost invariant, $\partial X$ is $H$--finite, so $\partial X$
is equal to $HF$ for some finite set $F$. Let $T$ be a finite generating set
for $H$, and let $R$ denote the diameter in $\Gamma$ of $F\cup TF$. Then
$N_{R}\partial X$ is connected. Now $N_{R}X=X\cup N_{R}\partial X$. Since
$\Gamma$ and $N_{R}\partial X$ are connected, it follows that so is $N_{R}X$,
as required.
\end{proof}

Note that, as $\partial X\ $is $H$--finite and $\Gamma$ is locally finite,
$N_{R}\partial X$ is also $H$--finite. Thus $N_{R}X$ is $H$--almost equal to
$X$ and so is equivalent to $X$. We will think of $N_{R}\partial X$ as being a
kind of boundary of $N_{R}X$. Thus we can replace $X$ by an equivalent
$H$--almost invariant set $N_{R}X$ which is connected and has connected
``boundary". Further this ``boundary" is $H$--finite, and clearly contains the
true boundary $\partial(N_{R}X)$ of $N_{R}X$.

If $(G,\mathcal{S})$ is a group system of finite type, and $X$ is
$\mathcal{S}$--adapted, we would like to make a similar construction in a
relative Cayley graph $\Gamma(G,\mathcal{S})$, but this is difficult as
$\Gamma(G,\mathcal{S})$ is not usually locally finite, even if $G\ $is
finitely generated. At first sight, it might seem natural to consider the
$R$--neighbourhood of $X$, or of an enlargement in $\Gamma(G,\mathcal{S})$,
but this is most unlikely to have $H$--finite coboundary. The proof of Lemma
\ref{R-nbhdisconnected} depended on the fact that a Cayley graph of a finitely
generated group is locally finite. As we are now considering graphs which need
not be locally finite, we will take a somewhat different point of view.

We start with an idea from the theory of relatively hyperbolic groups. Recall
that if $G$ is finitely generated, a choice of finite generating set
determines a Cayley graph $\Gamma$ for $G$ which in turn determines the word
metric on $G$, in which the distance from $e$ to $g$ equals the least number
of edges in any path in $\Gamma$ from $e$ to $g$. If $G$ is not finitely
generated, and $\Gamma(G,\mathcal{S})$ is a relative Cayley graph for the
group system $(G,\mathcal{S})$, then $G$ is the vertex set of the disconnected
subgraph $\Gamma(G:A)$ of $\Gamma(G,\mathcal{S})$. We will define a metric on
$G$ using paths in $\Gamma(G,\mathcal{S})$, which we will call an angle metric
on $G$.

\begin{definition}
\label{defnofanglemetric} Let $(G,\mathcal{S})$ be a group system of finite
type, let $A$ be a finite subset of $G$ which generates $G$ relative to
$\mathcal{S}$, and let $\Gamma(G,\mathcal{S})$ be the relative Cayley graph
for this system determined by $A$.

\begin{enumerate}
\item For each $S_{i}$ in $\mathcal{S}$, pick a countable generating set, and
assign length $n$ to the $n$--th generator of $S_{i}$. This determines a word
metric $d_{i}$ on $S_{i}$.

\item Let $e$ and $e^{\prime}$ be a pair of adjacent cone edges of
$\Gamma(G,\mathcal{S})$ meeting at a cone point $gS_{i}$, so the other
endpoints of $e$ and $e^{\prime}$ are elements $u$ and $v$ of $gS_{i}$. Then
the element $u^{-1}v$ of $G$ lies in $S_{i}$. We define \emph{the angle
between }$e$\emph{ and }$e^{\prime}$ to be the length of $u^{-1}v$ in the word
metric $d_{i}$ on $S_{i}$.

\item Let $\gamma$ be an edge path in $\Gamma(G,\mathcal{S})$ joining points
of $G$. The \emph{angle of }$\gamma$ is the sum of all the angles between
adjacent edges of $\gamma$ which meet in a cone point.

\item The \emph{angle length of }$\gamma$ is the sum of the angle of $\gamma$
with the number of edges of $\gamma$ which lie in $\Gamma(G:A)$.

\item Let $g$ and $h$ be elements of $G$. Their \emph{angle distance}, denoted
by $d_{\sphericalangle}(g,h)$, is the least angle length of any edge path in
$\Gamma(G,\mathcal{S})$ which joins $g$ and $h$.
\end{enumerate}
\end{definition}

\begin{remark}
It is clear that $d_{\sphericalangle}$ is a metric on $G$. Further
$d_{\sphericalangle}$ is invariant under the left action of $G$ on
$\Gamma(G,\mathcal{S})$.
This concept of cone angle has been used in the context of relative hyperbolic groups, see, for example, `Symbolic dynamics and relatively hyperbolic groups', Groups Geom. Dyn.~2, No.~2, 165--184 (2008).
\end{remark}

For a subset $V$ of $G$, we use the angle metric on $G$ to define the
$R$--neighbourhood of $V$ in $G$, which we again denote by $N_{R}V$. This is
the collection of elements of $G$ with angle distance at most $R$ from $V$,
i.e. $N_{R}V$ consists of those elements of $G$ which can be joined to a point
of $V$ by an edge path in $\Gamma(G,\mathcal{S})$ with angle length at most
$R$.

A crucial fact here is the obvious one that the number of elements of the
$S_{i}$'s with word length at most $R$ is finite. The following lemma uses
this to show that the angle metric $d_{\sphericalangle}$ on $G$ is proper,
i.e. metric balls are finite.

\begin{lemma}
\label{anglemetricisproper}Using the notation in Definition
\ref{defnofanglemetric}, for any element $g$ of $G$, and any integer $R$, the
$R$--neighbourhood $N_{R}(g)$ is finite.
\end{lemma}

\begin{proof}
Recall that $N_{R}(g)$ consists of those elements of $G$ which can be joined
to $g$ by an edge path in $\Gamma(G,\mathcal{S})$ with angle length at most
$R$. Let $h$ be a point of $N_{R}(g)$. Thus there is an edge path $\gamma$ in
$\Gamma(G,\mathcal{S})$ with angle length at most $R$ which joins $g$ to $h$.
We write $\gamma$ as the concatenation of paths $e_{1},\ldots,e_{k}$ where
each $e_{i}$ is either an edge of $\Gamma(G:A)$ or is a pair of consecutive
cone edges. This determines an expression of $hg^{-1}$ as a word $w=a_{1}%
a_{2}\ldots a_{k}$, where each $a_{i}$ is either an element of $A$ or its
inverse, or lies in some $S_{j}$. As $\gamma$ has angle length at most $R$,
the number $k$ of the $a_{i}$'s is at most $R$. Further if $a_{i}$ lies in
$S_{j}$, then $a_{i}$ lies in the finite set of elements of $S_{j}$ of length
at most $R$. The set of all such words is finite, which immediately implies
that $N_{R}(g)$ is finite, as required.
\end{proof}

We will use the following notation when $(G,\mathcal{S})$ is a group system of
finite type, and $A$ is a given finite subset of $G$ which generates $G$
relative to $\mathcal{S}$. Recall that $G$ is the vertex set of $\Gamma(G:A)$,
which is in turn contained in the relative Cayley graph $\Gamma(G,\mathcal{S}%
)$.

\begin{notation}
\label{notation}We will denote $\Gamma(G:A)$ by $\Gamma$, and will denote
$\Gamma(G,\mathcal{S})$ by $\widehat{\Gamma}$. Let $D^{R}$ denote the elements
of $G$ which lie at angle distance at most $R$ from the identity element
$1_{G}$ of $G$, so that $D^{R}=N_{R}(1_{G})$, and let $A^{R}$ denote the union
$A\cup D^{R}$. Note that Lemma \ref{anglemetricisproper} implies that $D^{R}$,
and hence $A^{R}$, is finite. Let $\Gamma^{R}$ denote the graph $\Gamma
(G:A^{R})$, and let $\widehat{\Gamma}^{R}$ denote the relative Cayley graph of
$(G,\mathcal{S})$ constructed from the graph $\Gamma^{R}$. Thus
$\widehat{\Gamma}^{R}$ can be constructed from $\widehat{\Gamma}$ by simply
adding non-cone edges which join $g$ to $gd$, for every $g$ in $G$,$\ $and $d$
in $D^{R}$. In particular, note that $\widehat{\Gamma}^{R}$ and
$\widehat{\Gamma}$ have the same cone edges, for all $R$.

If $Z$ is a vertex set in $\widehat{\Gamma}=\Gamma(G,\mathcal{S})$, we will
use $\delta Z$ to denote the coboundary of $Z$ in $\widehat{\Gamma}$, and will
use $\partial Z$ to denote the boundary of $Z$ in $\widehat{\Gamma}$. Also
$\delta^{R}Z$ denotes the coboundary of $Z$ in $\widehat{\Gamma}^{R}$, and
$\partial^{R}Z$ denotes the boundary of $Z$ in $\widehat{\Gamma}^{R}$.
\end{notation}

If $X$ is a $\mathcal{S}$--adapted $H$--almost invariant subset of $G$, Lemma
\ref{characterizingdX} tells us that there is an enlargement $Y$ of $X$ such
that $\delta Y$ is $H$--finite. (Recall that an enlargement $Y\ $of $X$ is a
vertex set in $\widehat{\Gamma}=\Gamma(G,\mathcal{S})$ such that $Y\cap
\Gamma=X$.) In general, such an enlargement is not unique, but it will be
convenient to make a specific choice. In addition, it would be technically
very convenient if $\partial Y$ were contained in $G$, so that $\partial Y$
contained no cone points. This need not be the case, but can be arranged by
enlarging $Y$ further. For future reference, we set out our choices in the
following definition.

\begin{definition}
\label{defnofstandardaugmentation}Let $X$ be a $\mathcal{S}$--adapted
$H$--almost invariant subset of $G$. Let $Y$ be the enlargement of $X$ with
the property that the cone point of $\widehat{\Gamma}=\Gamma(G,\mathcal{S})$
labeled $gS_{i}$ belongs to $Y$ if and only if the intersection $gS_{i}\cap X$
is non-empty, and $gS_{i}\cap X^{\ast}$ is $H$--finite. Then \emph{the
standard augmentation }$\widehat{X}$\emph{ of }$X$ is obtained from $Y$ by
adding each point of $G$ which lies in a cone edge of $\delta Y$.
\end{definition}

The following lemma lists the properties of $\widehat{X}$.

\begin{lemma}
\label{propertiesofXhat}Let $X$ be a $\mathcal{S}$--adapted $H$--almost
invariant subset of $G$, and let the enlargement $Y$ of $X$, and the standard
augmentation $\widehat{X}$ of $X$, be defined as in the preceding definition.
Then the following conditions all hold:

\begin{enumerate}
\item $Y$ is $H$--invariant and $\delta Y$ is $H$--finite, and $Y$ is the
maximal such enlargement of $X$ with the property that every cone point in $Y$
is joined to $X$ by a cone edge.

\item $\widehat{X}$ is a $H$--invariant vertex set of $\widehat{\Gamma}$, and
$\delta\widehat{X}$ is $H$--finite. Also $\partial\widehat{X}$ is contained in
$G$.

\item $\widehat{X}\cap G$ is a $\mathcal{S}$--adapted $H$--almost invariant
subset of $G$ and is equivalent to $X$.
\end{enumerate}
\end{lemma}

\begin{proof}
1) It is easy to see that $Y$ is $H$--invariant. As $X$ is $\mathcal{S}%
$--adapted, Lemma \ref{finitelymanydoublecosets} shows there are only
$H$--finitely many cosets $gS_{j}$ such that $gS_{j}\cap X$ and $gS_{j}\cap
X^{\ast}$ are both non-empty. Now it follows that $\delta Y$ is $H$--finite.
It is also easy to see that $Y$ is the maximal such enlargement of $X$ with
the property that every cone point in $Y$ is joined to $X$ by a cone edge.

2) As $Y$ is $H$--invariant, it follows that $\widehat{X}$ is also
$H$--invariant. As $\delta Y$ is $H$--invariant and $H$--finite, it follows
that $\widehat{X}$ is obtained from $Y$ by adding a subset of $G$ which is
also $H$--invariant and $H$--finite. As each element of $G$ is a vertex of
finite valence in $\widehat{\Gamma}$, we see that $\delta\widehat{X}$ is
$H$--finite. Further, our construction of $\widehat{X}$ implies that
$\partial\widehat{X}$ contains no cone points, so that $\partial\widehat{X}$
is contained in $G$.

3) As $\widehat{X}\cap G$ is obtained from $X$ by adding a $H$--invariant and
$H$--finite set, it follows that $\widehat{X}\cap G$ is itself a $H$--almost
invariant subset of $G$ and is equivalent to $X$. Thus $\widehat{X}\cap G$ is
also $\mathcal{S}$--adapted, as required.
\end{proof}

In what follows we will use the standard augmentation $\widehat{X}$ of $X$,
and the facts that $\partial\widehat{X}$ is $H$--finite and contained in $G$.
Instead of enlarging the vertex set $\widehat{X}$ to become connected in the
graph $\widehat{\Gamma}=\Gamma(G,\mathcal{S})$, as discussed at the start of
this section, we will fix $\widehat{X}$ and add edges to the graph to make
$\widehat{X}$ connected in the new graph.

\begin{remark}
\label{d^RZiscontainedinG}If $Z$ is a vertex set of the graph $\widehat{\Gamma
}=\Gamma(G,\mathcal{S})$ such that $\partial Z\subset G$, then $\partial
^{R}Z\subset G$, for every $R$. This is because $\widehat{\Gamma}^{R}$ is
obtained from $\widehat{\Gamma}$ by adding non-cone edges.
\end{remark}

The key technical result we need is the following.

\begin{lemma}
\label{H-finitesetinGbecomesconnectedinGL}Using the notation introduced above,
let $H$ be a finitely generated subgroup of $G$, and let $V$ be a
$H$--invariant and $H$--finite subset of $G$. Then there is $N$ such that $V$
is connected in $\Gamma^{N}$, i.e. the subgraph $\Gamma^{N}[V]$ of $\Gamma
^{N}$ is connected.
\end{lemma}

\begin{remark}
$\Gamma$ and $\Gamma^{N}$ themselves cannot be connected unless $G$ is
finitely generated.

If $V$ is connected in $\Gamma^{N}$, then $V$ is connected in $\Gamma^{R}$,
for any $R\geq N$. This is simply because $\Gamma^{R}$ contains $\Gamma^{N}$,
so that $\Gamma^{R}[V]$ contains $\Gamma^{N}[V]$.
\end{remark}

\begin{proof}
As $V$ is $H$--invariant and $H$--finite, we can write $V=HF$, for some finite
subset $F$ of $G$. Now let $T$ be a finite generating set for $H$, and
consider the finite subset $F\cup TF$ of $G$. For each pair of points of
$F\cup TF$, we choose a path in $\widehat{\Gamma}=\Gamma(G,\mathcal{S})$ which
joins them. Let $N$ denote the maximum angle length of all these paths. Then
$F\cup TF$ is connected in $\Gamma^{N}$. As $V=HF$, it follows that $V$ is
also connected in $\Gamma^{N}$, as required.
\end{proof}

Now we can prove the following analogue of Lemma \ref{R-nbhdisconnected}.

\begin{lemma}
\label{AlmostinvariantsetbecomesconnectedinGL}Using the notation introduced
above, let $H$ be a finitely generated subgroup of $G$, and let $X$ be a
$H$--almost invariant $\mathcal{S}$--adapted subset of $G$, with standard
augmentation $\widehat{X}$ in $\widehat{\Gamma}$. Then there is $N$ such that
$\partial^{N}\widehat{X}$ is connected in $\Gamma^{N}$, and $\widehat{X}$ is
connected in $\widehat{\Gamma}^{N}$, i.e. the graphs $\Gamma^{N}[\partial
^{N}\widehat{X}]$ and $\widehat{\Gamma}^{N}[\widehat{X}]$ are each connected.
Further, $\partial^{R}\widehat{X}$ is connected in $\Gamma^{R}$, and
$\widehat{X}$ is connected in $\widehat{\Gamma}^{R}$, for any $R\geq N$.
\end{lemma}

\begin{remark}
It follows that the set of $R$ such that $\partial^{R}\widehat{X}$ is
connected in $\Gamma^{R}$, and $\widehat{X}$ is connected in $\widehat{\Gamma
}^{R}$, consists of all integers $\geq R_{0}$, for some $R_{0}$.
\end{remark}

\begin{proof}
Recall that $\delta\widehat{X}$ is $H$--invariant and $H$--finite. Further the
same applies to $\delta^{R}\widehat{X}$, for all values of $R$, as $\delta
^{R}\widehat{X}$ is the coboundary of $\widehat{X}$ in the relative Cayley
graph $\widehat{\Gamma}^{R}$. Also recall that our construction of the
standard augmentation $\widehat{X}$ ensures that $\partial\widehat{X}\subset
G$. Now Remark \ref{d^RZiscontainedinG} tells us that $\partial^{R}%
\widehat{X}\subset G$, for all $R$.

As $\partial\widehat{X}$ is a $H$--invariant and $H$--finite subset of $G$,
Lemma \ref{H-finitesetinGbecomesconnectedinGL} implies there is $N$ such that
$\partial\widehat{X}$ is connected in $\Gamma^{N}$, i.e. $\Gamma^{N}%
[\partial\widehat{X}]$ is connected.

As $\widehat{\Gamma}$ is connected, any point of $\widehat{X}$ can be joined
to a point of $\partial\widehat{X}$, by a path in $\widehat{\Gamma}$ all of
whose vertices lie in $\widehat{X}$. As $\partial\widehat{X}$ is connected in
$\Gamma^{N}$, and hence in $\widehat{\Gamma}^{N}$, it follows that
$\widehat{X}$ itself is connected in $\widehat{\Gamma}^{N}$, as required.

Next we claim that $\partial^{N}\widehat{X}$ is connected in $\Gamma^{N}$,
i.e. that $\Gamma^{N}[\partial^{N}\widehat{X}]$ is connected. Note that
$\partial^{N}\widehat{X}$ contains $\partial\widehat{X}$, and consider a point
$x$ of $\partial^{N}\widehat{X}-\partial\widehat{X}$. Then $x$ lies in
$\widehat{X}\cap G$ and is joined to some point of $G-\widehat{X}$ by a path
$\lambda$ in $\widehat{\Gamma}$ of angle length at most $N$. Such a path must
meet $\partial\widehat{X}$. Hence $\lambda$ has a subpath $\mu$ joining $x$ to
a point $y$ of $\partial\widehat{X}$ such that all vertices of $\mu$ lie in
$\widehat{X}$. As $x$ and $y$ lie in $G$, it follows that $\mu$ also has angle
length at most $N$, so that $x$ is joined to $y$ by an edge of $\Gamma^{N}$.
This implies that $\partial^{N}\widehat{X}$ is connected in $\Gamma^{N}$, as claimed.

To complete the proof of the lemma, we need to show that $\partial
^{R}\widehat{X}$ is connected in $\Gamma^{R}$, and $\widehat{X}$ is connected
in $\widehat{\Gamma}^{R}$, for any $R\geq N$. As $\widehat{\Gamma}^{N}$ is
contained in $\widehat{\Gamma}^{R}$, it is trivial that $\widehat{X}$ is
connected in $\widehat{\Gamma}^{R}$. Now the argument in the preceding
paragraph shows that $\partial^{R}\widehat{X}$ is connected in $\Gamma^{R}$.
This completes the proof of the lemma.
\end{proof}

Armed with these ideas, we can now proceed to consider the number of a certain
type of ends of certain vertex sets in $\Gamma(G,\mathcal{S})$.

Let $G,\mathcal{S})$ be a group system of finite type, let $\Gamma
(G,\mathcal{S})$ be a relative Cayley graph, and let $\widehat{G}$ denote the
vertex set of $\Gamma(G,\mathcal{S})$. Let $K$ be a subgroup of $G$, and let
$Z$ be a $K$--invariant subset of $\widehat{G}$. If $A$ is a subset of $Z$, we
let $\delta_{Z}A$ denote the coboundary of $A$ in $Z$, when $Z$ is regarded as
a vertex set of $\widehat{\Gamma}=\Gamma(G,\mathcal{S})$. Recall from
Definition \ref{defnofcoboundaryinE} that the coboundary of $A$ in $Z$
consists of those edges of $\delta A$ whose endpoints both lie in $Z$. We also
let $\delta_{Z}^{R}A$ denote the coboundary of $A$ in $Z$, when $Z$ is
regarded as a vertex set of $\widehat{\Gamma}^{R}$. Let $\mathcal{B}_{Z}$
denote the set%
\[
\{A\subset Z:A\text{ is }K^{\prime}\text{--invariant, for some }K^{\prime}%
\in\lbrack K]\text{, and }\delta_{Z}^{R}A\text{ is }K\text{--finite, for every
}R\}.
\]
It is easy to see that $\mathcal{B}_{Z}$ is a Boolean algebra. For if $A$ and
$B$ are elements of $\mathcal{B}_{Z}$, then so is each corner of the pair
$(A,B)$. It follows that $A\cap B$ and $A+B$ each lie in $\mathcal{B}_{Z}$.
Now we make the following definition.

\begin{definition}
\label{defnofnumberofS-adaptedK-ends}Let $(G,\mathcal{S})$ be a group system
of finite type, let $\Gamma(G,\mathcal{S})$ be a relative Cayley graph, let
$K$ be a subgroup of $G$, and let $Z$ be a $K$--invariant subset of
$\widehat{G}$. The \emph{number of }$S$\emph{--adapted }$K$\emph{--ends of}
$Z$ is the dimension of the quotient of $\mathcal{B}_{Z}$ by the subalgebra of
$K$--finite sets.
\end{definition}

\begin{remark}
The above definition is clearly unchanged if we replace $Z$ by a $K$--almost
equal set. The terminology $\mathcal{S}$--adapted $K$--end is used because, in
some sense, this definition counts the number of $\mathcal{S}$--adapted ends
of the quotient $K\backslash Z$.

We will use this definition in the situation where $Z=\widehat{X}$ is the
standard augmentation of a $\mathcal{S}$--adapted $H$--almost invariant subset
$X$ of $G$, where $H$ contains $K$, as in Lemma
\ref{aisethasoneS-adaptedK-end}. This lemma plays a key role in section
\ref{CCC'swithstrongcrossing}.
\end{remark}

We will only be interested in situations where we can prove that the number of
$\mathcal{S}$--adapted $K$--ends is equal to $1$. A simple example of such an
argument is given in Lemma \ref{numberofS-adaptedK-endsofGis1} below. Recall
from Definition \ref{defnofendsof[H]} that

$e(G,[K]:\mathcal{S})=\sup\{e(G,K^{\prime}:\mathcal{S}):K^{\prime}\in\lbrack
K]\}$. It is not difficult to improve the proof of Lemma
\ref{numberofS-adaptedK-endsofGis1} to show that, in general, the number of
$\mathcal{S}$--adapted $K$--ends of $\widehat{G}$ is equal to
$e(G,[K]:\mathcal{S})$, but we have no need for such a result.

\begin{lemma}
\label{numberofS-adaptedK-endsofGis1}Let $(G,\mathcal{S})$ be a group system
of finite type, and let $K$ be a subgroup of $G$ such that
$e(G,[K]:\mathcal{S})=1$. Let $\Gamma(G,\mathcal{S})$ be any relative Cayley
graph of $(G,\mathcal{S})$.Then the number of $\mathcal{S}$--adapted $K$--ends
of $\widehat{G}$ is equal to $1$.
\end{lemma}

\begin{proof}
Recall that $\widehat{G}$ is the entire vertex set of the connected graph
$\Gamma(G,\mathcal{S})$, so is trivially $K$--invariant and connected in
$\Gamma(G,\mathcal{S})$. As $G$, and hence $\widehat{G}$, must be
$K$--infinite, it follows that the number of $\mathcal{S}$--adapted $K$--ends
of $\widehat{G}$ is at least $1$. Let $A$ be an element of $\mathcal{B}%
_{\widehat{G}}$. In particular, this implies that $A$ is a $K^{\prime}%
$--invariant subset of $\widehat{G}$, for some $K^{\prime}\in\lbrack K]$, and
that $\delta_{\widehat{G}}A=\delta A$ is $K$--finite. By Lemma
\ref{characterizingdX}, this implies that $A\cap G$ is a $K^{\prime}$--almost
invariant subset of $G$ which is also $\mathcal{S}$-adapted. As
$e(G,[K]:\mathcal{S})=1$, it follows that $A\cap G$ must be $K$--almost equal
to $G$ or the empty set, so that $A$ must be $K$--almost equal to
$\widehat{G}$ or the empty set. It follows that $\widehat{G}$ has exactly one
$\mathcal{S}$--adapted $K$--end, as required.
\end{proof}

Next we need the following technical result.

\begin{lemma}
\label{HintersectconjugateofSisK-finite}Let $(G,\mathcal{S})$ be a group
system of finite type, let $H$ be a $VPC(n+1)$ subgroup of $G$, and let $X$ be
a nontrivial $\mathcal{S}$--adapted $H$--almost invariant subset of $G$.
Suppose there is a nontrivial $\mathcal{S}$--adapted $H^{\prime}$--almost
invariant subset $Y$ of $G$ such that $X$ crosses $Y$ strongly, and let $K$
denote $H\cap H^{\prime}$. Then $K$ is $VPCn$, and $e(H,K)=2$.

Further, for each $g\in G$ and $S\in\mathcal{S}$, the intersection $H\cap
gSg^{-1}$ is $K$--finite.
\end{lemma}

\begin{remark}
\label{HintersectgSg-1}Note that $gSg^{-1}$ is the stabilizer of the coset
$gS$ of $S$ in $G$. Thus $H\cap gSg^{-1}$ is the set of elements of $H$ which
fix the cone point $gS$ of a relative Cayley graph $\Gamma(G,\mathcal{S})$ of
$(G,\mathcal{S})$.
\end{remark}

\begin{proof}
As $X$ crosses $Y$ strongly, we know that $H\cap Y$ and $H\cap Y^{\ast}$ are
both $H^{\prime}$--infinite. Now $H\cap Y$ and $H\cap Y^{\ast}$ are each
$K$--almost invariant subsets of $H$, and as they are $H^{\prime}$--infinite,
they are certainly $K$--infinite. Hence $H\cap Y$ and $H\cap Y^{\ast}$ are
nontrivial $K$--almost invariant subsets of $H$, so that $e(H,K)\geq2$. As $H$
is $VPC(n+1)$, Lemma \ref{Houghtonlemma} tells us that $K$ must be $VPCn$, and
that $e(H,K)=2$, as required.

Lemma \ref{Houghtonlemma} also tells us that $H$ has a finite index subgroup
$J$ which contains $K$ as a normal subgroup, such that $J/K$ is infinite
cyclic. Thus by replacing $H$ by the finite index subgroup $J$, which we
continue to denote by $H$, we can assume that the intersection $K=H\cap
H^{\prime}$ is normal in $H$, and that the quotient $H/K$ is infinite cyclic.

Now suppose there is $g\in G$ and $S\in\mathcal{S}$, such that the
intersection $H\cap gSg^{-1}$ is $K$--infinite. This implies that $H\cap
gSg^{-1}$ has infinite cyclic image in $H/K$. In particular, as $X$ crosses
$Y$ strongly, it follows that $gSg^{-1}\cap Y$ and $gSg^{-1}\cap Y^{\ast}$ are
both $H^{\prime}$--infinite. As $Y$ is $H^{\prime}$--almost invariant, we know
that $Yg$ is $H^{\prime}$--almost equal to $Y$. It follows that $gS\cap Y$ and
$gS\cap Y^{\ast}$ are both $H^{\prime}$--infinite, contradicting the
assumption that $Y$ is $\mathcal{S}$--adapted, by Lemma
\ref{S-adaptedimpliesgSintersectXisH-finite}. This contradiction completes the
proof that the intersection $H\cap gSg^{-1}$ must be $K$--finite.
\end{proof}

Now we start on the key technical results of this paper. The next lemma is a
generalization to our setting of Lemma 2.11 of \cite{D-Swenson}.

\begin{lemma}
\label{aisethasoneS-adaptedK-end}Let $(G,\mathcal{S})$ be a group system of
finite type with $e(G:\mathcal{S})=1$. Let $H$ be a $VPC(n+1)$ subgroup of
$G$, and let $X$ be a nontrivial $\mathcal{S}$--adapted $H$--almost invariant
subset of $G$. Suppose there is a nontrivial $\mathcal{S}$--adapted
$H^{\prime}$--almost invariant subset $Y$ of $G$ such that $X$ crosses $Y$
strongly, and let $K$ denote $H\cap H^{\prime}$. Suppose that
$e(G,[K]:\mathcal{S})=1$, and let $\widehat{\Gamma}=\Gamma(G,\mathcal{S})$ be
a relative Cayley graph of $(G,\mathcal{S})$. Let $\widehat{X}$ denote the
standard augmentation of $X$. Then

\begin{enumerate}
\item If $\partial^{R}\widehat{X}$ is connected in $\Gamma^{R}$, so that the
subgraph $\Gamma^{R}[\partial^{R}\widehat{X}]$ of $\Gamma^{R}$ is connected,
the quotient graph $K\backslash\Gamma^{R}[\partial^{R}\widehat{X}]$ has two ends.

\item the number of $\mathcal{S}$--adapted $K$--ends of $\widehat{X}$ is equal
to $1$.
\end{enumerate}
\end{lemma}

\begin{proof}
Lemma \ref{AlmostinvariantsetbecomesconnectedinGL} tells us there is $N$ such
that $\partial^{N}\widehat{X}$ is connected in $\Gamma^{N}$, and $\widehat{X}$
is connected in $\widehat{\Gamma}^{N}$. Recall that $\Gamma^{N}$ is locally
finite, so the same applies to any subgraph.

1) Suppose that $\partial^{R}\widehat{X}$ is connected in $\Gamma^{R}$, so
that the subgraph $\Gamma^{R}[\partial^{R}\widehat{X}]$ of $\Gamma^{R}$ is
connected. Recall that $\partial^{R}\widehat{X}$ is $H$--invariant and
$H$--finite. Thus $\Gamma^{R}[\partial^{R}\widehat{X}]$ is a $H$--invariant
and $H$--finite connected subgraph of $\Gamma^{R}$. Further $\Gamma
^{R}[\partial^{R}\widehat{X}]$ is locally finite. As $H$ acts freely on this
connected graph, it follows that the number of ends of the quotient graph
$K\backslash\Gamma^{R}[\partial^{R}\widehat{X}]$ is equal to $e(H,K)$. Now
Lemma \ref{HintersectconjugateofSisK-finite} implies that $e(H,K)=2$,
completing the proof of part 1).

2) Now we use the fact that there is $N$ such that $\partial^{N}\widehat{X}$
is connected in $\Gamma^{N}$, and $\widehat{X}$ is connected in
$\widehat{\Gamma}^{N}$. For notational simplicity, we let $Z$ denote
$\widehat{X}$, let $dZ$ denote $\partial^{N}Z$, let $W$ denote
$\widehat{\Gamma}^{N}[Z]$, and let $dW$ denote $\Gamma^{N}[dZ]$. Thus $W$ is a
connected subgraph of $\widehat{\Gamma}^{N}$, and $dW$ is a connected subgraph
of $\Gamma^{N}$. Further, part 1) tells us that the quotient graph
$K\backslash dW$ has two ends. As $Z$ is $K$--invariant and $K$--infinite, it
is immediate that $Z$ has at least one $\mathcal{S}$--adapted $K$--end.

Recall that $\mathcal{B}_{Z}$ is a Boolean algebra. It follows that if $Z$ has
more than two $\mathcal{S}$--adapted $K$--ends, we can partition $Z$ into
three $K$--infinite elements $Y_{1}$, $Y_{2}$ and $Y_{3}$ of $\mathcal{B}_{Z}%
$. Thus $\delta_{Z}^{R}Y_{i}$ is $K$--finite, for each $i$ and every $R$. Let
$K^{\prime}\in\lbrack K]$ such that $Z$ and each $Y_{i}$ are $K^{\prime}%
$--invariant. For each $i$, it is trivial that $\delta_{dZ}^{N}(Y_{i}\cap dZ)$
is contained in $\delta_{Z}^{N}Y_{i}$, and so is $K$--finite. Now part 1)
implies that some $Y_{i}\cap dZ$ is $K$--finite. We can assume that $i=3$.

Now $\delta^{N}Y_{3}=\delta_{Z}^{N}Y_{3}\cup(\delta^{N}Y_{3}\cap\delta
^{N}Z)\subset\delta_{Z}^{N}Y_{3}\cup\delta^{N}(Y_{3}\cap dZ)$. As each point
of $G$ is a vertex of finite valence in $\widehat{\Gamma}^{N}$, and $Y_{3}\cap
dZ$ is a $K$--finite subset of $G$, it follows that $\delta^{N}(Y_{3}\cap dZ)$
is $K$--finite. As $\delta_{Z}^{N}Y_{3}$ is $K$--finite, it follows that
$\delta^{N}Y_{3}$ is also $K$--finite. Now Lemma \ref{characterizingdX} tells
us that $Y_{3}\cap G$ is a $\mathcal{S}$--adapted $K^{\prime}$--almost
invariant subset of $G$. As $e(G,[K]:S)=1$, it follows that $Y_{3}\cap G$ must
be $K$--finite (as its complement in $G$ clearly cannot be $K$--finite).Thus
$Y_{3}$ itself is $K$--finite, which contradicts our assumption that each
$Y_{i}$ is $K$--infinite. This contradiction shows that the number of
$\mathcal{S}$--adapted $K$--ends of $Z$ is at most $2$.

We suppose now that the number of $\mathcal{S}$--adapted $K$--ends of $Z$
equals $2$, and will obtain a contradiction. We can partition $Z$ into
$K$--infinite elements $Y$ and $Z-Y$ of $\mathcal{B}_{Z}$. Again let
$K^{\prime}\in\lbrack K]$ such that $Z$ and $Y$ are $K^{\prime}$--invariant.
By replacing $H$ by a subgroup of finite index we can assume that $K^{\prime}$
is normal in $H$ with infinite cyclic quotient. Note that the argument in the
preceding paragraph implies that $Y\cap dZ$ and $(Z-Y)\cap dZ$ must each be
$K$--infinite. Also the coboundaries $\delta_{dZ}^{N}(Y\cap dZ)$ and
$\delta_{dZ}^{N}((Z-Y)\cap dZ)$ are each $K$--finite. It follows that in the
quotient $K^{\prime}\backslash dW$, the images of $Y\cap dZ$ and $(Z-Y)\cap
dZ$ each contain the vertex set of one end of $K^{\prime}\backslash dW$.

Now let $\overline{Y}$ denote the subgraph $\widehat{\Gamma}^{N}[Y]$ of
$\widehat{\Gamma}^{N}$. As $Y\subset Z$, we know that $\overline{Y}$ is a
subgraph of $W=\widehat{\Gamma}^{N}[Z]$. As $K^{\prime}$ is a finitely
generated group which acts on the connected graph $W=\widehat{\Gamma}^{N}[Z]$,
and $\delta_{Z}^{N}Y$ is $K$--finite, there is a $K^{\prime}$--invariant
$K^{\prime}$--finite connected subgraph $\Delta$ of $W$ which contains
$\delta_{Z}^{N}Y$. As $W$ is connected, it follows that $\overline{Y}%
\cup\Delta$ is also connected. Now Lemma
\ref{HintersectconjugateofSisK-finite} and Remark \ref{HintersectgSg-1}
together imply that the set of $h$ in $H$ which stabilize a given cone point
of $Z$ is $K$--finite. As $\Delta$ is $K^{\prime}$--finite, and $H$ is
$K$--infinite, it follows that there is $h\in H$ such that $h\Delta$ is
disjoint from $\Delta$. As $K^{\prime}\backslash H$ is infinite cyclic, we can
arrange that $h$ does not interchange its two ends, by replacing $h$ by
$h^{2}$ if needed. Thus we can assume that $h$ does not interchange the two
ends of $K^{\prime}\backslash dW$. As $\Delta$ is connected, it follows that
we must have $h(\overline{Y}\cup\Delta)\subset\overline{Y}\cup\Delta$ or
$\overline{Y}\cup\Delta\subset h(\overline{Y}\cup\Delta)$. By replacing $h$ by
its inverse if needed, we can assume that $h(\overline{Y}\cup\Delta
)\subset\overline{Y}\cup\Delta$. As we are assuming that the number of
$\mathcal{S}$--adapted $K$--ends of $Z$ equals $2$, it follows that the vertex
sets of $h(\overline{Y}\cup\Delta)$ and $\overline{Y}\cup\Delta$ are
$K$--almost equal. We let $E$ denote $\overline{Y}\cup\Delta$, so that
$hE\subset E$.

Next we claim that the intersection $\cap_{k\geq0}h^{k}E$ is empty. Any vertex
$v$ of $E$ can be joined by a path in $E$ to a point of $\Delta$. The length
of the shortest such path will be called the distance of $v$ from $\Delta$. As
$hE\subset E$ and $h\Delta$ is disjoint from $\Delta$, it follows that $hE$ is
also disjoint from $\Delta$. Thus any vertex of $hE$ has distance at least $1$
from $\Delta$. Hence any vertex of $h^{n}E$ has distance at least $1$ from
$h^{n-1}\Delta$. It follows that any vertex of $h^{k}E$ has distance at least
$k$ from $\Delta$. Hence no vertex of $E$ can lie in $\cap_{k\geq0}h^{k}E$, so
that this intersection is empty as claimed.

Recall that $hE\subset E$, and that their vertex sets are $K$--almost equal.
We let $V$ denote the vertex set of $E-hE$, so that $V$ is $K$--finite. The
fact that $\cap_{k\geq0}h^{k}E$ is empty implies that $\cup_{k\geq0}h^{k}V$
equals the vertex set of $E$. As $V$ is $K$--finite, this implies that the
vertex set of $E$ is $H$--finite. As $E$ equals $\overline{Y}\cup\Delta$, this
implies that $Y$ is $H$--finite. A similar argument shows that $Z-Y$ is
$H$--finite. But this implies that $Z$ itself is $H$--finite, which
contradicts the fact that $Z\cap G=X$ is a nontrivial $H$--almost invariant
subset of $G$. This contradiction completes the proof that the number of
$S$--adapted $K$--ends of $Z$ must equal $1$.
\end{proof}

\subsection{Applications}

An easy consequence of Lemma \ref{aisethasoneS-adaptedK-end} is the following,
which is closely related to Proposition 7.2 of \cite{SS2}.

\begin{lemma}
\label{strongcrossingissymmetric}Let $(G,\mathcal{S})$ be a group system of
finite type with $e(G:\mathcal{S})=1$. Let $H$ and $H^{\prime}$ be subgroups
of $G$, such that $H$ is $VPC(n+1)$, and let $X\ $and $Y$ be nontrivial
$\mathcal{S}$--adapted almost invariant subsets of $G$ over $H$ and
$H^{\prime}$ respectively. Suppose that, for any subgroup $M$ of $H$ with
infinite index, $e(G,M:\mathcal{S})=1$. If $X$ crosses $Y$ strongly, then $Y$
crosses $X$ strongly.
\end{lemma}

\begin{remark}
We do not assume that $H^{\prime}$ is $VPC(n+1)$. In fact we do not even
assume that $H^{\prime}$ is finitely generated.
\end{remark}

\begin{proof}
Let $K$ denote the intersection $H\cap H^{\prime}$. As $X$ crosses $Y$
strongly, Lemma \ref{HintersectconjugateofSisK-finite} tells us that $K$ must
be $VPCn$, and that $e(H,K)=2$. Thus $H$ has a finite index subgroup $J$ which
contains $K$ as a normal subgroup, such that $J/K$ is infinite cyclic. Thus by
replacing $H$ by the finite index subgroup $J$, which we continue to denote by
$H$, we can assume that the intersection $K=H\cap H^{\prime}$ is normal in
$H$, and that the quotient $H/K$ is infinite cyclic.

As $K$ has infinite index in $H$, we know that $e(G,[K]:\mathcal{S})=1$. Thus
we can apply Lemma \ref{aisethasoneS-adaptedK-end}. Let $\Gamma(G,\mathcal{S}%
)$ be a relative Cayley graph of $(G,\mathcal{S})$. Lemma
\ref{aisethasoneS-adaptedK-end}\ tells us that the number of $\mathcal{S}%
$--adapted $K$--ends of the standard augmentation $\widehat{X}$ of $X$ is
equal to $1$. As in the proof of that lemma we denote $\widehat{X}$ by $Z$. We
also let $\widehat{Y}$ denote the standard augmentation of $Y$, and consider
the subset $\widehat{Y}\cap Z$ of $Z$. As $X$ crosses $Y$ strongly, $Y\cap H$
and $Y^{\ast}\cap H$ are each $K$--infinite. It follows that $Y\cap Z$, and
hence $\widehat{Y}\cap Z$ is $K$--infinite, and so is $\widehat{Y}^{\ast}\cap
Z$, its complement in $Z$. Also $\widehat{Y}\cap Z$ and $\widehat{Y}^{\ast
}\cap Z$ are $K$--invariant.

Now suppose that the lemma is false so that $Y$ crosses $X$ weakly. As before,
we can assume that $X\cap H^{\prime}$ is $K$--finite. Thus Lemma
\ref{L-finiteintersection} tells us that $X\cap H^{\prime}g$ is $K$--finite,
for every $g\in G$, and the same holds for $Z\cap H^{\prime}g$. Now
$\partial^{R}\widehat{Y}$ is a $H^{\prime}$--finite subset of $G$, for every
$R$, and each vertex of $G$ has fixed finite valence in $\Gamma(G,\mathcal{S}%
)$. It follows that the coboundary $\delta_{Z}^{R}(\widehat{Y}\cap Z)$ is
$K$--finite, for every $R$. This implies that $\widehat{Y}\cap Z$ and
$\widehat{Y}^{\ast}\cap Z$ each lie in the Boolean algebra $\mathcal{B}_{Z}$,
which in turn implies that the number of $\mathcal{S}$--adapted $K$--ends of
$Z=\widehat{X}$ is at least $2$. As we know that the number of $\mathcal{S}%
$--adapted $K$--ends of $\widehat{X}$ is equal to $1$, this is a
contradiction, which completes the proof of the lemma.
\end{proof}

The following result adds a little to the previous lemma in the case when both
$H$ and $H^{\prime}$ are $VPC(n+1)$.

\begin{lemma}
\label{XcrossesYstronglyimpliese=2}Let $(G,\mathcal{S})$ be a group system of
finite type with $e(G:\mathcal{S})=1$. Let $H\ $and $H^{\prime}$ be $VPC(n+1)$
subgroups of $G$, and let $X\ $and $Y$ be nontrivial $\mathcal{S}$--adapted
almost invariant subsets of $G$ over $H$ and $H^{\prime}$ respectively.
Suppose that, for any subgroup $M$ of $H$ of infinite index,
$e(G,M:\mathcal{S})=1$. If $X$ crosses $Y$ strongly, then $e(G,[H]:\mathcal{S}%
)=2$.
\end{lemma}

\begin{proof}
As in the preceding lemma, we let $K$ denote the intersection $H\cap
H^{\prime}$. That lemma tells us that $K$ must be $VPCn$, and that $e(H,K)$
and $e(H^{\prime},K)$ must each be equal to $2$. We will not need the fact
that $e(H,K)=2$.

Note that $e(G,[H]:\mathcal{S})\geq e(G,H:\mathcal{S})\geq2$. Now we will
suppose that $e(G,[H]:\mathcal{S})\geq3$, and will obtain a contradiction. The
assumption that $e(G,[H]:\mathcal{S})\geq3$ implies that there are three
disjoint nontrivial $\mathcal{S}$--adapted almost invariant subsets $X_{1}$,
$X_{2}$ and $X_{3}$ of $G$ each over a subgroup $H^{\prime\prime}$ of finite
index in $H$. As $X$ crosses $Y$ strongly, each $X_{i}$ must also cross $Y$
strongly. Now we consider the disjoint intersections $X_{i}\cap H^{\prime}$.
Each such intersection is a $K^{\prime}$--almost invariant subset of
$H^{\prime}$, where $K^{\prime}=H^{\prime\prime}\cap H^{\prime}$. As
$e(H^{\prime},K)=2$, it follows that $e(H^{\prime},K^{\prime})\geq2$. As
$H^{\prime}$ is $VPC$, Lemma \ref{Houghtonlemma} implies that $e(H^{\prime
},K^{\prime})=2$. Thus at least one of the sets $X_{i}\cap H^{\prime}$ must be
$K^{\prime}$--finite, so that $Y$ crosses that $X_{i}$ weakly. But this
contradicts Lemma \ref{strongcrossingissymmetric}. This contradiction shows
that we must have $e(G,[H]:\mathcal{S})=2$, as required.
\end{proof}

Putting the preceding results together yields the following result which is a
generalisation of Proposition 7.5 of \cite{SS2}, and is the main result of
this section.

\begin{theorem}
\label{crossinginCCCsisallweakorallstrong}Let $(G,\mathcal{S})$ be a group
system of finite type with $e(G:\mathcal{S})=1$. Let $n$ be a non-negative
integer, and suppose that, for any $VPC(\leq n)$ subgroup $M$ of $G$, we have
$e(G,M:\mathcal{S})=1$. Let $\{X_{\lambda}\}_{\lambda\in\Lambda}$ be a family

of nontrivial $\mathcal{S}$--adapted almost invariant subsets of $G$, such
that each $X_{\lambda}$ is over a $VPC(n+1)$ subgroup. Let $E$ denote the set
of all translates of the $X_{\lambda}$'s and their complements, and consider
the collection of cross-connected components (CCCs) of $E$ as in the
construction of regular neighbourhoods in section \ref{algregnbhds}. Then the
following statements hold:

\begin{enumerate}
\item The crossings in a non-isolated CCC of $E$ are either all strong or are
all weak.

\item In a non-isolated CCC $\Phi$ with all crossings weak, the stabilisers of
the elements of $E$ which lie in $\Phi$ are all commensurable. If $H$ denotes
such a stabiliser, then $e(G,[H]:\mathcal{S})\geq4$.

\item In a non-isolated CCC $\Phi$ with all crossings strong, if $H$ is the
stabiliser of an element of $E$ which lies in $\Phi$, then
$e(G,[H]:\mathcal{S})=2$.
\end{enumerate}
\end{theorem}

\begin{remark}
If $X$ is an element of $E$ with stabiliser $H$, and if

$e(G,[H]:\mathcal{S})=3$, this result shows that $X$ crosses no element of
$E$. Thus $X$ must be an isolated element of $E$.
\end{remark}

\begin{proof}
Let $X$ be an element of $E$, and let $\Phi$ denote the CCC of $E$ which
contains $X$. If $\Phi$ is not isolated, then $X$ crosses some element $Y$ of
$E$. Let $H$ and $K$ denote the stabilisers of $X$ and $Y$ respectively.

If $X$ crosses $Y$ weakly, then Lemma \ref{strongcrossingissymmetric} shows
that $Y$ also crosses $X$ weakly. Hence Lemma \ref{weak-weakcrossing} shows
that $H$ and $K$ must be commensurable, and that $e(G,[H]:\mathcal{S})\geq4$
and $e(G,[K]:\mathcal{S})\geq4$. Now Lemma \ref{XcrossesYstronglyimpliese=2}
implies that neither $X$ nor $Y$ can cross any element of $E$ strongly. It
follows that all crossings in $\Phi$ must be weak, and that the stabilisers of
elements of $E$ which lie in $\Phi$ are all commensurable.

If $X$ crosses $Y$ strongly, then Lemma \ref{strongcrossingissymmetric} shows
that $Y$ also crosses $X$ strongly, and then Lemma
\ref{XcrossesYstronglyimpliese=2} shows that $e(G,[H]:\mathcal{S})=2$ and
$e(G,[K]:\mathcal{S})=2$. Now Lemma \ref{weak-weakcrossing} implies that
neither $X$ nor $Y$ can cross any element of $E$ weakly. It follows that all
crossings in $\Phi$ must be strong.
\end{proof}

Recall that Lemma \ref{strongcrossingissymmetric} applied to almost invariant
sets which need not both be over $VPC(n+1)$ subgroups of $G$. Thus part of the
preceding theorem can also be generalised to this setting. This is the result
we obtain.

\begin{lemma}
Let $(G,\mathcal{S})$ be a group system of finite type with $e(G:\mathcal{S}%
)=1$. Let $H\ $and $H^{\prime}$ be $VPC(n+1)$ subgroups of $G$, and let
$X\ $and $Y$ be nontrivial $\mathcal{S}$--adapted almost invariant subsets of
$G$ over $H$ and $H^{\prime}$ respectively. Suppose that, for any subgroup $M$
of $H$ of infinite index, $e(G,M:\mathcal{S})=1$, and that $X$ crosses $Y$
strongly. If $Z$ is a nontrivial $\mathcal{S}$--adapted almost invariant
subset of $G$ over some subgroup $\Sigma$, and if $X\ $and $Z$ cross, then $Z$
must cross $X$ strongly.
\end{lemma}

\begin{remark}
In this statement, $\Sigma$ is not assumed to be $VPC(n+1)$, nor even finitely
generated. The case when $\Sigma$ is $VPC(n+1)$ is a special case of Theorem
\ref{crossinginCCCsisallweakorallstrong}.
\end{remark}

\begin{proof}
As $X$ crosses $Y$ strongly, Lemma \ref{XcrossesYstronglyimpliese=2} implies
that $e(G,[H]:\mathcal{S})=2$.

We will suppose that $Z$ crosses $X$ weakly, and obtain a contradiction.

Lemma \ref{strongcrossingissymmetric} shows that $X$ must also cross $Z$
weakly. Now Lemma \ref{intersectionisL-a.i.} shows that one corner, say $X\cap
Z$, of the pair $(X,Z)$ is $\mathcal{S}$--adapted and almost invariant over
$H\cap\Sigma$. Let $M$ denote $H\cap\Sigma$.

If $M\ $has infinite index in $H$, then $e(G,M:\mathcal{S})=1$, so that $X\cap
Z$ must be a trivial $M$--almost invariant subset of $G$. As the complement of
$X\cap Z$ contains $X^{\ast}$ which is $H$--infinite, this complement is
certainly $H$--infinite, and hence $M$--infinite. Thus $X\cap Z$ must be
$M$--finite, and hence $H$--finite, so that $X\ $and $Z$ do not cross,
contradicting our supposition.

If $M\ $has finite index in $H$, then $X$ is also almost invariant over $M$.
As $e(G,[H]:\mathcal{S})=2$, we must have $e(G,M:\mathcal{S})=2$, so that any
$\mathcal{S}$--adapted $M$--almost invariant subset of $G$ must be trivial or
equivalent to $X$ or to $X^{\ast}$. Hence $X\cap Z$ must be trivial or
equivalent to $X$. In either case, it follows that $Z\ $and $X$ do not cross,
again contradicting our supposition.

The above contradiction shows that $Z$ must cross $X$ strongly, as required.
\end{proof}

At this stage we have proved the basic results we will need about the crossing
of almost invariant subsets of $G$ which are over $VPC$ groups. Next we need
some more detail in the case where such sets cross strongly.

\begin{lemma}
\label{HhasinfiniteindexnormaliserimpliesGisVPCn+2}Let $(G,\mathcal{S})$ be a
group system of finite type with $e(G:\mathcal{S})=1$. Let $n\geq0$, and let
$H\ $and $K$ be $VPC(n+1)$ subgroups of $G$, and let $X\ $and $Z$ be
nontrivial $\mathcal{S}$--adapted almost invariant subsets of $G$ over $H$ and
$K$ respectively. Suppose that, for any subgroup $M$ of $H$ of infinite index,
$e(G,M:\mathcal{S})=1$. Suppose also that $X$ crosses $Z$ strongly. If $H$ has
infinite index in $Comm_{G}(H)$, then $G$ must be $VPC(n+2)$.
\end{lemma}

\begin{proof}
The argument here is essentially the same as in the proof of Proposition
B.3.14 of \cite{SS2}.

We start by noting that as $X$ crosses $Z$ strongly, Lemma
\ref{XcrossesYstronglyimpliese=2} tells us that $e(G,[H]:\mathcal{S})=2$.

Let $g$ be an element of $Comm_{G}(H)$ and let $Y=gX$, so that $Y$ has
stabiliser $H^{g}$. Let $H^{\prime}$ denote the intersection $H\cap H^{g}$
which has finite index in both $H$ and in $H^{g}$ because $g$ lies in
$Comm_{G}(H)$. Thus $H^{\prime}\backslash X$ and $H^{\prime}\backslash Y$ are
both $\mathcal{S}$--adapted almost invariant subsets of $H^{\prime}\backslash
G$. As $H^{\prime}$ has finite index in $H$, we must have $e(G,H^{\prime
}:\mathcal{S})=2$, so that $H^{\prime}\backslash X$ and $H^{\prime}\backslash
Y$ are almost equal or almost complementary. It follows that $X$ is
$H$--almost equal to $Y$ or $Y^{\ast}$, i.e. $gX\sim X$ or $gX\sim X^{\ast}$.
As in Lemma 2.10 of \cite{SS1}, we let $\mathcal{K}$ denote $\{g\in G:gX\sim
X$ or $gX\sim X^{\ast}\}$. As $H$ has infinite index in $Comm_{G}(H)$, we have
shown that $H$ has infinite index in $\mathcal{K}$.

Now suppose that $G$ is finitely generated. We will show shortly that this
must be the case. Lemma 2.13 of \cite{SS1} tells us $\mathcal{K}$ must
normalize some subgroup $H^{\prime\prime}$ of finite index in $H$. Further, as
$H$ has infinite index in $\mathcal{K}$, the proof of this lemma tells us that
$H^{\prime\prime}\backslash\mathcal{K}$ has two ends and that $\mathcal{K}$
has finite index in $G$. As $H$ is $VPC(n+1)$, so is $H^{\prime\prime}$, and
it follows that $\mathcal{K}$ must be $VPC(n+2)$. Hence $G$ also is
$VPC(n+2)$, as required.

If $G$ is not finitely generated, we construct a standard $\mathcal{S}%
^{\infty}$--extension $\overline{G}$, as discussed in section
\ref{extensionsandcrossing}. Thus $\overline{G}$ is finitely generated, and
$X$ and $Z$ have standard extensions $\overline{X}$ and $\overline{Z}$ to
almost invariant subsets of $\overline{G}$. Recall that if $g$ lies in
$Comm_{G}(H)$, then $gX\sim X$ or $gX\sim X^{\ast}$. As $\overline{G}$ is a
standard $\mathcal{S}^{\infty}$--extension of $G$, Lemma
\ref{extensionsareequivalent} shows that $g\overline{X}\sim\overline{X}$ or
$g\overline{X}\sim\overline{X}^{\ast}$. We now let $\overline{K}$ denote
$\{g\in\overline{G}:g\overline{X}\sim\overline{X}$ or $g\overline{X}%
\sim\overline{X}^{\ast}\}$. As $H$ has infinite index in $Comm_{G}(H)$, we
have shown that $H$ has infinite index in $\overline{K}$. Now the arguments in

the preceding paragraph show that $\overline{G}$ must be $VPC(n+2)$. But this
implies that every subgroup of $\overline{G}$ must be finitely generated,
contradicting our assumption that $G$ is not finitely generated. This
contradiction completes the proof of the lemma.
\end{proof}

Recall that in Theorem \ref{crossinginCCCsisallweakorallstrong} we showed that
the crossings in a non-isolated cross-connected component (CCC) of the given
family of almost invariant subsets of $G$ must be all weak or all strong. The
next result will play a key role in understanding those CCCs in which all
crossings are strong.

\begin{lemma}
\label{intersectionsubgroupsarecommensurable}Let $(G,\mathcal{S})$ be a group
system of finite type such that $e(G:\mathcal{S})=1$. Suppose that $n\geq0$,
and that for any $VPC(\leq n)$ subgroup $M$ of $G$, we have $e(G,M:\mathcal{S}%
)=1$. Let $P$, $Q$ and $R$ be nontrivial $\mathcal{S}$--adapted almost
invariant subsets of $G$ over $VPC(n+1)$ subgroups $H$, $K$ and $L$
respectively, and suppose that no one of $P$, $Q$ and $R$ is equivalent to one
of the others or to its complement. Suppose that $P$ crosses $Q$ and $R$
strongly. Then either $H\cap K$, $H\cap L$ and $K\cap L$ are all commensurable
$VPCn$ subgroups of $G$, or $G$ is $VPC(n+2)$. In the last case, each group in
$\mathcal{S}$ must be commensurable with a subgroup of $H\cap K\cap L$, and be
$VPC(\leq n)$.
\end{lemma}

\begin{proof}
As $P$ crosses $Q$ and $R$ strongly, Theorem
\ref{crossinginCCCsisallweakorallstrong} tells us that $e(G,[H]:\mathcal{S})$,
$e(G,[K]:\mathcal{S})$ and $e(G,[L]:\mathcal{S})$ are all equal to $2$. Now
suppose that $K\cap L$ has finite index in $K$ and in $L$. Then $e(G,K\cap
L:\mathcal{S})=2$. As $Q$ and $R$ are nontrivial $\mathcal{S}$--adapted almost
invariant subsets of $G$ over $K\cap L$, it follows that $Q$ is equivalent to
$R$ or to $R^{\ast}$, contradicting our hypotheses. We conclude that $K\cap L$
cannot have finite index in $K$ and in $L$, so that $K\cap L$ must be
$VPC(\leq n)$.

As $P$ crosses $Q$ and $R$ strongly, Lemma
\ref{HintersectconjugateofSisK-finite} tells us that $H\cap K$ and $H\cap L$
are each $VPCn$. If $H\cap K$ and $H\cap L$ are commensurable, then $H\cap
K\cap L$, which equals their intersection, must also be $VPCn$, and hence must
have finite index in $K\cap L$. Hence the three groups $H\cap K$, $H\cap L$
and $K\cap L$ are all commensurable, completing the proof of the lemma in this
case. Thus in what follows we will assume that $H\cap K$ and $H\cap L$ are not commensurable.

As $P$ crosses $R$ strongly, we cannot have $e(H,H\cap L)=1$. Now Lemma
\ref{Houghtonlemma} implies that $e(H,H\cap L)=2$ and that $H$ has a finite
index subgroup $H_{1}$ which contains $H\cap L$ as a normal subgroup, and
$H_{1}/(H\cap L)$ is infinite cyclic. Similarly, as $R$ must also cross $P$
strongly, $L$ has a finite index subgroup $L_{1}$ which contains $H\cap L$ as
a normal subgroup, and $L_{1}/(H\cap L)$ is infinite cyclic. We now replace
$H$ and $L$ by these subgroups $H_{1}$ and $L_{1}$ of finite index without
changing notation. (Note that $H_{1}\cap L_{1}$ is equal to $H\cap L$.) Thus
we have arranged that $H\cap L$ is normal in both $H$ and in $L$, and each
quotient is infinite cyclic. Now consider the image of $H\cap K$ in the
infinite cyclic group $H/(H\cap L)$. As we are assuming that $H\cap K$ and
$H\cap L$ are not commensurable, this image must be infinite, and hence is of
finite index. It follows that the subgroup of $H$ generated by $H\cap K$ and
$H\cap L$ has finite index in $H$. Recall that $P$ crosses $R$ strongly, so
that $H\cap R$ and $H\cap R^{\ast}$ are both $L$--infinite. Equivalently the
image of $H/(H\cap L)$ in $L\backslash G$ has infinite intersection with each
of $L\backslash R$ and $L\backslash R^{\ast}$. As the image of $H\cap K$ in
the infinite cyclic group $H/(H\cap L)$ is infinite, it follows that $H\cap
K\cap R$ and $H\cap K\cap R^{\ast}$ are both $L$--infinite. Hence $K\cap R$
and $K\cap R^{\ast}$ are both $L$--infinite, so that $Q$ also crosses $R$
strongly. Thus $K\cap L$ must be $VPCn$.

At this stage, after replacing $H$ and $L$ by suitable subgroups of finite
index so that $H\cap L$ is normal in both $H$ and in $L$, we have found that
$H\cap K$, $H\cap L$ and $K\cap L$ are each $VPCn$, and no two are
commensurable in $G$. Thus $H\cap K\cap L$ must be $VPC(n-1)$. As above, we

can find finite index subgroups $K^{\prime}$ of $K$ and $L^{\prime}$ of $L$
which contain $H\cap K$ and $H\cap L$ respectively as normal subgroups with
infinite cyclic quotient. We consider the natural maps from the $VPCn$ group
$K^{\prime}\cap L^{\prime}$ to the infinite cyclic groups $K^{\prime}/(H\cap
K^{\prime})$ and $L^{\prime}/(H\cap L^{\prime})$. Each of these maps has
kernel the $VPC(n-1)$ group $H\cap K^{\prime}\cap L^{\prime}$, and so each has
infinite image. Hence there is $x\in K^{\prime}\cap L^{\prime}$ which maps to
an element of infinite order in each of $K^{\prime}/(H\cap K^{\prime})$ and
$L^{\prime}/(H\cap L^{\prime})$. In particular $x$ normalizes $H\cap
K^{\prime}$ and $H\cap L^{\prime}$, and so normalizes a subgroup $H^{\prime}$
of $H$ of finite index. It follows that $H^{\prime}$ has infinite index in its
normalizer $N_{G}(H^{\prime})$ in $G$. Now Lemma
\ref{HhasinfiniteindexnormaliserimpliesGisVPCn+2} implies that $G$ must be
$VPC(n+2)$, as required.

Finally if $S$ is a group in $\mathcal{S}$, we need to show that $S$ must be
commensurable with a subgroup of $H\cap K\cap L$. As $G$ is $VPC(n+2)$, and
$e(G,H)>1$, Lemma \ref{Houghtonlemma} tells us that $G$ has a finite index
subgroup $J$ which contains $H$ as a normal subgroup, such that $J/H$ is
infinite cyclic. As $P$ is adapted to $S$, one of $S\cap P$ and $S\cap
P^{\ast}$ must be $H$--finite, so the projection of $S\cap J$ into $J/H$ must
be trivial. It follows that $S$ is commensurable with a subgroup of $H$.
Similarly $S$ must be commensurable with a subgroup of $K$ and of $L$, so that
$S$ must be commensurable with a subgroup of $H\cap K\cap L$, as required. As
this intersection is $VPC(\leq n)$, the same holds for $S$.
\end{proof}

The following consequence will be useful in the next section.

\begin{lemma}
\label{allintersectionsubgroupsarecommensurable}Let $(G,\mathcal{S})$ be a
group system of finite type with $e(G:\mathcal{S})=1$. Suppose that $n\geq0$,
and that for any $VPC(\leq n)$ subgroup $M$ of $G$, we have $e(G,M:\mathcal{S}%
)=1$. Further suppose that $G\ $is not $VPC(n+2)$. For $i=1,\ldots,k$, let
$P_{i}$ be a nontrivial $\mathcal{S}$--adapted almost invariant subset of $G$
over a $VPC(n+1)$ subgroup $H_{i}$, such that $P_{i}$ crosses $P_{i+1}$
strongly, for $1\leq i\leq k-1$. Then either $H_{1}\cap H_{k}$ is
commensurable with $H_{1}\cap H_{2}$, or $P_{1}$ is equivalent to $P_{k}$ or
to $P_{k}^{\ast}$.
\end{lemma}

\begin{proof}
As $G\ $is not $VPC(n+2)$, Lemma \ref{intersectionsubgroupsarecommensurable}
tells us that the groups $H_{1}\cap H_{2}$, $H_{2}\cap H_{3},\ldots
,H_{k-1}\cap H_{k}$ are all commensurable $VPCn$ subgroups of $G$. Hence the
intersection $L$ of all these groups is also commensurable with $H_{1}\cap
H_{2}$, and so is $VPCn$. As $H_{1}\cap H_{k}$ contains $L$, it must be $VPCn$
or $VPC(n+1)$. In the first case, $L$ must be of finite index in $H_{1}\cap
H_{k}$, so that $H_{1}\cap H_{k}$ is commensurable with $H_{1}\cap H_{2}$ as
required. In the second case, $H_{1}$ and $H_{k}$ must be commensurable
subgroups of $G$. As $e(G,[H_{1}]:\mathcal{S})=2$, by Lemma
\ref{XcrossesYstronglyimpliese=2}, it follows that $P_{1}$ is equivalent to
$P_{k}$ or to $P_{k}^{\ast}$, as required.
\end{proof}

\section{CCCs with strong crossing\label{CCC'swithstrongcrossing}}

As in the previous section, we consider a group system $(G,\mathcal{S})$ of
finite type. As in Theorem \ref{crossinginCCCsisallweakorallstrong}, we
further assume that $e(G:\mathcal{S})=1$. Let $n$ be a non-negative integer,
and suppose that, for any $VPC(\leq n)$ subgroup $M$ of $G$, we have
$e(G,M:\mathcal{S})=1$. Let $\mathcal{X}=\{X_{\lambda}\}_{\lambda\in\Lambda}$
be a finite family of nontrivial almost invariant subsets of $G$, such that
each $X_{\lambda}$ is over a $VPC(n+1)$ subgroup and is adapted to
$\mathcal{S}$. Let $E(\mathcal{X})$ denote the set of all translates of the
$X_{\lambda}$'s and their complements, and consider the collection of
cross-connected components (CCCs) of $E(\mathcal{X})$ as in the construction
of algebraic regular neighbourhoods in section \ref{algregnbhds}. Theorem
\ref{crossinginCCCsisallweakorallstrong} tells us that crossings in a
non-isolated CCC of $E(\mathcal{X})$ are either all strong or all weak.

In this section, we will consider a non-isolated CCC $\Phi$ of $E(\mathcal{X}%
)$ in which all crossings are strong. Theorem
\ref{crossinginCCCsisallweakorallstrong} tells us that if $H$ is the
stabiliser of an element of $E(\mathcal{X})$ which lies in $\Phi$, then
$e(G,[H]:\mathcal{S})=2$. Our aim is to describe the structure of the
corresponding $V_{0}$--vertex of the algebraic regular neighbourhood of the
$X_{\lambda}$'s in $G$. We will show that this vertex is of $VPCn$%
--by--Fuchsian type relative to $\mathcal{S}$, in a sense to be defined shortly.

The following technical result is an immediate consequence of Lemma
\ref{cangetvgpinsystems}, and will be helpful for the arguments of this section.

\begin{lemma}
\label{equivalentimpliesequal}Let $(G,\mathcal{S})$ be a group system of
finite type such that $e(G:\mathcal{S})=1$. Suppose that $n\geq0$, and that
for any $VPC(\leq n)$ subgroup $M$ of $G$, we have $e(G,M:\mathcal{S})=1$.
Further suppose that $G$ is not $VPC(n+2)$. Let $\mathcal{X}=\{X_{\lambda
}\}_{\lambda\in\Lambda}$ be a finite family of nontrivial $\mathcal{S}%
$--adapted almost invariant subsets of $G$, such that each $X_{\lambda}$ is
over a $VPC(n+1)$ subgroup. Let $E(\mathcal{X})$ denote the set of all
translates of the $X_{\lambda}$'s and their complements, and let $\Phi$ be a
non-isolated CCC of $E(\mathcal{X})$ in which all crossings are strong.

Then there is $Y_{\lambda}$ equivalent to $X_{\lambda}$, for each $\lambda$,
such that the $Y_{\lambda}$'s are in very good position. Further if
$\mathcal{Y}$ denotes the family of the $Y_{\lambda}$'s, and $\Phi
(\mathcal{Y})$ denotes the CCC of $E(\mathcal{Y})$ which corresponds to $\Phi
$, then equivalent elements of $\Phi(\mathcal{Y})$ must be equal.
\end{lemma}

\begin{proof}
Lemma \ref{cangetvgpinsystems} tells us that there is $Y_{\lambda}$ equivalent
to $X_{\lambda}$, for each $\lambda$, such that the following conditions hold:

\begin{enumerate}
\item The $Y_{\lambda}$'s are in very good position.

\item If $U$ and $V$ are equivalent elements of $E(Y)$, and do not lie in the
same $E(Y_{\lambda})$, then $U\ $and $V$ are equal.

\item For each $\lambda$ such that $H_{\lambda}$ has small commensuriser in
$G$, if $U$ and $V$ are equivalent elements of $E(Y_{\lambda})$, then $U\ $and
$V$ are equal.
\end{enumerate}

Let $U\ $be an element of the CCC $\Phi$, which is almost invariant over $H$.
As we are assuming that $G$ is not $VPC(n+2)$, Lemma
\ref{HhasinfiniteindexnormaliserimpliesGisVPCn+2} implies that $H$ has small
commensuriser in $G$. Now conditions 2) and 3) together imply that equivalent
elements of $\Phi(\mathcal{Y})$ must be equal, as required.
\end{proof}

If we change a given family of $X_{\lambda}$'s as in Lemma
\ref{equivalentimpliesequal}, this will not alter their reduced algebraic
regular neighbourhood. But it will still be convenient, in the rest of this
section, to make the following assumption.

\begin{condition}
\label{conditionofverygoodposition}Our given family of $X_{\lambda}$'s
satisfies the conclusions of Lemma \ref{equivalentimpliesequal}.
\end{condition}

We recall the definition of $VPCn$--by--Fuchsian type given in \cite{SS2}.

\begin{definition}
Let $\Gamma$ be a minimal graph of groups decomposition of a group $G$. A
vertex $v$ of $\Gamma$ is of \textsl{$VPC$--by--Fuchsian type} if $G(v)$ is a
$VPC$--by--Fuchsian group, where the Fuchsian group is not finite nor
two-ended, and there is exactly one edge of $\Gamma$ which is incident to $v$
for each peripheral subgroup $\Sigma$ of $G(v)$ and this edge carries $\Sigma
$. If the length of the normal $VPC$ subgroup of $G(v)$ is $n$, we will say
that $G(v)$ is of \textsl{$VPCn$--by--Fuchsian type.}
\end{definition}

\begin{remark}
In this definition, a peripheral subgroup of $G(v)$ is the pre-image of a
peripheral subgroup of the Fuchsian quotient, and saying that an edge of
$\Gamma$ carries $\Sigma$ means that the image of the edge group in $G(v)$ is
equal to $\Sigma$.
\end{remark}

This needs to be modified a little when we are considering group systems.

\begin{definition}
\label{defnofVPC-by-FuchsianrelativetoS}Let $(G,\mathcal{S})$ be a group
system, and let $\Gamma$ be a minimal $\mathcal{S}$--adapted graph of groups
decomposition of $G$. A vertex $v$ of $\Gamma$ is of $VPC$%
\textsl{--by--Fuchsian type relative to }$\mathcal{S}$ if the following
conditions hold:

\begin{enumerate}
\item The vertex group $G(v)$ is $VPC$--by--Fuchsian, where the Fuchsian group
is not finite nor two-ended.

\item Each edge of $\Gamma$ which is incident to $v$ carries some peripheral
subgroup of $G(v)$, and for each peripheral subgroup $\Sigma$, there is at
most one edge of $\Gamma$ which is incident to $v$ and carries $\Sigma$.

\item For each peripheral subgroup $\Sigma$ of $G(v)$, if there is no edge of
$\Gamma$ which is incident to $v$ and carries $\Sigma$, then there is a group
$S_{i}$ in $\mathcal{S}$, such that a conjugate of $S_{i}$ is contained in
$\Sigma$, but $S_{i}$ is not conjugate commensurable with a subgroup of the
$VPC$ normal subgroup of $G(v)$.

\item If there is $S_{i}$ in $\mathcal{S}$, such that a conjugate of $S_{i}$
is contained in $G(v)$, then either $S_{i}\ $is conjugate to a subgroup of
some peripheral subgroup of $G(v)$, or $S_{i}$ is conjugate commensurable with
a subgroup of the $VPC$ normal subgroup of $G(v)$.
\end{enumerate}

If the length of the normal $VPC$ subgroup of $G(v)$ is $n$, we will say that
$G(v)$ is of $VPCn$\textsl{--by--Fuchsian type relative to }$S$\textsl{.}
\end{definition}

\begin{remark}
\label{remarkondefnofVPC-by-FuchsianrelS}Here the term Fuchsian group means a
discrete group of isometries of the hyperbolic plane.

For a $VPCn$--by--Fuchsian group, each peripheral subgroup is $VPC(n+1)$.

In the special case when $\Gamma$ consists of the single vertex $v$, so that
$G=G(v)$, there are no edges of $\Gamma$ incident to $v$. Thus condition 2) is
vacuous. Condition 3) tells us that, for each peripheral subgroup $\Sigma$ of
$G$, there is $S_{i}$ in $\mathcal{S}$ which is conjugate to a subgroup of
$\Sigma$, but is not conjugate commensurable with a subgroup of the $VPC$
normal subgroup $K$ of $G$. As $K$ is normal in $G$, we can simply say that
$S_{i}$ is not commensurable with a subgroup of $K$. As $G=G(v)$, each $S_{i}$
in $\mathcal{S}$ is contained in $G(v)$. Thus condition 4) tells us that each
$S_{i}$ in $\mathcal{S}$ is conjugate to a subgroup of some peripheral
subgroup of $G$, or $S_{i}$ is commensurable with a subgroup of $K$. Finally
if $G\ $has no peripheral subgroups, so that the orbifold quotient $G/K$ is
closed, then each $S_{i}$ in $\mathcal{S}$ must be commensurable with a
subgroup of $K$.
\end{remark}

The main result of this section is the following, which is proved at the end
of this section. This generalises Theorem 7.8 of \cite{SS2}.

\begin{theorem}
\label{strongcrossingimpliesFuchsiantype}Let $(G,\mathcal{S})$ be a group
system of finite type such that $e(G:\mathcal{S})=1$. Suppose that $n\geq0$,
and that for any $VPC(\leq n)$ subgroup $M$ of $G$, we have $e(G,M:\mathcal{S}%
)=1$. Further suppose that $G$ is not $VPC(n+2)$. Let $\{H_{\lambda
}\}_{\lambda\in\Lambda}$ be a finite family of $VPC(n+1)$ subgroups of $G$.
For each $\lambda\in\Lambda$, let $X_{\lambda}$ denote a nontrivial
$\mathcal{S}$--adapted $H_{\lambda}$--almost invariant subset of $G$. Let
$E(\mathcal{X})$ denote the set of all translates of the $X_{\lambda}$'s and
their complements, and let $\Gamma$ denote the algebraic regular neighbourhood
of the $X_{\lambda}$'s in $G$. Let $\Phi$ be a non-isolated CCC of
$E(\mathcal{X})$ in which all crossings are strong, and let $v$ denote the
corresponding vertex of $\Gamma$. Then $v$ is of $VPCn$--by--Fuchsian type
relative to $\mathcal{S}$.
\end{theorem}

As in earlier proofs of Dunwoody and Swenson \cite{D-Swenson} and Bowditch
\cite{B1}, our aim is to construct in a natural way an action of the group
$G(v)$, the stabilizer of $\Phi$, on the circle $S^{1}$. We will give a new
proof of the construction of this action which is geared to the setting of
this paper, but it is closely related to the original proof of Dunwoody and
Swenson \cite{D-Swenson}.

If $X$ and its complement $X^{\ast}$ are elements of $\Phi$, we want to
associate a pair of points of $S^{1}$. Such a pair of points will cut $S^{1}$
into two intervals, and we want further to associate one of these intervals to
$X$ and the other to $X^{\ast}$. Although all our arguments formally deal with
this situation, it seems helpful to think of associating to $X$ and $X^{\ast}$
the chord in the unit disk $D$ joining the pair of points in $S^{1}$. Further
such a chord cuts $D$ into two half disks, and we want to associate one of
these half disks with $X$ and the other with $X^{\ast}$. The aim is to do this
so that the stabilizer of the CCC $\Phi$ acts on $S^{1}$ and the entire
process is equivariant with respect to these actions. Then one shows that the
kernel of this action is $VPCn$ and that the quotient group acts on $S^{1}$ as
a convergence group. Now the work of Tukia \cite{Tukia}, Casson and Jungreis
\cite{Casson-Jungreis}\ and Gabai \cite{Gabai} implies that this action is
conjugate to a Fuchsian group action, so that the quotient group acts by
isometries on the hyperbolic plane $\mathbb{H}^{2}$, where $S^{1}$ is regarded
as the set of points at infinity of $\mathbb{H}^{2}$. It follows that $G(v)$
is $VPCn$--by--Fuchsian. Thus it is natural to think of the interior of the
disk $D$ as the hyperbolic plane, so that a chord (without its endpoints) can
be viewed as a hyperbolic geodesic. This process seems particularly natural in
the case when $n=1$, so that the elements of $\Phi$ are almost invariant over
$VPC1$ subgroups of $G$. As a $VPC1$ group is two-ended, we can think of the
boundary of an element of $\Phi$ as being two-ended, just like a hyperbolic
geodesic. If this process has been carried out equivariantly for all elements
of $\Phi$, then the family of chords so obtained will cut $D$ into subregions
which are intersections of the half disks determined by each chord. It is
natural to associate such intersections to intersections of the corresponding
elements of $\Phi$. Thus we consider the Boolean algebra generated by the
elements of $\Phi$. This motivates the terminology which we will shortly
introduce. In the next subsection we set the scene for the construction which
we will carry out in the following subsection.

\subsection{The setup}

We use the hypotheses and notation of Theorem
\ref{strongcrossingimpliesFuchsiantype}. Recall from Lemma
\ref{strongcrossingissymmetric} that if $X$ and $Y$ are elements of the CCC
$\Phi$ over $VPC(n+1)$ subgroups $H_{X}$ and $H_{Y}$ respectively, such that
$X\ $and $Y$ cross, then the intersection $H_{X}\cap H_{Y}$ is $VPCn$. Further
Lemmas \ref{strongcrossingissymmetric} and \ref{XcrossesYstronglyimpliese=2}
together imply that $e(H_{X},H_{X}\cap H_{Y})=2$ and $e(H_{Y},H_{X}\cap
H_{Y})=2$. As $G$ is assumed not to be $VPC(n+2)$, Lemmas
\ref{intersectionsubgroupsarecommensurable} and
\ref{allintersectionsubgroupsarecommensurable} tell us that for any other pair
of elements of $\Phi$, whether or not they cross, the corresponding
intersection is commensurable with $H_{X}\cap H_{Y}$. Further $H_{X}$ and
$H_{Y}$ have small commensuriser in $G$. We let $K$ denote a subgroup of $G$
commensurable with $H_{X}\cap H_{Y}$, and as usual let $[K]$ denote the
commensurability class of $K$. Note that for any element $X$ of the CCC $\Phi
$, we know that $e(H_{X},K\cap H_{X})=2$. Thus we can think of the boundary of
$(K\cap H_{X})\backslash X$ as being two-ended, just like the chord we want to
associate to $X$.

Let $\mathcal{B}$ be the Boolean algebra generated by the almost invariant
sets in the CCC $\Phi$. Note that, for any element of $\mathcal{B}$, there is
$K^{\prime}\in\lbrack K]$ which stabilizes that element. Hence, for any finite
collection of elements of $\mathcal{B}$, there is $K^{\prime}\in\lbrack K]$
which stabilizes each element in the collection. Suppose we have associated
chords and half disks to elements of $\Phi$. If two half disks have boundary
chords which cross, then the boundary of their intersection consists of the
union of two half chords. This motivates the definition of a half plane which
we will give below.

Let $X$ be an element of the CCC $\Phi$. Thus $X$ is almost invariant over
some $VPC(n+1)$ subgroup $H$ of $G$. As in the preceding section, we choose a
finite subset $A$ of $G$ which generates $G$ relative to $\mathcal{S}$. This
determines the graph $\Gamma(G:A)$, with vertex set $A$, which is in turn
contained in the relative Cayley graph $\Gamma(G,\mathcal{S})$. We will denote

$\Gamma(G:A)$ by $\Gamma$, and will denote $\Gamma(G,\mathcal{S})$ by
$\widehat{\Gamma}$. Let $\widehat{X}$ denote the standard augmentation of $X$
in $\widehat{\Gamma}$. Now Lemma \ref{AlmostinvariantsetbecomesconnectedinGL}
tells us there is $R$ such that $\partial^{R}\widehat{X}$ is connected in
$\Gamma^{R}$, and $\widehat{X}$ is connected in $\widehat{\Gamma}^{R}$. We
denote $\widehat{\Gamma}^{R}[\widehat{X}]$ by $W^{R}$, and denote $\Gamma
^{R}[\partial^{R}Z]$ by $dW^{R}$. Thus $W^{R}$ and $dW^{R}$ are connected
graphs. Now choose $K^{\prime}\in\lbrack K]$ such that $X$ is $K^{\prime}%
$--invariant. Then $dW^{R}$ is also $K^{\prime}$--invariant, and the number of
ends of the quotient graph $K^{\prime}\backslash dW^{R}$ is equal to $2$, by
Lemma \ref{aisethasoneS-adaptedK-end}. Thus there is a finite subcomplex of
this quotient whose complement has two infinite (hence one-ended) components.
Of course if we consider a different element of $\Phi$, we would need to
choose a different group in $[K]$. We now introduce some terminology which
allows us to avoid continually choosing different such groups.

\begin{definition}
\label{defnofsemi-line}Let $K^{\prime}$ be an element of $[K]$ such that $X$
is $K^{\prime}$--invariant, and let $R$ be such that $W^{R}$ and $dW^{R}$ are
connected. Let $p$ denote the projection map $\widehat{\Gamma}\rightarrow
K^{\prime}\backslash\widehat{\Gamma}$. As $dW^{R}$ is connected, $p(dW^{R})$
is two-ended. Let $F$ be a finite subgraph of $p(dW^{R})$ such that
$p(dW^{R})-F$ has two infinite components, and let $L$ denote the pre-image in
$\widehat{\Gamma}=\Gamma(G,\mathcal{S})$ of one of these infinite components.
We will say that $L$ is a \emph{semi-line}, and that two semi-lines are
equivalent if their vertex sets are within bounded Hausdorff distance of each
other in the angle metric on $G$.
\end{definition}

\begin{remark}
Note that, for all $R$ and $S$, the vertex sets of $dW^{R}$ and $dW^{S}$ are
within bounded Hausdorff distance in the angle metric on $G$.
\end{remark}

Let $\mathcal{E}$ denote the collection of equivalence classes of all
semi-lines obtained in this way from elements of $\Phi$. Thus each element of
$\Phi$ contributes two points to $\mathcal{E}$. We will call these points of
$\mathcal{E}$ the \textit{endpoints} of $\Phi$. The following result implies
that distinct elements of $\Phi$ have distinct endpoints.

\begin{lemma}
\label{semi-linemeetsa.i.set}Let $X$ and $Y$ be elements of the CCC $\Phi$,
such that $X\ $is not equal to $Y$ or to $Y^{\ast}$, and let $L$ be a
semi-line obtained from $X$. Then, for any $R$, the intersection
$L\cap\partial^{R}\widehat{Y}$ is $K$--finite.
\end{lemma}

\begin{remark}
\label{fourdistinctendpoints}This result implies that $L$ is not equivalent to
either of the semi-lines obtained from $Y$, so that $X$ and $Y$ together have
four distinct endpoints.
\end{remark}

\begin{proof}
As usual, we pick $K^{\prime}$ in $[K]$ which stabilizes both $X$ and $Y$,
and, by replacing $H_{X}$ by a suitable subgroup of finite index, we arrange
that $K^{\prime}$ is normal in $H_{X}$ with infinite cyclic quotient generated
by $h$ in $H_{X}$.

By Lemma \ref{AlmostinvariantsetbecomesconnectedinGL}, there is $R$ such that
each of $\widehat{X}$ and $\widehat{Y}$ is connected in $\widehat{\Gamma}^{R}%
$, and each of $\partial^{R}\widehat{X}$ and $\partial^{R}\widehat{Y}$ is
connected in $\Gamma^{R}$. As before, we will denote $\widehat{\Gamma}%
^{R}[\widehat{X}]$ by $W^{R}$,$\ $and denote $\Gamma^{R}[\partial
^{R}\widehat{X}]$ by $dW^{R}$.

Now suppose that the required result does not hold for some $R$ and hence for
all $R$ greater than some fixed value. Thus by increasing the value of $R$ if
needed, we can arrange that $L\cap\partial^{R}\widehat{Y}$ is $K$--infinite.
Suppose that $L$ was obtained from $dW^{R}$ by removing the pre-image of a
finite subgraph of $p(dW^{R})$. As $dW^{R}$ is $H_{X}$--finite, there is $h\in
H_{X}-K$, a vertex $z$ of $L$, and infinitely many distinct integers $n$ such
that $h^{n}z\in\partial^{R}\widehat{Y}$. As $\partial^{R}\widehat{Y}$ is
$H_{Y}$--finite, this implies that there is a finite collection of cosets
$(H_{Y})g$ whose union contains infinitely many distinct elements of the form
$h^{n}$. In particular there must be a coset $(H_{Y})g$ which contains two
distinct such elements, so that some non-zero power of $h$ lies in $H_{Y}$.
This implies that $H_{X}$ and $H_{Y}$ are commensurable. As $e(G,[H_{X}%
]:\mathcal{S})=2$, by Theorem \ref{crossinginCCCsisallweakorallstrong}, it
follows that $X\ $is equivalent to $Y$ or to $Y^{\ast}$. As we are assuming
Condition \ref{conditionofverygoodposition}, this implies that $X\ $is equal
to $Y$ or to $Y^{\ast}$. This contradiction completes the proof.
\end{proof}

Lemma \ref{AlmostinvariantsetbecomesconnectedinGL} implies that for any almost
invariant set $X$ in the CCC $\Phi$, there is $N$ such that $\partial
^{N}\widehat{X}$ is connected in $\Gamma^{N}$, and $\widehat{X}$ is connected
in $\widehat{\Gamma}^{N}$. We will now show that an analogous result holds for
any finite intersection of almost invariant sets in the CCC $\Phi$. If $P$ is
a finite intersection of almost invariant sets $X_{\lambda}$ in $\Phi$, we let
$\widehat{P}$ denote the corresponding intersection of their standard
augmentations $\widehat{X_{\lambda}}$. It will be convenient to refer to
$\widehat{P}$ as the standard augmentation of $P$. Note that, as $\partial
^{R}\widehat{P}$ is contained in the union of the $\partial^{R}%
\widehat{X_{\lambda}}$, for any $R$, we know that $\partial^{R}\widehat{P}%
\subset G$, for any $R$.

Here is the generalisation of Lemma
\ref{AlmostinvariantsetbecomesconnectedinGL} which we need.

\begin{lemma}
\label{R-nbhdofintersectionsisconnectedinrelativecase}Using notation as above,
if $P$ is a finite intersection of almost invariant sets in the CCC $\Phi$,
then there is $N$ such that $\partial^{N}\widehat{P}$ is connected in
$\Gamma^{N}$, and $\widehat{P}$ is connected in $\widehat{\Gamma}^{N}$.
Further, $\partial^{R}\widehat{P}$ is connected in $\Gamma^{R}$, and
$\widehat{P}$ is connected in $\widehat{\Gamma}^{R}$, for any $R\geq N$, and
$\Gamma^{R}[\partial^{R}\widehat{P}]$ is $K$--almost equal to a finite union
of semi-lines.
\end{lemma}

\begin{proof}
We will prove the lemma by induction on the number $n$ of almost invariant
sets in this intersection. When $n=1$, such sets are elements of the CCC
$\Phi$. Thus the result holds in this case, by Lemma
\ref{AlmostinvariantsetbecomesconnectedinGL} and the definition of a semi-line.

For the induction step, suppose that $n\geq2$, and that $P$ is the
intersection of $n$ elements of $\Phi$. Let $X$ denote one of these elements,
and let $Q$ denote the intersection of the remaining $(n-1)$ elements. By our
induction hypothesis, there is $N$ such that $\widehat{X}$ and $\widehat{Q}$
are each connected in $\widehat{\Gamma}^{N}$, and $\partial^{N}\widehat{X}$
and $\partial^{N}\widehat{Q}$ are each connected in $\Gamma^{N}$. Further, for
any $R\geq N$, the analogous statement holds, and each of $\Gamma^{R}%
[\partial^{R}\widehat{X}]$ and $\Gamma^{R}[\partial^{R}\widehat{Q}]$ is
$K$--almost equal to a finite union of semi-lines.

We let $W$ and $dW$ denote the connected graphs $\widehat{\Gamma}%
^{N}[\widehat{X}]$ and $\Gamma^{N}[\partial^{N}\widehat{X}]$, respectively. We
also let $\Omega$ and $d\Omega$ denote the connected graphs $\widehat{\Gamma
}^{N}[\widehat{Q}]$ and $\Gamma^{N}[\partial^{N}\widehat{Q}]$, respectively.
Let $L$ denote one of the semi-lines in $d\Omega$. Then Lemma
\ref{semi-linemeetsa.i.set} tells us that $L\cap\partial^{N}\widehat{X}$ is

$K$--finite, so that $L\cap dW$ is $K$--finite. It follows that $L\cap W$ is
$K$--finite or is $K$--almost equal to $L$. Thus $d\Omega\cap W$ is
$K$--almost equal to a finite union of semi-lines. A similar argument shows
that $dW\cap\Omega$ is $K$--almost equal to a union of at most two semi-lines.
Note that $\widehat{P}=\widehat{Q}\cap\widehat{X}$, and $\partial
^{N}\widehat{P}=[\partial^{N}\widehat{Q}\cup\partial^{N}\widehat{X}%
]\cap\widehat{P}$. Thus the edges of $\Gamma^{N}[\partial^{N}\widehat{P}]$ are
of three types. Those which join points of $\partial^{N}\widehat{Q}$, and so
lie in $d\Omega\cap W$, those which join points of $\partial^{N}\widehat{X}$,
and so lie in $dW\cap\Omega$, and those which join a point of $\partial
^{N}\widehat{X}$ to a point of $\partial^{N}\widehat{Q}$. We claim that the
number of edges of the third type is $K$--finite. Let $L$ be a semi-line in
$dW$. Then Lemma \ref{semi-linemeetsa.i.set} implies that, for any $R$, the
intersection $L\cap\partial^{R}\widehat{Q}$ is $K$--finite. In particular
$L\cap\partial^{2N}\widehat{Q}$ is $K$--finite. As $\Gamma^{N}$ is a graph
whose vertices have constant finite valence, it follows that the number of
edges of $\Gamma^{N}$ which join $L$ to $\partial^{N}\widehat{Q}$ is
$K$--finite. It follows that the number of edges of $\Gamma^{N}$ which join
$\partial^{N}\widehat{X}$ to $\partial^{N}\widehat{Q}$ is $K$--finite, which
proves the above claim. Hence $\Gamma^{N}[\partial^{N}\widehat{P}]$ is
$K$--almost equal to $(d\Omega\cap W)\cup(dW\cap\Omega)$, and so is
$K$--almost equal to a finite union of semi-lines. The same argument shows
that $\Gamma^{R}[\partial^{R}\widehat{P}]$ is $K$--almost equal to a finite
union of semi-lines for any $R\geq N$.

In order to complete the proof of the lemma, we need to arrange that
$\Gamma^{R}[\partial^{R}\widehat{P}]$ is connected, when $R$ is large enough.
It will then follow immediately that $\widehat{P}$ is connected in
$\widehat{\Gamma}^{R}$. At this point, we argue in a similar way to the proof
of Lemma \ref{AlmostinvariantsetbecomesconnectedinGL}.

As $\Gamma^{N}[\partial^{N}\widehat{P}]$ is $K$--almost equal to a finite
union of semi-lines, and each semi-line is a connected subgraph of $\Gamma
^{N}$, we could obtain a connected subgraph of $\widehat{\Gamma}^{N}$ by
adding a $K$--finite number of vertices and edges to $\Gamma^{N}[\partial
^{N}\widehat{P}]$. It follows that there is $R\geq N$ such that $\partial
^{N}\widehat{P}$ is connected in $\Gamma^{R}$. Now the argument used in the
proof of Lemma \ref{AlmostinvariantsetbecomesconnectedinGL} shows that
$\partial^{R}\widehat{P}$ is also connected in $\Gamma^{R}$. This completes
the proof of the lemma.
\end{proof}

Recall that in Definition \ref{defnofnumberofS-adaptedK-ends}, we defined the
number of $\mathcal{S}$--adapted $K$--ends of any $K$--invariant subset of $G$.

Next we define a half-plane.

\begin{definition}
\label{defnofhalfplane}An element $A$ of $\mathcal{B}$ is a \emph{half-plane}
if the number of $S$--adapted $K$--ends of the standard augmentation
$\widehat{A}$ of $A$ is equal to $1$, and there is $R$ such that $\Gamma
^{R}[\partial^{R}\widehat{A}]$ is $K$--almost equal to the union of two
distinct semi-lines.
\end{definition}

The definition implies that a half-plane determines two points of
$\mathcal{E}$, the collection of equivalence classes of all semi-lines
obtained from elements of $\Phi$. We will call these points of $\mathcal{E}$
the \textit{endpoints} of the half-plane.

Lemma \ref{aisethasoneS-adaptedK-end} combined with the definition of
semi-lines immediately implies that each element of the CCC $\Phi$ is a half-plane.

Next we will show that many more elements of $\mathcal{B}$ are half-planes.

\begin{definition}
\label{defnofcrossingofhalf-planes}We will say that two half-planes $X\ $and
$Y$ in $\mathcal{B}$ \emph{cross} if $X$ and $Y$ together have four distinct
endpoints, $X\amalg X^{\ast}$ cuts $\partial Y$ into two $K$--infinite sets,
and $Y\amalg Y^{\ast}$ cuts $\partial X$ into two $K$--infinite sets.
\end{definition}

\begin{lemma}
\label{XintersectYisahalf-plane} Let $X\ $and $Y$ be two half-planes in
$\mathcal{B}$ which cross. Then $X\cap Y$ is also a half-plane.
\end{lemma}

\begin{proof}
Let $E=X\cap Y$, and choose $N$ large enough so that $\widehat{X}$,
$\widehat{Y}$ and $\widehat{E}$ are each connected in $\widehat{\Gamma}^{N}$,
and $\partial^{N}\widehat{X}$, $\partial^{N}\widehat{Y}$ and $\partial
^{N}\widehat{E}$ are each connected in $\Gamma^{N}$. Of the two semi-lines
determined by $\widehat{X}$, exactly one is within a finite neighbourhood of
$\widehat{Y}$ and hence of $\widehat{E}$. And of the two semi-lines determined
by $\widehat{Y}$, exactly one is within a finite neighbourhood of
$\widehat{X}$ and hence of $\widehat{E}$. As $\Gamma^{N}[\partial
^{N}\widehat{E}]$ is $K$--almost equal to a finite union of semi-lines, by
Lemma \ref{R-nbhdofintersectionsisconnectedinrelativecase}, it follows that
$\Gamma^{N}[\partial^{N}\widehat{E}]$ is $K$--almost equal to the union of two
semi-lines. It remains to show that the number of $S$--adapted $K$--ends of
$\widehat{E}$ is equal to $1$.

Suppose that the number of $S$--adapted $K$--ends of $\widehat{E}$ is greater
than $1$, so that we can partition $\widehat{E}$ into two $K$--infinite sets
$A$ and $B$ in $\mathcal{B}_{\widehat{E}}$. Thus there is $K^{\prime}%
\in\lbrack K]$ such $A$ and $B$ are $K^{\prime}$--invariant, and
$\delta_{\widehat{E}}^{R}A$ and $\delta_{\widehat{E}}^{R}B$ are $K$--finite,
for all $R$. By replacing $K^{\prime}$ by a different element of $[K]$ if
needed, we can assume that $A$, $B$, $X$ and $Y$ are each $K^{\prime}%
$--invariant. We denote $\Gamma^{N}[\partial^{N}\widehat{E}]$ by $dW$. Let
$p:\Gamma\rightarrow K^{\prime}\backslash\Gamma$. We know that $p(dW)$ is
two-ended, one end corresponding to a semi-line obtained from $X$ and the
other to a semi-line obtained from $Y$.

Consider the sets $A\cap dW$ and $B\cap dW$. We claim that each must be
$K$--infinite. For suppose that $A\cap dW$ is $K$--finite. As $\delta
^{N}A=\delta_{\widehat{E}}^{N}A\cup(\delta^{N}A\cap\delta^{N}\widehat{E}%
)\subset\delta_{\widehat{E}}^{N}A\cup\delta^{N}(A\cap dW)$, it follows that
$\delta^{N}A$ is also $K$--finite, so that $A\cap G$ is a $K^{\prime}$--almost
invariant subset of $G$. As $e(G,K^{\prime}:\mathcal{S})=1$, it follows that
$A\cap G$, and hence $A$, is $K$--finite, a contradiction. This completes the
proof of the claim. As $\delta_{\widehat{E}}^{R}A$ and $\delta_{\widehat{E}%
}^{R}B$ are $K$--finite, for all $R$, it is immediate that $\delta_{dW}%
^{R}(A\cap dW)$ and $\delta_{dW}^{R}(B\cap dW)$ are also $K$--finite, for all
$R$. We conclude that each of $p(A\cap dW)$ and $p(B\cap dW)$ contains the
vertex set of precisely one end of $p(dW)$, and that they do not contain the
same end. In particular, by interchanging $A$ and $B$ if needed, we can
arrange that $B\cap\partial^{R}\widehat{Y}$ is $K$--finite, for all $R$.

We claim that $B$ lies in $\mathcal{B}_{\widehat{X}}$. Assuming this claim,
the fact that $B$ and $\widehat{X}-B$ are each $K$--infinite implies that
$\widehat{X}$ has at least two $S$--adapted $K$--ends. This will be a
contradiction as $\widehat{X}$ has only one $S$--adapted $K$--end, by Lemma
\ref{aisethasoneS-adaptedK-end}.

In order to prove the above claim, we need to show that $\delta_{\widehat{X}%
}^{R}B$ is $K$--finite, for all $R$. Note that $\delta_{\widehat{X}}%
^{R}B\subset\delta_{\widehat{E}}^{R}B\cup(\delta_{\widehat{X}}^{R}B\cap
\delta_{\widehat{X}}^{R}\widehat{E})\subset\delta_{\widehat{E}}^{R}%
B\cup(\delta_{\widehat{X}}^{R}B\cap\delta^{R}\widehat{Y})$. We know that
$\delta_{\widehat{E}}^{R}B$ is $K$--finite, for all $R$. Also $B\cap
\partial^{R}\widehat{Y}$ is $K$--finite, for all $R$. As $\partial
^{R}\widehat{Y}\subset G$, and each vertex of $G$ has fixed finite valence in
$\widehat{\Gamma}^{R}$, it follows that $\delta_{\widehat{X}}^{R}B\cap
\delta^{R}\widehat{Y}$ is $K$--finite, for all $R$. Hence $\delta
_{\widehat{X}}^{R}B$ is also $K$--finite, for all $R$. This completes the
proof of the lemma.
\end{proof}

We close this section with the following technical result which will play an
important role in the next section.

\begin{lemma}
\label{intersectionofdisjointhalf-planes}Let $X$ and $Y$ be two disjoint
half-planes. Let $W^{R}$ denote $\widehat{\Gamma}^{R}[\widehat{X}]$, let
$dW^{R}$ denote $\Gamma^{R}[\partial^{R}\widehat{X}]$, let $\Omega^{R}$ denote
$\widehat{\Gamma}^{R}[\widehat{Y}]$, and let $d\Omega^{R}$ denote $\Gamma
^{R}[\partial^{R}\widehat{X}]$. Then the following statements hold:

\begin{enumerate}
\item If $X$ and $Y$ have no common endpoint, then $\widehat{X}\cap
\widehat{Y}$ is $K$--finite.

\item If $X$ and $Y$ have one common endpoint, then either $\widehat{X}%
\cap\widehat{Y}$ is empty, or there is $R$ such that $dW^{R}\cap d\Omega^{R}$
contains $\widehat{X}\cap\widehat{Y}\cap G$ and is $K$--almost equal to a
semi-line whose equivalence class is the common endpoint of $X$ and $Y$.
\end{enumerate}
\end{lemma}

\begin{remark}
If $\widehat{X}\cap\widehat{Y}$ is empty, then $dW^{R}\cap d\Omega^{R}$ is
also empty, for all values of $R$. This is simply because the vertex sets of
the graphs $W^{R}$ and $\Omega^{R}$ are contained in $\widehat{X}$ and
$\widehat{Y}$ respectively.
\end{remark}

\begin{proof}
Recall that if $X$ is a $H$--almost invariant subset of $G$, then
$\widehat{X}$ is the standard augmentation of $X$, and $\widehat{X}\cap G$
contains $X$, and $(\widehat{X}\cap G)-X$ is $H$--finite. In particular,
$\widehat{X}\cap G$ is contained in $N_{R}X$, for some $R$, where $N_{R}X$
denotes the $R$--neighbourhood of $X$ in the angle metric on $G$. While $X$
and $X^{\ast}$ are disjoint, their standard augmentations $\widehat{X}$ and
$\widehat{X^{\ast}}$ need not be. But it follows that $\widehat{X}%
\cap\widehat{X^{\ast}}\cap G$ is contained in $N_{R}X\cap N_{R}X^{\ast}$.
Hence $\widehat{X}\cap\widehat{X^{\ast}}\cap G\subset dW^{R}$.

In general, a half-plane $X$ is an intersection of almost invariant subsets,
and $\widehat{X}$ denotes the intersection of the corresponding standard
augmentations. In particular, it follows that there is $R$ such that
$\widehat{X}\cap G$ is contained in $N_{R}X$. Thus if $X$ and $Y\ $are two
half-planes, there is $R$ such that $\widehat{X}\cap\widehat{Y}\cap G$ is
contained in $N_{R}X\cap N_{R}Y$. In particular, if $X\ $and $Y$ are disjoint,
it follows that $\widehat{X}\cap\widehat{Y}\cap G\subset dW^{R}\cap
d\Omega^{R}$.

Now let $N$ be such that $W^{N}$, $dW^{N}$, $\Omega^{N}$ and $d\Omega^{N}$ are
all connected graphs. As $X\ $and $Y$ are half-planes, each of $dW^{N}$ and
$d\Omega^{N}$ is $K$--almost equal to the union of two semi-lines, so that $X$
and $Y$ each have two endpoints.

1) If $X$ and $Y$ have no common endpoint, Lemma \ref{semi-linemeetsa.i.set}
implies that $dW^{R}\cap d\Omega^{R}$ is $K$--finite, for all $R$. As there is
$R$ such that this intersection contains $\widehat{X}\cap\widehat{Y}\cap G$,
it follows that $\widehat{X}\cap\widehat{Y}\cap G$ is $K$--finite. As each
point of $G$ has finite valence in $\widehat{\Gamma}^{R}$, and every edge of
$\widehat{\Gamma}^{R}$ has at least one vertex in $G$, it follows that
$\widehat{X}\cap\widehat{Y}$ is also $K$--finite, as required.

2) If $X$ and $Y$ have one common endpoint, then $dW^{N}$ and $d\Omega^{N}$
each contain a semi-line, whose equivalence class is that common endpoint. It
follows that either $dW^{R}\cap d\Omega^{R}$ is empty, for all $R$, or there
is $R$ such that the intersection $dW^{R}\cap d\Omega^{R}$ is $K$--almost
equal to such a semi-line. Hence either $\widehat{X}\cap\widehat{Y}$ is empty,
or there is $R$ such that $dW^{R}\cap d\Omega^{R}$ contains $\widehat{X}%
\cap\widehat{Y}\cap G$ and is $K$--almost equal to a semi-line whose
equivalence class is the common endpoint of $X$ and $Y$, as required.
\end{proof}

\subsection{The construction}

Now we are ready to start constructing an action of the group $G(v)$, the
stabilizer of $\Phi$, on the circle $S^{1}$. If $\{a,b\}$ and $\{c,d\}$ are
pairs of points in $S^{1}$, we will say that these pairs are \textit{unlinked}
if there is an interval of $S^{1}$ which contains both points of one pair but
neither of the other pair, and otherwise they are \textit{linked}. We will
construct an embedding of $\mathcal{E}$ into $S^{1}$, with the property that
if $X$ and $Y$ are two almost invariant sets in $\Phi$, which are neither
equal nor complementary, then $X\ $and $Y$ cross if and only if the pair of
points in $S^{1}$ coming from the endpoints of $X$ links the pair of points in
$S^{1}$ coming from the endpoints of $Y$. This property of our embedding
follows from condition 6) in Proposition \ref{canembedpointsincircle}. We need
some notation.

Let $X_{1},\dots,X_{n}$ be a cross-connected subset of the CCC $\Phi$, and let
$\mathcal{B}_{n}$ denote the Boolean algebra generated by the $X_{i}$'s.
Recall that $\widehat{X_{i}}$ denotes the standard augmentation of $X_{i}$ in
the relative Cayley graph $\Gamma(G:\mathcal{S})$, see Definition
\ref{defnofstandardaugmentation}, and that if $P$ is a finite intersection of
$X_{i}$'s or their complements, then $\widehat{P}$ denotes the corresponding
intersection of standard augmentations. Let $W_{i}^{R}$ denote the graph
$\widehat{\Gamma}^{R}[\widehat{X_{i}}]$, and let $dW_{i}^{R}$ denote the graph
$\Gamma^{R}[\partial^{R}\widehat{X_{i}}]$. We will also let $W_{i}^{\ast R}$
denote the graph $\widehat{\Gamma}^{R}[\widehat{X_{i}^{\ast}}]$, and let
$dW_{i}^{\ast R}$ denote the graph $\Gamma^{R}[\partial^{R}\widehat{X_{i}%
^{\ast}}]$. Finally let $W^{R}(P)$ denote the graph $\widehat{\Gamma}%
^{R}[\widehat{P}]$, and let $dW^{R}(P)$ denote the graph $\Gamma^{R}%
[\partial^{R}\widehat{P}]$. Recall from Lemma
\ref{AlmostinvariantsetbecomesconnectedinGL} and Lemma
\ref{R-nbhdofintersectionsisconnectedinrelativecase} that there is $N\ $such
that, for every $R\geq N$, and every $P$, the graphs $W_{i}^{R}$, $dW_{i}^{R}%
$, $W^{R}(P)$ and $dW^{R}(P)$ are connected. Further $dW^{R}(P)$ is
$K$--almost equal to a finite union of semi-lines. For each $i$, let $a_{i}$
and $a_{i}^{\prime}$ denote the two endpoints of $X_{i}$, i.e. the two
equivalence classes of semi-lines determined by $dW_{i}^{N}$, and let
$\mathcal{E}_{n}$ denote the collection of all these endpoints. Suppose we
have an embedding $\varphi:\mathcal{E}_{n}\rightarrow S^{1}$ and, that for
each $i$, we have a labelling by $X_{i}$ and $X_{i}^{\ast}$ of the two
intervals which form $S^{1}-\varphi\{a_{i},a_{i}^{\prime}\}$. Then we label
each intersection of such intervals by the intersection of their labels, which
is an element of $\mathcal{B}_{n}$. In particular, the image of $\varphi$ cuts
$S^{1}$ into $2n$ innermost intervals, each labelled by some element of
$\mathcal{B}_{n}$. For each such interval, we will call its label a
\textit{pie chart}. Thus a pie chart is an intersection of $n$ elements of the
set $\{X_{i},X_{i}^{\ast}:1\leq i\leq n\}$ such that, for each $i$, exactly
one of $X_{i}$ and $X_{i}^{\ast}$ occurs in the intersection. Note that this
description implies that any two distinct pie charts are disjoint subsets of
$G$. We will show that we can choose $\varphi$ and the labeling so that the
picture in $S^{1}$ reflects that in $G$. The precise statement follows.

\begin{proposition}
\label{canembedpointsincircle}Suppose that $n\geq1$, and $X_{1},\dots,X_{n}$
form a cross-connected subset of the CCC $\Phi$. Then there is an embedding
$\varphi:\mathcal{E}_{n}\rightarrow S^{1}$ and, for each $i$, a labeling by
$X_{i}$ and $X_{i}^{\ast}$ of the two intervals which form $S^{1}%
-\varphi\{a_{i},a_{i}^{\prime}\}$, such that the following conditions hold:

\begin{enumerate}
\item If $I$ is an innermost interval in $S^{1}$, the labeling pie chart is a
half-plane, and the two endpoints of the half-plane are mapped by $\varphi$ to
the two endpoints of $I$.

\item Let $I\ $and $J$ be innermost intervals in $S^{1}$ whose labeling pie
charts are $P\ $and $Q$ respectively.

\begin{enumerate}
\item If $I\ $and $J$ are disjoint, then the intersection $\widehat{P}%
\cap\widehat{Q}$ is $K$--finite.

\item If $I\cap J$ is the single point $x$, and if $a$ denotes the common
endpoint of $P$ and $Q$ which is mapped to $x$ by $\varphi$, then either the
intersection $\widehat{P}\cap\widehat{Q}$ is empty or there is $R$ such that
$dW^{R}(P)\cap dW^{R}(Q)$ contains $\widehat{P}\cap\widehat{Q}\cap G$ and is
$K$--almost equal to a semi-line which represents $a$.
\end{enumerate}

\item The union of all pie charts is $K$--almost equal to $G$.

\item If $A$ is a half-plane in $\mathcal{B}_{n}$ with endpoints $a,a^{\prime
}$ in $\mathcal{E}_{n}$, there is an interval $I$ joining $\varphi(a)$ to
$\varphi(a^{\prime})$ in $S^{1}$ such that $A$ is $K$--almost equal to the
union of all the pie charts labeling subintervals of $I$.

\item If $I$ is an interval in $S^{1}$ joining two points of $\varphi
(\mathcal{E}_{n})$, then the union of all the pie charts labeling subintervals
of $I$ is a half-plane.

\item Two half-planes $A\ $and $B$ in $\mathcal{B}_{n}$, with no common
endpoints, cross (Definition \ref{defnofcrossingofhalf-planes}) if and only if
the pair of points in $S^{1}$ coming from the endpoints of $A$ links the pair
of points in $S^{1}$ coming from the endpoints of $B$.
\end{enumerate}

Further, if $X_{1},\dots,X_{n-1}$ form a cross-connected subset of the CCC
$\Phi$, and we have an embedding $\varphi_{n-1}:\mathcal{E}_{n-1}\rightarrow
S^{1}$ and labelings such that the above six conditions hold, then $\varphi$
can be chosen to extend $\varphi_{n-1}$ and the labelings can also be chosen
to extend the given ones.

Finally if we have embeddings $\varphi$ and $\varphi^{\prime}$ of
$\mathcal{E}_{n}$ in $S^{1}$ and labelings satisfying all the above
conditions, then there is a homeomorphism $h:S^{1}\rightarrow S^{1}$ such that
$\varphi^{\prime}=h\circ\varphi$ and $h$ preserves the labelings, i.e. for
each interval $I$ of $S^{1}-\varphi(\mathcal{E}_{n})$, the labeling pie charts
of $I$ and $h(I)$ are the same.
\end{proposition}

\begin{remark}
The last part of the lemma tells us that the embedding $\varphi$ together with
the associated labeling is unique up to homeomorphism of $S^{1}$.
\end{remark}

\begin{proof}
The proof will be by induction on $n$, but we discuss the cases when $n=1,2$
separately. The case when $n=1$ is trivial as one can choose any embedding
$\varphi:\mathcal{E}_{1}\rightarrow S^{1}$, and any labeling of the two
complementary intervals $I\ $and $J$. This will automatically satisfy all six
conditions of the proposition, but condition 2)\ is vacuous as $I\cap J$ is
equal to two points. Finally the uniqueness statement of the proposition is trivial.

Next we consider the case $n=2$, and will show how to extend $\varphi_{1}$ to
an embedding $\varphi:\mathcal{E}_{2}\rightarrow S^{1}$, and how to label the
four complementary intervals. Let $L$ be a semi-line which determines one of
the endpoints of $X_{2}$. Lemma \ref{semi-linemeetsa.i.set} tells us that, for
any $R$, the intersection $L\cap dW_{1}^{R}$ is $K$--finite. As $K\backslash
L$ has one end, it follows that $L\cap W_{1}^{R}$ is $K$--finite or is
$K$--almost equal to $L$. The same argument applies to $L\cap W_{1}^{\ast R}$.
As $W_{1}^{R}\cap W_{1}^{\ast R}$ is empty or lies within bounded Hausdorff
distance of $dW_{1}^{R}$, it follows that $L$ is $K$--almost contained in
exactly one of $W_{1}^{R}$ or $W_{1}^{\ast R}$. The same holds for a semi-line
$L^{\prime}$ which determines the other endpoint of $X_{2}$. If both
semi-lines are $K$--almost contained in $W_{1}^{R}$ (or $W_{1}^{\ast R}$),
this would imply that $dW_{2}^{R}$ itself is $K$--almost contained in
$W_{1}^{R}$ (or $W_{1}^{\ast R}$), so that $X_{2}$ does not cross $X_{1}$. But
this contradicts the hypothesis that the family $\{X_{1},X_{2}\}$ is
cross-connected. Thus one of $L$ and $L^{\prime}$ is $K$--almost contained in
$W_{1}^{R}$ and the other is $K$--almost contained in $W_{1}^{\ast R}$. We
will now extend the embedding of $\mathcal{E}_{1}$ into $S^{1}$ to an
embedding of $\mathcal{E}_{2}$, and will also extend the given labeling. Let
$a_{2}$ and $a_{2}^{\prime}$ denote the endpoints of $X_{2}$ represented by
$W_{1}^{R}\cap dW_{2}^{R}$ and $W_{1}^{\ast R}\cap dW_{2}^{R}$ respectively.
If the interval $I$ is labelled by $X_{1}$, and the interval $J$ is labelled
by $X_{1}^{\ast}$, we embed $a_{2}$ in $I$ and embed $a_{2}^{\prime}$ in $J$.
Now let $L_{1}$ be a semi-line whose equivalence class is $a_{1}$. Arguing as
above shows that $L_{1}$ is $K$--almost contained in exactly one of $W_{2}%
^{R}$ and $W_{2}^{\ast R}$, say $W_{2}^{R}$. We label the interval of
$S^{1}-\varphi\{a_{2},a_{2}^{\prime}\}$ which contains $a_{1}$ by $X_{2}$, and
label the other interval of $S^{1}-\varphi\{a_{2},a_{2}^{\prime}\}$ by
$X_{2}^{\ast}$. We claim that this setup satisfies all of conditions 1)-6).
The four innermost intervals are labeled by one of $X_{1}\cap X_{2}$,
$X_{1}\cap X_{2}^{\ast}$, $X_{1}^{\ast}\cap X_{2}$ and $X_{1}^{\ast}\cap
X_{2}^{\ast}$. Lemmas \ref{aisethasoneS-adaptedK-end} and
\ref{XintersectYisahalf-plane} together show that each is a half plane,
proving condition 1). Recall that distinct pie charts are disjoint subsets of
$G$. Now Lemma \ref{intersectionofdisjointhalf-planes} shows that condition 2)
holds. Now it is easy to check each of the remaining four conditions. Note
that in condition 3), the union of all pie charts is clearly equal to $G$. The
uniqueness statement of the proposition follows because the extension of
$\varphi$ to an embedding of $\mathcal{E}_{2}$ must send the endpoints of
$X_{2}$ into the intervals labeled $X_{1}$ and $X_{1}^{\ast}$ as described,
and the labelings must also be as described.

Now we come to the induction step. Let $n\geq2$, and suppose the result holds
for the cross-connected family $X_{1},\dots,X_{n}$, so that we have an
embedding $\varphi_{n}:\mathcal{E}_{n}\rightarrow S^{1}$ and a labeling of the
complementary intervals which satisfies conditions 1)-6) of the Proposition.
In particular, we have $2n$ pie charts. We consider adding a new almost
invariant set $X_{n+1}$ in the CCC $\Phi$, so that the new family is still
cross-connected. To simplify notation we denote $X_{n+1}$ by $Y$. We further
let $W^{R}(Y)$ denote the graph $\widehat{\Gamma}^{R}[\widehat{Y}]$, and let
$dW^{R}(Y)$ denote the graph $\Gamma^{R}[\partial^{R}\widehat{Y}]$. Let $L$ be
a semi-line which determines one of the endpoints of $Y$. Lemma
\ref{semi-linemeetsa.i.set} tells us that for each $i$, and for any $R$, the
intersection $L\cap dW_{i}^{R}$ is $K$--finite. Now let $P$ be a pie chart
obtained from the family $X_{1},\dots,X_{n}$, let $W^{R}(P)$ denote the graph
$\widehat{\Gamma}^{R}[\widehat{P}]$, and let $dW^{R}(P)$ denote the graph
$\Gamma^{R}[\partial^{R}\widehat{P}]$. Then it follows that the intersection
$L\cap dW^{R}(P)$ is also $K$--finite. Condition 3) tells us that the union of
all the pie charts is $K$--almost equal to $G$. As $K\backslash L$ has one
end, it follows that $L$ is $K$--almost contained in $W^{R}(P)$, for some pie
chart $P$. Part 2) tells us that the intersection of the standard
augmentations of two pie charts with $G$ is $K$--finite or is within bounded
Hausdorff distance of a semi-line which determines an endpoint of some $X_{i}%
$. It follows that there is exactly one pie chart which $K$--almost contains
$L$. Let $P$ and $P^{\prime}$ denote the two pie charts of $\mathcal{B}_{n}$
which $K$--almost contain semi-lines $L$ and $L^{\prime}$ which determine the
two endpoints of $Y$. If $P\ $and $P^{\prime}$ were equal, this would imply
that $\partial^{R}\widehat{Y}$ itself is $K$--almost contained in $P$, so that
$Y$ does not cross any $X_{i}$. But this contradicts the hypothesis that the
family $X_{1},\dots,X_{n},Y$ is cross-connected. We conclude that $P\ $and
$P^{\prime}$ are not equal.

We denote by $I\ $and $I^{\prime}$ the innermost intervals in $S^{1}$ labelled
by the pie charts $P$ and $P^{\prime}$. We will extend $\varphi_{n}$ in the
obvious way so as to map the two endpoints of $Y$, one into $I\ $and the other
into $I^{\prime}$. It remains to decide how to extend the given labeling of
the complementary intervals in $S^{1}$.

The complement of the interiors of $I$ and $I^{\prime}$ has two components $U$
and $V$ each being an interval, unless $I$ and $I^{\prime}$ are adjacent, in
which case $U$ or $V$ is a point. Note that each of $U\ $and $V$ is a union of
innermost intervals, or is a point. Now consider an innermost interval $J$,
not equal to $I$ or $I^{\prime}$, and let $Q$ denote the pie chart which
labels $J$. Thus $Q$ is not equal to $P$ or $P^{\prime}$, so that
$dW^{R}(Y)\cap W^{R}(Q)$ is $K$--finite, for all $R$. It follows that
$\widehat{Y}\cap\widehat{Q}$ represents a $\mathcal{S}$--adapted $K$--end of
$\widehat{Q}$, the standard augmentation of $Q$. As $Q$ is a half-plane, by
condition 1), the number of $\mathcal{S}$--adapted $K$--ends of $\widehat{Q}$
is equal to $1$. It follows that $W^{R}(Q)$ must be $K$--almost contained in
$W^{R}(Y)$ or $W^{R}(Y^{\ast})$. In particular, if $u$ and $v$ denote the
endpoints of the interval $J$, then two semi-lines which determine $u$ and $v$
are both $K$--almost contained in $W^{R}(Y)$ or are both $K$--almost contained
in $W^{R}(Y^{\ast})$. As this holds for any innermost interval which is not
equal to $I$ or $I^{\prime}$, we conclude that all the semi-lines which
determine points of $\varphi(\mathcal{E}_{n})\cap U$ are $K$--almost contained
in $W^{R}(Y)$ or $K$--almost contained in $W^{R}(Y^{\ast})$, and the same
applies to semi-lines which determine points of $\varphi(\mathcal{E}_{n})\cap
V$. Since $Y$ and some $X_{i}$ cross, we cannot have every semi-line
$K$--almost contained in $W^{R}(Y)$, nor can we have every semi-line
$K$--almost contained in $W^{R}(Y^{\ast})$. We conclude that, after
interchanging $U$ and $V$ if needed, semi-lines which determine points of
$\varphi(\mathcal{E}_{n})\cap U$ are $K$--almost contained in $W^{R}(Y)$, and
semi-lines which determine points of $\varphi(\mathcal{E}_{n})\cap V$ are
$K$--almost contained in $W^{R}(Y^{\ast})$. It follows that pie charts which
label innermost intervals contained in $U$ are $K$--almost contained in $Y$,
and pie charts which label innermost intervals contained in $V$ are
$K$--almost contained in $Y^{\ast}$.

Now we can extend the given labeling of complementary intervals in $S^{1}$.
Let $a_{n+1}$ and $a_{n+1}^{\prime}$ denote the endpoints of $Y$ represented
by $P\cap N_{R}\partial Y$ and $P^{\prime}\cap N_{R}\partial Y$ respectively.
We embed $a_{n+1}$ in $I$ and embed $a_{n+1}^{\prime}$ in $I^{\prime}$. Of the
two intervals which form $S^{1}-\varphi\{a_{n+1},a_{n+1}^{\prime}\}$, one
contains $U$ and the other contains $V$. We label the one which contains $U$
by $Y$, and label the one which contains $V$ by $Y^{\ast}$.

We verify that this setup satisfies all the conditions 1)-6) of the proposition.

1) This condition follows from our induction assumption and Lemma
\ref{XintersectYisahalf-plane} except for the new innermost intervals. If $Q$
is one of the original $2n$ pie charts, other than $P\ $and $P^{\prime}$, it
is replaced either by $Q\cap Y$ or by $Q\cap Y^{\ast}$, depending on whether
$Q$ is $K$--almost contained in $Y$ or in $Y^{\ast}$. The intervals $I\ $and
$I^{\prime}$ are each the union of two new innermost intervals. The innermost
intervals contained in $I$ are labeled by $P\cap Y$ and $P\cap Y^{\ast}$, and
the innermost intervals contained in $I^{\prime}$ are labeled by $P^{^{\prime
}}\cap Y$ and $P^{^{\prime}}\cap Y^{\ast}$. These labels are four new pie
charts which replace $P\ $and $P^{\prime}$. We know that $Y$ and $P$ are
half-planes which cross, and the same applies to $Y$ and $P^{\prime}$. Hence
Lemma \ref{XintersectYisahalf-plane} tells us that each of the four new pie
charts is a half-plane. Further the two endpoints of one of the new pie charts
are mapped by $\varphi$ to the two endpoints of the interval labelled by that
pie chart. This proves 1).

2) Let $I\ $and $J$ be distinct innermost intervals in $S^{1}$ whose labeling
pie charts are $P\ $and $Q$ respectively. As noted just before the statement
of this theorem, distinct pie charts are disjoint, so that $P$ and $Q$ are
disjoint half-planes. Now conditions 2a)\ and 2b) follow immediately from
Lemma \ref{intersectionofdisjointhalf-planes}.

3) Recall that, by our induction assumption, the $2n$ pie charts associated to
the given embedding $\varphi_{n}$ have union $K$--almost equal to $G$. The
innermost intervals associated to $\varphi_{n}$ are unchanged when we extend
$\varphi_{n}$ to $\varphi$ apart from the fact that $I\ $and $I^{\prime}$ are
each subdivided into two new innermost intervals. The pie charts labeling
these new intervals are $P\cap Y$, $P\cap Y^{\ast}$, $P^{^{\prime}}\cap Y$ and
$P^{^{\prime}}\cap Y^{\ast}$. If $Q$ is one of the original $2n$ pie charts,
other than $P\ $and $P^{\prime}$, it is replaced either by $Q\cap Y$ or by
$Q\cap Y^{\ast}$, depending on whether $Q$ is $K$--almost contained in $Y$ or
in $Y^{\ast}$. It follows that the union of the $2(n+1)$ new pie charts is
$K$--almost equal to $G$, as required. Note that although, this union is equal
to $G$, when $n=2,$ it is unlikely to be equal to $G$ for any $n>2$.

4) Consider a half-plane $A\in\mathcal{B}_{n+1}$, and denote by $a,a^{\prime}$
the endpoints of $A$. Let $x=\varphi(a)$ and $x^{\prime}=\varphi(a^{\prime})$,
and let $J_{1}$ and $J_{2}$ denote the two subintervals of $S^{1}$ bounded by
$x$ and $x^{\prime}$. If $P$ is any pie chart, and so a half-plane, the two
endpoints of $P$ cannot link $x$ and $x^{\prime}$. Now, arguing exactly as in
our construction of $\varphi$ from $\varphi_{n}$, It follows that $P$ is
$K$--almost contained in $A$ or in $A^{\ast}$. Further the pie charts labeling
innermost subintervals of $J_{1}$ are all $K$--almost contained in $A$ or are
all $K$--almost contained in $A^{\ast}$, and the reverse must apply to $J_{2}%
$. Now choose notation so that the pie charts labeling innermost subintervals
of $J_{1}$ are all $K$--almost contained in $A$, and the remaining pie charts
are all $K$--almost contained in $A^{\ast}$. As the union of all pie charts is
$K$--almost equal to $G$, it follows that $A$ is $K$--almost equal to the
union of the pie charts labeling innermost subintervals of $J_{1}$, as required.

5) Let $J$ be an interval in $S^{1}$ joining two points $a$ and $a^{\prime}$
of $\varphi(\mathcal{E}_{n})$. We have to prove that the union $A$ of all the
pie charts labeling subintervals of $J$ is a half-plane. If $J$ is a union of
the original $2n$ intervals, condition 5) of our induction assumption tells us
that the union of all the original pie charts labeling subintervals of $J$ is
a half-plane. The new pie charts are obtained from the old by intersection
with $Y\ $or with $Y^{\ast}$, so that the new pie charts are $K$--almost equal
to the old ones. It follows that $A$ is $K$--almost equal to a half-plane, and
hence is actually a half-plane, as required. Otherwise $J$ contains exactly
one of the two intervals into which $I$ is cut by $a_{n+1}$, or $J$ contains
exactly one of the two intervals into which $I^{\prime}$ is cut by
$a_{n+1}^{\prime}$. Thus one of the endpoints of $J$ is $a_{n+1}$ or
$a_{n+1}^{\prime}$. Suppose one endpoint of $J$ is $a_{n+1}$ and that $J\cap
I^{\prime}$ is empty. Suppose the subinterval $L$ of $I$ contained in $J$ is
labeled by $P\cap Y$. The preceding argument shows that the union $B$ of all
the pie charts labeling subintervals of $J-L$ is a half-plane, and that $B$ is
$K$--almost contained in $Y$. Also the union of all the pie charts labeling
subintervals of $J\cup I$ is a half-plane. This union is clearly equal to
$B\cup P$. It follows that $A$ is $K$--almost equal to $(B\cup P)\cap Y$. Now
Lemma \ref{XintersectYisahalf-plane} shows that this is a half-plane, as it is
the intersection of two half-planes which cross. A similar analysis applies in
the other possible cases.

6) Let $A$ and $B$ be two half-planes corresponding to segments $[a,a^{\prime
}],[b,b^{\prime}]$ of the circle (where $a,a^{\prime},b,b^{\prime}$ are all
distinct). If the segments are disjoint, so that $\{a,a^{\prime}\}$ and
$\{b,b^{\prime}\}$ are not linked, then conditions 4) and 2a) together imply
that $\widehat{A}\cap\widehat{B}$ is $K$--finite, so that $A$ and $B$ do not
cross. If the segments overlap, so that $\{a,a^{\prime}\}$ and $\{b,b^{\prime
}\}$ are linked, we may assume $b\in\lbrack a,a^{\prime}]$ and $a\in\lbrack
b,b^{\prime}]$. Now condition 4) shows that a semi-line corresponding to $b$
is $K$--almost contained in $A$, a semi-line corresponding to $b^{\prime}$ is
$K$--almost contained in $A^{\ast}$, a semi-line corresponding to $a$ is
$K$--almost contained in $B$, and a semi-line corresponding to $a^{\prime}$ is
$K$--almost contained in $B^{\ast}$. It follows that $A$ and $B$ cross.

Finally the required uniqueness statement again follows as no other way of
extending $\varphi$ and of labeling the intervals is possible.

This completes the proof of Proposition \ref{canembedpointsincircle}.
\end{proof}

The above proposition deals with a finite cross-connected collection of almost
invariant sets. The next step is to extend this embedding result to the entire
family of almost invariant sets in the CCC $\Phi$. As $\Phi$ is
cross-connected, we can re-order this countable collection of sets
$\{X_{i}\}_{i\geq1}$ so that the subset $\{X_{1},\ldots,X_{n}\}$ is
cross-connected, for each $n\geq1$. Let $\mathcal{E}$ denote the collection of
all endpoints of $\Phi$. Proposition \ref{canembedpointsincircle} tells us we
can embed $\mathcal{E}_{n+1}$ into $S^{1}$, extending a given embedding of
$\mathcal{E}_{n}$. Thus there is an embedding $\varphi$ of $\mathcal{E}$ into
$S^{1}$, and a labelling of the complementary intervals of $\varphi
(\mathcal{E}_{n})$ such that conditions 1)-6) all hold for the embedding of
$\mathcal{E}_{n}$, for each $n\geq1$. Note that the group $G(v)$ permutes the
$X_{i}$'s and so acts naturally on $\mathcal{E}$. We need to show that this
action can be extended to an action on $S^{1}$. In order to do that we need
some more results about the crossing of elements of the CCC $\Phi$. Recall
that $\mathcal{X}=\{X_{\lambda}\}_{\lambda\in\Lambda}$ is a finite family of
nontrivial almost invariant subsets of $G$, such that each $X_{\lambda}$ is
over a $VPC(n+1)$ subgroup and is adapted to $\mathcal{S}$, that
$E(\mathcal{X})$ denote the set of all translates of the $X_{\lambda}$'s and
their complements, and that $\Phi$ is a cross-connected component of
$E(\mathcal{X})$.

\begin{lemma}
\label{numbercrossingtwoaxesisfinite}Let $X$ be a $H_{X}$--almost invariant
subset of $G$, and let $Y$ be a $H_{Y}$--almost invariant subset of $G$, such
that $X$ and $Y$ lie in $\Phi$, and $X$ is not equal to $Y$ or to $Y^{\ast}$.
Then the set of all elements of $E(\mathcal{X})$ which cross both $X$ and $Y$
is finite.
\end{lemma}

\begin{proof}
Consider $X_{\lambda}$ in the finite family $\mathcal{X}$. Thus $X_{\lambda}$
is a $H_{\lambda}$--almost invariant subset of $G$. As $H_{X}$ and
$H_{\lambda}$ are both finitely generated, Lemma 2.7 of \cite{Scott symmint}
tells us that the intersection number of $H_{X}\backslash X$ and $H_{\lambda
}\backslash X_{\lambda}$ is finite. This implies that the set of translates of
$X_{\lambda}$ which cross $X$ is $H_{X}$--finite. Similarly the set of
translates of $X_{\lambda}$ which cross $Y$ is $H_{Y}$--finite. It follows
that the set of translates of $X_{\lambda}$ which cross both $X$ and $Y$ is
$(H_{X}\cap H_{Y})$--finite. As this holds for each $X_{\lambda}$ in the
finite family $\mathcal{X}$, and as any element of $E(\mathcal{X})$ is a
translate of some $X_{\lambda}$, it follows that the set of all elements of
$E(\mathcal{X})$ which cross both $X$ and $Y$ is $(H_{X}\cap H_{Y})$--finite.
Next recall that $H_{X}\cap H_{Y}$ is commensurable with $K$, and so is
$K$--finite. As $K$ stabilizes every element of $\Phi$, it follows that the
set of all elements of $E(\mathcal{X})$ which cross both $X$ and $Y$ is
finite, as required.
\end{proof}

\begin{lemma}
\label{gYandYdonotcross}Let $X$ be a $H_{X}$--almost invariant subset of $G$,
and let $Y$ be a $H_{Y}$--almost invariant subset of $G$, such that $X$ and
$Y$ lie in $\Phi$. Suppose that $X$ and $Y$ cross, and let $g$ be an element
of $H_{X}$, which fixes the endpoints of $X$. Then $gY$ and $Y$ do not cross.
\end{lemma}

\begin{remark}
As $H_{X}$ preserves the two endpoints of $X$, either it fixes them or they
are fixed by a subgroup of $H_{X}$ of index two.
\end{remark}

\begin{proof}
Let $K$ denote the intersection $H_{X}\cap H_{Y}$.

If $g$ lies in $K$, then $gY=Y$, so the result is trivial. Thus we suppose
this is not the case. We consider the image in $S^{1}$ under $\varphi$ of the
endpoints of $X$ and $Y$, as shown in Figure 1(a). Thus $P\ $and $Q$ are the
images of the endpoints of $X$, and $R$ and $S$ are the images of the
endpoints of $Y$. Further $P$ and $Q$ divide $S^{1}$ into two intervals
labeled $X$ and $X^{\ast}$, and $R$ and $S$ divide $S^{1}$ into two intervals
labeled $Y$ and $Y^{\ast}$. Together these four points divide $S^{1}$ into
four intervals, labeled by the four corners of the pair $(X,Y)$, and each of
the endpoints of $gY$ must lie in one of these intervals. As $Y$ crosses $X$
so does $gY$, so the endpoints of $gY$ must lie one in the interval labeled
$X$ and the other in the interval labeled $X^{\ast}$. If they lie in the
intervals labeled $X\cap Y$ and $X^{\ast}\cap Y$, then $gY$ does not cross
$Y$, and the same applies if they lie in the intervals labeled $X\cap Y^{\ast
}$ and $X^{\ast}\cap Y^{\ast}$.

\begin{figure}[h]
\centering
\begin{subfigure}[b]{0.45\textwidth} 
\centering
\resizebox{\textwidth}{!}{
\begin{tikzpicture}
\newcommand{\outrad}{2cm}
\newcommand{\inrad}{0.4cm}
\draw (0,0) circle (\outrad);
\draw [thick] (0,-\outrad) --  (0,\outrad);
\draw [thick] (-\outrad,0) --  (\outrad,0);
\draw[fill] (\outrad,0) circle (1pt);
\node [right] at (\outrad,0) {\scriptsize R};
\draw[fill] (-\outrad,0) circle (1pt);
\node [left] at (-\outrad,0) {\scriptsize S};
\draw[fill] (0,\outrad) circle (1pt);
\node [above] at (0,\outrad) {\scriptsize P};
\draw[fill] (0,-\outrad) circle (1pt);
\node [below] at (0,-\outrad) {\scriptsize Q};
\node [right ] at (45:\outrad) {\scriptsize$X\cap Y$};
\node [left ] at (135:\outrad) {\scriptsize$X^{\ast}\cap Y$};
\node [right ] at (-45:\outrad) {\scriptsize$X\cap Y^{\ast}$};
\node [left ] at (-135:\outrad) {\scriptsize$X^{\ast}\cap Y^{\ast}$};
\end{tikzpicture}
} 
\caption{}
\end{subfigure}
\hfill
\begin{subfigure}[b]{0.45\textwidth} 
\centering
\resizebox{\textwidth}{!}{
\begin{tikzpicture}
\newcommand{\outrad}{2cm}
\newcommand{\inrad}{0.4cm}
\draw (0,0) circle (\outrad);
\draw [thick] (0,-\outrad) --  (0,\outrad);
\draw [thick] (-\outrad,0) --  (\outrad,0);
\draw[fill] (\outrad,0) circle (1pt);
\node [right] at (\outrad,0) {\scriptsize R};
\draw[fill] (-\outrad,0) circle (1pt);
\node [left] at (-\outrad,0) {\scriptsize S};
\draw[fill] (0,\outrad) circle (1pt);
\node [above] at (0,\outrad) {\scriptsize P};
\draw[fill] (0,-\outrad) circle (1pt);
\node [below] at (0,-\outrad) {\scriptsize Q};
\node [right ] at (45:\outrad) {\scriptsize$X\cap Y$};
\node [left ] at (135:\outrad) {\scriptsize$X^{\ast}\cap Y$};
\node [right ] at (-45:\outrad) {\scriptsize$X\cap Y^{\ast}$};
\node [left ] at (-135:\outrad) {\scriptsize$X^{\ast}\cap Y^{\ast}$};
\draw [thick] (-135:\outrad) --  (25:\outrad);
\draw[fill] (55:\outrad) circle (1pt);
\node [above] at (55:\outrad) {\scriptsize$g^{2}R$};
\draw[fill] (-115:\outrad) circle (1pt);
\node [below] at (-115:\outrad) {\scriptsize$g^{2}S$};
\draw[fill] (-135:\outrad) circle (1pt);
\node [right] at (-135:\outrad) {\scriptsize$gS$};
\draw[fill] (25:\outrad) circle (1pt);
\node [right] at (25:\outrad) {\scriptsize$gR$};
\end{tikzpicture}
} 
\caption{}
\end{subfigure}
\caption{}%
\label{fig1}%
\end{figure}

Now we will suppose that $gY$ crosses $Y$ and will obtain a contradiction. The
preceding discussion shows that the endpoints of $gY$ must either lie in the
intervals labeled $X\cap Y$ and $X^{\ast}\cap Y^{\ast}$, or in the intervals
labeled $X\cap Y^{\ast}$ and $X^{\ast}\cap Y$. By interchanging $Y$ and
$Y^{\ast}$ if needed we can arrange that they lie in the intervals labeled
$X\cap Y$ and $X^{\ast}\cap Y^{\ast}$. As $gX=X$, and $R$ lies in the interval
labeled $X$, it follows that $gR$ lies in the interval labeled $X\cap Y$, and
$gS$ lies in the interval labeled $X^{\ast}\cap Y^{\ast}$, as shown in Figure
1(b). As $P$ lies in the interval labeled $Y$, and $g$ preserves $P$, it
follows that $P$ lies in the interval labeled $gY$. Next we consider $g^{2}Y$.
As $gR$ lies in the interval labeled $X\cap Y$, we have $g^{2}R$ lies in the
interval labeled $X\cap gY$, and similarly $g^{2}S$ lies in the interval
$X^{\ast}\cap gY^{\ast}$, as shown in Figure 1(b). It follows that $g^{2}Y$
also crosses $X$ and $Y$. By repeating this argument, we will find that for
every $n\geq1$, the translate $g^{n}Y$ crosses both $X\ $and $Y$. But this
contradicts Lemma \ref{numbercrossingtwoaxesisfinite}, showing that $gY$ and
$Y\ $cannot cross.
\end{proof}

The proof of Lemma \ref{gYandYdonotcross} shows that if $g$ lies in $H_{X}-K$,
and fixes the endpoints of $X$, then either the images of the endpoints of
$gY$ lie in the intervals labeled $X\cap Y$ and $X^{\ast}\cap Y$, or they lie
in the intervals labeled $X\cap Y^{\ast}$ and $X^{\ast}\cap Y^{\ast}$. In the
first case, we will say that $g$ moves $Y$ towards $P$, and in the second case
that $g$ moves $Y$ towards $Q$. Note that if $g$ moves $Y$ towards $P$, then
$g^{-1}$ moves $Y$ towards $Q$.

\begin{lemma}
\label{anyintervalcontainsanotherpoint}Let $P\ $and $Q$ be distinct points in
$\varphi(\mathcal{E})$, and let $I$ denote one of the two intervals in $S^{1}$
with endpoints $P$ and $Q$. Then there is a point of $\varphi(\mathcal{E})$ in
the interior of $I$.
\end{lemma}

\begin{proof}
Recall that $\mathcal{E}$ is the set of endpoints of $\Phi$. Thus $P$ is the
image of an endpoint of an almost invariant set $X$ in $\Phi$, and $Q$ is the
image of an endpoint of an almost invariant set $Y$ in $\Phi$.

If $X$ and $Y$ are equal, the result is trivial, as we are assuming that
$\Phi$ is a non-isolated CCC, so there is some almost invariant set in $\Phi$
which crosses $X$, and its endpoints will lie one in each of the two intervals
with endpoints $P$ and $Q$.

If $X$ and $Y$ cross, we let $g$ be an element of $H_{X}-K$, which fixes the
endpoints of $X$. As discussed just before the statement of this lemma, either
$g$ or $g^{-1}$ must move $Y$ towards $P$. In particular, either $g$ or
$g^{-1}$ moves $Q$ into the interior of the interval $I$, as required.

Now suppose that $X$ and $Y$ do not cross. Suppose there is an almost
invariant set $Z$ in $\Phi$ such that $Z$ crosses both $X\ $and $Y$, and
consider the translates $g^{i}Z$ of $Z$. They all cross $X$, so Lemma
\ref{numbercrossingtwoaxesisfinite} tells us that only finitely many of them
can cross $Y$. By replacing $g$ by $g^{-1}$, we can assume that $g$ moves $Z$
towards $P$. There is $n\geq1$ such that $g^{n}Z$ does not cross $Y$. As
$g^{n}$ moves $Z$ towards $P$, it follows that $g^{n}$ moves one endpoint of
$Z$ into the interval $I$, as required.

In general, there need not be $Z$ which crosses both $X$ and $Y$, but as
$\Phi$ is cross-connected, there must be a sequence $Z_{1},\ldots,Z_{m}$ such
that $Z_{1}$ crosses $X$, $Z_{m}$ crosses $Y$, and $Z_{i}$ crosses $Z_{i+1}$,
for $1\leq i\leq m-1$. Again we consider $g$ in $H_{X}-K$, which fixes the
endpoints of $X$, such that $g$ moves $Z_{1}$ towards $P$. Only finitely many
of the translates $g^{n}Z_{1}$ can cross any of $Z_{2},\ldots,Z_{m}$ or $Y$.
It follows that there is $n\geq1$ such that $g^{n}$ moves one endpoint of
$Z_{1}$ into the interval $I$, as required.
\end{proof}

\begin{theorem}
Let $(G,\mathcal{S})$ be a group system of finite type such that $G$ is
finitely generated and $e(G:\mathcal{S})=1$. Suppose that $n\geq0$, and that
for any $VPC(\leq n)$ subgroup $M$ of $G$, we have $e(G,M:\mathcal{S})=1$.
Further suppose that $G$ is not $VPC(n+2)$. Let $\{X_{\lambda}\}_{\lambda
\in\Lambda}$ be a finite family of nontrivial $\mathcal{S}$--adapted almost
invariant subsets of $G$, such that each $X_{\lambda}$ is over a $VPC(n+1)$
subgroup. Let $E(\mathcal{X})$ denote the set of all translates of the
$X_{\lambda}$'s and their complements, and let $\Phi$ be a non-isolated CCC of
$E(\mathcal{X})$ in which all crossings are strong.

Then there is an embedding $\varphi$ of $\mathcal{E}$ into $S^{1}$, and a
labelling of the complementary intervals such that conditions 1)-6) all hold
for the embedding of $\mathcal{E}_{n}$, for each $n\geq1$. Further $\varphi$
can be chosen so that $\varphi(\mathcal{E})$ is dense in $S^{1}$, and the
natural action of $G(v)$ on $\varphi(\mathcal{E})$ has a unique extension to
an action of $G(v)$ on $S^{1}$ by homeomorphisms.
\end{theorem}

\begin{proof}
The natural action of $G(v)$ on $\mathcal{E}$ induces an action on
$\varphi(\mathcal{E})$ which we want to extend to an action of $G(v)$ on
$S^{1}$. We will choose $\varphi$ so that whenever we embed a new pair of
endpoints, as in the construction in the proof of Proposition
\ref{canembedpointsincircle}, we send each endpoint to the midpoint of the
innermost interval to which we map it. Now Lemma
\ref{anyintervalcontainsanotherpoint} implies that $\varphi(\mathcal{E})$ must
be dense in $S^{1}$, so that each element of $G(v)$ has at most one extension
to a homeomorphism of $S^{1}$. We will construct such an extension. We order
the elements of $\Phi$ so that the subset $\{X_{1},\ldots,X_{n}\}$ is
cross-connected, for each $n\geq1$. Consider the action of an element $g$ of
$G(v)$ on the set $\{X_{1},\ldots,X_{n}\}$. Recall the uniqueness part of
Proposition \ref{canembedpointsincircle}. This tells us that there is a
homeomorphism $h_{n}$ of the circle such that, for each endpoint $x$ of some
$X_{i}$, with $1\leq i\leq n$, we have $\varphi(gx)=h_{n}\varphi(x)$, and
$h_{n}$ preserves labels. Now we consider the limit as $n\rightarrow\infty$ of
these homeomorphisms. For every $k\geq n$, the homeomorphisms $h_{k}$ and
$h_{n}$ agree on $\varphi(\mathcal{E}_{n})$. It follows that the $h_{n}$'s
converge to a homeomorphism on the set $\varphi(\mathcal{E})$, and hence
converge to a homeomorphism of $S^{1}$. Now the uniqueness of these extensions
shows that the map from $G(v)$ to $Homeo(S^{1})$ constructed in this way, must
be a homomorphism, so that we have an action of $G(v)$ on $S^{1}$. This
completes the proof of the theorem.
\end{proof}

Next we need to show, as in Dunwoody-Swenson \cite{D-Swenson}, that the kernel
of the action of $G(v)$ on $S^{1}$ lies in $[K]$.

\begin{lemma}
\label{stabilizeroftwoaxesequalsK}Let $X$ and $Y$ be distinct and
non-complementary elements of the CCC $\Phi$. Then the following hold:

\begin{enumerate}
\item Let $Z$ be an element of the CCC $\Phi$, and let $h$ be an element of
$H_{X}\cap H_{Y}$. Then $hZ=Z$.

\item $H_{X}\cap H_{Y}$ is equal to the kernel of the action of $G(v)$ on
$S^{1}$.
\end{enumerate}
\end{lemma}

\begin{proof}
1) If $Z$ is equal to $X$ or $Y$ or one of their complements, the result is
trivial, so we will assume that $Z$ is not equal to $X$, $Y$, $X^{\ast}$ or
$Y^{\ast}$.

The given element $h$ of $H_{X}\cap H_{Y}$ fixes each of the points on $S^{1}$
corresponding to the endpoints of $X$ and $Y$, and hence preserves each of the
four intervals into which these four points divide $S^{1}$. If $e$ is an
endpoint of $Z$, then $e$ lies in one of these four intervals, say $J$. We
will show that $h$ must fix $e$. Now $h$ induces a homeomorphism of $J$ fixing
the endpoints. So if $he\neq e$, then $h^{k}e$ also cannot equal $e$, for any
value of $k>0$. But, as $\Phi$ is cross-connected, Lemma
\ref{allintersectionsubgroupsarecommensurable} implies that $H_{X}\cap H_{Z}$
and $H_{X}\cap H_{Y}$ are commensurable, so there is $k>0$ such that $h^{k}$
lies in $H_{X}\cap H_{Z}$. In particular, $h^{k}$ must preserve $Z$, so that
$h^{k}$ fixes $e$. This contradiction shows that $h$ itself must fix $e$, as
required. Similarly $h$ must fix the other endpoint of $Z$, so that $hZ=Z$.
This completes the proof of part 1) of the lemma.

2) Part 1) shows that each element of $H_{X}\cap H_{Y}$ fixes each of those
points of $S^{1}$ which correspond to endpoints of elements of the CCC $\Phi$.
As this set is dense in $S^{1}$, it follows that each element of $H_{X}\cap
H_{Y}$ fixes every point of $S^{1}$. Thus $H_{X}\cap H_{Y}$ is contained in
the kernel of the action of $G(v)$ on $S^{1}$. Conversely each element of this
kernel must preserve each element of the CCC $\Phi$, and so, in particular,
must preserve $X$ and $Y$, so that this kernel is contained in $H_{X}\cap
H_{Y}$. We conclude that $H_{X}\cap H_{Y}$ is equal to the kernel of the
action of $G(v)$ on $S^{1}$, thus completing the proof of part 2) of the lemma.
\end{proof}

In what follows, it will be convenient to denote the kernel of the action of
$G(v)$ on $S^{1}$ by $K$. If the action of $G(v)/K$ on $S^{1}$ is orientation
preserving, we will show that this action is a convergence group action using
Swenson's criteria in \cite{Swenson} as corrected in \cite{Swenson-correction}%
. If this action reverses orientation, we apply Swenson's criteria to the
orientation preserving subgroup of $G(v)/K$ of index $2$. For simplicity in
what follows, we will simply assume that $G(v)/K$ acts on $S^{1}$ preserving
orientation. We use the following standard terminology for actions of groups
on $S^{1}$. A pair of distinct points of $S^{1}$ is called an \textit{axis}.
Two axes $\{a,a^{\prime}\}$ and $\{b,b^{\prime}\}$ \textit{cross} if
$\{a,a^{\prime}\}$ and $\{b,b^{\prime}\}$ are linked. Recall that
$\mathcal{E}$ is the set of endpoints of elements of the CCC $\Phi$, and that
we have constructed an embedding of $\mathcal{E}$ into $S^{1}$. Thus to each
element $Z$ of the CCC $\Phi$, we have associated a pair of distinct points of
$S^{1}$, which are stabilised by the $VPC1$ group $H_{Z}/K$. We let
$\mathcal{A}$ denote the set of all axes obtained from elements of the CCC
$\Phi$. The action of $G(v)/K$ on $S^{1}$ induces an action on $\mathcal{A}$.

\begin{lemma}
\label{Swensonconditions}(Swenson \cite{Swenson, Swenson-correction}) Suppose
that the group $G(v)/K$ acts on $S^{1}$ preserving orientation. This action is
a convergence group action if the following hold:

\begin{enumerate}
\item The set of all points in all the axes in $\mathcal{A}$ is dense in
$S^{1}$.

\item For distinct axes $S$ and $T$ in $\mathcal{A}$, there are only finitely
many axes $R$ in $\mathcal{A}$ such that $R$ crosses both $S$ and $T$.

\item The set $\mathcal{A}$ is a discrete subset of $S^{1}\times S^{1}-\Delta
$, where $\Delta$ denotes the diagonal of $S^{1}\times S^{1}$, and is closed
in the set of all axes.

\item For each axis $A$ in $\mathcal{A}$, the stabilizer of $A$ in $G(v)/K$ is
infinite cyclic and acts as a convergence group on $S^{1}$.
\end{enumerate}
\end{lemma}

We will verify that these conditions hold, but first we need to give the
definition of a convergence group action.

\begin{definition}
\label{defnofconvergencegroupaction} If a group $G$ acts on a space $X$, this
action is a \emph{convergence group action} if for any sequence of distinct
elements of $G$, there is a subsequence $\{g_{i}\}$ and points $P$ and $N$ in
$X$ such that $g_{i}(x)$ tends to $P$ uniformly on compact subsets of
$X-\{N\}$ and $g_{i}^{-1}(x)$ tends to $N$ uniformly on compact subsets of
$X-\{P\}$.
\end{definition}

\begin{lemma}
If the group $G(v)/K$ acts on $S^{1}$ preserving orientation, this action is a
convergence group action.
\end{lemma}

\begin{proof}
We verify the four conditions in Lemma \ref{Swensonconditions}.

1) This condition is part of our construction of the action of $G(v)$ on
$S^{1}$.

2) This result is proved in Lemma \ref{numbercrossingtwoaxesisfinite}.

3) Suppose that $R_{i}$ is a sequence of distinct axes in $\mathcal{A}$ which
converge to some axis $R$ (which may or may not be in $\mathcal{A}$). We will
obtain a contradiction. This will imply both that $\mathcal{A}$ is discrete
and that $\mathcal{A}$ is closed in the set of all axes.

If no axis in $\mathcal{A}$ crosses $R$, the fact that $\mathcal{A}$ is
cross-connected implies that all axes in $\mathcal{A}$ lie on one side of $R$.
But this contradicts the fact that our construction of the action of $G(v)$ on
$S^{1}$ makes the set of all points in all the axes in $\mathcal{A}$ dense in
$S^{1}$. We conclude there must be an axis $P$ in $\mathcal{A}$ which crosses
$R$. Hence there is $n$ such that $P$ crosses each $R_{i}$, with $i>n$. If
there were a second axis $P^{\prime}$ in $\mathcal{A}$ which crosses $R$,
there would be $n^{\prime}$ such that $P^{\prime}$ crosses each $R_{i}$, with
$i>n^{\prime}$. But this would contradict part 2) of this lemma, as $P\ $and
$P^{\prime}$ would be distinct axes in $\mathcal{A}$ which are crossed by
infinitely many axes in $\mathcal{A}$. Hence $P$ is the only axis in
$\mathcal{A}$ which crosses $R$. As $\mathcal{A}$ is cross-connected, part 1)
implies there are axes $B$ and $B^{\prime}$ in $\mathcal{A}$ which cross $P$,
do not cross $R$ and are separated by $R$. Thus there is $m$ such that each
$R_{i}$, with $i>m$, crosses $P$, does not cross $B$ or $B^{\prime}$ and
separates $B$ from $B^{\prime}$. Thus there are infinitely many $R_{i}$'s
which cross $P$ and lie between $B$ and $B^{\prime}$. Let $X$, $X^{\prime}$
and $Y_{i}$ denote elements of the CCC $\Phi$ corresponding to $B$,
$B^{\prime}$ and $R_{i}$, respectively. Thus, after replacing some of these
almost invariant sets by their complements if needed, we have infinitely many
$Y_{i}$'s such that $X\subset Y_{i}\subset X^{\prime}$. (Note that these
inclusions follow from the facts that no two of these sets can cross, and that
the elements of $\Phi$ are in very good position.) But this contradicts Lemma
1.15 of \cite{SS1}. This contradiction completes the proof of part 3).

4) Let $A$ be an axis in $\mathcal{A}$. Thus $A$ corresponds to the two
endpoints of an element $Z$ of the CCC $\Phi$, and the stabilizer of $A$ is
the $VPC1$ group $H_{Z}/K$. Let $B$ be an axis in $\mathcal{A}$ which crosses
$A$, and let $a^{+}$, $a^{-}$, $b$ and $b^{\prime}$ denote the points in
$S^{1}$ which correspond to the endpoints of $A$ and $B$. Let $I\ $and
$I^{\prime}$ denote the two closed intervals into which $a^{+}$ and $a^{-}$
cut $S^{1}$, labeled so that $b$ lies in $I$ and $b^{\prime}$ lies in
$I^{\prime}$. Let $g$ be an element of $H_{A}-H_{B}$. Thus $g$ fixes $a$ and
$a^{\prime}$, as well as preserving $I\ $and $I^{\prime}$, as we are assuming
that $G(v)/K$ acts on $S^{1}$ preserving orientation. As $g$ does not preserve
$B$, the action of $g$ on $S^{1}$ must move both $b$ and $b^{\prime}$. As $g$
induces a homeomorphism of $I$, it follows that for all $n>0$ the points
$g^{n}b$ are distinct points of $I$. This sequence must converge to some point
$b_{0}$ in $I$. Similarly, the sequence $g^{n}b^{\prime}$ must converge to
some point $b_{0}^{\prime}$ in $I^{\prime}$. If $b_{0}$ and $b_{0}^{\prime}$
are distinct points of $S^{1}$, it follows that the sequence of axes
$g^{n}\{b,b^{\prime}\}$ converges to an axis which contradicts part 3). We
conclude that $b_{0}$ and $b_{0}^{\prime}$ must be equal and so must also be
equal to $a^{+}$ or to $a^{-}$. We can assume $b_{0}=b_{0}^{\prime}=a^{+}$, by
replacing $g$ by its inverse, if needed. Similarly the sequences $g^{-n}b$ and
$g^{-n}b^{\prime}$, must converge to a single point, and this point must be
$a^{-}$.

We have shown that any element of $H_{A}-H_{B}$ acts on $S^{1}$ as an element
of infinite order. As $K=H_{A}\cap H_{B},$ by Lemma
\ref{stabilizeroftwoaxesequalsK}, and $H_{A}/K$ is $VPC1$, it follows that
$H_{A}/K$ is torsion free and hence infinite cyclic. We let $g$ denote a
generator of this stabilizer.

Now consider two points $g^{k}b$ and $g^{l}b^{\prime}$ in $S^{1}$. They cut
$S^{1}$ into two intervals $I_{k,l}^{+}$ and $I_{k,l}^{-}$ labeled so that
$a^{+}$ lies in $I_{k,l}^{+}$, and $a^{-}$ lies in $I_{k,l}^{-}$. If $C$ is a
compact subset of $S^{1}-\{a^{+}\}$, there is $N$ such that $C$ is contained
in the interval $I_{-N,N}^{-}$. It follows that as $n\rightarrow\infty$,
$g^{-n}(x)\rightarrow a^{-}$ uniformly on $C$. Similarly if $C$ is a compact
subset of $S^{1}-\{a^{-}\}$, then as $n\rightarrow\infty$, $g^{n}%
(x)\rightarrow a^{+}$ uniformly on $C$. This implies that the action of the
stabilizer of $A$ on $S^{1}$ is a convergence group action. One simply takes
$P\ $and $N$ to be $a^{+}$ and $a^{-}$ or vice versa. This completes the proof
of part 4).
\end{proof}

Now we have a convergence group action on $S^{1}$ we can give the proof of
Theorem \ref{strongcrossingimpliesFuchsiantype} which we restate for the
reader's convenience.

\begin{theorem}
Let $(G,\mathcal{S})$ be a group system of finite type such that
$e(G:\mathcal{S})=1$. Suppose that $n\geq0$, and that for any $VPC(\leq n)$
subgroup $M$ of $G$, we have $e(G,M:\mathcal{S})=1$. Further suppose that $G$
is not $VPC(n+2)$. Let $\{H_{\lambda}\}_{\lambda\in\Lambda}$ be a finite
family of $VPC(n+1)$ subgroups of $G$. For each $\lambda\in\Lambda$, let
$X_{\lambda}$ denote a nontrivial $\mathcal{S}$--adapted $H_{\lambda}$--almost
invariant subset of $G$. Let $E(\mathcal{X})$ denote the set of all translates
of the $X_{\lambda}$'s and their complements, and let $\Gamma$ denote the
algebraic regular neighbourhood of the $X_{\lambda}$'s in $G$. Let $\Phi$ be a
non-isolated CCC of $E(\mathcal{X})$ in which all crossings are strong, and
let $v$ denote the corresponding vertex of $\Gamma$ which encloses $X$. Then
$v$ is of $VPCn$--by--Fuchsian type relative to $\mathcal{S}$.
\end{theorem}

\begin{proof}
We start by recalling from Definition \ref{defnofVPC-by-FuchsianrelativetoS}
the conditions which must be satisfied by $v$. These are:

\begin{enumerate}
\item The vertex group $G(v)$ is $VPCn$--by--Fuchsian, where the Fuchsian
group is not finite nor two-ended.

\item Each edge of $\Gamma$ which is incident to $v$ carries some peripheral
subgroup of $G(v)$, and for each peripheral subgroup $\Sigma$, there is at
most one edge of $\Gamma$ which is incident to $v$ and carries $\Sigma$.

\item For each peripheral subgroup $\Sigma$ of $G(v)$, if there is no edge of
$\Gamma$ which is incident to $v$ and carries $\Sigma$, then there is a group
$S_{i}$ in $\mathcal{S}$, such that a conjugate of $S_{i}$ is contained in
$\Sigma$, but $S_{i}$ is not conjugate commensurable with a subgroup of the
$VPC$ normal subgroup of $G(v)$.

\item If there is $S_{i}$ in $\mathcal{S}$, such that a conjugate of $S_{i}$
is contained in $G(v)$, then either $S_{i}\ $is conjugate to a subgroup of
some peripheral subgroup of $G(v)$, or $S_{i}$ is conjugate commensurable with
a subgroup of the $VPC$ normal subgroup of $G(v)$.
\end{enumerate}

Now we verify these conditions.

1) The conclusion of the preceding work in this section is that the group
$G(v)$ has a $VPCn$ normal subgroup $K$ such that $G(v)/K$ has an action on
the circle $S^{1}$, for which the orientation preserving subgroup acts as a
convergence group. If $G(v)$ is finitely generated we can apply the result,
due to Casson and Jungreis \cite{Casson-Jungreis}, and to Gabai \cite{Gabai}
independently, that any convergence group action on $S^{1}$ is conjugate to
the action of a Fuchsian group. If $G(v)$ is not finitely generated, we have
the same result due to Tukia \cite{Tukia}. After this conjugation, we have an
action of $G(v)/K$, or a subgroup of index $2$, by isometries on the
hyperbolic plane $\mathbb{H}^{2}$. In the second case, we can further
conjugate the action of $G(v)/K$ on $\mathbb{H}^{2}$ so that it extends to an
action of the entire group by isometries on $\mathbb{H}^{2}$. We will say that
$G(v)/K$ is a Fuchsian group in this case also, although usually this
terminology is reserved for orientation preserving groups of isometries of
$\mathbb{H}^{2}$. As $K$ is $VPCn$, we see that $G(v)$ is a $VPCn$%
--by--Fuchsian group. Further, it is clear from the construction of the action
of $G(v)$ on $S^{1}$ that, as $\Phi$ is a non-isolated CCC, $G(v)/K$ is not
finite nor two-ended. This completes the proof that condition 1) holds.

2) Let $e$ be an edge of $\Gamma$ incident to $v$, and let $G(e)$ denote the
group associated to $e$. Then $G(e)$ is a subgroup of $G(v)$ and so has an
action on $S^{1}$ with kernel $G(e)\cap K$. We let $A$ denote the quotient of
$G(e)$ by this kernel. As $K$ is $VPCn$, we know that $G(e)\cap K$ is $VPCl$,
for some $l\leq n$. Suppose that $A$ is finite. It follows that $G(e)$ itself
is $VPCl$. As $\Gamma$ is $\mathcal{S}$--adapted, so is the splitting of
$G\ $over $G(e)$ associated to $e$. As $G$ has no $\mathcal{S}$--adapted
splittings over any $VPC(\leq n)$ subgroup, this is a contradiction. This
shows that $A$ must be infinite. We let $A_{0}$ denote the orientation
preserving subgroup of $A$, so that $A_{0}$ is also infinite.

Now each element of $A_{0}$ must act on $\mathbb{H}^{2}$ by an elliptic,
hyperbolic or parabolic isometry. If each element of $A_{0}$ were elliptic,
this would imply that $A_{0}$ is finite. Thus $A_{0}$ must contain a
non-elliptic element. Suppose that $A_{0}$ contains a hyperbolic element $a$,
and let $\lambda$ denote the axis of $a$. Now recall that the almost invariant
sets in the CCC $\Phi$ are cross-connected, so that the corresponding
collection of axes in $\mathbb{H}^{2}$ is also cross-connected. In addition
the set of endpoints of these axes is dense in $S^{1}$ by our construction. It
follows that $\lambda$ must cross some axis $\mu$ associated to an $H$--almost
invariant element $X$ of the CCC $\Phi$. Let $\overline{X}$ denote the half
space in $\mathbb{H}^{2}$ determined by $\mu$ which is associated to $X$.
After replacing $a$ by $a^{-1}$ if needed, we have inclusions $\ldots
a^{-2}\overline{X}\varsubsetneq a^{-1}\overline{X}\varsubsetneq\overline
{X}\varsubsetneq a\overline{X}\varsubsetneq a^{2}\overline{X}\varsubsetneq
\ldots$, and it follows from the construction above that we have inclusions
$\ldots\alpha^{-2}X\varsubsetneq\alpha^{-1}X\varsubsetneq X\varsubsetneq\alpha
X\varsubsetneq\alpha^{2}X\varsubsetneq\ldots$, where $\alpha$ denotes an
element of $G(e)$ which maps to $a$ under the quotient map $G(e)$ to $A$. This
implies that $G(e)\cap X$ and $G(e)\cap X^{\ast}$ are both $H$--infinite, so
that any almost invariant subset of $G$ over $G(e)$ must cross $X$ strongly.
But, as $\Gamma$ is an algebraic regular neighbourhood of the $X_{i}$'s, the
splitting of $G$ determined by $e$ cannot cross any element of the CCC $\Phi$.
This contradiction shows that $A_{0}$ cannot contain any hyperbolic elements.
It follows that $A_{0}$ must consist entirely of parabolic and elliptic
elements. As $A_{0}$ is an infinite discrete group of orientation preserving
isometries of $\mathbb{H}^{2}$, it follows that every element of $A_{0}$ must
be parabolic, that $A_{0}$ fixes a point on $S^{1}$, and that $A_{0}$ is
infinite cyclic. Recall that $G(e)\cap K$ is $VPCl$, for some $l\leq n$. It
follows that $G(e)$ itself is $VPC(l+1)$. As $\Gamma$ is $\mathcal{S}%
$--adapted, so is the splitting of $G\ $over $G(e)$ associated to $e$. As $G$
has no $\mathcal{S}$--adapted splittings over any $VPC(\leq n)$ subgroup, this
implies that $l=n$, so that $G(e)$ must be of finite index in the peripheral
subgroup $\Sigma$ of $G(v)$ associated to the cusp fixed by $A_{0}$.

At this point, we have shown that each edge of $\Gamma$ which is incident to
$v$ carries a subgroup of finite index in some peripheral subgroup of $G(v)$.
It remains to show that for each peripheral subgroup $\Sigma$ of $G(v)$, there
is at most one such edge, and if such an edge exists it must carry $\Sigma$.
Suppose there are edges $e_{1},\ldots,e_{k}$ of $\Gamma$ which are incident to
$v$ and carry a subgroup of finite index in $\Sigma$. We consider a refinement
of $\Gamma$ to a new graph of groups structure $\Gamma^{\prime}$ for $G$
defined as follows. Attach a new edge $E$ to $\Gamma$ with one end at $v$, the
other vertex $w$ not being attached, and with associated groups
$G(w)=G(E)=\Sigma$. Then for each $i$, slide $e_{i}$ along $E$ to the other
vertex $w$. The splitting of $G$ associated with $E$ is over the $VPC(n+1)$
group $\Sigma$, is adapted to $\mathcal{S}$, and does not cross any translate
of any $X_{\lambda}$. Hence this splitting is enclosed by some $V_{1}$--vertex
$v^{\prime}$ of the original graph of groups $\Gamma$. This implies that there
is a path $\lambda$ in $\Gamma^{\prime}$ whose first edge is $E$, and whose
final vertex is $v^{\prime}$ such that all interior vertices of $\lambda$ have
valence $2$ in $\Gamma^{\prime}$, and such that every edge group of $\lambda$
is equal to $\Sigma$, and each interior vertex also carries $\Sigma$. As
$G(v)$ properly contains $\Sigma$, it follows that $v$ is the initial vertex
of $\lambda$, and that $w$ is an interior vertex of $\lambda$. Hence $w$ has
valence $2$ in $\Gamma^{\prime}$, so that $k$ must equal $1$. Further the edge
$e_{1}$ is the second edge of $\lambda$, so that $G(e_{1})$ must equal
$\Sigma$, proving the required result.

This completes the proof that condition 2) holds.

3) \ Let $\Sigma$ be a peripheral subgroup of $G(v)$, such that there is no
edge of $\Gamma$ which is incident to $v$ and carries $\Sigma$. In the
quotient orbifold given by the action of $G(v)/K$ on the hyperbolic plane,
there is an infinite simple geodesic with both ends in the cusp associated to
$\Sigma$. This geodesic determines a free product splitting of $G(v)/K$ and
hence a splitting of $G(v)$ over $K$. Clearly this splitting of $G(v)$ is
adapted to all peripheral subgroups of $G(v)$ other than $\Sigma$. Hence this
splitting of $G(v)$ over $K$ can be extended to a splitting $\sigma$ of $G$
itself over $K$. Again $\sigma$ is adapted to all peripheral subgroups of
$G(v)$ other then $\Sigma$. As $K$ is $VPCn$, our hypothesis on $G$ tells us
that $\sigma$ cannot be $\mathcal{S}$--adapted, so that there is $S_{i}$ in
$\mathcal{S}$ such that $\sigma$ is not adapted to $S_{i}$.

We consider how $S_{i}$ can intersect $G(v)$. As $\Gamma$ is adapted to
$\mathcal{S}$, each $S_{i}$ fixes some vertex of the $G$--tree $T$. If $S_{i}$
does not fix $v$, then $S_{i}\cap G(v)$ is a subgroup of some peripheral
subgroup of $G(v)$ other than $\Sigma$. As $\sigma$ is adapted to all such
peripheral subgroups of $G(v)$, this implies that $\sigma$ is adapted to
$S_{i}$, which is a contradiction. We conclude that $S_{i}$ must fix $v$, so
that $S_{i}$ is contained in $G(v)$. As each $X_{i}$ is adapted to
$\mathcal{S}$, it follows, as for $G(e)$ in part 2), that the image of $S_{i}$
in $G(v)/K$ cannot contain any hyperbolic elements, and so must be finite or
an infinite cyclic parabolic subgroup. In the first case, $S_{i}$ is
commensurable with a subgroup of $K$, and in the second case, $S_{i}$ is
contained in a peripheral subgroup of $G(v)$. The same discussion applies to
any conjugate of $S_{i}$ in $G$. If the intersection with $G(v)$ of each
conjugate of $S_{i}$ were commensurable with a subgroup of $K$, then $\sigma$
would be adapted to $S_{i}$, which is again a contradiction. We conclude that
$S_{i}$ is conjugate to a subgroup of some peripheral subgroup $\Sigma
^{\prime}$ of $G(v)$, but is not conjugate commensurable with a subgroup of
$K$. If $\Sigma^{\prime}$ were not equal to $\Sigma$, the fact that the
splitting $\sigma$ of $G$ is adapted to $\Sigma^{\prime}$ would imply that
$\sigma$ was adapted to $S_{i}$, which is a contradiction. It follows that
$\Sigma^{\prime}$ must equal $\Sigma$.

We have shown that if there is no edge of $\Gamma$ which is incident to $v$
and carries $\Sigma$, then there is a group $S_{i}$ in $\mathcal{S}$, such
that a conjugate of $S_{i}$ is contained in $\Sigma$, but $S_{i}$ is not
conjugate commensurable with a subgroup $K$, which completes the proof that
condition 3) holds.

4) In particular, the preceding discussion shows that if there is $S_{i}$ in
$\mathcal{S}$, such that a conjugate of $S_{i}$ is contained in $G(v)$, then
either $S_{i}\ $is conjugate to a subgroup of some peripheral subgroup of
$G(v)$, or $S_{i}$ is conjugate commensurable with a subgroup of $K$, which
completes the proof that condition 4) holds. This completes the proof of
Theorem \ref{strongcrossingimpliesFuchsiantype}.
\end{proof}

\section{Accessibility results\label{accessibilityresults}}

In the following section we will consider what can be said about almost
invariant sets which cross weakly. But first we need some accessibility results.

In \cite{SS2}, an important role was played by an accessibility result,
Theorem 7.11 of \cite{SS2}. Unfortunately the proof of Theorem 7.11 given in
\cite{SS2} is incorrect, so we will give a correct argument below in Theorem
\ref{accessibilityinabsolutecase}. We then generalise this result to the
setting of this paper.

Before proving our accessibility results, we will recall some standard
terminology. In a graph of groups decomposition, a vertex is
\textit{inessential} if it has valence two, it is not the vertex of a loop,
and each edge group includes by an isomorphism into the vertex group. (Note
that a vertex is inessential if and only if it is isolated, as defined
immediately before Definition \ref{defnofalgregnbhd}.) Also a vertex is
\textit{reducible} if it has valence two, it is not the vertex of a loop, and
one of the incident edge groups includes by an isomorphism into that vertex
group. Let $\Gamma$ be a graph of groups and write $\pi_{1}(\Gamma)$ for the
fundamental group of $\Gamma$. If $e$ is an edge of $\Gamma$, we say that a
graph of groups structure $\Gamma^{\prime}$ is \textit{obtained from }$\Gamma
$\textit{ by collapsing} $e$ if the underlying graph of $\Gamma^{\prime}$ is
obtained from $\Gamma$ by collapsing $e$. In addition, if $p:\Gamma
\mathcal{\rightarrow}\Gamma^{\prime}$ denotes the natural projection map, we
require that each vertex $v$ of $\Gamma^{\prime}$ has associated group equal
to $\pi_{1}(p^{-1}(v))$. These conditions imply that $\pi_{1}(\Gamma)$ and
$\pi_{1}(\Gamma^{\prime})$ are naturally isomorphic. We will say that $\Gamma$
is a \textit{refinement} of $\Gamma^{\prime}$ obtained by splitting at the
vertex $v=p(e)$. We will say that $\Gamma$ is a \textit{proper refinement} of
$\Gamma^{\prime}$ if the edge splitting associated to $e$ induces a splitting
of the vertex group $G(v)$.

Now we come to the proof of the accessibility result we need. This result is
equivalent to that of Theorem 7.11 of \cite{SS2} apart from the assumption
that $\Gamma_{k+1}$ is a proper refinement of $\Gamma_{k}$. The word proper
was omitted in the statement of Theorem 7.11 of \cite{SS2}, as discussed in
section \ref{accessibility} of the appendix. In the proof below, we use the
finite presentation of $G$ to tell us that there is a finite $2$--complex $K$
with $\pi_{1}(K)$ equal to $G$, and then apply the result of Dunwoody
\cite{Dunwoody} that there is an upper bound on the number of non-parallel
disjoint tracks one can have in $K$. In \cite{Dunwoody}, Dunwoody also pointed
out that if $G\ $is almost finitely presented, one effectively has the same
finiteness result. Thus this accessibility result holds if $G\ $is almost
finitely presented.

\begin{theorem}
\label{accessibilityinabsolutecase}Let $G$ be a finitely presented group such
that $e(G)=1$. Let $n$ be a non-negative integer, and suppose that, for any
$VPC(\leq n)$ subgroup $M$ of $G$, we have $e(G,M)=1$. Let $\Gamma_{k}$ be a
sequence of minimal graphs of groups decompositions of $G$ without inessential
vertices and with all edge groups being $VPC(n+1)$, and suppose that
$\Gamma_{k+1}$ is a proper refinement of $\Gamma_{k}$, for each $k$. Then the
sequence $\Gamma_{k}$ must terminate.
\end{theorem}

\begin{remark}
\label{remarkonaccessibilityproof}The methods of Dunwoody in \cite{Dunwoody}
show that there must be a sequence of such refinements of $\Gamma_{1}$ which
terminates. But the above theorem asserts that any such sequence must
terminate. Note that this theorem does not assert any uniform bound for the
length of such a sequence, and in fact there is no such bound, as will be
shown in Example \ref{examplewithmanyrefinements} below. Thus this
accessibility result differs from all previous such results, as they all
produce a uniform bound on the length of sequences of refinements.
\end{remark}

\begin{proof}
The main result of \cite{B-F} gives a bound on the complexity of reduced
minimal graphs of groups decompositions of $G$ with all edge groups being
small groups. As $VPC$ groups are small, this yields an immediate proof of the
present theorem in the case when the $\Gamma_{k}$'s are all reduced, i.e. they
have no reducible vertices.

Now we will consider the general case and suppose that the sequence of
$\Gamma_{k}$'s is infinite. The facts that each $\Gamma_{k}$ is minimal with
no inessential vertices and that $\Gamma_{k+1}$ is a proper refinement of
$\Gamma_{k}$ imply that the sequence $\Gamma_{k}$ must stabilize apart from
refinements which introduce reducible vertices. As $\Gamma_{k+1}$ is a proper
refinement of $\Gamma_{k}$, the associated edge groups must be strictly
descending. By going to a subsequence and collapsing some edges of the
$\Gamma_{k}$'s to a point, we can arrange that all of the following hold.

\begin{itemize}
\item For each $k$, the graph of groups $\Gamma_{k}$ has $k$ edges
$e_{1},\ldots,e_{k}$ and is homeomorphic to an interval or the circle. For
$1\leq i\leq k-1$, let $v_{i}$ denote the vertex $e_{i}\cap e_{i+1}$. Let
$v_{0}$ denote the other vertex of $e_{1}$, and let $v_{k}$ denote the other
vertex of $e_{k}$. Thus if $\Gamma_{k}$ is a circle, we have $v_{0}=v_{k}$.

\item The graph of groups $\Gamma_{k+1}$ is obtained from $\Gamma_{k}$ by
splitting at the vertex $v_{k}$, and the extra edge is $e_{k+1}$.

\item Let $H_{i}$ denote the group associated to $e_{i}$. Then the inclusion
of $H_{i}$ into $G(v_{i})$ is an isomorphism, and $H_{i}$ strictly contains
$H_{i+1}$, for $1\leq i\leq k-1$.
\end{itemize}

As $G$ is finitely presented, there is a finite simplicial $2$--complex $K$
with $\pi_{1}(K)$ equal to $G$. Let $\widetilde{K}$ denote the universal cover
of $K$, and let $T_{k}$ denote the universal covering $G$--tree of $\Gamma
_{k}$. Let $q_{k}:T_{k}\rightarrow T_{k-1}$ denote the natural collapsing map.
We now pick $G$--equivariant linear maps $p_{k}:\widetilde{K}\rightarrow
T_{k}$ such that $p_{k}=q_{k+1}p_{k+1}$. Let $W_{k}$ denote the union of the
midpoints of the edges of $T_{k}$ and consider $p_{k}^{-1}(W_{k})$. This is a
$G$--invariant pattern in $\widetilde{K}$, which projects to a finite pattern
$L_{k}$ in $K$. By construction of the maps $p_{k}$, we have $L_{k}\subset
L_{k+1}$. As in the paper \cite{Dunwoody-Fenn} of Dunwoody and Fenn, we let
$T_{k}^{\prime}$ denote the dual graph to the pattern $p_{k}^{-1}(W_{k})$ in
$\widetilde{K}$. As $\widetilde{K}$ is simply connected, this graph is a tree,
and the covering action of $G$ on $\widetilde{K}$ induces an action of $G$ on
$T_{k}^{\prime}$ such that $G\backslash T_{k}^{\prime}$ is a finite graph of
groups structure for $G$. We also have a natural simplicial, and
$G$--equivariant, map $T_{k}^{\prime}\rightarrow T_{k}$. Now Dunwoody showed
in \cite{Dunwoody} that there is an upper bound on the number of non-parallel
disjoint tracks one can have in $K$. It follows that there is $N$ such that if
$k>N$, the tree $T_{k+1}^{\prime}$ is obtained from $T_{k}^{\prime}$ by
subdividing some edges. By collapsing the edges $e_{1},\ldots,e_{N}$ of each
of the $\Gamma_{k}$'s, we can arrange that this holds for all values of $k$.
Thus every tree $T_{k}^{\prime}$ is obtained from $T_{1}^{\prime}$ by
subdividing edges. It is possible that $T_{1}^{\prime}$ is not minimal. In
this case we will replace it by its unique minimal subtree, so we will assume
in what follows that $T_{1}^{\prime}$ is minimal. Note that as each $T_{k}$ is
minimal, the maps $T_{k}^{\prime}\rightarrow T_{k}$ must each be surjective.

Recall that $G$ acts on $T_{k}$ with quotient $\Gamma_{k}$. Thus we can find
an edge path $E_{1},\ldots,E_{k}$ in $T_{k}$, such that, for each $i$, the
edge $E_{i}$ maps to the edge $e_{i}$ of $\Gamma_{k}$ and has stabilizer equal
to $H_{i}$. If $f$ is an edge of $T_{k}^{\prime}$ which maps to $E_{i}$, the
stabilizer of $f$ is a subgroup of $H_{i}$. As $T_{1}^{\prime}$ is minimal, so
is $T_{k}^{\prime}$. In particular, any edge of $T_{k}^{\prime}$ determines a
splitting of $G$. As $G$ has no splittings over $VPC$ subgroups of length
$\leq n$, it follows that the stabilizer of $f$ must be $VPC(n+1)$, and so is
a subgroup of finite index in $H_{i}$, and hence in $H_{1}$. As the quotient
$G\backslash T_{k}^{\prime}$ is finite, it follows that there are only
finitely many edges of $T_{k}^{\prime}$ which map to $E_{i}$. We will denote
by $I(f)$ the index in $H_{1}$ of the stabilizer of $f$. We let $N_{i}$ denote
the maximum value of $I(f)$ as $f$ varies over those edges of $T_{k}^{\prime}$
which map to $E_{i}$. If $f^{\prime}$ is an edge of $T_{k}^{\prime}$ which
maps to the edge $e_{i}$ of $\Gamma_{k}$, then $f^{\prime}$ must map to a
translate $gE_{i}$ of $E_{i}$, for some $g$ in $G$. Thus the stabilizer of
$f^{\prime}$ is a subgroup of $H_{i}^{g}$, the stabilizer of $gE_{i}$. We let
$I(f^{\prime})$ denote the index in $H_{1}^{g}$ of the stabilizer of
$f^{\prime}$. Note that $I(f^{\prime})=I(g^{-1}f^{\prime})$, so is finite and
not greater than $N_{i}$.

Consider the sequence $N_{i}$, and suppose that this sequence is bounded by
some constant $M$. Then any edge of $T_{k}^{\prime}$ which maps to $E_{i}$ has
stabilizer which is a subgroup of $H_{i}$ and has index in $H_{1}$ which is at
most $M$. As the indices of the $H_{i}$'s in $H_{1}$ tend to infinity, this is
an immediate contradiction. We conclude that the sequence $N_{i}$ cannot be
bounded. By collapsing more edges of the $\Gamma_{k}$'s, we can arrange that
the sequence $N_{i}$ is strictly increasing.

Let $B_{k}$ denote the subgroup of $G$ associated to the vertex $v_{k}$ of
$\Gamma_{k}$. Thus for each $k$, the group $B_{k-1}$ is the amalgamated free
product $H_{k-1}\ast_{H_{k}}B_{k}$. Recall that if $\Gamma_{k}$ is an
interval, then $v_{k}$ is an end vertex, but if $\Gamma_{k}$ is a circle, then
$v_{k}=v_{0}$. For simplicity in our analysis, we will assume in what follows
that $k>2$.

As $N_{k-1}<N_{k}$, there must be an edge $e_{k}^{\prime}$ of $T_{k}^{\prime}$
such that $e_{k}^{\prime}$ maps to the edge $e_{k}$ of $\Gamma_{k}$, and
$I(e_{k}^{\prime})>N_{k-1}$. As $T_{k}^{\prime}$ is obtained from
$T_{k-1}^{\prime}$ by subdividing edges, there must be a path $\mu$ in
$T_{k}^{\prime}$ which has all internal vertices of valence $2$, and starts
with $e_{k}^{\prime}$ and ends with some edge $e_{k-1}^{\prime}$ of
$T_{k}^{\prime}$ which maps to the edge $e_{k-1}$ of $\Gamma_{k}$. There is a
subpath $\lambda$ of $\mu$ which starts with $e_{k}^{\prime}$ and ends with an
edge $e^{\prime}$ such that each edge of $\lambda$, apart from $e^{\prime}$,
maps to the edge $e_{k}$ in $\Gamma_{k}$, and $e^{\prime}$ does not map to
$e_{k}$. As all the internal vertices of $\lambda$ are of valence $2$ in
$T_{k}^{\prime}$, it follows that every edge of $\lambda$ has the same
stabilizer, which we denote by $J$.

As $I(e^{\prime})\leq N_{k-1}$ and $I(e_{k}^{\prime})>N_{k-1}$, there must be
consecutive edges $f$ and $f^{\prime}$ of $\lambda$ such that $I(f)>N_{k-1}$,
and $I(f)\neq I(f^{\prime})$. Let $F$ and $F^{\prime}$ denote the images of
$f$ and $f^{\prime}$ in $T_{k}$. As $I(f)>N_{k-1}$, the edge $F$ must be a
translate of $E_{k}$. As $\Gamma_{k}$ is an interval or circle, the edge
$F^{\prime}$ must be a translate of $E_{k-1}$, or of $E_{k}$, or of $E_{1}$.
Note that as $k>2$, the last case can only occur if $\Gamma_{k}$ is a circle.
We will consider these three cases separately.

\begin{case}
$F^{\prime}$ is a translate of $E_{k-1}$.
\end{case}

By replacing $f$ and $f^{\prime}$ by suitable translates, we can arrange that
$F^{\prime}$ is equal to $E_{k-1}$. Thus we have that $J$ is a subgroup of
$H_{k-1}$, and $I(f^{\prime})$ is equal to the index of $J$ in $H_{1}$. Now
$F$ is a translate $gE_{k}$ of $E_{k}$ which is incident to $E_{k-1}$, so that
$J$ is also a subgroup of $H_{k}^{g}$. As $k>2$, it follows that $E_{k-1}\cap
gE_{k}$ lies above the vertex $v_{k-1}$ of $\Gamma_{k}$, so that $E_{k-1}\cap
gE_{k}$ is equal to $E_{k-1}\cap E_{k}$. Further $g$ lies in the stabilizer of
$E_{k-1}\cap E_{k}$ which is equal to $H_{k-1}$. Thus $H_{k-1}^{g}$ is equal
to $H_{k-1}$, so that $I(f)$, which is the index of $J$ in $H_{1}^{g}$, must
equal $I(f^{\prime})$. This contradicts our assumption that $I(f)\neq
I(f^{\prime})$.

\begin{case}
$F^{\prime}$ is a translate of $E_{k}$.
\end{case}

By replacing $f$ and $f^{\prime}$ by suitable translates, we can arrange that
$F$ is equal to $E_{k}$, and that $F^{\prime}$ is equal to $gE_{k}$, for some
$g$ in $G$. If $F^{\prime}$ is equal to $F$, then $I(f)$ and $I(f^{\prime})$
must be equal, which is a contradiction. Thus $F^{\prime}$ is not equal to
$F$, and $g$ must lie in the stabilizer of the vertex $w=F^{\prime}\cap
F=F^{\prime}\cap E_{k}$. If $w$ projects to $v_{k-1}$, then $w$ equals
$E_{k-1}\cap E_{k}$, and so the stabilizer of $w$ equals $H_{k-1}$. Hence
$H_{k}^{g}$ is equal to $H_{k}$, which again implies that $I(f)$ and
$I(f^{\prime})$ must be equal. This contradiction shows that $w$ must project
to $v_{k}$, and so the stabilizer of $w$ equals $B_{k}$.

Now denote $I(f)$ by $d$. We have found that $J$ is contained in $H_{k}$, and
is contained in $H_{1}$ with index $d$, and that there is $g$ in $B_{k}$ such
that $J$ is also contained in $H_{k}^{g}$, and is contained in $H_{1}^{g}$
with index which is not equal to $d$. Recall that $B_{k}=H_{k}\ast_{H_{k+1}%
}B_{k+1}$. Lemma \ref{indexofconjugates}\ below shows that there is a
conjugate $J^{\prime}$ of $J$ and an element $g^{\prime}$ of $B_{k+1}$ such
that $J^{\prime}$ is a subgroup of $H_{k+1}$ with index $d$ in $H_{1}$, and
$J^{\prime}$ is also a subgroup of $H_{k+1}^{g^{\prime}}$ whose index in
$H_{1}^{g^{\prime}}$ is not equal to $d$. By repeatedly applying this lemma we
can find, for each $i$, a subgroup of $H_{i}$ which has index $d$ in $H_{1}$.
This is a contradiction as the indices of the $H_{i}$'s in $H_{1}$ tend to infinity.

We conclude that if either of the above two cases holds for any value of $k$,
we obtain a contradiction, so the sequence $\Gamma_{k}$ must terminate. In
particular, this holds if each $\Gamma_{k}$ is an interval. This leaves the
following case to consider.

\begin{case}
Each $\Gamma_{k}$ is a circle, and Cases 1) and 2) do not occur for any value
of $k>2$.
\end{case}

In this case, for each value of $k>2$, there must be adjacent edges $f_{k}$
and $f_{k}^{\prime}$ of $T_{k}^{\prime}$, such that $f_{k}\cap f_{k}^{\prime}$
is a vertex of valence two, $I(f_{k})>N_{k-1}$, and $I(f_{k})\neq
I(f_{k}^{\prime})$, and $f_{k}$ maps to $e_{k}$ in $\Gamma_{k}$, and
$f_{k}^{\prime}$ maps to $e_{1}$ in $\Gamma_{k}$. Thus if $F$ and $F^{\prime}$
denote the images of $f$ and $f^{\prime}$ in $T_{k}$, then $F$ is a translate
of $E_{k}$, and $F^{\prime}$ is a translate of $E_{1}$. By replacing $f$ and
$f^{\prime}$ by suitable translates, we can arrange that $F^{\prime}$ is equal
to $E_{1}$. Let $w$ denote the vertex $F\cap F^{\prime}=F\cap E_{1}$, which
must project to the vertex $v_{0}=v_{k}$ of $\Gamma_{k}$.

We consider the collection $C_{k}$ of all edges of $T_{k}^{\prime}$ whose
image in $T_{k}$ is incident to $w$. This collection is invariant under the
stabilizer $A$ of $w$ and is $A$--finite. As $T_{k+1}^{\prime}$ is a
refinement of $T_{k}^{\prime}$ obtained by subdividing some edges, the natural
collapsing map $T_{k+1}^{\prime}\rightarrow T_{k}^{\prime}$ induces an
$A$--equivariant bijection from $C_{k+1}$ to $C_{k}$. Each $f_{k}$ determines
an element of $C_{k}$. As $C_{k}$ is $A$--finite, there must be $k$, and
$l>k$, such that $f_{k}$ and $f_{l}$ determine corresponding $A$--orbits. By
replacing $f_{l}$ by a suitable translate, we can assume that $f_{k}$ and
$f_{l}$ determine corresponding elements of $C_{k}$ and $C_{l}$. Thus the
pre-image in $T_{l}^{\prime}$ of the edge $f_{k}$ in $T_{k}^{\prime}$ is a
subdivision of $f_{k}$ which contains the edge $f_{l}$. We abuse notation and
denote by $f_{k}$ the edge in this subdivision which maps onto $f_{k}$. Then
we have a path $\lambda$ in $T_{l}^{\prime}$ between $f_{l}$ and $f_{k}$ in
which all vertices have valence $2$, and no edges project to $e_{1}$ in
$\Gamma_{l}$. As $I(f_{l})>N_{l-1}$, there must be consecutive edges of
$\lambda$ which yield Case 1) or 2)\ above, which contradicts the assumption
we have made in this case.

This contradiction completes the proof that the sequence $\Gamma_{k}$ must
terminate as required.
\end{proof}

Now we give the proof of the lemma we quoted above.

\begin{lemma}
\label{indexofconjugates}Let $K$ be a group which has a splitting as
$A\ast_{C}B$, where $C$ has finite index in $A$. Let $J$ be a subgroup of $A$
of finite index $d$, and suppose there is $g$ in $K$ such that $J$ is also a
subgroup of the conjugate $A^{g}$ of $A$, and the index of $J$ in $A^{g}$ is
not equal to $d$. Then there is a conjugate $J^{\prime}$ of $J$ and an element
$k$ of $B$ such that $J^{\prime}$ is a subgroup of $C$ with index $d$ in $A$,
and $J^{\prime}$ is also a subgroup of $C^{k}$ whose index in $A^{k}$ is not
equal to $d$.
\end{lemma}

\begin{proof}
Let $T$ denote the $K$--tree associated with the given splitting of $K$. Thus
there is an edge $e$ of $T$ with stabilizer equal to $C$, whose vertices $v$
and $w$ have stabilizers $A$ and $B$ respectively. We note that $g$ cannot lie
in $A$, as the indices of $J$ in $A$ and in $A^{g}$ differ.

As $A$ fixes $v$, the conjugate $A^{g}$ fixes $gv$. As $J$ is a subgroup of
$A$ and of $A^{g}$, it follows that $J$ fixes both $v$ and $gv$. As $g$ does
not lie in $A$, the vertices $v$ and $gv$ are distinct, and $J$ must fix the
path $\lambda$ in $T$ which joins $v$ and $gv$. Let $f$ be an edge of
$\lambda$, so that $f$ is equal to $he$ for some $h$ in $K$. As $J$ fixes $f$,
it is a subgroup of the stabilizer $C^{h}$ of $he$. We consider the index
$I(f)$ of $J$ in $A^{h}$. Note that $I(f)$ is well defined, as $he=h^{\prime
}e$ if and only if $h^{-1}h^{\prime}$ lies in $C$. If $f$ is incident to $v$,
then $A^{h}$ is equal to $A$, so that $I(f)$ is equal to $d$. If $f$ is
incident to $gv$, then $A^{h}$ is equal to $A^{g}$, so that $I(f)$ is not
equal to $d$. It follows that there are two consecutive edges $f=he$ and
$f^{\prime}=h^{\prime}e$ of $\lambda$ such that $I(f)=d$ and $I(f^{\prime
})\neq d$. Let $J^{\prime}$ denote $h^{-1}Jh$, and let $e^{\prime}$ denote
$h^{-1}f^{\prime}$. Then $J^{\prime}$ fixes the adjacent edges $h^{-1}f=e$ and
$h^{-1}f^{\prime}=e^{\prime}$. As $e^{\prime}$ is incident to $e$, there is
$k$ in the stabilizer of $e\cap e^{\prime}$ such that $e^{\prime}=ke$. Thus
$k$ lies in $A$ or in $B$. Further $J^{\prime}$ is contained in the stabilizer
$C$ of $e$, and has index $d$ in $A$, and $J^{\prime}$ is also contained in
the stabilizer $C^{k}$ of $e^{\prime}$, and has index in $A^{k}$ which is not
equal to $d$. As these indices differ, it is not possible that $k$ lies in
$A$, so that $k$ must lie in $B$. This completes the proof of the lemma.
\end{proof}

In Remark \ref{remarkonaccessibilityproof}, we noted that Theorem
\ref{accessibilityinabsolutecase} does not provide a bound for the length of
the sequence $\Gamma_{k}$ of graphs of groups. Here is the promised example to
show there is no such bound.

\begin{example}
\label{examplewithmanyrefinements}In this example, the group $G$ will be the
Baumslag-Solitar group $BS(4,2)=\left\langle a,t:t^{-1}a^{2}t=a^{4}%
\right\rangle $. We will consider graphs of groups structures for $G$ in which
all edge and vertex groups are infinite cyclic. Thus the structure is
determined by integers associated to the end of each edge which denote the
index of the edge group in that vertex group.

The subgroup of $BS(4,2)$ generated by $a^{2}$ and $t$ is isomorphic to the
Baumslag-Solitar group $BS(2,1)$. Consider the natural graph of groups
structure for $BS(2,1)$ with one vertex $w$, which carries $\left\langle
a^{2}\right\rangle $, the cyclic group generated by $a^{2}$, and one edge
which is a loop $l$ whose ends are labelled by $1$ and $2$. Now construct a
new graph of groups by adding a new vertex $v$ which carries $\left\langle
a\right\rangle $, and an edge $e$ which joins $v$ and $w$. The edge $e$ is
labelled by $2$ at the $v$--end and by $1$ at the $w$--end. It is easy to see
that the fundamental group of this graph of groups is $BS(4,2)$.

Now we come to the key step. We slide one end of $e$ round the loop $l$. The
result is to leave the graph of groups unchanged except that the label on $e$
at the $w$-end becomes $2$, and now $w$ carries $\left\langle ta^{2}%
t^{-1}\right\rangle $. Repeat this step a total of $n$ times, to change the
label on $e$ at the $w$-end to $2^{n}$, and change the group carried by $w$ to
$\left\langle t^{n}a^{2}t^{-n}\right\rangle $. Denote this graph of groups by
$\Gamma_{1}$. Now construct a graph of groups $\Gamma_{n}$ by subdividing $e$
into $n$ sub edges, which are all labelled by $2$ and $1$, apart from the edge
incident to $v$ which is labelled by $2$ and $2$. Clearly there is a sequence,
of length $n$, of refinements of $\Gamma_{1}$, whose $n$--th step is
$\Gamma_{n}$.
\end{example}

The accessibility result of Theorem \ref{accessibilityinabsolutecase} can
easily be extended to the case of adapted splittings so long as we make
suitable restrictions on the group system involved. Recall from Definition
\ref{defnoff.p.relfamilyofsubgroups} that if $(G,\mathcal{S})$ is a group
system of finite type, then $G$ is finitely presented relative to
$\mathcal{S}$, if there is a finite relative generating set $\Omega$ such that
the kernel of the natural epimorphism $F(\Omega)\ast S_{1}\ast\ldots\ast
S_{n}\rightarrow G$ is normally generated by a finite set. The result we
obtain is the following.

\begin{theorem}
\label{accessibilityforsystemswhenGisfp}Let $(G,\mathcal{S})$ be a group
system of finite type such that $G$ is finitely presented relative to
$\mathcal{S}$, and $e(G:\mathcal{S})=1$. Let $n$ be a non-negative integer,
and suppose that, for any $VPC(\leq n)$ subgroup $M$ of $G$, we have
$e(G,M:\mathcal{S})=1$. Let $\Gamma_{k}$ be a sequence of $\mathcal{S}%
$--adapted minimal graphs of groups decompositions of $G$ without inessential
vertices and with all edge groups being $VPC(n+1)$. If $\Gamma_{k+1}$ is a
proper refinement of $\Gamma_{k}$, for each $k$, then the sequence $\Gamma
_{k}$ must terminate.
\end{theorem}

\begin{remark}
Recall that $e(G,M:\mathcal{S})\leq e(G,M)$, as observed just before Lemma
\ref{splittingisadaptediffa.i.setis}. Thus if $e(G)=1$, and $e(G,M)=1$ for any
$VPC(\leq n)$ subgroup $M$ of $G$, then the hypotheses of the theorem hold.
\end{remark}

\begin{proof}
We wish to apply the proof of Theorem \ref{accessibilityinabsolutecase}, but
the difficulty is that now $G$ may have splittings over some $VPC(\leq n)$
subgroups. But we do know that no such splitting can be $\mathcal{S}%
$--adapted. The first step in the proof of Theorem
\ref{accessibilityinabsolutecase} was to apply the accessibility result of
Bestvina and Feighn in \cite{B-F}. Their result applies exactly the same in
the present situation.

The next step was to choose a finite simplicial $2$--complex $K$ with $\pi
_{1}(K)$ equal to $G$, let $\widetilde{K}$ denote the universal cover of $K$,
let $T_{k}$ denote the universal covering $G$--tree of $\Gamma_{k}$, and pick
$G$--equivariant maps $p_{k}:\widetilde{K}\rightarrow T_{k}$ such that
$p_{k}=q_{k+1}p_{k+1}$.

In the present situation, $G$ need not be finitely presented, so we may be
unable to find a finite simplicial $2$--complex $K$ with $\pi_{1}(K)$ equal to
$G$. In order to apply this method to prove Theorem
\ref{accessibilityforsystemswhenGisfp}, we need to refine our choice of $K$.

As $G$ is finitely presented relative to $\mathcal{S}$, we can construct a
$2$--dimensional complex $K$ with $\pi_{1}(K)\cong G$ as follows. Start with
disjoint $2$--complexes $X_{i}$ with $\pi_{1}(X_{i})\cong S_{i}$. Next add
$(n-1)$ edges joining the $X_{i}$'s to obtain a connected $2$--complex $X$
with $\pi_{1}(X)\cong S_{1}\ast\ldots\ast S_{n}$. Now we can form $K$ by
adding a finite number of vertices, $1$--simplices and $2$--simplices.

Next we refine our choice of $G$--equivariant maps $p_{k}:\widetilde{K}%
\rightarrow T_{k}$ by insisting that $p_{k}$ map each component of the
pre-image of $X_{i}$ to a vertex of $T_{k}$, for each $i$. This can be done as
the assumption that $\Gamma_{k}$ is $\mathcal{S}$--adapted means that each
$S_{i}$ fixes some vertex of $T_{k}$. It follows that if $W_{k}$ denotes the
union of the midpoints of the edges of $T_{k}$, then $p_{k}^{-1}(W_{k})$ is
disjoint from the pre-image of each $X_{i}$. Thus $p_{k}^{-1}(W_{k})$ is a
$G$--invariant pattern in $\widetilde{K}$, which projects to a pattern $L_{k}$
in $K$, disjoint from each of the $X_{i}$'s. It follows that $L_{k}$ lies in
the subcomplex $K^{\prime}$ of $K$ obtained by removing the interiors of all
the $2$--simplices of $X$, and then removing the interiors of all the
$1$--simplices of $X$ which now have no incident $2$--simplex. Clearly
$K^{\prime}$ is a finite subcomplex of $K$, so that the pattern $L_{k}$ must
also be finite. As $K^{\prime}$ is finite, we can again apply Dunwoody's
result in \cite{Dunwoody} to tell us that there is an upper bound on the
number of non-parallel disjoint tracks in $K^{\prime}$. Now we can continue as
in the proof of Theorem \ref{accessibilityinabsolutecase}, to consider the
trees $T_{k}^{\prime}$ and the natural map $T_{k}^{\prime}\rightarrow T_{k}$.
The key result in that proof was that if $f$ is an edge of $T_{k}^{\prime}$
which maps to an edge $E$ of $T_{k}$, then the stabilizer of $f$ must have
finite index in the stabilizer of $E$. This was because $f$ had an associated
splitting of $G$ over the stabilizer of $f$, and $G$ had no splittings over
$VPC(\leq n)$ subgroups. In the present situation, our construction of $K$ and
our choice of the map $p_{k}:\widetilde{K}\rightarrow T_{k}$ implies that any
splitting of $G$ associated to an edge of $T_{k}^{\prime}$ must be
$\mathcal{S}$--adapted. As $G$ has no $\mathcal{S}$--adapted splittings over
$VPC(\leq n)$ subgroups, it again follows that the stabilizer of $f$ must have
finite index in the stabilizer of $E$. After this, the rest of the proof of
Theorem \ref{accessibilityinabsolutecase} applies unchanged to complete the
proof of Theorem \ref{accessibilityforsystemswhenGisfp}.
\end{proof}

\section{CCCs with weak crossing\label{CCC'swithweakcrossing}}

As in previous sections, we start by considering a group system
$(G,\mathcal{S})$ of finite type. As in Theorem
\ref{crossinginCCCsisallweakorallstrong}, we further assume that
$e(G:\mathcal{S})=1$. Let $n$ be a non-negative integer, and suppose that, for
any $VPC(\leq n)$ subgroup $M$ of $G$, we have $e(G,M:\mathcal{S})=1$. Let
$\{X_{\lambda}\}_{\lambda\in\Lambda}$ be a finite family of nontrivial
$\mathcal{S}$--adapted almost invariant subsets of $G$, such that each
$X_{\lambda}$ is over a $VPC(n+1)$ subgroup. Let $E(\mathcal{X})$ denote the
set of all translates of the $X_{\lambda}$'s and their complements, and
consider the collection of cross-connected components (CCCs) of $E(\mathcal{X}%
)$ as in the construction of regular neighbourhoods in section
\ref{algregnbhds}. Theorem \ref{crossinginCCCsisallweakorallstrong} tells us
that crossings in a non-isolated CCC of $E(\mathcal{X})$ are either all strong
or all weak.

In this section, we will consider a non-isolated CCC $\Phi$ of $E(\mathcal{X}%
)$ in which all crossings are weak. Theorem
\ref{crossinginCCCsisallweakorallstrong} tells us that the stabilisers of the
elements of $E$ which lie in $W$ are all commensurable, and if $H$ denotes
such a stabiliser, then $e(G,[H]:\mathcal{S})\geq4$. Our aim is to get
information about the structure of the corresponding $V_{0}$--vertex $v$ of
the algebraic regular neighbourhood of the $X_{\lambda}$'s in $G$. Clearly
$G(v)$ is contained in $Comm_{G}(H)$.

In \cite{SS2}, the authors considered the special case of this situation where
$\mathcal{S}$ is empty, and so $G$ is finitely generated. The work of Dunwoody
and Roller \cite{Dunwoody-Roller} shows that $v$ must enclose a splitting of
$G$ over some subgroup commensurable with $H$. In outline, the argument of
Dunwoody and Roller \cite{Dunwoody-Roller} proceeds as follows. (See also
\cite{D-Swenson} for a later account.) Given an almost invariant set $X$ in
the CCC $\Phi$, either $X$ is associated to a splitting of $G$, or $X$ crosses
some translate $gX$. In the second case, let $Y$ be a corner of the pair
$(X,gX)$. In the given situation, $X$ and $gX$ are over subgroups of $G$
commensurable with $H$, so that $Y$ is also an almost invariant set over a
subgroup of $G$ commensurable with $H$. Further $Y$ must be nontrivial, as $X$
crosses $gX$. Thus $Y$ is an element of the CCC $\Phi$, and so we can repeat
the argument. An elegant combinatorial argument due to Bergman \cite{Bergman},
modified slightly by Dunwoody and Roller, defines a complexity such that if
$X$ crosses $gX$, then one of the four corners of the pair $(X,gX)$ must have
strictly less complexity than $X$. Further any strictly descending sequence of
complexities obtained by this process must have finite length. It follows that
after finitely many steps, one must find an almost invariant set in the CCC
$\Phi$, which is associated to a splitting of $G$.

In the setting of this paper, Bergman's combinatorics apply in just the same
way, so long as $G$ is finitely generated. Note that, as $X$ and $gX$ are
adapted to $\mathcal{S}$, Lemma \ref{intersectionisL-a.i.}\ shows that any
corner of the pair $(X,gX)$ is also adapted to $\mathcal{S}$. The following
result summarises the above discussion.

\begin{lemma}
\label{weakcrossingCCCenclosessplittingifGfg}Let $(G,\mathcal{S})$ be a group
system of finite type such that $G$ is finitely generated, and
$e(G:\mathcal{S})=1$. Let $n$ be a non-negative integer, and suppose that, for
any $VPC(\leq n)$ subgroup $M$ of $G$, we have $e(G,M:\mathcal{S})=1$. Let
$\{X_{\lambda}\}_{\lambda\in\Lambda}$ be a finite family of nontrivial
$\mathcal{S}$--adapted almost invariant subsets of $G$, such that each
$X_{\lambda}$ is over a $VPC(n+1)$ subgroup. Let $E(\mathcal{X})$ denote the
set of all translates of the $X_{\lambda}$'s and their complements, let $W$ be
a non-isolated CCC of $E(\mathcal{X})$ in which all crossings are weak, and
let $v$ be the corresponding $V_{0}$--vertex of the algebraic regular
neighbourhood of the $X_{\lambda}$'s in $G$.

Then $G(v)$ is contained in $Comm_{G}(H)$, and $v$ encloses a $\mathcal{S}%
$--adapted splitting of $G$ over some subgroup commensurable with $H$. Further
the almost invariant set associated to this splitting lies in the Boolean
algebra generated by the $X_{\lambda}$'s.
\end{lemma}

\begin{remark}
If $W$ is an isolated CCC of $E(\mathcal{X})$, the result of this lemma is immediate.
\end{remark}

Now we can use the above result to show that the conclusion holds even when
$G$ is not finitely generated.

\begin{lemma}
\label{weakcrossingCCCenclosessplitting}Let $(G,\mathcal{S})$ be a group
system of finite type such that $e(G:\mathcal{S})=1$. Let $n$ be a
non-negative integer, and suppose that, for any $VPC(\leq n)$ subgroup $M$ of
$G$, we have $e(G,M:\mathcal{S})=1$. Let $\{X_{\lambda}\}_{\lambda\in\Lambda}$
be a finite family of nontrivial $\mathcal{S}$--adapted almost invariant
subsets of $G$, such that each $X_{\lambda}$ is over a $VPC(n+1)$ subgroup.
Let $E(\mathcal{X})$ denote the set of all translates of the $X_{\lambda}$'s
and their complements, let $W$ be a non-isolated CCC of $E(\mathcal{X})$ in
which all crossings are weak, and let $v$ be the corresponding $V_{0}$--vertex
of the algebraic regular neighbourhood of the $X_{\lambda}$'s in $G$.

Then $G(v)$ is contained in $Comm_{G}(H)$, and $v$ encloses a $\mathcal{S}%
$--adapted splitting of $G$ over some subgroup commensurable with $H$. Further
the almost invariant set associated to this splitting lies in the Boolean
algebra generated by the $X_{\lambda}$'s.
\end{lemma}

\begin{remark}
If $W$ is an isolated CCC of $E(\mathcal{X})$, the result of this lemma is immediate.
\end{remark}

\begin{proof}
As in the proof of Lemma \ref{weakcrossingCCCenclosessplittingifGfg}, it is
clear that $G(v)$ is contained in $Comm_{G}(H)$. We will now proceed as in
section \ref{extensionsandcrossing} to reduce to the case where $G\ $is
finitely generated.

Let $\mathcal{S}^{\infty}$ denote the subfamily of $\mathcal{S}$ which
consists of all those groups in $\mathcal{S}$ which are not contained in a
finitely generated subgroup of $G$. As each $S_{j}$ in $\mathcal{S}^{\infty}$
is countable, there is a finitely generated group $A_{j}$ which contains
$S_{j}$. We form $\overline{G}$ by amalgamating $G$ with $A_{j}$ over $S_{j}$,
for each $S_{j}$ in $\mathcal{S}^{\infty}$. Thus $\overline{G}$ is the
fundamental group of a graph $\Gamma$ of groups which has one vertex $V$ with
associated group $G$, and has edges each with one end at $V$ and the other end
at a vertex $V_{j}$ with associated group $A_{j}$. The edge joining $V$ to
$V_{j}$ has associated group $S_{j}$. As $G$ is finitely generated relative to
$\mathcal{S}$, it follows that $\overline{G}$ is finitely generated.

Now suppose that $H$ is a finitely generated subgroup of $G$ and that $X$ is a
nontrivial $\mathcal{S}$--adapted $H$--almost invariant subset of $G$. Note
that as $H$ is finitely generated, no $S_{j}$ in $\mathcal{S}^{\infty}$ can be
conjugate commensurable with a subgroup of $H$. As discussed immediately after
Lemma \ref{uniqueenlargement}, it follows that $X$ has a canonical extension
to a $H$--almost invariant subset $\overline{X}$ of $\overline{G}$. Further
Lemma \ref{extensionsareadaptedtoAunionS} tells us that $\overline{X}$ is
adapted to the family $\mathcal{A}\cup\mathcal{S}$. Now the group system
$(\overline{G},\mathcal{A}\cup\mathcal{S})$ is of finite type, so we are ready
to apply Lemma \ref{weakcrossingCCCenclosessplittingifGfg}. Let $\{X_{\lambda
}\}_{\lambda\in\Lambda}$ be a finite family of nontrivial $\mathcal{S}%
$--adapted almost invariant subsets of $G$, such that each $X_{\lambda}$ is
over a $VPC(n+1)$ subgroup. Let $E(\mathcal{X})$ denote the set of all
translates of the $X_{\lambda}$'s and their complements, let $W$ be a
non-isolated CCC of $E(\mathcal{X})$ in which all crossings are weak, and let
$v$ be the corresponding $V_{0}$--vertex of the algebraic regular neighbourhood
of the $X_{\lambda}$'s in $G$. Then $\{\overline{X_{\lambda}}\}_{\lambda
\in\Lambda}$ is a finite family of nontrivial $(\mathcal{A}\cup\mathcal{S}%
)$--adapted almost invariant subsets of $\overline{G}$, such that each
$\overline{X_{\lambda}}$ is over a $VPC(n+1)$ subgroup. Let $E(\overline
{\mathcal{X}})$ denote the set of all translates of the $\overline{X_{\lambda
}}$'s and their complements, and let $\overline{W}$ denote $\{\overline{U}\in
E(\overline{\mathcal{X}}):U\in W\}$. Then $\overline{W}$ is a CCC of
$E(\overline{\mathcal{X}})$ in which all crossings are weak. Now the proof of
Lemma \ref{weakcrossingCCCenclosessplittingifGfg} shows that either
$\overline{X}$ determines a splitting of $\overline{G}$ or there is a corner
of some pair $(X,gX)$ with strictly less Bergman complexity. After repeating
this argument for finitely many steps, we obtain an almost invariant set
$\overline{Z}$ in the CCC $\overline{W}$ which determines a $(\mathcal{A}%
\cup\mathcal{S})$--adapted splitting of $\overline{G}$ over some subgroup
commensurable with $H$. Then $\overline{Z}\cap G$ is a nontrivial element $Z$
of the CCC $W$, so that $Z$ determines a $\mathcal{S}$--adapted splitting of
$G$ over some subgroup commensurable with $H$. As $Z$ is enclosed by $v$, this
splitting is also enclosed by $v$.
\end{proof}

Next we discuss an alternative approach to proving Lemma
\ref{weakcrossingCCCenclosessplittingifGfg}. In the special case when
$\mathcal{S}$ is empty, Niblo \cite{Niblo, Niblo2} gave an alternative
approach to the result of Dunwoody and Roller \cite{Dunwoody-Roller}. His
approach uses the results of \cite{NibloSageevScottSwarup} on very good
position. We describe those parts of Niblo's arguments in \cite{Niblo} and
\cite{Niblo2} which are relevant for the present application.

Consider a finitely generated group $G$, and let $H$ be a finitely generated
subgroup of $G$ such that $e(G,H)\geq2$, so there is a nontrivial $H$--almost
invariant subset $X$ of $G$. As shown in \cite{NibloSageevScottSwarup}, and
discussed in section \ref{goodposition}, we can assume that $X$ is in very
good position. Suppose that whenever $gX$ crosses $X$ then $g$ lies in
$Comm_{G}(H)$. Let $C(X)$ denote the cubing constructed by Sageev from $X$, as
discussed in section \ref{cubings}. In Proposition 7 of \cite{Niblo}, Niblo
shows that each hyperplane in $C(X)$ is compact. Next he considers the
$2$--skeleton $C(X)^{(2)}$ of $C(X)$. This $2$--complex need not be locally
finite in general, but in \cite{Niblo2}, Niblo shows that the compactness of
the hyperplanes in $C(X)$ implies that $C(X)^{(2)}$ is locally finite away
from the vertices. Thus the ideal $2$--complex obtained by removing the
vertices from $C(X)^{(2)}$ is actually locally finite, so the theory of
patterns and tracks as defined by Dunwoody \cite{Dunwoody} can be used. In
\cite{Niblo2}, Niblo subdivides $C(X)^{(2)}$ into triangles without
introducing new vertices. By identifying each triangle with an ideal
hyperbolic triangle, one can define the length of a pattern in $C(X)^{(2)}$ as
was done by Jaco and Rubinstein \cite{Jaco-Rubinstein} when considering
surfaces in $3$--manifolds, and by Scott and Swarup
\cite{Scott-SwarupAlgAnnulus} when considering patterns in locally finite $2$--complexes.

A finite pattern in $C(X)^{(2)}$ is \textit{essential} if it separates
$C(X)^{(2)}$ into two unbounded pieces. In \cite{Niblo2}, Niblo shows that if
there is a finite essential pattern in $C(X)^{(2)}$, then either there is a
shortest such pattern $t$, or there is such a pattern $s$ arbitrarily close to
a vertex. A shortest essential pattern $t$ must be connected and so is a
track. Further such $t$ is automatically $G$--equivariant, by the cut and
paste arguments of \cite{Jaco-Rubinstein} and \cite{Scott-SwarupAlgAnnulus}.
The pattern $s$ can also be chosen to be connected, and $G$--equivariant, so
that in all cases, the existence of a finite essential pattern in $C(X)^{(2)}$
implies the existence of a $G$--equivariant finite essential track.

Now the intersection with $C(X)^{(2)}$ of a hyperplane in $C(X)$ is a finite
essential track, so the discussion of previous paragraph shows that
$C(X)^{(2)}$ contains a $G$--equivariant finite essential track $t$. The
essentiality of $t$ implies that $t$ yields a splitting of $G$ over the
stabilizer of $t$. Finally, in \cite{Niblo2}, Niblo shows that the stabilizer
of any finite track in $C(X)^{(2)}$ must be commensurable with $H$, so that
$G$ has a splitting over a subgroup commensurable with $H$. Further, if $v$
denotes the $V_{0}$--vertex of the algebraic regular neighbourhood of $X$ in
$G$ which encloses $X$, then this splitting is enclosed by $v$.

In the setting of this paper, almost exactly the same argument can be used.
Consider a group system $(G,\mathcal{S})$ of finite type such that $G$ is
finitely generated, and let $H$ be a finitely generated subgroup of $G$, such
that $e(G,H:\mathcal{S})\geq2$. Let $X$ be a nontrivial $\mathcal{S}$--adapted
$H$--almost invariant subset of $G$, which is in very good position, and
suppose that if $gX$ crosses $X$ then $g$ lies in $Comm_{G}(H)$. Let $C(X)$
denote the cubing constructed by Sageev from $X$. Applying the preceding
argument to $C(X)^{(2)}$ produces a splitting of $G$ over a subgroup
commensurable with $H$. But now we need to know in addition that this
splitting is adapted to $\mathcal{S}$. Let $S$ be a group in $\mathcal{S}$. If
$S$ fixes a vertex of $C(X)$, then any splitting of $G$ obtained from a
$G$--equivariant finite essential track in $C(X)^{(2)}$ must be adapted to
$S$. Conveniently Lemma \ref{cubinghasafixedvertex} implies that if $S$ is not
conjugate commensurable with a subgroup of $H$, then $S$ must fix a vertex of
$C(X)$. This leaves the special case when $S$ is conjugate commensurable with
a subgroup of $H$. In this case, it need not be true that $S$ fixes a vertex
of $C(X)^{(2)}$ (see Example \ref{exampleofnofixedvertexinC(Y)}), but this
does not matter as any splitting of $G$ over any subgroup commensurable with a
subgroup of $H$ is automatically adapted to $S$, by Lemma
\ref{adaptedtoSimpliesadaptedtoS'commS}. Thus in all cases, any splitting of
$G$ obtained by Niblo's arguments is adapted to $\mathcal{S}$, as required.
This completes our discussion of an alternative approach to proving Lemma
\ref{weakcrossingCCCenclosessplittingifGfg}.

So far in this section we have considered the algebraic regular neighbourhood
of a finite family of almost invariant subsets of a group system
$(G:\mathcal{S})$, which is known to exist by part 2) of Theorem
\ref{algregnbhdexists}. For the rest of this section, we will be interested in
the problem of trying to form the algebraic regular neighbourhood of an
infinite family of almost invariant subsets of $(G:\mathcal{S})$. For this
discussion, we need to restrict the groups considered. We consider a group
system $(G,\mathcal{S})$ of finite type, such that $G$ is finitely presented,
and each $S_{i}$ in $\mathcal{S}$ is finitely generated. We continue to assume
that $e(G:\mathcal{S})=1$, and that there is a non-negative integer $n$ such
that, for any $VPC(\leq n)$ subgroup $M$ of $G$, we have $e(G,M:\mathcal{S}%
)=1$. Let $H$ be a $VPC(n+1)$ subgroup of $G$, such that $e(G,H:\mathcal{S}%
)\geq4$. Let $Q(H)$ denote the collection of all nontrivial $\mathcal{S}%
$--adapted almost invariant subsets of $G$ which are over a subgroup of $G$
commensurable with $H$, and let $B(H)$ denote the collection of equivalence
classes of elements of $Q(H)$.

In Theorem 8.2 of \cite{SS2}, the authors used the result of Lemma
\ref{weakcrossingCCCenclosessplitting}, in the case when $n=0$ and
$\mathcal{S}$ is empty, to show that the Boolean algebra $B(H)$ is generated
by finitely many almost invariant sets which are associated to splittings of
$G$. An important role in their proof was played by an accessibility result,
Theorem 7.11 of \cite{SS2}, which is why they needed to assume that $G\ $is
finitely presented. Unfortunately the proof of Theorem 7.11 given in
\cite{SS2} is incorrect, but we gave a corrected statement and argument in
Theorem \ref{accessibilityinabsolutecase}. We will use an extension of this
accessibility result to generalise the proof of Theorem 8.2 of \cite{SS2} to
the setting of this paper.

The next results are not needed for the pr0of of the relative torus theorem proved in the next section. But they are needed for the discreteness of pretrees considered when forming regular neighbourhoods, when infinite families of almost invariant sets are involved. An alternate approach would be to use the relative torus theorem and accessibility results to form JSJ trees as in Guiradel and Levitt \cite{Guirardel-Levitt2}  to form trees of cylinders and see that the regular neighbourhoods considered are subdivisions of those trees of cylinders. This approach will be discussed in section 16.

\begin{theorem}
\label{Booleanalgebraisf.g.}Let $(G,\mathcal{S})$ be a group system of finite
type such that $G$ is finitely presented, each $S_{i}$ in $\mathcal{S}$ is
finitely generated, and $e(G:\mathcal{S})=1$. Let $n$ be a non-negative
integer, and suppose that, for any $VPC(\leq n)$ subgroup $M$ of $G$, we have
$e(G,M:\mathcal{S})=1$. Let $H$ be a $VPC(n+1)$ subgroup of $G$, and let
$Q(H)$ denote the collection of all nontrivial $\mathcal{S}$--adapted almost
invariant subsets of $G$ which are over a subgroup of $G$ commensurable with
$H$. Let $B(H)$ denote the collection of all equivalence classes of elements
of $Q(H)$. Then the Boolean algebra $B(H)$ is generated by finitely many
almost invariant sets which are associated to $\mathcal{S}$--adapted
splittings of $G$.
\end{theorem}

\begin{proof}
This result is proved in the same way as Theorem 8.2 of \cite{SS2}. We will
outline the arguments noting the changes from the proof of that theorem.

Note that all the crossings between elements of $B(H)$ must be weak, by
Theorem \ref{crossinginCCCsisallweakorallstrong}. The accessibility result of
Theorem \ref{accessibilityforsystemswhenGisfp} tells us that there is a finite
$\mathcal{S}$--adapted graph of groups decomposition $\mathcal{G}$ of $G$,
such that all edge groups of $\mathcal{G}$ are commensurable with $H$, and
$\mathcal{G}$ cannot be properly refined using a $\mathcal{S}$--adapted
splitting of $G$ over a subgroup which is commensurable with $H$. An
alternative way of expressing this condition is to say that if $G$ possesses a
$\mathcal{S}$--adapted splitting $\sigma$ over a subgroup commensurable with
$H$ which has intersection number zero with the edge splittings of
$\mathcal{G}$, then $\sigma$ is conjugate to one of these edge splittings. Let
$X_{1},\ldots,X_{n}$ denote almost invariant subsets of $G$ associated to the
edge splittings of $\mathcal{G}$, and let $H_{i}$ denote the stabiliser of
$X_{i}$. Let $A(H)$ denote the subalgebra of $B(H)$ generated over
$Comm_{G}(H)$ by the equivalence classes of the $X_{i}$'s. We will show that
$B(H)=A(H)$, which will complete the proof of the theorem.

Let $E$ denote the collection of translates of the $X_{i}$'s by elements of
$Comm_{G}(H)$. Let $Y$ be an element of $Q(H)$, so that $Y$ is a nontrivial
$\mathcal{S}$--adapted almost invariant subset of $G$ over a subgroup $K$
commensurable with $H$. We will show that the equivalence class of $Y$ lies in
the subalgebra $A(H)$ of $B(H)$. It will then follow that $B(H)=A(H)$, as required.

As in the proof of Theorem 8.2 of \cite{SS2}, $Y$ crosses only finitely many
elements of $E$. For the intersection number $i(H_{i}\backslash X{_{i}%
},K\backslash Y)$ is finite and is the number of double cosets $KgH_{i}$ such
that $gX_{i}$ crosses $Y$. If $gX_{i}$ crosses $Y$, it must do so weakly by
Theorem \ref{crossinginCCCsisallweakorallstrong} which also tells us that the
stabilisers of $gX_{i}$ and $Y$ are commensurable. Now, as in the proof of
Theorem 8.2 of \cite{SS2}, a simple calculation shows that $KgH_{i}$ is the
union of finitely many cosets $gH_{i}$ of $H_{i}$, so that there are only
finitely many translates of the $X_{i}$'s which cross $Y$. We conclude that
$Y$ crosses only finitely many of the translates of the $X_{i}$'s by $G$ and
that these must all lie in $E$.

If $Y$ crosses no elements of $E$, then $Y$ is enclosed by some vertex $v$ of
$\mathcal{G}$. In particular, the stabiliser $K$ of $Y$ is a subgroup of
$G(v)$. We claim that $Y\cap G(v)$ must be a trivial almost invariant subset
of $G(v)$. Assuming this claim, it follows, as in the proof of Theorem 8.2 of
\cite{SS2}, that $Y$ is equivalent to the union of a finite number of almost
invariant subsets of $G$ associated to edges of $\mathcal{G}$, so that the
equivalence class of $Y$ lies in the subalgebra $A(H)$ of $B(H)$, as required.
In order to prove our claim, we suppose that $Y\cap G(v)$ is a nontrivial
$K$--almost invariant subset of $G(v)$. Let $E(v)$ denote the family of
subgroups of $G(v)$ associated to the edges of $\mathcal{G}$ incident to $v$.
Note that $G(v)$ is finitely generated, that $E(v)$ is a finite collection of
subgroups of $G(v)$, and that $Y\cap G(v)$ is adapted to $E(v)$. Thus we can
apply Lemma \ref{weakcrossingCCCenclosessplitting} to the $K$--almost
invariant set $Y\cap G(v)$ in the group system $(G(v),E(v))$. This yields a
$E(v)$--adapted splitting $\tau$ of $G(v)$ over some subgroup commensurable
with $K$, which extends to a splitting $\sigma$ of $G$ over the same subgroup
such that $\sigma$ is enclosed by $v$. In particular, we can refine
$\mathcal{G}$ at $v$ by inserting an edge with associated splitting equal to
$\sigma$. Our choice of $\mathcal{G}$ now implies that $\sigma$ must be
conjugate to one of the edge splittings of $\mathcal{G}$, so that $\tau$ is
not a splitting of $Stab(v)$. This contradiction completes the proof of our claim.

As in the proof of Theorem 8.2 of \cite{SS2}, we can now show that $Y$ lies in
the Boolean algebra $A(H)$ by induction on the number of elements of $E$ which
cross $Y$. The induction step itself is unchanged. This completes the proof of
Theorem \ref{Booleanalgebraisf.g.}.
\end{proof}

Then as in Proposition 8.6 of \cite{SS2}, we deduce 

\begin{proposition} 
With notation as in 14.5, consider the CCCs of $\bar{E}$ which consist of elements of $Q(H)$. Then there is exactly one such CCC which is infinite, and there are finitely many (possibly zero) other CCCs all of which are isolated. The infinite CCC has stabiliser equal to $Comm_G(H)$ and encloses every element of $Q(H)$.
\end{proposition}

\begin{remark} 
Let $\mathcal{F}_{n+1}$ denote the equivalence classes of all relative almost invariant sets in $(G,\mathcal{S})$ and consider the pretree $P$ of all CCCs formed from $\mathcal{F}_{n+1}$. Then it follows as in sections 10 and 12 of \cite{SS2} that $P$ is discrete and we can form the regular neighbourhood.
\end{remark}

The general principle is that accessibilty leads to the finite generation of boolean algabras and this leads to the discreteness of the preetrees considered. This approach is also used in section 17.

\section{The Relative Torus Theorem\label{section:relativetorustheorem}}

In this section we will prove relative versions of the Torus Theorem which are
far more general than that proved in \cite{SSpdnold}. The results we obtain
are the following. Note that Theorem \ref{relativetorustheorem} does not
assume that $G$ is finitely generated. And even when $G$ is finitely
generated, Theorem \ref{relativetorustheorem} is more general than that proved
in \cite{SSpdnold}. In section \ref{multi-lengthdecompositions}, we prove an
even more general result, Theorem \ref{n-canonicalrelativetorustheorem}.

For clarity, we start by stating a special case of Theorem
\ref{relativetorustheorem}.

\begin{theorem}
\label{specialcaseofrelativetorustheoremforfgG}Let $(G,\mathcal{S})$ be a
group system of finite type, such that $e(G:\mathcal{S})=1$, and $G$ and each
$S_{i}$ in $\mathcal{S}$ are finitely generated. Suppose that at least one
$S_{i}$ is not $VPC$.

Let $n$ be a non-negative integer, and suppose that, for any $VPC(\leq n)$
subgroup $M$ of $G$, we have $e(G,M:\mathcal{S})=1$. If there is a nontrivial
$\mathcal{S}$--adapted almost invariant subset $X$ of $G$ over a $VPC(n+1)$
subgroup of $G$, then there is a $\mathcal{S}$--adapted splitting of $G$ over
some $VPC(n+1)$ subgroup of $G$.
\end{theorem}

Note that the above result is false if we remove the assumption that at least
one $S_{i}$ is not $VPC$. A simple example occurs when $G\ $is one of the
Euclidean triangle groups $\Delta(2,3,6)$, $\Delta(2,4,4)$ or $\Delta(3,3,3)$,
the family $\mathcal{S}$ is empty, and $n=0$. For $G$ has no splitting over
any subgroup, but as $G$ has a subgroup of finite index isomorphic to
$\mathbb{Z}\times\mathbb{Z}$, it has many nontrivial almost invariant sets
over infinite cyclic subgroups.

Now we state the general result.

\begin{theorem}
\label{relativetorustheorem}(The Relative Torus Theorem) Let $(G,\mathcal{S})$
be a group system of finite type, such that $e(G:\mathcal{S})=1$. Let $n$ be a
non-negative integer, and suppose that, for any $VPC(\leq n)$ subgroup $M$ of
$G$, we have $e(G,M:\mathcal{S})=1$. Suppose there is a nontrivial
$\mathcal{S}$--adapted almost invariant subset $X$ of $G$, which is over a
$VPC(n+1)$ subgroup of $G$. Then one of the following holds:

\begin{enumerate}
\item $G$ is $VPC(n+2)$.

\item There is a $\mathcal{S}$--adapted splitting of $G$ which is over some
$VPC(n+1)$ subgroup of $G$.

\item $G$ is of $VPCn$--by--Fuchsian type relative to $\mathcal{S}$.
\end{enumerate}
\end{theorem}

\begin{remark}
In case 3), Remark \ref{remarkondefnofVPC-by-FuchsianrelS} tells us that for
each peripheral subgroup $\Sigma$ of $G$, there is $S_{i}$ in $\mathcal{S}$
which is conjugate to a subgroup of $\Sigma$, but is not conjugate
commensurable with a subgroup of the $VPCn$ normal subgroup $K$ of $G$. As $K$
is normal in $G$, we can simply say that $S_{i}$ is not commensurable with a
subgroup of $K$. Further each $S_{i}$ in $\mathcal{S}$ is either commensurable
with a subgroup of $K$ or is conjugate to a subgroup of some peripheral
subgroup of $G$. Finally if $G\ $has no peripheral subgroups, so that the
orbifold quotient $G/K$ is closed, then each $S_{i}$ in $\mathcal{S}$ is
commensurable with a subgroup of $K$.

It follows that in cases 1) and 3), each $S_{i}$ in $\mathcal{S}$ must be
$VPC(\leq n+1)$. Thus if there is $S_{i}$ in $\mathcal{S}$ which is not
$VPC(\leq n+1)$, only case 2) can occur, which proves Theorem
\ref{specialcaseofrelativetorustheoremforfgG}. In particular, if there is
$S_{i}$ in $\mathcal{S}$ which is not finitely generated, only case 2) can
occur. Finally if $G$ is not finitely generated, the fact that $(G,\mathcal{S}%
)$ is a group system of finite type implies that there is $S_{i}$ in
$\mathcal{S}$ which is not finitely generated, so that again only case 2) can occur.

The proof below shows that part 2) of the conclusion of this theorem can be
refined to state that either there is a splitting of $G$ over a group
commensurable with $H_{X}$, or there is a graph of groups structure for $G$ in
which $X$ is enclosed by a vertex of $VPCn$--by--Fuchsian type relative to
$\mathcal{S}$.
\end{remark}

\begin{proof}
We will use the results of sections \ref{algregnbhds},
\ref{crossingofa.i.setsoverVPCgroups}, \ref{CCC'swithstrongcrossing} and
\ref{CCC'swithweakcrossing} to prove this result.

Let $\Gamma$ denote the algebraic regular neighbourhood of $X$ in $G$. This
exists by part 2) of Theorem \ref{algregnbhdexists}. Then $\Gamma$ has a
single $V_{0}$--vertex $v$. Theorem \ref{crossinginCCCsisallweakorallstrong}
shows that the crossings in a non-isolated CCC of $E(\mathcal{X})$ are either
all strong or are all weak. Theorem \ref{strongcrossingimpliesFuchsiantype}
shows that if $G$ is not $VPC(n+2)$, then each non-isolated CCC of
$E(\mathcal{X})$ in which all crossings are strong yields a $V_{0}$--vertex of
$\Gamma$ of $VPCn$--by--Fuchsian type relative to $\mathcal{S}$. Lemma
\ref{weakcrossingCCCenclosessplitting} shows that each non-isolated CCC of
$E(\mathcal{X})$ in which all crossings are weak yields a $V_{0}$--vertex of
$\Gamma$ which encloses a $\mathcal{S}$--adapted splitting of $G$ over a
$VPC(n+1)$ subgroup.

If $G$ is $VPC(n+2)$, we have case 1) of the theorem. In what follows we
assume that $G\ $is not $VPC(n+2)$.

If $v$ is isolated, then the two edges of $\Gamma$ which are incident to $v$
each determine a splitting of $G$ over the stabilizer of $X$, which is
$VPC(n+1)$. Further, as $\Gamma$ is $\mathcal{S}$--adapted, this splitting is
also $\mathcal{S}$--adapted. Thus we have case 2) of the theorem.

If $v$ arises from a CCC of $E(\mathcal{X})$ in which all crossings are weak,
then Lemma \ref{weakcrossingCCCenclosessplitting} shows that $v$ encloses a
$\mathcal{S}$--adapted splitting of $G$ over a $VPC(n+1)$ subgroup. Thus we
again have case 2) of the theorem.

Otherwise $v$ is of $VPCn$--by--Fuchsian type relative to $\mathcal{S}$. Now
there are two cases depending on whether $\Gamma$ is equal to $v$ or not. If
$\Gamma$ is not equal to $v$, there is an edge $e$ incident to $v$, and $e$
determines a splitting of $G$ over a $VPC(n+1)$ subgroup which is a peripheral
subgroup of $G(v)$. As before this splitting is $\mathcal{S}$--adapted, so
that again we have case 2) of the theorem.

If $\Gamma$ is equal to $v$, then $G=G(v)$ is of $VPCn$--by--Fuchsian type
relative to $\mathcal{S}$. Thus we have case 3) of the theorem.
\end{proof}

We remark that Houghton showed, at the end of section 3 of
\cite{Houghton-infinity ends}, that any countably infinite locally finite
group has infinitely many ends. Such a group cannot have a splitting over any
subgroup. Thus taking the product of such a group with a $VPC(n+1)$ group $H$
yields a non-finitely generated group $G$ such that $e(G,H)$ is infinite,
$e(G,M)=1$ for any $VPC(\leq n)$ subgroup $M$ of $G$, and $G$ has no splitting
over any $VPC(n+1)$ subgroup. Thus the hypothesis of Theorem
\ref{relativetorustheorem} that the system $(G,\mathcal{S})$ is of finite type
cannot be omitted.

It still seems that this construction should be a good source of
counterexamples to Theorem \ref{relativetorustheorem}, but such examples are
excluded by the hypothesis that $(G,\mathcal{S})$ is of finite type. To see
this, let $B$ denote a countably infinite locally finite group, let $H$ denote
a $VPC(n+1)$ group, and let $G=B\times H$. Then $e(G,H)=e(G/H)=e(B)$ is
infinite. Suppose that $\mathcal{S}$ is a family of subgroups of $B$ such that
the system $(B,\mathcal{S})$ is of finite type. Then the system
$(G,\mathcal{S})$ is also of finite type. For any $VPC(\leq n)$ subgroup $M$
of $G$, we have $e(G,M:\mathcal{S})\leq e(G,M)=1$, so that $e(G,M:\mathcal{S}%
)=1$. As the conclusion of Theorem \ref{relativetorustheorem} fails, it
follows that there is no nontrivial $\mathcal{S}$--adapted almost invariant
subset $X$ of $G$, which is over a $VPC(n+1)$ subgroup of $G$. In particular,
$e(G,H:\mathcal{S})=1$. Now $e(G,H:\mathcal{S})=e(B:\mathcal{S})$, so we
conclude that even though $e(B)$ is infinite, if $(B,\mathcal{S})$ is a system
of finite type, then $e(B:\mathcal{S})=1$.


\section[Adapted decompositions and the Guiradel-Levitt approach]{
Adapted JSJ decompositions  over VPC groups and the Guiradel-Levitt approach}

In part 2 of Theorem 9.12 we showed that if $(G,\cS)$ is a  group system of finite type and $\cF$ is a finite family of almost invariant sets over finitely generated subgroups, then \cF\ has an algebraic  regular neighbourhood in $G$. In Theorem 16.1 below we extend this to obtain \cS-adapted JSJ decompositions of $G$ analogues to Theorem 12.3 of \cite{SS2}. The proof is similar and the required results are in sections 11-14, in particular Theorem 11.32, Theorem 12.7, accessibility results of section 13, Theorem 14.5 and its corollary.

\begin{theorem}\label{thm16.1}
Let $(G,\cS)$ be a group system of finite type such that $G$ is finitely
presented relative to \cS, and $e(G : \cS) = 1$. Suppose that
$G$ has no \cS-adapted splitting over a $VPC(<n)$ subgroup, and let $\cF_n$ denote
the collection of equivalence classes of all nontrivial \cS-adapted almost invariant
subsets of G which are over $VPCn$ groups.

Then the family $\cF_n$ has a reduced algebraic regular neighbourhood $\Gamma_n$ in $G$.
$\Gamma_n$ is a minimal $\cS$-adapted bipartite graph of groups structure for $G$.
 For
each $V_1$-vertex $v$ of $\Gamma_n$, the associated group $G(v)$ is finitely generated.

Further, either $G$ is $VPC(n + 2)$ or each $V_0$-vertex $v$ of $\Gamma_n$ satisfies one of
the following conditions:
\begin{enumerate}
\item $v$ is isolated, so that $G(v)$ is $VPCn$.
\item $G(v)$ is of $VPC(n-1)$-by-Fuchsian type relative to \cS.
\item $G(v)$ is the full commensuriser $\comm_G(H)$ for some $VPCn$ subgroup $H$,
such that $e(G;H : \cS) \ge 2$.
\end{enumerate}
$\Gamma_n$ consists of a single vertex if and only if $\cF_n$ is empty, or G itself 
satisfies
one of the above three conditions.
\end{theorem}

More general theorems are proved in the next section.

With the relative torus theorem available Theorem 16.1 can be more quickly deduced from the methods of Guiradel and Levitt in \cite{Guirardel-Levitt2}. We next outline their procedure in the absolute case. Note their notation of $V_0$ and $V_1$ vertices differs from ours. We follow our notation as in Theorem 16.1 above and the theorems in section 12 of \cite{SS2}.

In what follows $T$ with be the JSJ-tree for $G$ as described by Dunwoody and Swenson. We note that this includes among its vertices all quadratically hanging (QH) subgroups of $G$ up to conjugacy. The edge groups are all \vpcn and thus the vertex groups are also (almost) finitely presented.

A cylinder in $T$ is an equivalence class of edges up to commensurability and is a subtree $Y$ of $T$. The tree of cylinders $T_c$ is constructed as follows:
$V_1(T_c)$ is the set of vertices of $T$ belonging to at least two cylinders and $V_0(T_c)$ is the set of cylinders of $T$. An edge is $(v,y)$ where
$v\in V_1(T_c)$, $y\in V_0(T_c)$ where $v\in Y$.

Now the regular neighbourhood in this setup is constructed as follows:
Consider the edges $e$ in $Y$ and the almost invariant sets $Z_e$, $Z_e^*$ (a base point $v_0$ is chosen) defined by $e$. If \C\ is the commensurabilty class of edges groups of $Y$, $\mathcal{B}_\mathcal{C}$ is the Boolean algebra generated by  $Z_e$. And $\mathcal{B}=\cup_\mathcal{C} \mathcal{B}_\mathcal{C}$ as $\mathcal{C}$ varies over the commensurability classes of edge groups of cylinders $Y$ of $T$. The cross connected components $\mathcal{H}$ of $\mathcal{B}$ are defined as in the previous sections. A \df{star} is a subset $\Sigma\subset\mathcal{H}$ such that if $C, C'\in\Sigma$ there is no $C''\in\mathcal{H}$ between $C$ and $C'$. Now $RN(\mathcal{B})$ is the bipartite graph with $\mathcal{H}\cup\mathcal{S}$ ($\mathcal{S}R$ is the set of stars) as vertex set and edges $(C,\Sigma)$ where $C\in\Sigma$.

One of the main results of [20, Thm 3.3] is that $RN(\B)$ is equvariantly isomorphic to a subdivision of $T_c$.  In Theorem 3.11 they extend this to the almost invariant sets defined by the $QH$ vertices of $T$, denoted by $QH(T)$. In Theorem 4.2 they show that the equivalence class of any almost invariant set over a \vpcn\ group in $G$ is in $\B(T) \cup QH(T)$. 

We now sketch some of the ideas in the proof of Theorem 3.3. Clearly there is a map from the $V_0$ vertices of $RN(\B)$ to the $V_0$ vertices of $T_c$. The ambiguity is clarified by the following construct:

A subset $S$ of the vertices $V(T)$ of T (often considered together with the edges in $T$ which join the vertices of $S$) is called a \df{special forest} based in $Y$ if $\delta S$ is finite and contained in $Y$ ($Y$ a cylinder). Such a forest defines an almost invariant set $X_S$ over 
$H=\cap_{e\in\delta S} \stab(e)$ by $X_S=\{ g\| gv_0\in S\}$ with a base vertex $v_0$. Just as in the case of an edge $X_S$ is an almost invariant set over $H\in\mathcal{C}$. They first show that $X\to X_S$ defines an isomorphism of Boolean algebras. We also have $S$ is finite iff $X_S$ is $H$-finite. Moreover, $X_S$, $X_{S'}$ intersect iff S, S' are based in the same cylinder.

To consider further this relation they consider the number of points $\#\partial Y$ in the boundary of $Y$ in $T$. This number is $\ge2$ as $Y$ is the convex hull of $\partial Y$. The cases when  $\#\partial Y$ is two or three are easy. They show that if $\#\partial Y>3$
there is exactly one non-isolated \ccc. And there are possibly several isolated ones over H$\in \mathcal{C}$. These come from points $v$ of $\partial Y$ which have finite valance in $Y$. Such a pair $(v,Y)$ is subdivided with a vertex at the midpoint and the isolated \ccc corresponding to $v$ is mapped to this midpoint. A similar case occurs when $\#\partial Y=3$. This describes rather precisely how the subdivision occurs.

This procedure, rather than that in \cite{SS2} gives some extra information. Notice that $V_1$-vertices of $T_c$ remain the same though the edge groups change. Thus we have

\begin{proposition}
The $V_1$-vertices of RN have  finitely presented stabilisers relative to their edge groups.
\end{proposition}

As mentioned before, their proof extends to that of Theorem 16.1 using the methods of the relative torus theorem.


\section{Multi-length adapted JSJ
decompositions}
\label{multi-lengthdecompositions}

In \cite{SS2}, Scott and Swarup considered multi-length JSJ decompositions. In
this section we will recall their results, discuss adapted versions of these
decompositions, and also discuss how to deal with multi-length JSJ
decompositions in general.

We start by recalling some of their terminology. Let $G$ be a one-ended
finitely generated group and let $X$ be a nontrivial almost invariant subset
over a subgroup $H$ of $G$. For $n\geq1$, they defined $X$ to be
$n$\textsl{--canonical} if $X$ crosses no nontrivial $K$--almost invariant
subset of $G$, for which $K$ is $VPC$ of length at most $n$. Let $G$ be a
one-ended, finitely presented group which does not split over $VPC$ subgroups
of length $<n$, and let $\mathcal{F}_{n,n+1}$ denote the collection of
equivalence classes of all nontrivial almost invariant subsets of $G$ which
are over a $VPCn$ subgroup, together with the equivalence classes of all
$n$--canonical almost invariant subsets of $G$ which are over a $VPC(n+1)$
subgroup. In Theorem 13.12 of \cite{SS2}, they showed that the family
$\mathcal{F}_{n,n+1}$ of almost invariant subsets of $G$ had an algebraic
regular neighbourhood $\Gamma_{n,n+1}$ in $G$.

Before continuing it seems worth discussing the following natural question. We
will say that $X$ is $n$--\textsl{canonical for splittings} if $X$ crosses no
$K$--almost invariant set which is associated to a splitting of $G$, for which
$K$ is $VPC$ of length at most $n$. It is trivial that if $X$ is
$n$--canonical it is also $n$--canonical for splittings. The question is
whether or when the converse holds, so that the concepts are equivalent. The
following simple example shows that when $n=1$, these concepts are not
equivalent for almost invariant sets over a $VPC1$ group, nor for almost
invariant sets over a $VPC2$ group.

\begin{example}
Let $M$ and $N$ be closed orientable $3$--manifolds, each of which is the
product of a surface, not $S^{2}$, with the circle, and construct a space $W$
by gluing $M$ to $N$ along a fibre of each. Let $G$ denote $\pi_{1}(W)$, and
let $H$ denote the infinite cyclic fundamental group $H$ of any fibre of $M$
or $N$. Thus $G\ $splits over $H$, and we let $\sigma$ denote this splitting.
Clearly $G$ normalizes and hence commensurises $H$. It is immediate that
$\Gamma_{1}(G)$ is simply a $V_{0}$--vertex of commensuriser type labelled
$G$. It follows from Remark \ref{XnotcrossinganyXiisenclosedbysomeV1-vertex}
that any almost invariant subset over $G$ over a finitely generated subgroup
cannot be $1$--canonical. However we will give two examples of almost
invariant subsets of $G$ which are $1$--canonical for splittings, showing that
the concepts of $n$--canonical for splittings and $n$--canonical are distinct,
when $n=1$.

Let $X$ denote a $H$--almost invariant subset of $G$ associated to the
splitting $\sigma$, and let $\Gamma$ denote the associated graph of groups
structure with one edge. Also let $T\ $be an essential vertical torus in $M$,
disjoint from $M\cap N$, and denote $\pi_{1}(T)$ by $A$. Thus $A$ is free
abelian of rank $2$ and contains $H$. As $e(\pi_{1}(M),A)=2$, there is a
unique (up to equivalence and complementation) nontrivial $A$--almost
invariant subset $Y$ of $\pi_{1}(M)$. As $H$ is a subgroup of $A$, it follows
that $Y$ is adapted to $H$. Hence $Y$ extends to an $A$--almost invariant
subset $Z$ of $G$, which is enclosed by a vertex of $\Gamma$ and so does not
cross any translate of $X$.

We claim that both $X$ and $Z$ are $1$--canonical for splittings. To prove
this claim it suffices to show that the splitting $\sigma$ of $G$ over $H$
given by the construction of $G$ is the only splitting of $G$ over a $VPC1$
subgroup (up to conjugacy). For by our choices of $X$ and $Z$, neither crosses
this splitting.

Now we can prove the claim. Note that $H$ is a maximal cyclic subgroup of $G$,
as it is a maximal cyclic subgroup of $\pi_{1}(M)$ and of $\pi_{1}(N)$. Now
suppose that $G$ splits over a $VPC1$ subgroup $K$, and consider the
corresponding $G$--tree $T$. Note that as $G$ is torsion free, $K\ $must also
be infinite cyclic. As $\pi_{1}(M)$ and $\pi_{1}(N)$ do not split over any
trivial or infinite cyclic subgroup, it follows that they each fix a vertex of
$T$. As $G$ does not fix any vertex of $T$, we have distinct vertices $v$ and
$w$ of $T$ such that $\pi_{1}(M)$ fixes $v$, and $\pi_{1}(N)$ fixes $w$. Let
$\lambda$ denote the edge path in $T$ joining $v$ and $w$. As $G$ is generated
by the stabilizers of $v$ and $w$, it follows that $\lambda$ injects into the
quotient graph $G\backslash T$, and so consists of a single edge. As $H$ fixes
both $v$ and $w$, it also fixes $\lambda$. As any edge stabiliser of $T$ is
infinite cyclic, the fact that $H$ is a maximal cyclic subgroup of $G$ implies
that $H$ is equal to the stabilizer of $\lambda$, so that $H=K$. The claim follows.
\end{example}

Note that the existence of such examples allows one to define a natural
refinement of the $JSJ$ decomposition $\Gamma_{n}$ of $G$. For let $\sigma$ be
a splitting of $G$ over a $VPCn$ subgroup which is $n$--canonical for
splittings but not $n$--canonical. By the definition of $\Gamma_{n}$, the
splitting $\sigma$ is enclosed by some $V_{0}$--vertex $v$ of $\Gamma_{n}$, so
we can refine $\Gamma_{n}$ at $v$ by inserting an edge whose associated
splitting of $G$ is $\sigma$. We can repeat for other such splittings, and
this process must terminate by accessibility, as proved in Theorem
\ref{accessibilityinabsolutecase}. The result is a natural and canonical
refinement of $\Gamma_{n}$.

We define the adapted analogue of $n$--canonical as follows.

\begin{definition}
\label{defnofn-canonicalrelativetoS}Let $(G,\mathcal{S})$ be a group system of
finite type, and let $X$ be a nontrivial $\mathcal{S}$--adapted almost
invariant subset of $G$ over a subgroup $H$. If $n\geq1$, then $X$ is
$n$\textsl{--canonical relative to }$\mathcal{S}$ if $X$ crosses no nontrivial
$\mathcal{S}$--adapted $K$--almost invariant subset of $G$, for which $K$ is
$VPC(\leq n)$.
\end{definition}

Now we can obtain the following analogue for group systems of Theorem 13.12 of
\cite{SS2}. We will briefly discuss the arguments needed after the statement.

\begin{theorem}
\label{JSJfortwosuccessivelengths}\label{thm17.3}
Let $(G,\mathcal{S})$ be a group system of
finite type such that $G$ is finitely presented, each $S_{i}$ in $\mathcal{S}$
is finitely generated, and $e(G:\mathcal{S})=1$. Let $n$ be a non-negative
integer, and suppose that, for any $VPC(<n)$ subgroup $M$ of $G$, we have
$e(G,M:\mathcal{S})=1$. $L$et $\mathcal{F}_{n,n+1}$ denote the collection of
equivalence classes of all nontrivial $\mathcal{S}$--adapted almost invariant
subsets of $G$ which are over a $VPCn$ subgroup, together with the equivalence
classes of all nontrivial $\mathcal{S}$--adapted almost invariant subsets of
$G$ which are over a $VPC(n+1)$ subgroup and are $n$--canonical relative to
$\mathcal{S}$.

Then the family $\mathcal{F}_{n,n+1}$ has an algebraic regular neighbourhood
$\Gamma_{n,n+1}$ in $G$.

$\Gamma_{n,n+1}$ is a minimal $\mathcal{S}$--adapted bipartite graph of groups
structure for $G$. Further, either $G$ is $VPC(n+2)$, or each $V_{0}$--vertex
$v$ of $\Gamma_{n,n+1}$ satisfies one of the following conditions:

\begin{enumerate}
\item $v$ is isolated, so that $G(v)$ is $VPC$ of length $n$ or $n+1$.

\item $v$ is of $VPC$-by-Fuchsian type relative to $\mathcal{S}$, where the
$VPC$ group has length $n-1$ or $n$.

\item $G(v)$ is the full commensuriser $Comm_{G}(H)$ for some $VPC$ subgroup
$H$ of length $n$ or $n+1$, such that $e(G,H:\mathcal{S})\geq2$.
\end{enumerate}

$\Gamma_{n,n+1}$ consists of a single vertex if and only if $\mathcal{F}%
_{n,n+1}$ is empty, or $G$ itself satisfies one of the above three conditions.
\end{theorem}

The proof of Theorem \ref{thm17.3} used the arguments in
sections \ref{crossingofa.i.setsoverVPCgroups}, \ref{CCC'swithstrongcrossing}
and \ref{CCC'swithweakcrossing}. We need to discuss how to generalise those
arguments to cover the case of $n$--canonical almost invariant subsets of $G$
over $VPC(n+1)$ subgroups.

In section \ref{crossingofa.i.setsoverVPCgroups}, the main result was Theorem
\ref{crossinginCCCsisallweakorallstrong} on the symmetry of strong crossing
and of weak crossing of $\mathcal{S}$--adapted almost invariant subsets of $G$
which are over a $VPC(n+1)$ subgroup. The analogous result we need in this
section is the following.

\begin{theorem}
\label{crossingintwostageCCCsisallweakorallstrong}Let $(G,\mathcal{S})$ be a
group system of finite type, such that $G$ and each $S$ in $\mathcal{S}$ are
finitely generated and $e(G:\mathcal{S})=1$. Let $n$ be a non-negative
integer, and suppose that, for any $VPC(<n)$ subgroup $M$ of $G$, we have
$e(G,M:\mathcal{S})=1$. Let $\{X_{\lambda}\}_{\lambda\in\Lambda}$ be a family
of nontrivial $\mathcal{S}$--adapted almost invariant subsets of $G$ which are
either over a $VPCn$ subgroup, or are over a $VPC(n+1)$ subgroup and are
$n$--canonical relative to $\mathcal{S}$.

Let $E$ denote the set of all translates of the $X_{\lambda}$'s and their
complements, and consider the collection of cross-connected components (CCCs)
of $E$ as in the construction of regular neighbourhoods in chapter 3 of
\cite{SS2}. Then the following statements hold:

\begin{enumerate}
\item The crossings in a non-isolated CCC of $E$ are either all strong or are
all weak.

\item In a non-isolated CCC $\Phi$ with all crossings weak, the stabilisers of
the elements of $E$ which lie in $W$ are all commensurable. If $H$ denotes
such a stabiliser, then $e(G,[H]:\mathcal{S})\geq4$.

\item In a non-isolated CCC $\Phi$ with all crossings strong, if $H$ is the
stabiliser of an element of $E$ which lies in $\Phi$, then
$e(G,[H]:\mathcal{S})=2$.
\end{enumerate}
\end{theorem}

Let $W$ denote a CCC of $E$ in the above theorem. As any $X_{\lambda}$ in $W$
which is over a $VPC(n+1)$ subgroup cannot cross any $X_{\mu}$ which is over a
$VPCn$ subgroup, it follows that $W$ consists entirely of $X_{\lambda}$'s over
$VPCn$ subgroups, or entirely of $X_{\lambda}$'s over $VPC(n+1)$ subgroups. In
the first case, the result of Theorem
\ref{crossingintwostageCCCsisallweakorallstrong} follows immediately from
Theorem \ref{crossinginCCCsisallweakorallstrong}. However, in the second case,
this is not immediate as the hypothesis of Theorem
\ref{crossinginCCCsisallweakorallstrong} that, for any $VPC(\leq n)$ subgroup
$M$ of $G$, we have $e(G,M:\mathcal{S})=1$, has been weakened in Theorem
\ref{crossingintwostageCCCsisallweakorallstrong} by replacing $VPC(\leq n)$ by
$VPC(<n)$. We need to use the additional hypothesis of Theorem
\ref{crossingintwostageCCCsisallweakorallstrong} that if $X_{\lambda}$ in $E$
is over a $VPC(n+1)$ subgroup, it is $n$--canonical relative to $\mathcal{S}$.

We need to consider the various lemmas in section
\ref{crossingofa.i.setsoverVPCgroups} which were used to prove Theorem
\ref{crossinginCCCsisallweakorallstrong} and discuss how to modify them to
provide a proof of Theorem \ref{crossingintwostageCCCsisallweakorallstrong}.
Many of the lemmas used to prove Theorem
\ref{crossinginCCCsisallweakorallstrong} do not use the hypothesis that, for
any $VPC(\leq n)$ subgroup $M$ of $G$, we have $e(G,M:\mathcal{S})=1$, and so
they do not need any modification.

The first lemma in section \ref{crossingofa.i.setsoverVPCgroups} which needs
modifying is Lemma \ref{weak-weakcrossing}. We give below the statement needed
for this section. This result is the generalization to the setting of group
systems of Proposition 13.5 of \cite{SS2}, and the proof we give here is also
a generalization of the proof of Proposition 13.5 of \cite{SS2} to our setting.

\begin{lemma}
\label{weak-weakcrossingforn-canonical}Let $(G,\mathcal{S})$ be a group system
of finite type, such that $G$ and each $S$ in $\mathcal{S}$ are finitely
generated, and $e(G:\mathcal{S})=1$. Let $H$ and $K$ be $VPC(n+1)$ subgroups
of $G$, and let $L$ denote $H\cap K$. Let $X$ and $Y$ be nontrivial
$\mathcal{S}$--adapted almost invariant subsets of $G$ over $H$ and $K$
respectively, such that $X$ and $Y$ are $n$--canonical relative to
$\mathcal{S}$. Suppose that, for any subgroup $M$ of $H$ or of $K$ which is
$VPC(<n)$, we have $e(G,M:\mathcal{S})=1$. Suppose further that $X$ crosses
$Y$ weakly and $Y$ crosses $X$ weakly. Then $H$ and $K$ must be commensurable,
and $e(G,[H]:\mathcal{S})\geq4$.
\end{lemma}

\begin{proof}
The first part of the proof of Lemma \ref{weak-weakcrossing} applies unchanged
to show that one corner of the pair $(X,Y)$ is a nontrivial $\mathcal{S}%
$--adapted $L$--almost invariant set. Thus $e(G,L:\mathcal{S})\geq2$. Now the
hypotheses of the lemma imply that $L$ must be $VPCn$ or $VPC(n+1)$. In the
second case, $L$ has finite index in $H$ and in $K$, so that $H$ and $K$ are
commensurable, as required. It follows that each corner of the pair $(X,Y)$ is
a nontrivial $\mathcal{S}$--adapted $L$--almost invariant set, so that
$e(G,L:\mathcal{S})\geq4$. It follows that $e(G,[H]:\mathcal{S})\geq4$, as required.

It remains to show that the case where $L$ is $VPCn$ leads to a contradiction
using the hypothesis that $X$ and $Y$ are $n$--canonical relative to
$\mathcal{S}$. Note that unlike the situation in Lemma \ref{weak-weakcrossing}%
, the fact that $e(G,L:\mathcal{S})\geq2$ does not force $L$ to be $VPC(n+1)$.
We start by applying Lemma \ref{Houghtonlemma} to the $VPCn$ subgroup $L$ of
$H$. This tells us that $H$ has a finite index subgroup $J$, and $L$ has a
subgroup $M$ of index $1$ or $2$, such that $J$ contains $M$ as a normal
subgroup, and $J/M$ is infinite cyclic. Recall that we are working in a
relative Cayley graph $\Gamma(G,\mathcal{S})$ of the given group system. In
what follows we will also use the associated angle metric on $G$. See
Definition \ref{defnofanglemetric}.

As $X$ crosses $Y$ weakly and $Y$ crosses $X$ weakly, without loss of
generality, we can assume that $H\cap Y$ is $K$--finite, and that $X\cap K$ is
$H$--finite. Thus $H\cap Y$ and $X\cap K$ are both $L$--finite, and the corner
$X\cap Y$ of the pair $(X,Y)$ is a nontrivial $\mathcal{S}$--adapted
$L$--almost invariant set. We denote this corner by $Z$.

We will use the terminology in Notation \ref{notation} and Definition
\ref{defnofstandardaugmentation}, and the results of Lemmas
\ref{propertiesofXhat} and \ref{AlmostinvariantsetbecomesconnectedinGL}. Thus
if we let $\widehat{W}$ denote the standard augmentation in $\widehat{\Gamma}$
of a nontrivial $\mathcal{S}$--adapted almost invariant subset $W$ of $G$ over
a finitely generated subgroup of $G$, then $\widehat{W}\cap G$ is equivalent
to $W$, and there is $N$ such that $\partial^{N}\widehat{W}$ is connected in
$\Gamma^{N}$, and $\widehat{W}$ is connected in $\widehat{\Gamma}^{N}$, i.e.
the graphs $\Gamma^{N}[\partial^{N}\widehat{W}]$ and $\widehat{\Gamma}%
^{N}[\widehat{W}]$ are each connected. Further these graphs are connected for
all integers greater than $N$. As $J$, $K$ and $M$ are all finitely generated,
we can apply Lemma \ref{AlmostinvariantsetbecomesconnectedinGL} to $X$, $Y$
and $Z$. Further there is $N$ such that all three of $\partial^{N}\widehat{X}%
$, $\partial^{N}\widehat{Y}$ and $\partial^{N}\widehat{Z}$ are connected in
$\Gamma^{N}$, so that all three of $\widehat{X}$, $\widehat{Y}$ and
$\widehat{Z}$ are connected in $\widehat{\Gamma}^{N}$. As $Z\subset X$ and
$L\subset H$, it is easy to check that $\widehat{Z}\subset\widehat{X}$.
Similarly we also have $\widehat{Z}\subset\widehat{Y}$.

Consider the images of all these connected graphs in the quotient
$M\backslash\widehat{\Gamma}^{N}$. Let $c$ be an element of $J$ which maps to
a generator of $J/M$, and let $C$ denote the infinite cyclic subgroup of
$J\ $generated by $c$. Then $C$ acts on $M\backslash\widehat{\Gamma}^{N}$
preserving the image of $X$. Let $x$ denote a point of this image. As $H\cap
Y$ is $K$--finite, the distance of $c^{k}x$ from the image of $Y$ tends to
infinity as $k$ tends to $\pm\infty$. As the image of $\partial^{N}%
\widehat{Z}$ in $M\backslash\widehat{\Gamma}^{N}$ is finite, it follows that
there is $k$ such that $c^{k}(\partial^{N}\widehat{Z})$ does not meet the
image of $\widehat{Y}$. Now consider the action of $C$ on $\widehat{\Gamma
}^{N}$. Then $c^{k}(\partial^{N}\widehat{Z})$ does not meet $\widehat{Y}$. In
particular $c^{k}(\partial^{N}\widehat{Z})$ does not meet $\partial
^{N}\widehat{Y}$. As $\partial^{N}\widehat{Y}$ is connected in $\Gamma^{N}$
and does not meet $c^{k}(\partial^{N}\widehat{Z}),$ either $\partial
^{N}\widehat{Y}$ is contained in $c^{k}\widehat{Z}$, or $\partial
^{N}\widehat{Y}\cap c^{k}\widehat{Z}$ is empty. In the first case, we would
have $\partial^{N}\widehat{Y}\subset c^{k}\widehat{Z}\subset c^{k}%
\widehat{X}=\widehat{X}$, which would contradict the assumption that $X$ and
$Y$ cross. It follows that $\partial^{N}\widehat{Y}\cap c^{k}\widehat{Z}$ is
empty. As $c^{k}(\partial^{N}\widehat{Z})$ does not meet $\widehat{Y}$, it
follows that $\widehat{Y}\cap c^{k}\widehat{Z}$ is empty. Thus we have
$Z\subset Y$, and $c^{k}Z\subset Y^{\ast}$. As $C$ normalizes $M$, the set
$Z\cup c^{k}Z$ is a $M$--almost invariant subset of $G$ which crosses $Y$,
contradicting the hypothesis that $Y$ is $n$--canonical relative to
$\mathcal{S}$. This contradiction completes the proof of the lemma.
\end{proof}

The next lemma which needs modifying is Lemma \ref{aisethasoneS-adaptedK-end},
which contains the hypothesis that $e(G,K;\mathcal{S})=1$. Here is the
statement we need in this section.

\begin{lemma}
\label{aisethasoneadaptedK-endifn-canonical}Let $(G,\mathcal{S})$ be a group
system of finite type, such that $G$ and each $S$ in $\mathcal{S}$ are
finitely generated, and $e(G:\mathcal{S})=1$. Let $H$ be a $VPC(n+1)$ subgroup
of $G$, and let $X$ be a nontrivial $\mathcal{S}$--adapted $H$--almost
invariant subset of $G$. Suppose there is a nontrivial $\mathcal{S}$--adapted
$H^{\prime}$--almost invariant subset $Y$ of $G$ such that $X$ crosses $Y$
strongly, and let $K$ denote $H\cap H^{\prime}$. Finally suppose that $X$ and
$Y$ are $n$--canonical relative to $\mathcal{S}$.

Let $\widehat{\Gamma}=\Gamma(G,\mathcal{S})$ be a relative Cayley graph of
$(G,\mathcal{S})$, and let $\widehat{X}$ denote the standard augmentation of
$X$. Then

\begin{enumerate}
\item If $\partial^{R}\widehat{X}$ is connected in $\Gamma^{R}$, so that the
subgraph $\Gamma^{R}[\partial^{R}\widehat{X}]$ of $\Gamma^{R}$ is connected,
the quotient graph $K\backslash\Gamma^{R}[\partial^{R}\widehat{X}]$ has two ends.

\item the number of $\mathcal{S}$--adapted $K$--ends of $\widehat{X}$ is equal
to $1$.
\end{enumerate}
\end{lemma}

\begin{proof}
The statement of Lemma \ref{aisethasoneadaptedK-endifn-canonical} differs from
the statement of Lemma \ref{aisethasoneS-adaptedK-end} as follows. Lemma
\ref{aisethasoneadaptedK-endifn-canonical} assumes that $G$ and each $S$ in
$\mathcal{S}$ are finitely generated, and that $X$ and $Y$ are $n$--canonical
relative to $\mathcal{S}$. Lemma \ref{aisethasoneS-adaptedK-end} does not
contain any of these assumptions. Instead Lemma
\ref{aisethasoneS-adaptedK-end} contains the assumption that
$e(G,[K]:\mathcal{S})=1$.

In the proof of Lemma \ref{aisethasoneS-adaptedK-end}, the assumption that
$e(G,[K]:\mathcal{S})=1$ is used at just one point, in the third paragraph of
the proof of part 2) of the lemma. The argument in the lemma shows that
$Y_{3}\cap G$ is a $\mathcal{S}$--adapted $K^{\prime}$--almost invariant
subset of $G$, for some $K^{\prime}\in\lbrack K]$. It follows from the
assumption that $e(G,[K]:\mathcal{S})=1$ that $Y_{3}\cap G$ is a trivial such
set, and hence $K$--finite, which contradicts the initial assumptions on the
$Y_{i}$'s. Thus the assumption that $e(G,[K]:\mathcal{S})=1$ is far more
general than needed. We just need to be able to show that $Y_{3}\cap G$ is trivial.

Thus we can reduce the proof of Lemma
\ref{aisethasoneadaptedK-endifn-canonical} to that of Lemma
\ref{aisethasoneS-adaptedK-end} by using the hypotheses of Lemma
\ref{aisethasoneadaptedK-endifn-canonical} to prove that, for any $K^{\prime
}\in\lbrack K]$, if $Z$ is a $\mathcal{S}$--adapted $K^{\prime}$--almost
invariant subset of $G$ such that $Z\subset X$, then $Z$ must be trivial.

Let $\widehat{Z}$ denote the standard augmentation of $Z$. By Lemma
\ref{AlmostinvariantsetbecomesconnectedinGL}, there is $N$ such that
$\partial^{N}\widehat{Z}$ is connected in $\Gamma^{N}$, and $\widehat{Z}$ is
connected in $\widehat{\Gamma}^{N}$, i.e. the graphs $\Gamma^{N}[\partial
^{N}\widehat{Z}]$ and $\widehat{\Gamma}^{N}[\widehat{Z}]$ are each connected.
Let $W$ denote the connected graph $\widehat{\Gamma}^{N}[\widehat{Z}]$, and
let $dW$ denote the connected graph $\Gamma^{N}[\partial^{N}\widehat{Z}]$. As
$\partial Z$ is $K^{\prime}$--finite, so are $\partial^{N}\widehat{Z}$ and
$\Gamma^{N}[\partial^{N}\widehat{Z}]=dW$. By increasing $N$ if needed, we can
assume in addition that $\partial^{N}\widehat{Y}$ is connected in $\Gamma^{N}%
$, and $\partial^{N}\widehat{Y^{\ast}}$ is connected in $\Gamma^{N}$.

By replacing $H$ by a subgroup of finite index, and replacing $K^{\prime}$ by
its intersection with this subgroup, which we denote by $K$, we can assume
that $K$ is normal in $H$ with infinite cyclic quotient. Let $c$ denote an
element of $H$ which projects to a generator of $H/K$. As $X$ crosses
$Y\ $strongly, we know that the sequences $\{d(c^{k},\partial Y):k\geq1\}$ and
$\{d(c^{-k},\partial Y):k\geq1\}$ tend to infinity, one in $Y$, the other in
$Y^{\ast}$. By replacing $c$ by $c^{-1}$ if needed, we can assume that the
sequence $\{c^{k}:k\geq1\}$ eventually lies in $Y$. As $dW$ is $K$--finite, it
follows that the sequence $c^{k}(dW)$ eventually lies in $\widehat{Y}%
-\partial^{N}\widehat{Y}$, and that the sequence $c^{-k}(dW)$ eventually lies
in $\widehat{Y^{\ast}}-\partial^{N}\widehat{Y^{\ast}}$. We choose $k$ large
enough so that $c^{k}(dW)\subset\widehat{Y}-\partial^{N}\widehat{Y}$ and
$c^{-k}(dW)\subset\widehat{Y^{\ast}}-\partial^{N}\widehat{Y^{\ast}}$. As
$c^{k}(dW)\subset\widehat{Y}-\partial^{N}\widehat{Y}$, the connectedness of
$dW$, $W$ and $\Gamma^{N}[\partial^{N}\widehat{Y}]$ implies that either
$c^{k}(W)\subset\widehat{Y}$ or that $\Gamma^{N}[\partial^{N}\widehat{Y}%
]\subset c^{k}(W)$. As $c$ preserves $X$, the second case yields the
inclusions $\partial^{N}\widehat{Y}\subset c^{k}W\subset\widehat{X}$
contradicting the fact that $Y\ $crosses $X$ strongly. It follows that
$c^{k}(W)\subset\widehat{Y}$, so that $c^{k}(Z)\subset Y$. Similarly we see
that $c^{-k}(Z)\subset Y^{\ast}$. Thus $c^{k}Z\cup c^{-k}Z$ is a nontrivial
$\mathcal{S}$--adapted $K$--almost invariant subset of $G$ which crosses $Y$,
which contradicts the hypothesis that $Y$ is $n$--canonical, as required.
\end{proof}

The next lemma which needs modifying is Lemma \ref{strongcrossingissymmetric},
as it has the assumption that $e(G,M:\mathcal{S})=1$ if $M$ is any subgroup of
$H$ of infinite index. But this assumption was only needed in order to quote
Lemma \ref{aisethasoneS-adaptedK-end}. Thus the only modification needed to
the proof of the revised version of Lemma \ref{strongcrossingissymmetric} is
to quote Lemma \ref{aisethasoneadaptedK-endifn-canonical} instead of Lemma
\ref{aisethasoneS-adaptedK-end}.

Lemma \ref{XcrossesYstronglyimpliese=2} needs a similarly trivial modification.

This completes our discussion of how to modify the lemmas in section
\ref{crossingofa.i.setsoverVPCgroups} used to prove Theorem
\ref{crossinginCCCsisallweakorallstrong}. Once we have proved Theorem
\ref{crossingintwostageCCCsisallweakorallstrong}, the arguments in section
\ref{CCC'swithstrongcrossing} apply essentially unchanged to the
$n$--canonical case.

At this point we can prove the following generalization of Theorem
\ref{relativetorustheorem}.

\begin{theorem}
\label{n-canonicalrelativetorustheorem}(The $n$--canonical Relative Torus
Theorem) Let $(G,\mathcal{S})$ be a group system of finite type, such that
$e(G:\mathcal{S})=1$. Let $n$ be a non-negative integer, and suppose that, for
any $VPC(<n)$ subgroup $M$ of $G$, we have $e(G,M:\mathcal{S})=1$. Suppose
there is a nontrivial $\mathcal{S}$--adapted almost invariant subset $X$ of
$G$, which is over a $VPC(n+1)$ subgroup of $G$ and is $n$--canonical relative
to $\mathcal{S}$. Then one of the following holds:

\begin{enumerate}
\item $G$ is $VPC(n+2)$.

\item There is a $\mathcal{S}$--adapted splitting of $G$ which is over some
$VPC(n+1)$ subgroup of $G$ and is $n$--canonical relative to $\mathcal{S}$.

\item $G$ is of $VPCn$--by--Fuchsian type relative to $\mathcal{S}$.
\end{enumerate}
\end{theorem}

\begin{remark}
In case 3), Definition \ref{defnofVPC-by-FuchsianrelativetoS} tells us that
for each peripheral subgroup $\Sigma$ of $G$, there is $S_{i}$ in
$\mathcal{S}$ which is conjugate to a subgroup of $\Sigma$, but is not
conjugate commensurable with a subgroup of the $VPCn$ normal subgroup $K$ of
$G$. As $K$ is normal in $G$, we can simply say that $S_{i}$ is not
commensurable with a subgroup of $K$. As $X$ is $\mathcal{S}$--adapted, the
proof of part 2) of Theorem \ref{strongcrossingimpliesFuchsiantype} shows that
each $S_{i}$ in $\mathcal{S}$ is either commensurable with a subgroup of $K$
or is conjugate to a subgroup of some peripheral subgroup of $G$. Finally if
$G\ $has no peripheral subgroups, so that the orbifold quotient $G/K$ is
closed, then each $S_{i}\in\mathcal{S}$ is commensurable with a subgroup of
$K$.

It follows that in cases 1) and 3), each $S_{i}$ in $\mathcal{S}$ must be
$VPC(\leq n+1)$. Thus if there is $S_{i}$ in $\mathcal{S}$ which is not
$VPC(\leq n+1)$, only case 2) can occur. In particular, if there is $S_{i}$ in
$\mathcal{S}$ which is not finitely generated, only case 2) can occur. Finally
if $G$ is not finitely generated, the fact that $(G,\mathcal{S})$ is a group
system of finite type implies that there is $S_{i}$ in $\mathcal{S}$ which is
not finitely generated, so that again only case 2) can occur.

The proof below shows that part 2) of the conclusion of this theorem can be
refined to state that either there is a splitting of $G$ over a group
commensurable with $H_{X}$, or there is a graph of groups structure for $G$ in
which $X$ is enclosed by a vertex of $VPCn$--by--Fuchsian type relative to
$\mathcal{S}$.
\end{remark}

\begin{proof}
We will argue as in the proof of Theorem \ref{relativetorustheorem}, but using
versions of the results quoted in that proof modified as earlier in this section.

Let $\Gamma$ denote the algebraic regular neighbourhood of $X$ in $G$. This
exists by part 2) of Theorem \ref{algregnbhdexists}. Then $\Gamma$ has a
single $V_{0}$--vertex $v$. Theorem
\ref{crossingintwostageCCCsisallweakorallstrong} shows that the crossings in a
non-isolated CCC of $E(\mathcal{X})$ are either all strong or are all weak.
Theorem \ref{strongcrossingimpliesFuchsiantype} as generalised using Theorems
\ref{crossingintwostageCCCsisallweakorallstrong},
\ref{aisethasoneadaptedK-endifn-canonical} and applying the arguments from
section \ref{CCC'swithstrongcrossing} shows that if $G$ is not $VPC(n+2)$,
then each non-isolated CCC of $E(\mathcal{X})$ in which all crossings are
strong yields a $V_{0}$--vertex of $\Gamma$ of $VPCn$--by--Fuchsian type
relative to $\mathcal{S}$. Lemmas \ref{weakcrossingCCCenclosessplitting} and
\ref{weak-weakcrossingforn-canonical} together show that each non-isolated CCC
of $E(\mathcal{X})$ in which all crossings are weak yields a $V_{0}$--vertex
of $\Gamma$ which encloses a $\mathcal{S}$--adapted splitting of $G$ over a
$VPC(n+1)$ subgroup, which is also $n$--canonical.

If $G$ is $VPC(n+2)$, we have case 1) of the theorem. In what follows we
assume that $G\ $is not $VPC(n+2)$.

If $v$ is isolated, then the two edges of $\Gamma$ which are incident to $v$
each determine a splitting of $G$ with associated almost invariant set $X$.
Thus we have case 2) of the theorem.

If $v$ arises from a CCC of $E(\mathcal{X})$ in which all crossings are weak,
then Lemma \ref{weakcrossingCCCenclosessplitting} shows that $v$ encloses a
$\mathcal{S}$--adapted splitting of $G$ over a $VPC(n+1)$ subgroup. Thus we
again have case 2) of the theorem.

Otherwise $v$ is of $VPCn$--by--Fuchsian type relative to $\mathcal{S}$. Now
there are two cases depending on whether $\Gamma$ is equal to $v$ or not. If
$\Gamma$ is not equal to $v$, there is an edge $e$ incident to $v$, and $e$
determines a splitting of $G$ over a $VPC(n+1)$ subgroup which is a peripheral
subgroup of $G(v)$. As before this splitting is $\mathcal{S}$--adapted, so
that again we have case 2) of the theorem.

If $\Gamma$ is equal to $v$, then $G=G(v)$ is of $VPCn$--by--Fuchsian type
relative to $\mathcal{S}$. Thus we have case 3) of the theorem.
\end{proof}

In section \ref{accessibilityresults}, the accessibility results proved in
Theorem \ref{accessibilityinabsolutecase} and its extensions need to be
generalised to the $n$--canonical case. We now give the required extension.

\begin{theorem}
\label{accessibilityforn-canonicalsplittings}Let $(G,\mathcal{S})$ be a group
system of finite type such that $G$ is finitely presented, each $S_{i}$ in
$\mathcal{S}$ is finitely generated, and $e(G:\mathcal{S})=1$. Let $n$ be a
non-negative integer, and suppose that, for any $VPC(<n)$ subgroup $M$ of $G$,
we have $e(G,M:\mathcal{S})=1$. Let $\Gamma_{k}$ be a sequence of minimal
$\mathcal{S}$--adapted graphs of groups decompositions of $G$ without
inessential vertices and with all edge groups being $VPC(n+1)$, such that each
edge splitting is $n$--canonical. If $\Gamma_{k+1}$ is a proper refinement of
$\Gamma_{k}$, for each $k$, then the sequence $\Gamma_{k}$ must terminate.
\end{theorem}

\begin{proof}
Theorem \ref{accessibilityforsystemswhenGisfp} tells us that there is a
minimal graph of groups structure $\Gamma_{0}$ for $G$, without inessential
vertices, in which all edge groups are $VPCn$ and no further proper refinement
is possible. As each edge splitting of each $\Gamma_{k}$ is $n$--canonical,
there is a common refinement $\Gamma_{k}^{\prime}$ of $\Gamma_{0}$ and
$\Gamma_{k}$.

Suppose that the sequence $\Gamma_{k}$ does not terminate. As $\Gamma_{0}$ has
only finitely many vertices, by collapsing suitable edges of each $\Gamma
_{k}^{\prime}$, we can arrange that there is a vertex $V$ of $\Gamma_{0}$ such
that each $\Gamma_{k}^{\prime}$ is obtained from $\Gamma_{0}$ by refining at
$V$. This gives rise to a sequence $\Gamma_{k}(V)$ of minimal graphs of groups
structures for $G(V)$ without inessential vertices and with all edge groups
being $VPC(n+1)$. Further each $\Gamma_{k}(V)$ is adapted to the family
$\mathcal{S}$ and is also adapted to the family $E(V)$ of subgroups of $G(V)$
associated to the edges of $\Gamma_{k}$ incident to $V$. As $G$ is finitely
presented and all groups in $E(V)$ are finitely generated, it follows that
$G(V)$ is also finitely presented. Also each group in the family
$\mathcal{S}\cup E(V)$ is finitely generated. Finally the maximality of
$\Gamma_{0}$ implies that $G(V)$ has no splittings over $VPCn$ subgroups which
are adapted to $\mathcal{S}\cup E(V)$. Thus the Relative Torus Theorem
\ref{relativetorustheorem} implies that $G(V)$ must be $VPC(n+2)$ or of
$VPCn$--by--Fuchsian type relative to $\mathcal{S}$, or that for any $VPC(\leq
n)$ subgroup $M$ of $G$, we have $e(G,M:\mathcal{S})=1$. Now Theorem
\ref{accessibilityforsystemswhenGisfp} implies that the sequence $\Gamma
_{k}(V)$ must terminate, so that the original sequence $\Gamma_{k}$ must terminate.
\end{proof}

In \cite{SS2}, Scott and Swarup considered multi-length JSJ decompositions of
length greater than $2$. Example \ref{exampletoshowGamma123doesnotexist} in
the appendix of this paper shows that in general there is no analogue of
Theorem 13.12 of \cite{SS2} for such decompositions. But Scott and Swarup
showed that there is such an analogue, which they denoted by $\Gamma
_{1,2,\ldots,n}$, if one restricts to virtually abelian subgroups of $G$ in
place of virtually polycyclic subgroups. We will state the analogous result
which holds in the setting of group systems.

\begin{theorem}
\label{multilengthJSJforVAsubgroups}Let $(G,\mathcal{S})$ be a group system of
finite type such that $G$ is finitely presented, each $S_{i}$ in $\mathcal{S}$
is finitely generated, and $e(G:\mathcal{S})=1$. Let $\mathcal{F}%
_{1,2,\ldots,n}$ denote the collection of equivalence classes of all
nontrivial $\mathcal{S}$--adapted almost invariant subsets of $G$ which are
over a virtually abelian subgroup of rank $i$, for $1\leq i\leq n$, and are
$(i-1)$--canonical relative to $\mathcal{S}$ with respect to abelian groups.

Then the family $\mathcal{F}_{1,2,\ldots,n}$ has an algebraic regular
neighbourhood $\Gamma_{1,2,...,n}$ in $G$.

Each $V_{0}$--vertex $v$ of $\Gamma_{1,2,...,n}$ satisfies one of the
following conditions:

\begin{enumerate}
\item $v$ is isolated, so that $G(v)$ is $VA(\leq n)$.

\item $G(v)$ is of $VA(<n)$--by--Fuchsian type relative to $\mathcal{S}$.

\item $G(v)$ is the full commensuriser $Comm_{G}(H)$ for some $VA(\leq n)$
subgroup $H$, such that $e(G,H:\mathcal{S})\geq2$.
\end{enumerate}

$\Gamma_{1,2,...,n}$ consists of a single vertex if and only if $\mathcal{F}%
_{1,2,\ldots,n}$ is empty, or $G$ itself satisfies one of the above three conditions.
\end{theorem}

\begin{remark}
In case 2), where $G(v)$ is of $VA$--by--Fuchsian type relative to
$\mathcal{S}$, the groups associated to edges incident to $v$ are $VPC$, but
need not be $VA$.
\end{remark}



\subsection{Further remarks}

We observed in Section 16 that Theorem \ref{thm16.1} can be deduced more quickly by using the Guiradel-Levitt approach of
\cite{Guirardel-Levitt2} to adapted a.i.\ sets.
Now we discuss how much of Theorem 17.3 can be deduced using their approach.
Let $\cF_{n,n+1}$ and $\Gamma_{n,n+1}(G,\cS)$ be as in the hypothesis of Theorem \ref{JSJfortwosuccessivelengths}.
We now form $\Gamma_n(G,\cS)$ and consider the $V_1$-vertices of $\Gamma_n(G,\cS)$.
Let $T$ denote the tree corresponding to $\Gamma_n(G,\cS)$ and $\cF_{n+1}$ denote the equivalence classes of adapted a.i.\ sets in 
$\cF_{n,n+1}$. If $X$ represents an element of $\cF_{n+1}$, which is the subset of $\cF_{n,n+1}$ corresponding to $VPC(n+1)$ groups,
then $X$ is enclosed in a $V_1$-vertex $v$ of 
$\Gamma_n(G,\cS)$. We consider a $\bar{v}$ in $T$ above that encloses $X$. We need to consider the ccc's of such 
$X$ in $v$ as $v$ varies and the corresponding ccc's in $\Gamma_{n,n+1}(G,\cS)$.
It may happen that the $X$ which are enclosed in $v$ may be trivial in $v$.
We first consider the edge system for $v$. This consists of the edges of $v$ in $\Gamma_{n}(G,\cS)$, say $E(v)$.
We consider a.i.\ sets $X$ of $\cF_{n+1}$; these are enclosed in $v$. Thus, any such $X$ is adapted to $G(e)$, $e\in E(v)$ and those $\cS$ which are conjugate to subgroups of $\cS$. From now on it is understood that we consider only such $X$.


\begin{proposition} \label{prop17.12}
Consider an $X$ over a $VPC(n+1)$ group $H$ coming from $\cF_{n+1}$.
Suppose that $X\cap G(v)$ is $H$-finite then $H$ is commensurable with an edge group $G(e)$ of $v$.
If $X$ and $Y$ cross in $G$ but not in $G(v)$, then they are both over an $H$ commensurable with some edge group $G(e)$ of $v$. Moreover, any such $X$ with $X\cap G(e)$ $H$-finite faces only finitely many edges of $G(v)$.
\end{proposition}

We proceed to the proof of the above proposition.
In \cite{Guirardel-Levitt2} $T$ is formed from the JSJ tree $T_{J}$.

Now, Proposition 14.6 of \cite{SS2} asserts that $X$ defines a non-trivial partition of the edges $E(\bar{v})$ of $\bar{v}$.
Consider the two sides $\Sigma(X)$ and $\Sigma(X^*)$ of $X$ in $T$ with $\Sigma(X)$ being the trivial side, that is, $X\cap G(\bar{v})$ is $H$-finite.
If $\bar{e}$ is an edge on the trivial side, $G(\bar{e})$ is a subgroup of
$H$ and thus $G(\bar{e})$ is $VPC(n)$ or $VPC(n+1)$. We first observe that $G(\bar{e})$ cannot be $VPC(n)$. To see this consider the other vertex $\bar{w}$ of $\bar{e}$. 
If $G(\bar{e})$ is $VPC(n)$, $\bar{w}$ must be a $QH$-vertex. Consider $w$, $e$ in $\Gamma_n(G,\cS)$, which are the images of $\bar{w}$, $\bar{e}$ respectively .

Thus in $G(w)*_{G(e)}G(v)$, $H\subset G(v)$ commensurises $G(e)$ which implies that there is another $V_0$-vertex of $\Gamma_n(G,\cS)$ between $w$ and $v$. This contradiction shows that $G(e)$ is $VPC(n+1)$.
Since this is true for any edge $e$ on the trivial side, it follows that there are only finitely many edges on the trivial side of $X$ and all those have groups commensurable with $H$. This proves the first part of Proposition \ref{prop17.12}. For the second part, we observe that if $X$, $Y$ are enclosed in $v$ and cross strongly in $G$, they also cross strongly in $G(v)$. We may assume that the crossing is weak and both $X$, $Y$ are over a $VPC$ group $H$.
Thus one $X^{(*)}\cap Y^{(*)}$ is trivial and the first part of Proposition \ref{prop17.12} shows that $H$ is commensurable with an edge group $G(e)$ of $v$.

We also conclude, 

\begin{remark}
If $X$, $Y\in\cF_{n+1}$ cross in $G$, they also cross in $v$ unless $X$, $Y$ are both over an edge group of $v$ with $VPC(n+1)$ stabiliser. Note that for any $X$ enclosed in $v$ any extension of $X\cap G(v)$ is equivalent to $X$.
For trivial $X\cap G(v)$ we use the partition of the boundary. Also, the edge splittings of $v$ give a.i.\ sets which do not cross any element of $\cF_{n+1}$.
\end{remark}

In general, it is not clear that the edge groups of $\Gamma(G,\cS)$ represent a.i.\ sets which are adapted to $\cS$.
In the above case, the edge groups of $v$ which are $VPC(n+1)$ come from the boundary points of cylinders with infinite valence (from \cite{Guirardel-Levitt2}).
Since the original JSJ tree comes from a finite graph, it follows that

\begin{remark}
The $VPC(n+1)$ edge groups of a $V_1$-vertex $v$ in $\Gamma_n(G,\cS)$ define almost invariant sets which are adapted to $\cS$.
\end{remark}

All the a.i.\ sets which have non-trivial commensurisers and adapted to $\cS$ give new $V_0$-vertices of commensuriser type if $X$ is non-trivial in $G(v)$. The same is true for $G(e)$, if there is  such  a $X$ with group $G(e)$ enclosed in $G(v)$. 
For edge groups without such a $X$ the situation is unclear from $\cite{Guirardel-Levitt2}$, though Theorem \ref{thm17.3} asserts that there is a vertex
of commensuriser type which encloses the a.i.\ sets given by $G(e)$. Perhaps there is some way of splitting the commensuriser of $G(e)$ in these cases, but it is not clear.
One way around  seems to be add along the above $VPC (n+1)$ edges, finitely presented groups which do not split along 
$VPC (i)$ groups for $i$ less than equal to $n+1$. This discussion is meant to suggest further work on these topics.

So the missing cases are the commensuriser vertices of $\Gamma_{n,n+1}(G,\cS)$ which correspond to edge subgroups $G(e)$ of $v$ in $\Gamma_n(G,\cS)$ for which there is no non-trivial a.i.\ set enclosed in $v$.
These have group $\comm_{G(e)}G$ by Proposition 13.9 of \cite{SS2}.
It is not clear how to split off  these from $G(v)$ from what we know about $\Gamma_n(G,\cS)$.
From what we know about $\Gamma_{n,n+1}(G,\cS)$, such a group also gives another $V_0$-vertex in $v$. 
Thus, we can recover most of $\Gamma_{n,n+1}(G,\cS)$ using \cite{Guirardel-Levitt2} twice.

Also, we have examples in \cite{SS2} and the Appendix here that the approach of the main part of section 17 does not extend to the case of three a.i.\ sets over three successive VPC groups.
However, the procedure of \cite{Guirardel-Levitt2} can be continued indefinitely as in \cite{Guirardel-Levitt4}.
Using a combination of Theorem \ref{thm17.3} as above with \cite{Guirardel-Levitt2} approach we can continue and obtain complete analogues of JSJ decomposition over all $VPC(i)$ groups analogous to the one in \cite{Guirardel-Levitt4}.

\section{Comparisons with topology\label{comparisonswithtopology}}

Let $M$ be an orientable Haken $3$--manifold with fundamental group $G$. Note
that if $G\ $is infinite, then $M$ is aspherical, so that $G$ has
cohomological dimension at most $3$. In particular $G$ is torsion free. If we
want to apply the preceding algebraic results to $G$, we need to consider
$VPCk$ subgroups of $G$. Such subgroups can only occur for $k\leq3$, and if we
assume that $G$ is not $VPC3$, then such subgroups can only occur for $k\leq
2$. As $G\ $is torsion free, the only nontrivial such subgroups are isomorphic
to $\mathbb{Z}$, $\mathbb{Z}\times\mathbb{Z}$ or the fundamental group of the
Klein bottle, so the only $\pi_{1}$--injective surfaces in $M$ we need
consider are annuli and tori. As discussed in the introduction, there is a
relative version of the JSJ decomposition of $M$ which partly motivated the
work in this paper. We will compare the results obtained using the topological
point of view and the algebraic point of view.

First we will give a brief summary of the topological theory, emphasising
those points which are closely related to the algebraic theory in this paper.

Let $S$ be a compact subsurface (not necessarily connected) of $\partial M$
such that both $S$ and $\partial M-S$ are incompressible in $M$. Let $\Sigma$
denote the closure of $\partial M-S$. One considers $\pi_{1}$--injective maps
of tori and annuli into $M$. For annuli, one assumes in addition that the
boundary is mapped into $S$. Such a map of a torus or annulus into $M$ is said
to be \textit{essential in }$(M,S)$ if it is not homotopic, as a map of pairs,
into $S$ or into $\Sigma$. Jaco and Shalen \cite{JS}, and Johannson \cite{JO},
defined what can reasonably be called the JSJ decomposition of the pair
$(M,S)$. The standard JSJ decomposition is the special case where $S$ equals
$\partial M$ so that $\Sigma$ is empty.

A compact codimension--$0$ submanifold $W$ of $M$ whose frontier consists of
annuli and tori is \textit{essential in }$(M,S)$ if each frontier component of
$W$ is an essential surface in $(M,S)$. In addition $W$ is \textit{simple} in
$(M,S)$ if any essential map of the annulus or torus into $M$ which has image
in $W$ can be homotoped, as a map of pairs, into the frontier of $W$.

Jaco and Shalen \cite{JS} and Johannson \cite{JO} proved that there is a
family $\mathcal{T}$ of disjoint essential annuli and tori embedded in
$(M,S)$, unique up to isotopy, and with the following properties. The
manifolds obtained by cutting $M$ along $\mathcal{T}$ are simple in $(M,S)$ or
are Seifert fibre spaces or $I$--bundles over surfaces. In fact, $\mathcal{T}$
can be characterised as the minimal family of annuli and tori with this
property. The Seifert and $I$--bundle pieces of $M$ are said to be
\textit{characteristic}. The characteristic submanifold $V(M,S)$ of $(M,S)$
consists essentially of the union of the characteristic pieces of the manifold
obtained from $M$ by cutting along $\mathcal{T}$. However, if two
characteristic pieces of $M$ have a component $F$ of $\mathcal{T}$ in common,
we add a second copy of $F$ to the family $\mathcal{T}$, thus separating the
two characteristic pieces by a copy of $F\times I$, which we regard as a
non-characteristic piece of $M$. Similarly, if two non-characteristic pieces
of $M$ have a component $F$ of $\mathcal{T}$ in common, we add a second copy
of $F$ to the family $\mathcal{T}$, thus separating the two non-characteristic
pieces by a copy of $F\times I$, which we regard as a characteristic piece of
$(M,S)$. Any essential map of the annulus or torus into $M$ can be homotoped,
as a map of pairs, to lie in $V(M,S)$, which is called the Enclosing Property
of $V(M,S)$. Note that the frontier of $V(M,S)$ is usually not equal to
$\mathcal{T}$. Some annuli or tori in $\mathcal{T}$ may appear twice in this
frontier. This discussion brings out a somewhat confusing fact about the
characteristic submanifold, which is that both $V(M,S)$ and its complement can
have components which are homeomorphic to $F\times I$, where $F$ is an annulus
or torus. One other basic point to note is that it is quite possible that
$\mathcal{T}$ is empty. In this case, either $V(M,S)$ is equal to $M$ or it is
empty, so that either $M$ is a Seifert fibre space or an $I$--bundle, or
$(M,S)$ admits no essential annuli and tori.

In order to complete this description of $V(M,S)$, we need to say a little
more about its frontier. If $W$ is a component of $V(M,S)$ which is an
$I$--bundle over a compact surface $L$, then the frontier $fr(W)$ of $W$ in
$M$ is the restriction of the bundle to $\partial L$, so that $fr(W)$ is
homeomorphic to $\partial L\times I$. If $W$ is a Seifert fibre space
component of $V(M,S)$, then there is a Seifert fibration on $W$ such that the
frontier of $W$ in $M$ consists of vertical annuli and tori.

As we stated in the introduction, in most cases, an embedded annulus or torus
in $M$ which is essential in $(M,S)$ determines a splitting of $\pi_{1}(M)$
over $\mathbb{Z}$ or $\mathbb{Z}\times\mathbb{Z}$, as appropriate, which is
clearly adapted to the family of subgroups of $\pi_{1}(M)$ which consists of
the fundamental groups of the components of $\Sigma$. For notational
simplicity, we will denote this family of subgroups by $\pi_{1}(\Sigma)$.
Further, in most cases, any annulus or torus in $M$ which is essential in
$(M,S)$ determines a nontrivial almost invariant subset of $\pi_{1}(M)$ over
$\mathbb{Z}$ or $\mathbb{Z}\times\mathbb{Z}$, as appropriate, which is adapted
to $\pi_{1}(\Sigma)$. Obviously, those annuli and tori in $M$ which are
essential in $(M,S)$ but do not determine a nontrivial almost invariant subset
of $\pi_{1}(M)$ do not have any analogue in the algebraic theory of this
paper. In order to have a topological theory which is more analogous to the
algebraic theory, we will modify our idea of essentiality. We will say that an
annulus or torus in $(M,S)$ is \textit{strongly essential in} $(M,S)$ if it is
essential in $M$, i.e. it is $\pi_{1}$--injective and not properly homotopic
into $\partial M$. If a component $F$ of $\mathcal{T}$ is inessential in $M$,
it cuts $M$ into two pieces one of which, $W_{F}$, is homeomorphic to $F\times
I$. It is possible that $W_{F}$ contains other components of $\mathcal{T}$,
but any $\pi_{1}$--injective annulus or torus in $(M,S)$ which is contained in
$W_{F}$ must also be inessential in $M$. Now remove from $V(M,S)$ any
component which is contained in some $W_{F}$, for some component $F$ of
$\mathcal{T}$. The resulting submanifold $V^{\prime}(M,S)$ of $M$ encloses all
annuli and tori in $(M,S)$ which are strongly essential in $(M,S)$. We let
$\mathcal{T}^{\prime}$ denote the family of annuli and tori in $\mathcal{T}$
which are strongly essential in $(M,S)$. As before the frontier of $V^{\prime
}(M,S)$ is usually not equal to $\mathcal{T}^{\prime}$. Some annuli or tori in
$\mathcal{T}^{\prime}$ may appear twice in the frontier of $V^{\prime}(M,S)$.

We recall the idea of a canonical surface in $M$, which was defined in
\cite{SS4}. An embedded essential annulus or torus $F$ in $M$ is called
\textit{canonical} if any essential map of the annulus or torus into $M$ can
be properly homotoped to be disjoint from $F$. We extend this to the relative
case in the natural way. If $F$ is an embedded strongly essential annulus or
torus in $(M,S)$, we will define it to be \textit{canonical in }$(M,S)$ if any
strongly essential map of the annulus or torus into $(M,S)$ can be homotoped,
as a map of pairs, to be disjoint from $F$. The Enclosing Property for
$V^{\prime}(M,S)$ clearly implies that any annulus or torus in $\mathcal{T}%
^{\prime}$ is canonical.

As discussed in the introduction, the guiding idea behind this paper is that
$V^{\prime}(M,S)$ should be thought of as a regular neighbourhood of the
family of all strongly essential annuli and tori in $(M,S)$. By this we mean
that every such map is properly homotopic into $V^{\prime}(M,S)$ and that
$V^{\prime}(M,S)$ is minimal, up to isotopy, among all essential submanifolds
of $M$ with this property. It will be convenient to say that the collection of
all such maps \textit{fills} $V^{\prime}(M,S)$, when $V^{\prime}(M,S)$ has
this minimality property. The word \textquotedblleft fill\textquotedblright%
\ is used in the same way to describe certain curves on a surface. A subtle
point which arises here is that there are exceptional cases where $V^{\prime
}(M,S)$, as defined by Jaco-Shalen and Johannson, is not filled by the
collection of all essential annuli and tori in $M$. In these cases, the
regular neighbourhood idea does not quite correspond to the topological JSJ-decomposition.

Let $V^{\prime\prime}(M,S)$ denote the strongly essential submanifold of
$(M,S)$ which encloses every strongly essential annulus and torus in $(M,S)$
and is filled by them. As in the absolute case when $S$ equals $\partial M$,
it is easy to see that $V^{\prime\prime}(M,S)$ is only slightly different from
$V^{\prime}(M,S)$. Clearly $V^{\prime\prime}(M,S)$ is a submanifold of
$V^{\prime}(M,S)$, up to isotopy. Further, it follows that, as for $V^{\prime
}(M,S)$, the isotopy classes of the frontier components of $V^{\prime\prime
}(M,S)$ are precisely those of the canonical annuli and tori in $(M,S)$. Hence
the difference between $V^{\prime\prime}(M,S)$ and $V^{\prime}(M,S)$ is that
certain exceptional components of $V^{\prime}(M,S)$ are discarded. Here is a
description of the exceptional components, most of which are solid tori. A
solid torus component $W$ of $V^{\prime}(M,S)$ will fail to be filled by
annuli strongly essential in $(M,S)$ when its frontier consists of $3$ annuli
each of degree $1$ in $W$, or when its frontier consists of $1$ annulus of
degree $2$ in $W$, or when its frontier consists of $1$ annulus of degree $3$
in $W$. Another exceptional case occurs when a component $W$ of $V^{\prime
}(M,S)$ lies in the interior of $M$ and is homeomorphic to the twisted
$I$--bundle over the Klein bottle. Then any incompressible torus in $W$ is
homotopic into the boundary of $W$, so that $W$ is not filled by tori. In all
these cases, one obtains $V^{\prime\prime}(M,S)$ from $V^{\prime}(M,S)$ by
first replacing $W$ by a regular neighbourhood of its frontier in $M$, and
then removing any redundant product components from the resulting submanifold.

In the classical case, when $S=\partial M$, we will omit $S$ from our
notation. If $M$ is an orientable Haken $3$--manifold with fundamental group
$G$, and if $M$ has incompressible boundary, then $\Gamma_{1,2}^{c}(G)$
corresponds precisely to the JSJ decomposition of $M$ with characteristic
submanifold $V(M)$, and $\Gamma_{1,2}(G)$ corresponds to the decomposition of
$M$ determined by $V^{\prime}(M)$. This is discussed in the introduction of
\cite{SS2} and relies on the results of \cite{SS4}. 
We believe that the
relative version of this asserts that $\Gamma_{1,2}^{c}(G,\pi_{1}(\Sigma))$
corresponds to $V^{\prime}(M,S)$, and that $\Gamma_{1,2}(G,\pi_{1}(\Sigma))$
corresponds to $V^{\prime\prime}(M,S)$ and that this can be proved by working
through a relative version of \cite{SS4}, but we do not give any of the details.

\pagebreak

\appendix

\section{Corrections and additions to \cite{SS2}\label{correctionstomonster}}

In this appendix, written by the authors of \cite{SS2}, we discuss and correct
some errors in that paper. Three of these errors affect the general theory and
thus the entire paper, and one of these is quite serious. Correcting these
errors requires changes to the paper in several places. The remaining errors
do not have a global impact on the paper, so are simpler to correct. We will
list each error and explain its significance, and state how they are dealt
with. We will leave all the details to later in this appendix.

\subsection{The three major errors}

\begin{itemize}
\item (Accessibility) The most serious error is that the accessibility result,
Theorem 7.11 on page 103 of \cite{SS2}, is incorrect as stated. The same
applies to Theorems 7.13 and 16.2. This affects the proofs of the two main
results of \cite{SS2}, Theorems 10.1 and 12.3, which assert the existence of
JSJ decompositions. We give a corrected statement in Theorem
\ref{accessibilityinabsolutecase} of this paper. This requires substantial
modifications to the proof, which we also give in this paper. The applications
of this result in \cite{SS2} need minor modifications in light of the revised
statement. We will discuss the details in subsection \ref{accessibility}.

\item (Splittings over non-finitely generated groups) A second important error
is that in Theorem 5.16 of \cite{SS2}, we asserted that two splittings of a
finitely generated group are compatible if and only if they have intersection
number zero. Unfortunately this is incorrect. We will give a counterexample,
due to Guirardel, later in this appendix. However, it is fairly easy to
correct our development of the theory of algebraic regular neighbourhoods in
\cite{SS2}. We discuss the details in subsection
\ref{splittingsovernonfggroups}.

\item (Isolated invertible almost invariant sets) The last of these errors is
rather technical in nature and is not often a problem. Lemma 3.14 of
\cite{SS2} is incorrect. Even in the special case when $n=k=1$, this lemma
implies the following incorrect assertion.

\textit{Let }$G$\textit{ denote a finitely generated group, and let }%
$H$\textit{ be a finitely generated subgroup of }$G$\textit{. Let }$X$\textit{
be a nontrivial }$H$\textit{--almost invariant subset of }$G$\textit{, such
that }$X$\textit{ is isolated. Then there is an almost invariant set }%
$Z$\textit{ equivalent to }$X$\textit{, such that }$Z$\textit{ is not
invertible, and is in good enough position.}

In subsection \ref{invertibleaisets} we will discuss a counterexample to this
statement and the solution to the problem this error creates.
\end{itemize}

\subsection{The remaining errors}

\textbf{Page 28:} There is an error in the second sentence of Example 2.34 of
\cite{SS2}. In this example, $G$ denotes the free group of rank $2$. The
sentence in question asserts that if $C$ is a subgroup of $G$ which is not
finitely generated, then $e(G,C)=\infty$. This is incorrect. For example, if
$C$ is the kernel of the abelianisation map $G\rightarrow\mathbb{Z}^{2}$, then
$e(G,C)=e(\mathbb{Z}^{2})=1$. However this does not invalidate Example 2.34.
It is easy to see that there are many subgroups $C$ of $G$ which are not
finitely generated for which $e(G,C)=\infty$. For a specific example, let $H$
be a subgroup of $G$ of index $2$, so that $H$ is free of rank $3$, and let
$C$ denote the kernel of some surjection from $H$ to the free group $F_{2}$ of
rank $2$. Then $e(G,C)=e(H,C)=e(F_{2})=\infty$.

Note that for the argument in Example 2.34 it is irrelevant that
$e(G,C)=\infty$. All that is needed is that $e(G,C)>1$ in order that there be
a nontrivial $C$--almost invariant subset of $G$.

\bigskip

\textbf{Page 63:} The proof of Lemma 4.14 of \cite{SS2} is incorrect, though
the result is correct. In the last paragraph of page 63, the application of
Lemma 2.31 needs the group $K$ to be finitely generated, which need not be the
case. To correct this proof, we proceed in two stages. In the first stage we
assume that the base vertex $w$ of the tree $T$ equals the vertex $v$ which
encloses the given $H$--almost invariant set $A$. As $H\ $is finitely
generated, the proof of Lemma 4.14 is correct in this case. In terms of the
terminology introduced in section \ref{enclosing} of this paper, we have shown
that $A$ is equivalent to a $H$--almost invariant subset $B(A)$ of $G$ which
is strictly enclosed by $v$ with basepoint $v$. For the second stage of the
proof of Lemma 4.14, we simply apply Lemma
\ref{strictenclosingdoesnotdependonbasepoint} of this paper which implies that
$A$ is equivalent to a $H$--almost invariant subset of $G$ which is strictly
enclosed by $v$ with basepoint $w$.

\bigskip\ 

\textbf{Page 132:} Theorem 10.8 of \cite{SS2} is false. This clarifies why it
is so important for our theory to consider \textit{all} almost invariant
subsets not just those which correspond to splittings. The generalisations of
Theorem 10.8 in Theorems 12.6, 13.15 and 14.11 are also false. In subsection
\ref{detailsonremainingerrors} we will describe a counterexample, and discuss
in detail why the listed results are false.

\bigskip

\textbf{Page 137:} There is a bad misprint in line -4. The formula
$G=A\ast_{H}(D\ast H)$ should read $G=A\ast_{H}(D\times H)$. Also we neglected
to point out that $G$ is one-ended, but this is easily verified.

\bigskip

\textbf{Page 138:} Example 11.2 of \cite{SS2} is incorrect. It is asserted
there that $\Gamma(G)$ is the graph of groups given by $G=A\ast_{K}\left[
(K\ast L)\times H\right]  \ast_{L}B$, i.e. that $\Gamma(G)$ has a $V_{0}%
$--vertex with associated group $(K\ast L)\times H$, and has two $V_{1}%
$--vertices with associated groups $A$ and $B$.

In fact the methods of Guirardel and Levitt in \cite{Guirardel-Levitt2} show
that if one has a $V_{0}$--vertex $v$ of $\Gamma(G)$ of commensuriser type, so
that $G(v)=Comm_{G}(H)$, where $e(G,H)\geq2$, then for each edge $e$ incident
to $v$, the edge group $G(e)$ contains a subgroup commensurable with $H$.

\bigskip

\textbf{Page 144:} The statement of Theorem 12.3 of \cite{SS2} is not quite
correct. There is a special case when $\Gamma_{n}$ consists of a single
$V_{0}$--vertex. In this case there is a possibility which is not mentioned in
the statement. Namely $G$ may be $VPC(n+1)$. This possibility is contained in
the paper by Dunwoody and Swenson but we omitted it in error when we applied
their results.

The same omission occurs in the statements of Theorems 12.5, 12.6, 13.12,
13.13, 14.5, and 14.6.

\bigskip

\textbf{Page 161:} Example 14.1 of \cite{SS2} is wrong. Recall that in chapter
14 of \cite{SS2}, we gave a construction of the regular neighbourhood
$\Gamma_{1,2,\ldots,n}$ in which we restrict attention to almost invariant
sets over virtually abelian subgroups. The point of Example 14.1 was to show
that this construction does not work if one considers almost invariant sets
over virtually polycyclic subgroups. In subsection
\ref{detailsonremainingerrors}, after discussing the error in Example 14.1, we
will give a new example which demonstrates the phenomenon which Example 14.1
was supposed to demonstrate. Thus it is still correct to say that the
construction of $\Gamma_{1,2,\ldots,n}$ does not work if one considers almost
invariant sets over virtually polycyclic subgroups.

\subsection{Details on the three major errors}

\subsubsection{Accessibility\label{accessibility}}

Let $\Gamma_{1}$ and $\Gamma_{2}$ be graphs of groups decompositions of a
group $G$. We will say that $\Gamma_{2}$ is a \textit{proper} refinement of
$\Gamma_{1}$ if it is obtained from $\Gamma_{1}$ by splitting at a vertex so
that the induced splitting of the vertex group is nontrivial. A sequence of
proper refinements will also be called a proper refinement. The statement of
Theorem 7.11 on page 103 of \cite{SS2} needs to be modified as in the
statement of Theorem \ref{accessibilityinabsolutecase} of this paper. The only
important difference is that the refinements of Theorem 7.11 are assumed to be
proper in Theorem \ref{accessibilityinabsolutecase}. The proof of Theorem 7.11
in \cite{SS2} implicitly assumes this in the second sentence of the proof,
where we assert that we have only to bound the length of chains of splittings
of $G$ over descending subgroups. But the argument given in \cite{SS2} is
still incorrect.

All of the above comments apply also to Theorems 7.13 and 16.2. Here is a
detailed list of consequential changes needed to \cite{SS2}.

The first paragraph of the proof of Theorem 8.2 uses Theorem 7.11, and this
paragraph needs changing as follows. Lines 1-6 of this paragraph are fine. But
the next sentence is incorrect. It should assert that if $G$ possesses a
splitting $\sigma^{\prime}$ over a two-ended subgroup $C^{\prime}$
commensurable with $H$ which has intersection number zero with the edge
splittings of $\mathcal{G}$, then this splitting is enclosed by some $V_{0}%
$--vertex $v$ of $\mathcal{G}$, and determines a trivial splitting of the
vertex group $G(v)$. This does not imply that $\sigma^{\prime}$ is conjugate
to one of the edge splittings of $\mathcal{G}$. This requires changing the
proof of Theorem 8.2 on lines 17-18 of page 111. Our choice of $\mathcal{G}$
does not imply that $\sigma$ must be conjugate to one of the edge splittings
of $\mathcal{G}$, as claimed, but it does imply the result of the next
sentence which is all we need.

At the end of the proof of Proposition 8.4, we quote Theorem 7.11 to show that
one cannot have an unbounded chain $\sigma_{i}$ of compatible splittings of
$G$ over strictly descending subgroups, and this does follow from the
corrected version of Theorem 7.11. The use of Theorem 7.11 in the proof of
Theorem 9.2 is correct as $\Gamma_{i+1}$ is a proper refinement of $\Gamma
_{i}$ for each $i$.

Theorems 7.11 and 7.13 are used in several other places. The key point to note
is that the only problem which might occur is the existence of an unbounded
chain $\sigma_{i}$ of compatible splittings of $G$ over strictly
\emph{ascending} subgroups $H_{i}$. If this occurs and all the $H_{i}$'s are
$VPC$ of the same length, then they are all commensurable, and hence each
$H_{i}$ has large commensuriser. In particular, when constructing an algebraic
regular neighbourhood, the splittings $\sigma_{i}$ must all be enclosed by a
single $V_{0}$--vertex of large commensuriser type.

\subsubsection{Splittings over non-finitely generated
groups\label{splittingsovernonfggroups}}

Let $G$ be a finitely generated group. A HNN extension $\sigma$ of $G$ is said
to be \textit{ascending} if $G=A\ast_{C}$ where at least one of the two
injections of $C$ into $A$ is an isomorphism. Note that if both injections are
isomorphisms, then $A$ is normal in $G$ with infinite cyclic quotient. The
difficulty arises when one considers ascending HNN extensions $G=A\ast_{C}$
for which $A$ is not finitely generated. Such extensions are not at all
unusual, as all that is needed is a surjection from $G$ to $\mathbb{Z}$ whose
kernel $A$ is not finitely generated. Whether or not $A$ is finitely
generated, ascending HNN\ extensions have the property that they can only be
compatible with ascending HNN\ extensions. (Recall that two splittings
$\sigma$ and $\tau$ of a group $G$ are compatible if $G$ is the fundamental
group of a graph of groups with two edges such that the associated edge
splittings are $\sigma$ and $\tau$.) The precise result is the following.

\begin{lemma}
Let $\sigma$ be an ascending HNN extension of a group $G$. If $\sigma$ is
compatible with a splitting $\tau$ of $G$, then $\tau$ is also an ascending
HNN extension.
\end{lemma}

\begin{proof}
If $\sigma$ is compatible with $\tau$, then $G$ is the fundamental group of a
graph $\Gamma$ of groups such that $\Gamma$ has two edges $s$ and $t$ and the
associated edge splittings are $\sigma$ and $\tau$ respectively. As $\sigma$
is HNN, it follows that $\Gamma$ has at most two vertices. We claim that
$\Gamma$ has two vertices each of valence $2$, so that $\Gamma$ is a circle.
For otherwise, the edge $s$ would be a loop, and the subgraph given by $t$
would carry the vertex group $A$ of $\Gamma$. As $\tau$ is a splitting of $G$,
collapsing $t$ yields the loop $s$ with vertex group $B$ properly containing
$A$, and $\sigma$ is the splitting of $G$ associated to this loop. As $\sigma$
is ascending, one of the inclusions of the edge group of $s$ into $B$ is an
isomorphism. But this is impossible as each of these inclusions factors
through $A$. Now we know that $\Gamma$ is a circle, it is easy to see that
$\tau$ is also an ascending HNN extension.
\end{proof}

Note that compatible ascending HNN extensions need not be conjugate. To
construct examples, let $K$ be a subgroup of a group $H$ such that $K$ itself
has a subgroup $L$ isomorphic to $H$. Let $G$ be the fundamental group of a
graph of groups with underlying graph a circle, with two vertices labeled by
$H$ and $K$ and two edges labeled by $K\ $and $L$, so that the inclusion of
the edge group $L$ into the vertex group $H$ is an isomorphism and the other
inclusions are clear. If $K\ $and $L$ are not isomorphic, the two edge
splittings cannot be conjugate. A simple such example can be found with $H$
(and hence $L$) free of rank $2$, and $K$ free of rank $3$.

In Theorem 5.16 of \cite{SS2}, we asserted that two splittings of a finitely
generated group are compatible if and only if they have intersection number
zero, but Guirardel gave us a counterexample. Both splittings in his example
are ascending HNN extensions over non-finitely generated groups. However such
splittings are the only source of problems, and it is fairly easy to correct
our development of the theory of algebraic regular neighbourhoods. This
requires many changes to the paper including a modification of the definition
of an algebraic regular neighbourhood (Definition 6.1 of \cite{SS2}).

Here is Guirardel's example.

\begin{example}
\label{incompatiblebutintersectionnozero}Let $G$ denote the free group on two
generators $a$ and $b$, and let $f:G\rightarrow\mathbb{Z}$ be given by
$f(a)=0$, and $f(b)=1$. Let $K$ denote the kernel of $f$, so that $K$ is
freely generated by the elements $b^{k}ab^{-k}$, $k\in\mathbb{Z}$. Let $\tau$
denote the splitting of $G$ as $K\ast_{K}$, where both injections of the edge
group into the vertex group are isomorphisms. Let $H$ denote the subgroup of
$G$ generated by the elements $b^{k}ab^{-k}$, $k\geq0$, and let $\sigma$
denote the splitting of $G$ as $H\ast_{H}$, in which one inclusion is the
identity and the other is conjugation by $b$. We will show that the splittings
$\sigma$ and $\tau$ have intersection number zero but cannot be compatible.

Let $S$ and $T$ denote the $G$--trees corresponding to $\sigma$ and $\tau$
respectively. Thus $T$ is a copy of the real line, with a vertex at each
integer point. Think of $T$ as the $x$--axis in the plane with $S$ lying in
the plane above $T$. We can describe $S$ pictorially by saying that each
vertex of $S$ has integer $x$--coordinate, no edge of $S$ is vertical, and at
each vertex $v$ of $S$, there is exactly one edge incident to $v$ from the
right, and infinitely many edges incident to $v$ from the left. Now the
projection of the plane onto the $x$--axis induces a $G$--equivariant map
$p:S\rightarrow T$. Let $s$ be an edge of $S$ and let $t$ denote the edge
$p(s)$ of $T$, and orient $s$ and $t$ to point to the left. Choose a base
vertex $\ast$ for $S$ and let $p(\ast)$ be the base vertex of $T$. Thus we
have $\varphi:G\rightarrow V(S)$ given by $\varphi(g)=g(\ast)$, for all $g\in
G$. If we remove the interior of $s$ from $S$, we are left with two subtrees
of $S$. The one which contains the terminal vertex of $s$ is denoted by
$Y_{s}$, and we let $Z_{s}$ denote $\varphi^{-1}(Y_{s})$. Similarly removing
the interior of $t$ from $T$ yields two half lines in $T$, and the one which
contains the terminal vertex of $t$ is denoted by $Y_{t}$, and we let $Z_{t}$
denote $\varphi^{-1}(p^{-1}(Y_{t}))$. The sets $Z_{s}$ and $Z_{t}$ are almost
invariant subsets of $G$ over $H$ and $K$ respectively which are associated to
the splittings $\sigma$ and $\tau$ of $G$. Clearly $Z_{s}\subset Z_{t}$. As
$gZ_{t}$ is equivalent to $Z_{t}$ for every $g$ in $G$, we have $Z_{s}\leq
gZ_{t}$, for every $g$ in $G$. Thus $\sigma$ and $\tau$ have intersection
number zero. Now consider the set $E$ of all translates of $Z_{s}$ and $Z_{t}$
and their complements in $G$. If $\sigma$ and $\tau$ were compatible, $E$
would correspond to the edges of a $G$--tree. In turn this would imply that
there must some translate of $Z_{t}$ (or of $Z_{t}^{\ast}$) which is $\leq
Z_{s}$. As $Z_{s}\leq gZ_{t}$, for every $g$ in $G$, this is clearly impossible.

It is interesting to see how our construction of an algebraic regular
neighbourhood in \cite{SS2} fails in this case. As no two elements of $E$
cross, each element of $E$ is isolated in $E$ and so forms a CCC by itself. It
follows immediately that the CCCs of $E$ form a pretree. The problem is that
this pretree is not discrete. For if $b$ acts on $T$ by translating one unit
to the left, then we have the inclusions $\ldots\subset b^{2}Z_{t}\subset
bZ_{t}\subset Z_{t}$ and $Z_{s}\leq b^{k}Z_{t}$ for each $k\geq0$. Thus there
are infinitely many CCCs of $E$ between $Z_{s}$ and $Z_{t}$.
\end{example}

\bigskip

The error which causes all the problems occurs in the proof of Proposition 5.7
of \cite{SS2}. On page 72, lines -13 to -12, of \cite{SS2}, we assert that
\textquotedblleft there must be an element $g$ of $G$ such that $gX\subset
U$." This is not correct in general. In order to appreciate the problem, we
need to discuss the proof of Proposition 5.7. Recall that we have a family
$\{X_{\lambda}\}$ of almost invariant subsets of $G$ in good position such
that the regular neighbourhoods of this family can be constructed as in
chapter 3 of \cite{SS2}. Recall also that $E$ denotes the collection of all
translates of the $X_{\lambda}$'s and $X_{\lambda}^{\ast}$'s, and that $X$ is
a nontrivial $H$--almost invariant subset of $G$ which crosses no element of
$E$.\ Proposition 5.7 asserts that $X$ is enclosed by a $V_{1}$--vertex of
$\Gamma$ so long as either $H\ $is finitely generated or $X$ is associated to
a splitting of $G$. Our proof starts by showing that $X$ is sandwiched between
two elements of $E$, i.e. there are elements $U$ and $V$\ of $E\ $such that
$U\leq X\leq V$. If this condition holds then the remainder of our proof is
correct. Further our proof that $X$ must be sandwiched between two elements of
$E$ is correct, so long as $H$ is finitely generated. In the case when $H$ is
not finitely generated, so that $X$ is associated to a splitting $\sigma$ of
$G$, our proof is also correct so long as $\sigma$ is not an ascending
HNN\ extension. This was the case which we discussed incorrectly on page 72,
line -12. For convenience in what follows we will say that a splitting of a
group which is an ascending HNN extension over a non-finitely generated group
is \textit{special}.

Given a collection $E$ of almost invariant subsets of $G$, we will say that an
almost invariant subset $X$ of $G$ is \textit{sandwiched by} $E$ if there are
elements $U$ and $V$ of $E$ such that $U\leq X\leq V$. We can now correct the
statement of part 2) of Proposition 5.7 of \cite{SS2} by simply adding the
hypothesis that $X$ be sandwiched by $E$. The published proof of Proposition
5.7 shows that this is automatic except possibly when $X$ is associated to a
special splitting of $G$. This weakens Proposition 5.7, but almost all of the
applications will still follow from this weakened version. Note that the proof
of this weakened form of Proposition 5.7 is substantially shorter.

\bigskip

Here is a list of consequential changes needed in \cite{SS2}:

Lemmas 5.10 and 5.15 need an extra sandwiching assumption. In the last
paragraph of the proof of Lemma 5.10, on page 75 of \cite{SS2}, our arguments
use the assumptions that $A$ is sandwiched between two elements of $E$ and
that any element of $E$ is sandwiched between two translates of $A$ and
$A^{\ast}$. The first assumption follows from the hypotheses of the lemma, but
the second does not and needs to be added to the list of hypotheses. Of
course, the second assumption is correct unless some edge splitting of
$\Gamma$ is special, and in this case, we know that $\Gamma$ must be a circle.
The same point arises on page 77 in the proof of Lemma 5.15 of \cite{SS2}.

The existence part of Theorem 5.16 is correct so long as we exclude special
splittings. Otherwise the two splittings $\sigma$ and $\tau$ described above
form a counterexample. But the uniqueness part of Theorem 5.16 is correct, so
that Theorem 5.17 is still correct.

Lemma 5.19 is correct.

Lemma 5.21 is also correct. We note that the assumption in this lemma and
several later results that the regular neighbourhood has been constructed as in
chapter 3 avoids many of the above worries about sandwiching. In particular,
it implies that if $P$ denotes the pretree of CCCs of $E$, the set of all
translates of all the $X_{\lambda}$'s and $X_{\lambda}^{\ast}$'s, then $P$ is discrete.

On pages 80-83, in the discussion of algebraic regular neighbourhoods of
almost invariant sets over possibly non-finitely generated groups, we need to
add a sandwiching assumption in order to prove existence for finite families.
For example, we could assume that each $X_{i}$ is sandwiched by $E(X_{j})$,
for each $i$ and $j$.

Proposition 5.23 also needs an additional sandwiching assumption.

In part 2) of the definition of an algebraic regular neighbourhood (Definition
6.1 of \cite{SS2}), we need to add the assumption that the almost invariant
set associated to the splitting $\sigma$ is sandwiched by $E$.

Our main existence result (Theorem 6.6) is fine if the $H_{i}$'s are all
finitely generated. In general, we need to add a sandwiching assumption. For
example, we could assume that each $X_{i}$ is sandwiched by $E(X_{j})$, for
each $i$ and $j$.

Our main uniqueness result (Theorem 6.7 of \cite{SS2}) is correct as stated.
The proof uses Lemma 5.10 crucially, but the sandwiching assumption which
needs to be added to the statement of Lemma 5.10 holds automatically in the
present context.

The change in the definition of an algebraic regular neighbourhood means that
we need to change the statements of many later theorems which describe the
properties of the regular neighbourhoods we construct. For example, part (5)
of Theorem 9.4 of \cite{SS2} needs the addition of a sandwiching assumption.
Note that we also used Proposition 5.7 to prove part (9) of Theorem 9.4, but
this needs no change as the splittings considered are over finitely generated
subgroups of $G$. The same comments apply to all the similar later results.

\subsubsection{Isolated invertible almost invariant
sets\label{invertibleaisets}}

Lemma 3.14 of \cite{SS2} is incorrect even in the special case when $n=k=1$.
In that case, the lemma implies the following incorrect statement.

\medskip

\textit{Let }$G$\textit{ denote a finitely generated group, and let }%
$H$\textit{ be a finitely generated subgroup of }$G$\textit{. Let }$X$\textit{
be a nontrivial }$H$\textit{--almost invariant subset of }$G$\textit{, such
that }$X$\textit{ is isolated. Then there is an almost invariant set }%
$Z$\textit{ equivalent to }$X$\textit{, such that }$Z$\textit{ is not
invertible, and is in good enough position.}

\bigskip

Here is a counterexample to this statement.

\bigskip

\begin{example}
\label{examplewheregoodpositionimpliesinvertible}Let $G$ denote the group
$\mathbb{Z}_{2}\ast\mathbb{Z}_{2}$, let $H$ denote the trivial subgroup and
let $\sigma$ denote the given splitting of $G$ over $H$. Let $X$ be a
$H$--almost invariant subset of $G$ associated to $\sigma$. As $X$ is
associated to a splitting, it crosses none of its translates and so is
isolated in the set $E(\mathcal{X})$ which consists of all translates of $X$
and $X^{\ast}$. Let $Z$ be an almost invariant set which is equivalent to $X$.
If $Z$ is in good enough position, we claim that $Z$ must be invertible. Thus
we have the required counterexample to the above statement.

Here is the proof of the claim that $Z$ must be invertible. The key point is
that every translate of $Z$ is equivalent to $Z$ or $Z^{\ast}$, so that two
translates of $Z$ never cross. Thus the fact that $Z$ is in good enough
position implies that $Z$ is in good position.

Let $\Gamma$ denote the Cayley graph of $G$ with respect to the two natural
generators, the generators of the $\mathbb{Z}_{2}$ factors. Thus $\Gamma$
consists of two disjoint copies $l$ and $l^{\prime}$ of the real line with
vertices at each integer point, together with pairs of edges joining
corresponding vertices of $l$ and $l^{\prime}$, so that $\Gamma$ is "an
infinite ladder with double steps". Let $\alpha$ and $\alpha^{\prime}$ be rays
in $l$ and $l^{\prime}$ respectively each containing the positive end of its
line. Then the set $W$ of vertices of $\alpha\cup\alpha^{\prime}$ forms a
nontrivial almost invariant subset of $G$. As $G$ has $2$ ends, we know that
any nontrivial almost invariant subset of $G$ is almost equal to $W$ or to
$W^{\ast}$. Next observe that $W$ is invertible, as the two generators of
$G\ $act on $\Gamma$by reflecting and interchanging $l$ and $l^{\prime}$, and
the kernel of the natural map from $G$ to $\mathbb{Z}_{2}$ is an infinite
cyclic group acting by translations on each of $l$ and $l^{\prime}$.

At this point, we return to the given almost invariant subset $Z$ of $G$ such
that $E(Z)$ is nested. Without loss of generality, we can assume that $Z$ is
almost equal to $W$. Now let $b$ denote the least element of $Z\cap l$, let
$b^{\prime}$ denote the least element of $Z\cap l^{\prime}$, and let $\beta$
and $\beta^{\prime}$ denote the rays starting at $b$ and $b^{\prime}$ which
are almost equal to $\alpha$ and $\alpha^{\prime}$ respectively. Thus $Z$ is
equal to $\beta\cup\beta^{\prime}$, which we denote by $U$, with some finite
subset $F$ removed. Assume that $F$ is non-empty. Let $g$ denote translation
of $\Gamma$ by one step such that $gU \subset U$, and consider the corners of
the pair $(Z,gZ)$. Then $Z \cap gZ^{\ast}$ is finite and contains $b$ and
$b^{\prime}$, and $gZ \cap Z^{\ast}$ is also finite and contains the least
element of $F$. Thus the pair $(Z,gZ)$ has two non-empty small corners, which
contradicts the fact that $Z$ is in good position. Hence $Z$ must be equal to
$U=\beta\cup\beta^{\prime}$ and so be invertible, as required.
\end{example}

We resolve this problem as in Definition
\ref{defn9.8} of this paper. This was also discussed
in Examples 3.10 and 3.11 on pages 44 and 45 of \cite{SS2}.

\medskip

This involves corresponding changes in the statements and proofs at several
points in chapters 3, 4, 5 and 6 of \cite{SS2}:

Pages 50-51: the statement of Summary 3.16.

Pages 51-52: the proof of Lemma 3.17.

Pages 53-54: the discussion of the construction of an algebraic regular
neighbourhood of an infinite family.

Page 62: the proof of part 2) of Lemma 4.10.

Page 67: Lemma 5.1 and its proof are correct, but the comment just before
Lemma 5.1 that "it suffices to prove that the $V_{0}$--vertices of
$\Gamma(\{X_{\lambda}\}_{\lambda\in\Lambda}:G)$ enclose the given $X_{\lambda
}$'s in the case when the $X_{\lambda}$'s are in good position and isolated
$X_{\lambda}$'s are not invertible" is not correct. Thus Lemma 5.1 should be
stated and proved without this last assumption.

Page 75: the proof of Lemma 5.10.

Page 78: the proof of Theorem 5.16.

Pages 80-83: the discussion of the construction of an algebraic regular
neighbourhood of a family of almost invariant sets over groups which need not
be finitely generated.

Page 89: the proof of Theorem 6.6.

\bigskip

\subsection{Details on the remaining errors\label{detailsonremainingerrors}}

\bigskip

\textbf{Page 132:} Here is a counterexample to Theorems 10.8, 12.6, 13.15 and
14.11 of \cite{SS2}.

\begin{example}
Let $K$ denote the Baumslag-Solitar group $BS(1,2)=\left\langle a,t:tat^{-1}%
=a^{2}\right\rangle $. Thus $K$ has a natural expression as a HNN extension
with infinite cyclic vertex group generated by $a$. It follows that $K$ is
finitely presented and torsion free. The map from $K$ to the integers
$\mathbb{Z}$ given by killing $a$ has kernel $C$ which is isomorphic to the
additive group of the dyadic rationals $\mathbb{Z}[\frac{1}{2}]$. If we choose
the isomorphism so that $a$ corresponds to $1$, then $t^{-1}at$ corresponds to
$\frac{1}{2}$ and in general $t^{-n}at^{n}$ corresponds to $\frac{1}{2^{n}}$.
We let $A$ denote the cyclic subgroup of $C$ generated by $a$ and let $A_{n}$
denote the cyclic subgroup of $C$ generated by $t^{-n}at^{n}$. Thus $A=A_{0}$
and the union of all the $A_{n}$'s equals $C$. It is easy to see that $K$
admits no free product splitting. As $K$ is torsion free, it follows that $K$
admits no splitting over any finite subgroup and so is one-ended. Next we let
$G$ denote the group $K\times\mathbb{Z}$. Then $G$ admits no splitting over
any $VPC\left(  \leq1\right)  $ subgroup. Let $H$ be an isomorphic copy of
$G$, let $D$, $B$ and $B_{n}$ denote the subgroups of $H$ which correspond to
$C$, $A$ and $A_{n}$ respectively, and define $\overline{G}$ to be
$G\ast_{A=B}H$, the double of $G$ along $A$. Note that $\overline{G}$ is also
finitely presented, and it is easy to see that $\overline{G}$ is also
one-ended. Thus the regular neighbourhood $\Gamma_{1}(\overline{G})$ exists.
Note that $\overline{G}$ commensurises $A$ so that $A$ has large commensuriser
in $\overline{G}$.

Now $\overline{G}$ admits splittings $\sigma_{i}$ over $A_{i}$, and $\tau_{i}$
over $B_{i}$, which can be described in the following simple way. We define
$\sigma_{i}$ by writing%
\[
\overline{G}=G\ast_{A}H=(G\ast_{A_{i}}A_{i})\ast_{A}H=G\ast_{A_{i}}(A_{i}%
\ast_{A}H),
\]
and define $\tau_{i}$ by writing%
\[
\overline{G}=G\ast_{B}H=G\ast_{B}(B_{i}\ast_{B_{i}}H)=(G\ast_{B}B_{i}%
)\ast_{B_{i}}H.
\]
Note that $\sigma_{0}=\tau_{0}$. It is easy to see that all these splittings
are compatible. In fact the definitions already show that $\sigma_{i}$ and
$\sigma_{0}$ are compatible and that $\tau_{i}$ and $\tau_{0}$ are compatible.
The lemma below states that these are the only splittings of $\overline{G}$
over a two-ended subgroup, up to conjugacy. It follows that the family
$\mathcal{S}_{1}$ of all splittings of $\overline{G}$ over two-ended subgroups
contains an infinite collection of isolated splittings of $\overline{G}$,
which shows that $\mathcal{S}_{1}$ cannot have a regular neighbourhood.
\end{example}

Thus Theorems 10.8, 12.6, 13.15 and 14.11 are all false.

\begin{lemma}
Any splitting of $\overline{G}$ over a two-ended subgroup is conjugate to some
$\sigma_{i}$ or $\tau_{j}$.
\end{lemma}

\begin{proof}
Suppose that $\overline{G}$ has a splitting $\sigma$ over a two-ended
subgroup. Let $T$ be the corresponding $\overline{G}$--tree, so that
$\overline{G}\backslash T$ has a single edge. Recall that $G$ admits no
splitting over any $VPC(\leq1)$ subgroup. It follows that the subgroups
$G\ $and $H$ of $\overline{G}$ must each fix a vertex of $T$. These vertices
must be distinct as otherwise $\overline{G}$ itself would fix a vertex of $T$.
Also note that $G$ and $H$ can each fix only a single vertex of $T$, as
otherwise $G$ or $H$ would fix some edge of $T$ and hence itself be
$VPC(\leq1)$, which is not the case. Let $v$ and $w$ denote the vertices fixed
by $G\ $and $H$ respectively, and let $\lambda$ denote the edge path in $T$
which joins them. As $G$ and $H$ together generate $\overline{G}$, it follows
that $\sigma$ is an amalgamated free product. Further, by considering
canonical forms of elements, it is easy to see that $\lambda$ must consist of
a single edge. Thus $\overline{G}=P\ast_{R}Q$, where $P$, $Q$ and $R$ denote
the stabilisers of $v$, $w$ and $\lambda$ respectively. Note that this
splitting of $\overline{G}$ is conjugate to $\sigma$. As $G\subset P$ and
$H\subset Q$, it follows that $G\cap H=A\subset R$. As $R$ and $A$ are each
two-ended, it follows that $R$ contains $A$ with finite index. In particular
we have the well known fact that $R$ must be conjugate into $G$ or into $H$.
Now consider the subgroup $GR$ of $P$ generated by $G$ and $R$. If $R$ is
contained in $G$, this subgroup equals $G$. Otherwise the fact that $R$ lies
in some conjugate of $G$ or $H$ implies that $GR=G\ast_{A}R$. Similarly the
subgroup $HR$ of $Q$ generated by $H$ and $R$ is equal either to $H$ or to
$R\ast_{A}H$. Now the inclusions $GR\subset P$ and $HR\subset Q$ induce a
natural injection of $GR\ast_{R}HR$ into $\overline{G}=P\ast_{R}Q$. As $G$ and
$H$ together generate $\overline{G}$, this injection must also be surjective
and hence an isomorphism. It follows that $GR=P$ and $HR=Q$. If $R$ lies in
$G$ and in $H$, the fact that $R$ contains $A$ implies that $R=A$. In this
case $GR=G$ and $HR=H$, and the splitting $\sigma$ is conjugate to $\sigma
_{0}$. If $R$ lies in $G$, the facts that $G$ is isomorphic to $K\times
\mathbb{Z}$ and $R$ contains $A$ with finite index imply that $R$ must equal
$A_{n}$, for some $n\geq0$. In this case $GR=G$ and $HR=A_{n}\ast_{A}H$, and
the splitting $\sigma$ is conjugate to $\sigma_{n}$. Similarly if $R$ lies in
$H$, then $R$ must equal $B_{n}$, for some $n\geq0$, and in this case the
splitting $\sigma$ is conjugate to $\tau_{n}$. This leaves the case when $R$
is not contained in $G$ or in $H$. In this case we have an isomorphism between
$G\ast_{A}H$ and $(G\ast_{A}R)\ast_{R}(R\ast_{A}H)=G\ast_{A}R\ast_{A}H$. But
this is impossible as we know that $R$ is conjugate into $G$ or $H$. The
result follows.
\end{proof}

\bigskip

\textbf{Page 161:} We will start by showing that Example 14.1 of \cite{SS2} is
wrong. This was supposed to be an example of a one-ended group $G$ with
incommensurable polycyclic subgroups $H$ and $K$ of length $3$, and
$2$--canonical almost invariant sets $X$ and $Y$ over $H$ and $K$ respectively
which cross weakly. The sets $X$ and $Y$ described in the example do cross
weakly but they are not $2$--canonical, as we will now show.

Recall that $G$ is constructed by amalgamating $H$, $K$ and a third group $L$
along a certain infinite cyclic subgroup $C$. Thus $G$ has subgroups
$H\ast_{C}K$, $H\ast_{C}L$, $K\ast_{C}L$. Further $G$ can be expressed as the
amalgamated free product of the first two groups over $H$, and $X$ is an
$H$--almost invariant subset of $G$ associated to this splitting. Similarly,
the first and third groups give an amalgamated free product decomposition of
$G$ over $K$, and $Y$ is associated to this splitting. Recall that $\Gamma$
denotes the graph of groups structure for $G$ which is a tree with four
vertices carrying the subgroups $C$, $H$, $K$ and $L$ such that each edge
group is $C$. We will label the vertices $c$, $h$, $k$ and $l$
correspondingly. Note that all three edges are incident to $c$.

The error in our argument occurs where we claim that if $W$ is a nontrivial
almost invariant subset of $G$ over a two-ended subgroup $A$, then $W$ is
enclosed by the vertex $l$ of $\Gamma$ which carries $L$. The two edge
splittings given by the edges $ch$ and $ck$ of $\Gamma$ are clearly not
enclosed by $l$, and so the associated $C$--almost invariant subsets of $G$
are also not enclosed by $l$.

Now we will show that $X\ $and $Y$ are not $2$--canonical. In fact, they are
not $1$--canonical. Recall that $X$ is associated to the splitting $\sigma$ of
$G$ as $(H\ast_{C}K)\ast_{H}(H\ast_{C}L)$. We claim that this splitting is not
compatible with the splitting $\tau$ of $G$ over $C$ as $H\ast_{C}(K\ast
_{C}L)$. This means that $X$ crosses the $C$--almost invariant subset of $G$
associated to $\tau$, so that $X$ is not $1$--canonical. A similar argument
shows that $Y$ is also not $1$--canonical. To prove our claim suppose that
$\sigma$ and $\tau$ are compatible. This implies that $G$ is the fundamental
group of a graph $\Gamma^{\prime}$ of groups which has two edges such that the
associated edge splittings are $\sigma$ and $\tau$. As neither splitting is
HNN, $\Gamma^{\prime}$ must be homeomorphic to an interval with vertices $P$
and $R$ at the endpoints and one interior vertex $Q$. Choose notation so that
$\sigma$ is the edge splitting associated to $PQ$, and that $\tau$ is the edge
splitting associated to $QR$. Thus the edge group associated to $QR$ is $C$.
Also the group $G_{P}$ must be $H\ast_{C}K$ or $H\ast_{C}L$, and the group
$G_{R}$ must be $H$ or $K\ast_{C}L$. If $G_{R}$ is $H$, the fact that in
either case $G_{P}$ also contains $H$ implies that $H$ is contained in each
edge group of $\Gamma^{\prime}$. But this is impossible as $C$ cannot contain
$H$, as $C$ is cyclic and $H$ is not. It follows that $G_{R}$ must be
$K\ast_{C}L$. But $G_{P}$ contains one of $K$ or $L$, which implies that $K$
or $L$ is contained in $C$ which is again impossible. This completes the proof
of the claim.

Now we give a new example which demonstrates the phenomenon which Example 14.1
was supposed to demonstrate. Thus it is still correct to say that the
construction of $\Gamma_{1,2,\ldots,n}$ does not work if one considers almost
invariant sets over virtually polycyclic subgroups.

\begin{example}
\label{exampletoshowGamma123doesnotexist}Here is an example of a one-ended,
finitely presented group $G$ which has $2$--canonical splittings $\sigma_{1}$
and $\sigma_{2}$ over $VPC3$ subgroups $C_{1}$ and $C_{2}$ respectively, such
that $\sigma_{1}$ and $\sigma_{2}$ cross weakly and $C_{1}$ and $C_{2}$ are
not commensurable.

Let $B$ denote the free abelian group of rank $2$, let $\theta$ be a
hyperbolic automorphism of $B$ (i.e. $\theta$ has two real eigenvalues with
absolute value not equal to $1$), and let $C$ be the extension of $B$ by
${\mathbb{Z}}$ determined by $\theta$. Thus $B$ is normal in $C$ with quotient
isomorphic to ${\mathbb{Z}}$. Fix an element $\alpha$ of $C$ such that
conjugation of $B$ by $\alpha$ induces the automorphism $\theta$, and let $A$
denote the infinite cyclic subgroup of $C$ generated by $\alpha$. Thus
$A\ $projects onto the infinite cyclic quotient of $C$ by $B$. A crucial
property of $C$, which follows from the hyperbolicity of $\theta$, is that any
$VPC2$ subgroup of $C$ must be a subgroup of $B$. Let $nA$ denote the subgroup
of $A$ of index $n$. As $A$ is a maximal cyclic subgroup of $C$, it follows
that $Z_{C}(nA)=A$, where $Z_{C}(X)$ denotes the centraliser in $C$ of $X$.
Note that $C$ is the fundamental group of a closed $3$--manifold $M$ which is
a torus bundle over the circle.

For any integer $i$, let $B_{i}$ denote a copy of $B$, let $\theta_{i}$ denote
the automorphism of $B_{i}$ which corresponds to $\theta$, let $C_{i}$ denote
the corresponding copy of $C$, and let $\alpha_{i}$ denote the corresponding
element of $C_{i}$. For any pair $i,j$ of distinct integers, let $C_{ij}$
denote the extension of $B_{i}\oplus B_{j}$ by ${\mathbb{Z}}$, given by the
automorphism $\theta_{i}\oplus\theta_{j}$. Fix an element $\alpha_{ij}$ of
$C_{ij}$ such that conjugation of $B_{i}\oplus B_{j}$ by $\alpha_{ij}$ induces
the automorphism $\theta_{i}\oplus\theta_{j}$, and let $A_{ij}$ denote the
infinite cyclic subgroup of $C_{ij}$ generated by $\alpha_{ij}$. Thus there is
a natural inclusion of $C_{i}$ into $C_{ij}$ which sends $\alpha_{i}$ to
$\alpha_{ij}$, and hence sends $A_{i}$ to $A_{ij}$.

We note the following facts about the $C_{ij}$'s. Each $C_{ij}$ is one-ended
and torsion free, and any $VPC2$ subgroup of $C_{ij}$ must be a subgroup of
$B_{i}\oplus B_{j}$. It follows that each $C_{ij}$ has no nontrivial almost
invariant subsets over any $VPC$ subgroup of length $\leq2$. It is easy to see
that $Z_{C_{ij}}(nA_{ij})=A_{ij}$.

We construct a group $H=C_{13}\ast_{C_{1}}C_{01}\ast_{C_{0}}C_{02}\ast_{C_{2}%
}C_{24}$, where the inclusions of the edge groups into the vertex groups are
the natural ones. Thus we have in $H$ the equations $\alpha_{13}=\alpha
_{1}=\alpha_{01}=\alpha_{0}=\alpha_{02}=\alpha_{2}=\alpha_{24}$, and we abuse
notation and denote this element by $\alpha$ and the cyclic subgroup of $H$ it
generates by $A$. The natural homomorphisms $C_{ij}\rightarrow\mathbb{Z}$ all
fit together to yield a homomorphism $H\rightarrow\mathbb{Z}$ which maps $A$
onto $\mathbb{Z}$. We also choose a group $D$ which is the fundamental group
of a closed hyperbolic $3$--manifold and has $H_{1}(D)\cong\mathbb{Z}$. Then
$D$ is one-ended and torsion free. It has no nontrivial almost invariant
subsets over any $VPC1$ subgroup and contains no $VPC2$ subgroups. Thus $D$
has no nontrivial almost invariant subsets over any $VPC$ subgroup of length
$\leq2$. We define $G=H\ast_{A}D$, where $A$ is identified with a cyclic
subgroup of $D$ which maps onto $H_{1}(D)$. Thus the homomorphism
$H\rightarrow\mathbb{Z}$ extends to one from $G$ to $\mathbb{Z}$ which maps
$A$ onto $\mathbb{Z}$. As $H$ and $D$ are both torsion free, so is $G$. Note
that as $A$ is a maximal cyclic subgroup of $D$, it follows that $Z_{D}(A)=A$.

The group $H$ is the fundamental group of a graph $\Delta$ of groups with
underlying graph an interval divided into three edges. The edge groups are
$C_{1}$, $C_{0}$ and $C_{2}$, and we denote the corresponding edges by $e_{1}%
$, $e_{0}$ and $e_{2}$ respectively. The vertex groups are $C_{13}$, $C_{01}$,
$C_{02}$ and $C_{24}$, and we denote the corresponding vertices by $v_{13}$,
$v_{01}$, $v_{02}$ and $v_{24}$ respectively. The group $G$ is the fundamental
group of a graph $\Gamma$ of groups obtained from $\Delta$ by adding one edge
$e$ with associated group $A$. One end of $e$ is at $v_{13}$ and the other end
has associated group $D$.

We take the splitting $\sigma_{1}$ of $G$ to be the edge splitting of $\Gamma$
associated to the edge $e_{1}$. Thus $\sigma_{1}$ is a splitting of $G$ of the
form%
\[
G=\left\langle C_{13},D\right\rangle \ast_{C_{1}}\left\langle C_{01}%
,C_{02},C_{24}\right\rangle .
\]
To describe the splitting $\sigma_{2}$ of $G$, we first slide the edge $e$ so
as to be attached to $v_{24}$ instead of $v_{13}$. Now $\sigma_{2}$ is the
edge splitting of this new graph of groups which is associated to the edge
$e_{2}$. Thus $\sigma_{2}$ is a splitting of $G$ of the form%
\[
G=\left\langle C_{13},C_{01},C_{02}\right\rangle \ast_{C_{2}}\left\langle
C_{24},D\right\rangle .
\]
It is easy to check that $\sigma_{1}$ and $\sigma_{2}$ are not compatible. If
$\sigma_{1}$ and $\sigma_{2}$ cross strongly, then some conjugate of $C_{1}$
must intersect $C_{2}$ in a $VPC2$ subgroup. But our construction of $G$ shows
that conjugates of $C_{1}$ and $C_{2}$ intersect trivially or in a conjugate
of $A$. It follows that $\sigma_{1}$ and $\sigma_{2}$ must cross weakly. Note
also that $C_{1}$ and $C_{2}$ are not commensurable subgroups of $G$.

It remains to prove that the splittings $\sigma_{1}$ and $\sigma_{2}$ are
$2$--canonical. We will show that $G$ has no nontrivial almost invariant
subsets over any $VPC2$ subgroup, and that it has only one (up to equivalence,
complementation and translation) nontrivial almost invariant subset over a
$VPC1$ subgroup, which arises from the splitting $\tau$ of $G$ over $A$
associated to the edge $e$ of $\Gamma$. Thus $\tau$ is the splitting of $G$ as
$G=H\ast_{A}D$. As each of $\sigma_{1}$ and $\sigma_{2}$ is clearly compatible
with $\tau$, it follows that they do not cross the associated almost invariant
subset over $A$, and so are $2$--canonical, as required.
\end{example}

\label{exampleofsplittings}Before getting down to the details of the proof, it
is worth discussing how to make sense of this phenomenon, as it is actually a
very common one. As $G$ has only one nontrivial almost invariant subset over a
$VPC1$ subgroup, which arises from the splitting $\tau$ of $G$ over $A$, it
follows that $\Gamma_{1}(G)$ has a single $V_{0}$--vertex carrying $A$ and two
$V_{1}$--vertices carrying $H$ and $D$, respectively. The construction of the
splittings $\sigma_{1}$ and $\sigma_{2}$ of $G$ shows that they induce
compatible splittings of $H$. If one starts with these compatible splittings
of $H$, one can ask how to extend them to splittings of $G$. The answer is
that each has at least two natural extensions determined simply by the choice
of vertex $v_{13}$, $v_{01}$, $v_{02}$ or $v_{24}$ to which the edge $e$ is
attached. Each choice of vertex yields extensions which are compatible, so the
problem with this example is that in some sense, we have chosen the "wrong"
extensions. This is an example of a general phenomenon.

First we prove the following technical result.

\begin{lemma}
\label{technicalresult}Let $K$ be a subgroup of $G$, and suppose that, for any
conjugate $L$ of $K$ in $G$,

\begin{enumerate}
\item the number $e(C_{ij},L\cap C_{ij})=1$, for $ij=13$, $01$, $02$ or $24$, and

\item the number $e(D,L\cap D)=1$, and

\item $L\cap C_{i}$ has infinite index in $C_{i}$, for $i=0$, $1$ or $2$, and

\item $L\cap A$ has infinite index in $A$.
\end{enumerate}

Then $e(G,K)=1$.
\end{lemma}

\begin{proof}
Corresponding to the graph of groups decomposition $\Gamma$ of $G$, there is a
graph of spaces decomposition of a space $X$ with fundamental group $G$. The
covering space $X_{K}$ of $X$ with fundamental group $K$ is a graph of spaces,
where each vertex space has fundamental group equal to the intersection of $K$
with a conjugate of some $C_{ij}$ or of $D$, and each edge space has
fundamental group equal to the intersection of $K$ with a conjugate of $C_{0}%
$, $C_{1}$, $C_{2}$ or $A$. The hypotheses imply that each vertex space has
$1$ end and that each edge space is non-compact. It follows easily that
$X_{K}$ has $1$ end, so that $e(G,K)=1$ as required.
\end{proof}

Now we summarise the main properties of the group $G$ in Example
\ref{exampletoshowGamma123doesnotexist}.

\begin{lemma}
\label{propertiesofG}The group $G$ has the following properties.

\begin{enumerate}
\item $G$ has $1$ end.

\item If $K$ is a $VPC1$ subgroup of $G$, then either $e(G,K)=1$ or $K$ is
conjugate commensurable with $A$.

\item $Comm_{G}(A)=A$.

\item If $K$ is a $VPC2$ subgroup of $G$, then $e(G,K)=1$.
\end{enumerate}
\end{lemma}

\begin{proof}
1) Each $C_{ij}$ is one-ended, as is $D$, and each of $C_{1}$, $C_{2}$ and $A$
is infinite. Now we apply Lemma \ref{technicalresult}, with $K$ equal to the
trivial group, to see that $e(G)=1$, as required.

2) Let $K$ be a $VPC1$ subgroup of $G$. Thus the intersection of $K$ with any
subgroup of $G$ is $VPC(\leq1)$. Now we apply Lemma \ref{technicalresult}. It
follows that conditions 1)-3) of that lemma are satisfied. Hence $e(G,K)=1$
unless condition 4) fails. This would mean that there is a conjugate $L$ of
$K$ such that $L\cap A$ has finite index in $A$. As $K$ and $A$ are both
$VPC1$, it follows that $e(G,K)=1$ unless $K$ is conjugate commensurable with
$A$.

3) Let $g$ denote an element of $Comm_{G}(A)$. As there is a homomorphism of
$G$ to $\mathbb{Z}$ which maps $A$ onto $\mathbb{Z}$, it follows that $g$ must
lie in $Z_{G}(kA)$, the centraliser of $kA$ in $G$, for some non-zero $k$. Now
recall that $G$ is the fundamental group of the graph $\Gamma$ of groups. As
$A$ is contained in each vertex group, and the centraliser in each vertex
group of $kA$ is equal to $A$, it follows that $Z_{G}(kA)=A$. Hence
$Comm_{G}(A)=A$, as required.

4) Let $K$ be a $VPC2$ subgroup of $G$. Thus the intersection of $K$ with any
subgroup of $G$ is $VPC(\leq2)$. Now we apply Lemma \ref{technicalresult}. It
follows that condition 1) and 3) of that lemma hold. As the intersection of
any conjugate of $K$ with $D$ must be $VPC(\leq1)$, it also follows that
condition 2) holds. Hence $e(G,K)=1$ unless condition 4) fails. This would
mean that there is a conjugate $L$ of $K$ such that $L\cap A$ has finite index
in $A$. As any $VPC2$ group$\ $has a subgroup of finite index isomorphic to
${\mathbb{Z}}\times{\mathbb{Z}}$, it follows that any $VPC1$ subgroup of $K$
has large commensuriser in $K$. Thus this would imply that $A$ has large
commensuriser in $G$. Now part 3) of this lemma shows that this also is
impossible. It follows that if $K$ is a $VPC2$ subgroup of $G$, then
$e(G,K)=1$, as required.
\end{proof}

Next we consider almost invariant subsets of $G$ which are over a $VPC1$ subgroup.

\begin{lemma}
$G$ has only one (up to equivalence, complementation and translation)
nontrivial almost invariant subset over a $VPC1$ subgroup, which arises from
the splitting $\tau$ of $G$ over $A$ as $G=H*_{A}D$.
\end{lemma}

\begin{proof}
Suppose that $G$ has a nontrivial almost invariant subset over a $VPC1$
subgroup $K$, so that $e(G,K)>1$. Recall from part 2) of Lemma
\ref{propertiesofG} that if $K$ is a $VPC1$ subgroup of $G$, then either
$e(G,K)=1$ or $K$ is conjugate commensurable with $A$. As $Comm_{G}(A)=A$, by
part 3) of Lemma \ref{propertiesofG}, this implies that $K$ is conjugate to a
subgroup of $A$.

Now suppose that $K$ is a subgroup of $A$ and, as in the proof of Lemma
\ref{technicalresult}, consider the cover $X_{K}$ of $X$ with fundamental
group $K$. Part 3) of Lemma \ref{propertiesofG} tells us that $Comm_{G}(A)=A$,
so that exactly one edge space of $X_{K}$ is compact. Thus, as in the proof of
Lemma \ref{technicalresult}, each of the two complementary components of this
edge space has one end. It follows that $e(G,K)=2$, for any subgroup $K$ of
finite index in $A$. Hence $G$ has only one (up to equivalence,
complementation and translation) nontrivial almost invariant subset over a
$VPC1$ subgroup. As the splitting $\tau$ of $G$ over $A$ has such an almost
invariant subset of $G$ associated, the result follows.
\end{proof}

This completes our discussion of errors in \cite{SS2}.

\end{document}